\documentclass[12pt,letterpaper]{report}
\usepackage{amsfonts,amsmath, amssymb,amsthm,amscd}
\usepackage[usenames,dvipsnames]{color}
\usepackage{bm}
\usepackage{mathrsfs}
\usepackage{yfonts}
\usepackage[mpexclude,DIV13]{typearea}
\usepackage{verbatim}
\usepackage{hyperref}
\usepackage{graphicx}
\usepackage[latin1]{inputenc}
\usepackage{latexsym}
\usepackage{lscape}
\usepackage{epsfig}
\usepackage{subfigure}
\usepackage{comment}
\usepackage{url}
\usepackage{setspace}
\usepackage{makeidx}
\makeindex
\include{bibtex}

\makeatletter
\def\printnotation{{%
\def\indexname{Index of notation}
\begin{theindex}
\@input{\jobname.ntn}
\end{theindex}
}}
\makeatother

\makeglossary

\begin{document}

\bibliographystyle{alpha}
\newcommand{\cn}[1]{\overline{#1}}
\newcommand{\e}[0]{\epsilon}
\newcommand{\EE}{\ensuremath{\mathbb{E}}}
\newcommand{\qq}[1]{(q;q)_{#1}}
\newcommand{\A}{\ensuremath{\mathcal{A}}}
\newcommand{\GT}{\ensuremath{\mathbb{GT}}}
\newcommand{\link}{\ensuremath{Q}}
\newcommand{\PP}{\ensuremath{\mathbb{P}}}
\newcommand{\frakP}{\ensuremath{\mathfrak{P}}}
\newcommand{\frakQ}{\ensuremath{\mathfrak{Q}}}
\newcommand{\frakq}{\ensuremath{\mathfrak{q}}}
\newcommand{\R}{\ensuremath{\mathbb{R}}}
\newcommand{\Rplus}{\ensuremath{\mathbb{R}_{+}}}
\newcommand{\C}{\ensuremath{\mathbb{C}}}
\newcommand{\Z}{\ensuremath{\mathbb{Z}}}
\newcommand{\Weyl}[1]{\ensuremath{\mathbb{W}}^{#1}}
\newcommand{\Zgzero}{\ensuremath{\mathbb{Z}_{>0}}}
\newcommand{\Zgeqzero}{\ensuremath{\mathbb{Z}_{\geq 0}}}
\newcommand{\Zleqzero}{\ensuremath{\mathbb{Z}_{\leq 0}}}
\newcommand{\Q}{\ensuremath{\mathbb{Q}}}
\newcommand{\T}{\ensuremath{\mathbb{T}}}
\newcommand{\Y}{\ensuremath{\mathbb{Y}}}
\newcommand{\M}{\ensuremath{\mathbf{M}}}
\newcommand{\MM}{\ensuremath{\mathbf{MM}}}
\newcommand{\W}[1]{\ensuremath{\mathbf{W}}_{(#1)}}
\newcommand{\WM}[1]{\ensuremath{\mathbf{WM}}_{(#1)}}
\newcommand{\Zsd}{\ensuremath{\mathbf{Z}}}
\newcommand{\Fsd}{\ensuremath{\mathbf{F}}}
\newcommand{\Gsd}{\ensuremath{\mathbf{G}}}
\newcommand{\symBM}{\ensuremath{\mathbf{W}}}
\newcommand{\symFE}{\ensuremath{\mathbf{S}}}
\newcommand{\qWhitP}{\ensuremath{\Upsilon}}
\newcommand{\qWhitQ}{\ensuremath{\Psi}}
\newcommand{\Real}{\ensuremath{\mathrm{Re}}}
\newcommand{\Imag}{\ensuremath{\mathrm{Im}}}
\newcommand{\re}{\ensuremath{\mathrm{Re}}}

\newcommand{\bfk}{\ensuremath{\bar{f}_{\kappa}}}
\newcommand{\btk}{\ensuremath{\bar{t}_{\kappa}}}
\newcommand{\bgk}{\ensuremath{\bar{g}_{\kappa}}}

\newcommand{\Sym}{\ensuremath{\mathrm{Sym}}}

\newcommand{\bfone}{\ensuremath{\mathbf{1}}}

\newcommand{\whitenoise}{\ensuremath{\mathscr{\dot{W}}}}
\newcommand{\alphaW}[1]{\ensuremath{\mathbf{\alpha W}}_{(#1)}}
\newcommand{\alphaWM}[1]{\ensuremath{\mathbf{\alpha WM}}_{(#1)}}
\newcommand{\malpha}{\ensuremath{\hat{\alpha}}}
\newcommand{\walpha}{\ensuremath{\alpha}}
\newcommand{\edge}{\textrm{edge}}
\newcommand{\dist}{\textrm{dist}}

\newcommand{\OO}[0]{\Omega}
\newcommand{\F}[0]{\mathfrak{F}}
\newcommand{\poly}[0]{R}
\def \Ai {{\rm Ai}}
\def \SS {\mathcal{S}}
\newcommand{\poles}{\mathbb{A}}
\def \ss {\mathcal{X}}
\newcommand{\var}{{\rm var}}
\newcommand\Pker[2]{P^{#1}_{#2}}  %%P kernel
\newcommand\thetar{\hat\theta}  %%theta row - changes in time
\newcommand\thetac{\theta}  %%theta column
\newcommand\thetacv{\theta}  %%theta column vector, the earlier \thetac_{1:N}
\newcommand\thetarc{\gamma}  %%theta row column -- both parameters.

\newcommand\BaxKer[5]{Q^{#1\to #2}_{#3}(#4,#5)}
\newcommand\BaxKerfour[4]{Q^{#1\to #2}_{#3}(#4)}
\newcommand\BaxOp[3]{Q^{#1\to #2}_{#3}}

\newcommand\DBaxKer[5]{\hat{Q}^{#1\to #2}_{#3}(#4,#5)}
\newcommand\DBaxOp[3]{\hat{Q}^{#1\to #2}_{#3}}

\newcommand{\Res}[1]{\underset{{#1}}{\mathrm{Res}}}
\newcommand{\Resfrac}[1]{\mathrm{Res}_{{#1}}}

\newcommand{\ul}[2]{\underline{#1}_{#2}}
\newcommand{\qhat}[1]{\widehat{#1}^{q}}
\newcommand{\G}[0]{\mathfrak{G}}
\newcommand{\La}[0]{\Lambda}
\newcommand{\la}[0]{\lambda}
\newcommand{\ta}[0]{\theta}
\newcommand{\w}[0]{\omega}
\newcommand{\ra}[0]{\rightarrow}
\newcommand{\vectoro}{\overline}
\newtheorem{theorem}{Theorem}[section]
\newtheorem{partialtheorem}{Partial Theorem}[section]
\newtheorem{conj}[theorem]{Conjecture}
\newtheorem{lemma}[theorem]{Lemma}
\newtheorem{proposition}[theorem]{Proposition}
\newtheorem{corollary}[theorem]{Corollary}
\newtheorem{claim}[theorem]{Claim}
\newtheorem{formal}[theorem]{Critical point derivation}
\newtheorem{experiment}[theorem]{Experimental Result}
\newtheorem{prop}{Proposition}

\def\todo#1{\marginpar{\raggedright\footnotesize #1}}
\def\change#1{{\color{green}\todo{change}#1}}
\def\note#1{\textup{\textsf{\color{blue}(#1)}}}

\theoremstyle{definition}
\newtheorem{remark}[theorem]{Remark}

\theoremstyle{definition}
\newtheorem{example}[theorem]{Example}

\theoremstyle{definition}
\newtheorem{definition}[theorem]{Definition}

\theoremstyle{definition}
\newtheorem{definitions}[theorem]{Definitions}

%\title{Macdonald processes\newline \large{Alexei Borodin, Ivan Corwin}}

\title{Macdonald processes}
 \author{Alexei Borodin, Ivan Corwin}
% \date{December 2004}
 \maketitle

%\author[A. Borodin]{Alexei Borodin}

%\author[I. Corwin]{Ivan Corwin}
%\address{I. Corwin\\
%  Courant Institute of Mathematical Sciences\\
%  New York University\\
%  251 Mercer Street\\
%  New York, NY 10012, USA}
%\email{corwin@cims.nyu.edu}

%\date{October 24, 2011}

\begin{abstract}
%Macdonald processes are measures on Gelfand-Tsetlin (GT) patterns specified in terms of nonnegative specializations of Macdonald symmetric functions as well as q,t in (0,1). Implications of our study include: (1) We compute expectations of a rich family of observables for these processes. (2) When t=0 we find a Fredholm determinant formula for a certain $q$-Laplace transform. (3) We introduce dynamics on GT patterns which preserves the class of Macdonald processes and leads to a new ``integrable'' 2d interacting particle systems with a 1d particle system marginal called $q$-TASEP. (4) As q->0 these particles crystallize on a lattice and fluctuations around this limit converge to Whittaker processes introduced by O'Connell in studying directed polymers. (5) We provide a Fredholm determinant for the Laplace transform of the O'Connell-Yor polymer partition function and taking its asymptotics prove KPZ universality for the polymer (free energy fluctuation exponent 1/3 and Tracy-Widom GUE limit law) (6) Under intermediate disorder scaling we recover the Laplace transform for the KPZ equation with narrow wedge initial data. (7) We provide contour integral formulas for a wide array of polymer moments. (8) We provide a new ansatz for solving quantum many body systems such as the delta Bose gas.

Macdonald processes are probability measures on sequences of partitions defined in terms of nonnegative specializations of the Macdonald symmetric functions and two Macdonald parameters $q,t \in [0,1)$. We prove several results about these processes, which include the following.

(1) We explicitly evaluate expectations of a rich family of observables for these processes. (2) In the case $t=0$, we find a Fredholm determinant formula for a $q$-Laplace transform of the distribution of the last part of the Macdonald-random partition. (3) We introduce Markov dynamics that preserve the class of Macdonald processes and lead to new ``integrable'' 2d and 1d interacting particle systems. (4) In a large time limit transition, and as $q$ goes to 1, the particles of these systems crystallize on a lattice, and fluctuations around the lattice converge to O'Connell's Whittaker process that describe semi-discrete Brownian directed polymers. (5) This yields a Fredholm determinant for the Laplace transform of the polymer partition function, and taking its asymptotics we prove KPZ universality for the polymer (free energy fluctuation exponent $1/3$ and Tracy-Widom GUE limit law). (6) Under intermediate disorder scaling, we recover the Laplace transform of the solution of the KPZ equation with narrow wedge initial data. (7) We provide contour integral formulas for a wide array of polymer moments. (8) This results in a new ansatz for solving quantum many body systems such as the delta Bose gas.
\end{abstract}

\setcounter{tocdepth}{4}
\tableofcontents
\hypersetup{linktocpage}

%note {Things to change later}
%\begin{enumerate}
%\item The $a_i$ vs $A_i$ issue should be fixed.
%\item All of the different operators and kernels need some care to distinguish between.
%\item In the Macdonald measure formulas we just write $\lambda_N$ whereas $\lambda_N^N$ is more precise and is more like what we do in the $T_N^N$ case.
%\item Consider putting in a symbol index and regular index of where terms are introduced and defined.
%\item Open problems section?
%\item Introduce macros for many of the symbols which are used to avoid mistakes.
%\item change labeling to be like part - sect - subsec.
%\item Idea: based on arXiv 0907.2045 or 0812.2306. they provide different types of eigenfunctions at $q$-Whittaker level which are not the same as our $q$-Whittaker functions. Showing the relation between then seems to provide interesting $q$ combinatorial identitie.
%\item Idea: try to symmetrized the pre-asymptotic version of Bump-Stade identity.
%\end{enumerate}

\chapter{Introduction}
%There is a philosophical question of where our accents are. Originally I thought that computing asymptotics in Neil's model is the main goal. I'm not so sure anymore. The $q$-model has a TASEP interpretation, and one could emphasize that it's the asymptotics of $q$-TASEP that we are computing (this would require direct asymptotic transitions from $q$-case to crossover and TW).  We could try to stress that we are making the Bethe ansatz a la Dotsenko rigorous, although taking this path does not feel natural to me. There is also a statement that the limit of the continuous Brownian polymer is TW (Dotsenko's main claim), should we stress it? Our work gives some explanation of where TW formulas are coming from - another possibility for emphasizing. Finally, I'm a fan of 2d growth models, and Theorem 26 shows an amazing crystalline structure that arises in a local 2d growth model (law of large numbers level) - I've never seen anything like this before.

%Thoughts on focus: The big surprise is that we discover Fredholm determinant of a very particular type at a level where there is no principle behind its existence. This miracle extends down to positive temperature polymer models and in fact legitimizes the replica trick.

%\begin{enumerate}
%\item mention connection of macdonald functions to Calogero-Sutherland model \cite{AOS}
%\end{enumerate}
The principal goal of the present paper is to develop a formalism of Macdonald processes -- a class of probability measures on Gelfand-Tsetlin patterns -- and to apply it to studying interacting particle systems and directed random polymers.

Our inspiration comes from three sides: The well-known success of the determinantal point processes in general, and the Schur processes in particular, in analyzing totally asymmetric simple exclusion processes (TASEPs) and last passage percolation problems (in one space dimension); the recent breakthrough \cite{ACQ,SaSp} in finding exact solutions to the Kardar-Parisi-Zhang stochastic nonlinear PDE (KPZ equation, for short), based either on Tracy-Widom's integrability theory for the partially asymmetric simple exclusion process (PASEP) or on the replica trick coupled with Bethe ansatz solutions of the quantum delta Bose gas; and O'Connell's description of a certain random directed polymers in terms of eigenfunctions of the quantum Toda lattice known as class one Whittaker functions (for root systems of type A).

The present work ties these developments together.

The Macdonald processes are defined in terms of the Macdonald symmetric functions, a remarkable two-parameter family of symmetric functions discovered by Macdonald in the late 1980s \cite{M}. The parameters are traditionally denoted by $q,t$. On the diagonal $q=t$, the Macdonald symmetric functions turn into the Schur functions, and the Macdonald processes turn into the Schur processes of \cite{OkResh,OkWedge}. The key feature that makes the Schur processes very useful is the determinantal structure of their correlation functions. This feature is not present in the general Macdonald case, which is one reason why Macdonald processes have not been extensively studied yet (see, however, \cite{ForRains, Vuletic}). We find a way of utilizing the Macdonald operators -- a sequence of difference operators that are diagonalized by Macdonald symmetric functions -- for explicit evaluations of averages for a rich class of observables of Macdonald processes.

In the case $t=0$, these averages lead us to a Fredholm determinant representation of a $q$-Laplace transform of the distribution of the last part of a Macdonald-random partition.
The appearance of the Fredholm determinant seems remarkable. As to our best knowledge, the Macdonald symmetric functions (either general ones or those for $t=0$) do not carry natural determinantal structures (as opposed to the Schur functions).

In a different direction, we exhibit a family of Markov chains that preserve the class of Macdonald processes. The construction is parallel to that of \cite{BF, Bor-Sch-Dyn} in the Schur case, and it is based on a much earlier idea of \cite{DiaconisFill}. In the case $t=0$, the continuous time version of these stochastic dynamics can be viewed as an interacting particle system in two space dimensions (whose Schur degeneration was studied in \cite{BF}). A certain one dimensional Markovian subsystem has a particularly simple description -- it is a $q$-deformation of TASEP where a particle's jump rate equals $1-q^{gap}$, with `gap' being the number of open spaces to the right of the particle; we call it $q$-TASEP. The Macdonald formalism thus provides explicit expressions for averages of many observables of these interacting particle systems started from certain special initial conditions. In particular, we obtain a Fredholm determinant formula for the $q$-Laplace transform of the distribution of any particle in $q$-TASEP started from the so-called step initial condition.

We then focus on asymptotic behavior of certain specializations of the Macdonald processes. In what follows, the parameter $t$ is assumed to be 0.

We find that as $q\to 1$, in the so-called Plancherel specialization, the particles in the Macdonald-random Gelfand-Tsetlin pattern crystalize onto a perfect lattice (law of large
numbers behavior). The scaled fluctuations of particles around this limiting lattice converge to the Whittaker process introduced by O'Connell \cite{OCon} as the image of the O'Connell-Yor semi-discrete directed polymer under a continuous version of the tropical RSK correspondence. The process is supported on triangular arrays of real numbers. (This degeneration is parallel to the recently discovered degeneration of the Macdonald symmetric functions to the class one Whittaker function, see \cite{GLOqlim}.) The Markov dynamics on Macdonald processes converge to dynamics on these triangular arrays which are not the same as those induced by the tropical RSK correspondence, but which were also introduced in \cite{OCon} as a geometric version of Warren's process \cite{Warren}.
For a different choice of specialization, the fluctuations above converge to another version of the Whittaker process introduced in \cite{COSZ} as the image of Sepp\"{a}l\"{a}inen's log-gamma discrete directed polymer under the tropical RSK correspondence of A.N. Kirillov \cite{Kir}.

Under the scalings to Whittaker processes, the Fredholm determinant formulas converge to similar formulas for the Laplace transform of the partition functions for these polymers (hence characterizing the free energy fluctuations). As one application of these new Fredholm determinant formulas we prove (via steepest descent analysis) that as time $\tau$ goes to infinity, these polymers are in the KPZ
universality class, i.e., their free energy fluctuates like $\tau^{1/3}$ with limiting law given by the Tracy-Widom GUE distribution. Under a different scaling, known as intermediate disorder scaling, we show how to recover the Laplace transform for the partition function of the point-to-point continuum directed random polymer, i.e., the solution to the stochastic heat equation with delta initial data, or the exponential of the solution to the KPZ equation with narrow wedge initial data. Those familiar with the non-rigorous replica approach of \cite{Dot} or \cite{CDR} may observe that the unjustifiable analytic continuations and rearrangements necessary to sum the divergent moment series in those works (a.k.a. the replica trick), can be seen as shadows of the rigorous manipulations and calculations performed at the higher level of the Macdonald processes.

From the point of view of the KPZ equation, we observe that the $q$-TASEP may now be used as its `integrable discretization' in a similar way as PASEP has been used, and we expect to present more results in this direction in future works.

It has been known for a while that moments of exactly solvable polymer models solve certain quantum many body systems with delta interactions. This fact, together with Bethe ansatz computations, has been widely used in physics literature to analyze the polymers. The formalism of the Macdonald processes results in certain explicit integral formulas for those moments, that are thus solutions of the quantum many-body systems.

Those contour integral solutions provide a new and seemingly robust ansatz for solving such systems without appealing to the Bethe ansatz. As one application we provide new formulas for the solution to the delta Bose gas, with the difference between attractive and repulsive interactions only coming in as a sign change in the denominator of our contour integral integrand -- a striking contrast with the Bethe ansatz approach, where the two cases lead to very different sets of eigenfunctions. The range of potential applications of these new formulas remains to be investigated.

The introduction below provides a more detailed description of our results. We would like to note that on many occasions, because of space considerations, the introduction will only give a sample statement, and for the full scale of the result the reader is encouraged to look into the body of the paper (we provide exact references to relevant statements throughout the introduction).

\subsection{Outline}

This paper is split into six chapters:
\begin{itemize}
\item Chapter one is the introduction. We have attempted to make this fairly self-contained (at the cost of repeating certain definitions, discussions and results again in the main parts of the paper).
\item Chapter two contains the definition of Macdonald processes, the calculation of certain observables as contour integrals, and the development of dynamics which preserve these processes. Additionally, a review of properties of Macdonald symmetric functions is provided in this part for the readers' convenience.
\item Chapter three takes the Macdonald parameter $t=0$ (resulting in $q$-Whittaker processes) and contains the basic result that $q$-Laplace transforms of certain elements of these processes are Fredholm determinants. Specializing the $q$-Whittaker functions and focusing on a certain projection of the dynamics on the processes results in $q$-TASEP whose one-particle transition probabilities are characterized by a Fredholm determinant.
\item Chapter four includes the weak (in the sense of distribution) limit $q\to 1$ of $q$-Whittaker processes. Two $q$-Whittaker specializations (called Plancherel and pure alpha) are focused on. They lead to Whittaker processes which were previously discovered by O'Connell \cite{OCon} and Corwin, O'Connell, Sepp\"{a}l\"{a}inen and Zygouras \cite{COSZ}. Laplace transforms for certain elements of those are now given by Fredholm determinants, and exponential moments are given by contour integrals of a particularly simple form. Using the Fredholm determinant, a limit theorem is proved showing Tracy-Widom GUE fluctuations.
\item Chapter five uses the connection developed in \cite{OCon,COSZ} to relate the above mentioned formulas to the study of directed polymers in random media and via another limiting procedure, to the continuum directed random polymer and the KPZ equation.
\item Chapter six shows how the contour integral formulas developed here provide a new ansatz for solving a certain class of integrable quantum many body systems such as the delta Bose gas.
%These systems arise when finding moments for polymers and are the basis for the non-rigorous replica approach of statistical physics. In fact, one could say that this paper makes the replica trick rigorous by lifting it up to a higher algebraic level where $q$-moments identify distributions.
\end{itemize}

\subsection{Acknowledgements}

The authors would like to thank P.L. Ferrari, V. Gorin, A. Okounkov, G. Olshanski, J. Quastel, E. Rains, T. Sasamoto, T. Sepp\"{a}l\"{a}inen and H. Spohn for helpful discussions pertaining to this paper. We are particularly grateful to N. O'Connell for generously sharing his insights on this subject with us through this project. Additional thanks go out to the participants of the American Institute of Mathematics workshop on the ``Kardar-Parisi-Zhang equation and universality class'' where AB first spoke on these results. AB was partially supported by the NSF grant DMS-1056390. IC was partially supported by the NSF through PIRE grant OISE-07-30136 as well as by Microsoft Research through the Schramm Memorial Fellowship.
%MSR, PIRE, IMPA, AIM, NSF, Tomohiro, Patrik, O'Connell, Rains, Okounkov?, Quastel? Seppalainen? Aim participants (Bertini, Hairer, Dembo)?

\section{Macdonald processes}

The {\it ascending Macdonald process} $\M_{asc}(a_1,\dots,a_N;\rho)$ is the probability distribution on sequences of partitions
\begin{equation}\label{ascseqnintro}
\varnothing \prec \lambda^{(1)}\prec \lambda^{(2)}\prec\cdots \prec\lambda^{(N)}
\end{equation}
(equivalently Gelfand-Tsetlin patterns, or column-strict Young tableaux) indexed by positive variables $a_1,\ldots, a_N$ and a single Macdonald nonnegative specialization $\rho$ of the algebra of symmetric functions, with
\begin{equation*}
\M_{asc}(a_1,\dots,a_N;\rho)(\la^{(1)},\dots,\la^{(N)})= \frac{P_{\la^{(1)}}(a_1)P_{\la^{(2)}/\la^{(1)}}(a_2)\cdots P_{\la^{(N)}/\la^{(N-1)}}(a_N) Q_{\la^{(N)}}(\rho)}{\Pi(a_1,\dots,a_N;\rho)}\,.
\end{equation*}

Here by a partition $\lambda$ we mean an integer sequence $\lambda=(\lambda_1\geq \lambda_2\geq \ldots \geq 0)$ with finitely many nonzero entries $\ell(\lambda)$, and we say that $\mu \prec \lambda$ if the two partitions interlace: $\mu_i\leq \lambda_i \leq \mu_{i-1}$ for all meaningful $i$'s. In standard terminology, $\mu\prec\lambda$ is equivalent to saying that the skew partition $\lambda/\mu$ is a horizontal strip.

%A partition is a sequence $\lambda = (\lambda_1,\lambda_2,\ldots)$ of nonnegative integers such that $\lambda_1\geq~\lambda_2\geq~\cdots$. A partition $\lambda$ can be graphically represented as a Young diagram with $\lambda_1$ left justified boxes in the top row, $\lambda_2$ in the second row, and so on. Given two diagrams $\lambda$ and $\mu$ such that $\lambda\supset \mu$ (as a set of boxes), we call the set-difference $\theta = \lambda -\mu$ a {\it skew Young diagram}. A skew Young diagram $\theta$ is a horizontal strip in each column $\theta$ has at most one box. In this case we write $\mu \prec \lambda$. A column-strict Young tableaux is a sequence of partitions $\lambda^{(k)}=(\lambda^{(k)}_{1},\ldots, \lambda^{(k)}_k)$ for $1\leq k\leq N$ as in (\ref{ascseqnintro}). We will generally refer to such a triangular array also as a Gelfand-Tsetlin (GT) pattern.

The functions $P$ and $Q$ are Macdonald symmetric functions which are indexed by (skew) partitions and implicitly depend on two parameters $q,t\in (0,1)$. Their remarkable properties are developed in Macdonald's book \cite[Section VI]{M} and reviewed in Section~\ref{MacDefSEC} below. The evaluation of a Macdonald symmetric function on a positive variable $a$ (as in $P_{\lambda/\mu}(a)$) means to restrict the function to a single variable and then substitute the value $a$ in for that variable. This is a special case of a Macdonald nonnegative specialization $\rho$ which is an algebra homomorphism of the algebra of symmetric functions $\Sym$ to $\C$ that takes skew Macdonald symmetric functions to nonnegative real numbers (notation: $P_{\lambda/\mu}(\rho)\geq 0$ for any partitions $\lambda$ and $\mu$). Restricting the Macdonald symmetric functions to a finite number of variables (i.e., considering Macdonald polynomials) and then substituting nonnegative numbers for these variables constitutes such a specialization.

We will work with a more general class which can be thought of as unions of limits of such finite length specializations as well as limits of finite length {\it dual} specializations. Let $\{\alpha_i\}_{i\ge 1}$, $\{\beta_i\}_{i\ge 1}$, and $\gamma$ be nonnegative numbers, and $\sum_{i=1}^\infty(\alpha_i+\beta_i)<\infty$. Let $\rho$ be a specialization of $\Sym$ defined by
\begin{equation}\label{specdefeqn}
\sum_{n\ge 0} g_n(\rho) u^n= \exp(\gamma u) \prod_{i\ge 1} \frac{(t\alpha_iu;q)_\infty}{(\alpha_i u;q)_\infty}\,(1+\beta_i u)=: \Pi(u;\rho).
\end{equation}
Here $u$ is a formal variable and $g_n=Q_{(n)}$ is the $(q,t)$-analog of the complete homogeneous symmetric function $h_n$. Since $g_n$ forms a $\Q[q,t]$ algebraic basis of $\Sym$, this uniquely defines the specialization $\rho$. This $\rho$ is a Macdonald nonnegative specialization (see Section~\ref{NNspecSEC} for more details).

Finally, the normalization for the ascending Macdonald process is given by
\begin{equation*}
\Pi(a_1,\dots,a_N;\rho) = \prod_{i=1}^{N} \Pi(a_i;\rho),
\end{equation*}
as follows from a generalization of Cauchy's identity for Schur functions. It is not hard to see that the condition of the partition function $\Pi(a_1,\dots,a_N;\rho)$ being finite is equivalent to $a_i\alpha_j<1$ for all $i,j$, and hence we will always assume this holds.

The projection of $\M_{asc}$ to a single partition $\lambda^{(k)}$, $k=1,\dots,N$, is the {\it Macdonald measure}
\begin{equation*}
\MM(a_1,\ldots,a_k;\rho)(\la^{(k)})= \frac{P_{\lambda^{(k)}}(a_1,\ldots,a_k) Q_{\lambda^{(k)}}(\rho)} {\Pi(a_1,\ldots,a_k;\rho)}\,.
\end{equation*}

Macdonald processes (which receive a complete treatment in Section~\ref{MacProcessSec}) were first introduced and studied by Vuletic \cite{Vuletic} in relation to a generalization of MacMahon's formula for plane partitions (see also the earlier work of Forrester and Rains \cite{ForRains} in which a Macdonald measure is introduced). Setting $q=t$, both $P$ and $Q$ become Schur functions and the processes become the well-known Schur processes introduced in \cite{OkResh,OkWedge}. One may similarly degenerate the Macdonald symmetric functions to Hall-Littlewood and Jack functions (and similarly for processes). In what follows, we will focus on a less well-known degeneration to Whittaker functions.

\section{Computing Macdonald process observables}
Assume we have a linear operator $\mathcal{D}$ in the space of functions in $N$ variables whose restriction to the space of symmetric polynomials diagonalizes in the basis of Macdonald polynomials: $\mathcal{D} P_\la=d_\lambda P_\la$ for any partition $\la$ with $\ell(\la)\le N$. Then we can apply $\mathcal{D}$ to both sides of the identity
\begin{equation*}
\sum_{\la:\ell(\la)\le N} P_{\lambda}(a_1,\dots,a_N) Q_{\lambda}(\rho)=\Pi(a_1,\dots,a_N;\rho).
\end{equation*}
Dividing the result by $\Pi(a_1,\dots,a_N;\rho)$ we obtain
\begin{equation}\label{tag8intro}
\langle d_\lambda \rangle_{\MM(a_1,\dots,a_N;\rho)}=\frac{\mathcal{D}\Pi(a_1,\dots,a_N;\rho)} {\Pi(a_1,\dots,a_N;\rho)}\,,
\end{equation}
where $\langle \cdot\rangle_{\MM(a_1,\dots,a_N;\rho)}$ represents averaging $\cdot$ over the specified Macdonald measure.
If we apply $\mathcal{D}$ several times we obtain
\begin{equation*}
\langle d_\lambda^k \rangle_{\MM(a_1,\dots,a_N;\rho)}=\frac{\mathcal{D}^k \Pi(a_1,\dots,a_N;\rho)} {\Pi(a_1,\dots,a_N;\rho)}\,.
\end{equation*}
If we have several possibilities for $\mathcal{D}$ we can obtain formulas for averages of the observables equal to products of powers of the corresponding eigenvalues. One of the remarkable features of Macdonald polynomials is that there exists a large family of such operators for which they form the eigenbasis (and this fact can be used to define the polynomials). These are the Macdonald difference operators.

In what follows we fix the number of independent variables to be $N\in \Zgzero$. For any $u\in\R$ and $1\le i\le N$, define the {\it shift operator} $T_{u,x_i}$ by
\begin{equation*}
(T_{u,x_i}F)(x_1,\dots,x_N)=F(x_1,\dots,ux_i,\dots,x_N).
\end{equation*}
For any subset $I\subset\{1,\dots,N\}$, define
\begin{equation*}
A_I(x;t)=t^{\frac{r(r-1)}2}\prod_{i\in I,\,j\notin I}\frac{tx_i-x_j}{x_i-x_j}\,.
\end{equation*}

Finally, for any $r=1,2,\ldots,n$, define the {\it Macdonald difference operators}
\begin{equation*}
D_N^r=\sum_{\substack{I\subset\{1,\ldots,N\}\\|I|=r}} A_I(x;t)\prod_{i\in I} T_{q,x_i}.
\end{equation*}

\begin{proposition}\cite[VI,(4.15)]{M}\label{prop5intro}
For any partition $\lambda=(\lambda_1\geq \lambda_2\geq \cdots)$ with $\lambda_m=0$ for $m>N$
\begin{equation*}
D_N^r P_\lambda(x_1,\dots,x_N)=e_r(q^{\la_1}t^{N-1},q^{\la_2}t^{N-2},\dots,q^{\la_N}) P_\lambda(x_1,\dots,x_N).
\end{equation*}
Here $e_r$ is the elementary symmetric function, $e_r(x_1,\ldots,x_N) = \sum_{1\leq i_1<\cdots<i_r\leq N} x_{i_1}\cdots x_{i_r}$.
\end{proposition}
Although the operators $D_N^r$ do not look particularly simple, they can be represented by contour integrals by properly encoding the shifts in terms of residues (see Section~\ref{difopSEC} for proofs and more general results and discussion).
\begin{proposition}\label{prop6intro}
Assume that $F(u_1,\dots,u_N)=f(u_1)\cdots f(u_N)$. Take $x_1,\ldots,x_N>0$ and assume that $f(u)$ is holomorphic and nonzero in a complex neighborhood of an interval in $\R$ that contains $\{x_j,qx_j\}_{j=1}^N$. Then for $r=1,2,\dots,N$
\begin{equation}\label{tag9intro}
(D_N^r F)(x)= F(x)\cdot \frac {1}{(2\pi \iota)^r r!}\oint\cdots\oint
\det\left[\frac 1{tz_k-z_\ell}\right]_{k,\ell=1}^r
\prod_{j=1}^r  \left(\prod_{m=1}^N \frac{tz_j-x_m}{z_j-x_m}\right)
\frac{f(qz_j)}{f(z_j)}\,dz_j,
\end{equation}
where each of $r$ integrals is over positively oriented contour encircling $\{x_1,\ldots,x_N\}$ and no other singularities of the integrand.
\end{proposition}

Observe that $\Pi(a_1,\dots,a_N;\rho)=\prod_{i=1}^N \Pi(a_i;\rho)$. Hence, Proposition~\ref{prop6intro} is suitable for evaluating the right-hand side of equation (\ref{tag8intro}). Observe also that for any fixed $z_1,\dots,z_r$, the right-hand side of equation (\ref{tag9intro}) is a multiplicative function of $x_j$'s. This makes it possible to iterate the procedure. The simplest case is below (the fully general case can be found in Proposition~\ref{prop8FULL}).

\begin{proposition}\label{prop8intro}
Fix $k\ge 1$. Assume that $F(u_1,\dots,u_N)=f(u_1)\cdots f(u_N)$. Take $x_1,\dots,x_N >0$ and assume that $f(u)$ is holomorphic and nonzero in a complex neighborhood of an interval in $\R$ that contains $\{q^ix_j\mid i=0,\ldots,k,j=1\ldots,N\}$.
Then
\begin{equation*}
\frac{\bigl({(D_n^1)}^k F\bigr)(x)}{F(x)}=\frac {(t-1)^{-k}}{(2\pi \iota)^k} \oint\cdots\oint \prod_{1\le a<b\le k} \frac{(tz_a-qz_b)(z_a-z_b)}{(z_a-qz_b)(tz_a-z_b)} \prod_{c=1}^k\left(\prod_{m=1}^N \frac{tz_c-x_m}{z_c-x_m}\right)
\frac{f(qz_c)}{f(z_c)}\,\frac{dz_c}{z_c}\,,
\end{equation*}
where the $z_c$-contour contains $\{qz_{c+1},\ldots,qz_k,x_1,\ldots,x_N\}$ and no other singularities for $c=1,\dots,k$.
\end{proposition}

In Proposition~\ref{prop10} we describe another family of difference operators that are diagonalized by Macdonald polynomials.

\section{The emergence of a Fredholm determinant}
Macdonald polynomials in $N$ variables with $t=0$ are also known of as {\it $q$-deformed $\mathfrak{gl}_{N}$ Whittaker functions} \cite{GLOqlim}. We now denote the ascending Macdonald process with $t=0$ as $\M_{asc,t=0}(a_1,\dots,a_N;\rho)$ and the Macdonald measure as $\MM_{t=0}(a_1,\dots,a_N;\rho)$. In the paper we refer to these as {\it $q$-Whittaker processes} and {\it $q$-Whittaker measures}, respectively.

% and are related via
%\begin{equation}\label{134intro}
%\QWhit_{x_1,\ldots, x_N}(\lambda) = P_{\lambda}(x_1,\ldots, x_N).
%\end{equation}
%These are eigenfunctions for the quantum $q$-deformed $\mathfrak{gl}_{N}$-Toda chain (see e.g. \cite{Ru}, \cite{Et} or \cite{GLOqlim}).

The partition function for the corresponding $q$-Whittaker measure $\MM_{t=0}(a_1,\dots,a_N;\rho)$ simplifies as
\begin{equation*}
\sum_{\lambda\in \Y(N)} P_\la(a_1,\dots,a_N)Q_\la(\rho)=\Pi(a_1,\dots,a_N;\rho)=\prod_{j=1}^N \exp(\gamma a_j) \prod_{i\ge 1} \frac{1+\beta_i a_j}{(\alpha_i a_j;q)_\infty}\,,
\end{equation*}
where $\rho$ is determined by $\{\alpha_j\}$, $\{\beta_j\}$, and $\gamma$ as before, and we assume $\alpha_i a_j<1$ for all $i,j$ so that the series converge.

Let us take the limit $t\to 0$ of Proposition~\ref{prop8intro}. For the sake of concreteness let us focus on a particular {\it Plancherel} specialization in which the parameter $\gamma>0$ but all $\alpha_i\equiv \beta_i\equiv 0$ for $i\geq 1$. Let us also assume that all $a_i\equiv 1$. For compactness of notation, in this case we will replace $\MM_{t=0}(a_1,\ldots,a_N;\rho)$ with $\MM_{Pl}$. Similar results for general specializations and values of $a_i$, as well as the exact statements and proofs of the results given below can be found in Section~\ref{qWhitSec}.

Write $\mu_k=\left\langle q^{k\lambda_N}\right\rangle_{\MM_{Pl}}$. Then
\begin{equation}\label{mukdefintro}
\mu_k= \frac{(-1)^k q^{\frac{k(k-1)}{2}}}{(2\pi \iota)^k} \oint \cdots \oint \prod_{1\leq A<B\leq k} \frac{z_A-z_B}{z_A-qz_B} \frac{f(z_1)\cdots f(z_k)}{z_1\cdots z_k} dz_1\cdots dz_k,
\end{equation}
where $f(z) = (1-z)^{-N} \exp\{(q-1)\gamma z\}$ and where the $z_j$-contours contain $\{q z_{j+1},\ldots, q z_k, 1\}$ and not 0. For example when $k=2$, $z_2$ is integrated along a small contours around 1, and $z_1$ is integrated around a slightly larger contour which includes 1 and the image of $q$ times the $z_2$ contour.

We may combine these $q$-moments into a generating function
\begin{equation*}
\sum_{k\geq 0} \frac{(\zeta/(1-q))^k}{k_q!} \left\langle q^{k\lambda_N}\right\rangle_{\MM_{Pl}} = \left\langle \frac{1}{\left(\zeta q^{\lambda_N};q\right)_{\infty}}\right\rangle_{\MM_{Pl}},
\end{equation*}
where $k_q!= (q;q)_n/(1-q)^n$ and $(a;q)_k=(1-a)\cdots (1-aq^{k-1})$ (when $k=\infty$ the product is infinity, though convergent since $|q|<1$). This generator function is, via the $q$-Binomial theorem, a $q$-Laplace transform in terms of the $e_q$ $q$-exponential function which can be inverted to give the distribution of $\lambda_N$ (see Proposition~\ref{qlaplaceinverse}). As long as $|\zeta|<1$ the convergence of the right-hand side above follows readily, and justifies the interchanging of the summation and the expectation.

\begin{remark}
It is this sort interchange of summation and expectation which can not be justified when trying to recover the Laplace transform of the partition function of the continuum directed random polymer from its moments. In what follows we show how this generating function can be written as a Fredholm determinant. The non-rigorous replica approach manipulations of Dotsenko \cite{Dot} and Calbrese, Le Doussal and Rosso \cite{CDR} can be seen as a shadow of the rigorous calculations we now detail (in this way, some of what follows can be thought of as a rigorous version of the replica approach).
\end{remark}

We can now describe how to turn the above formulas for $q$-moments of $\lambda_N$ into a Fredholm determinant for the $q$-Laplace transform. To our knowledge there is no a priori reason as to why one should expect to find a nice Fredholm determinant for the transform (and all of its subsequent degenerations). In Section~\ref{qWhitSecFormulas} we provide similar statements for general specializations and values of $a_i$, as well as the exact statements and proofs of the results given below.

In order to turn the series of contour integrals (\ref{mukdefintro}) into a Fredholm determinant, it would be nice if all contours were the same. This can be achieved by either shrinking all contours to a single small circle around 1, or blowing all contours up to a single large circle containing 0 and 1. We focus here on the small contour formulas and resulting Fredholm determinant (in Section~\ref{largecontourformulas} we provide a large contour Fredholm determinant as well, which has a very familiar form for those readers familiar with Tracy and Widom's ASEP formulas \cite{TW3}).

It follows by careful residue book-keeping (see Proposition~\ref{mukprop}) that
\begin{equation}\label{mukintro}
\frac{\mu_k}{k_q!}=  \sum_{\substack{\lambda\vdash k\\ \lambda=1^{m_1}2^{m_{2}}\cdots}} \frac{1}{m_1!m_2!\cdots} \, \frac{(1-q)^{k}}{(2\pi \iota)^{\ell(\lambda)}} \oint \cdots \oint \det\left[\frac{1}{w_i q^{\lambda_i}-w_j}\right]_{i,j=1}^{\ell(\lambda)} \prod_{j=1}^{\ell(\lambda)}  f(w_j)f(qw_j)\cdots f(q^{\lambda_j-1}w_j) dw_j,
\end{equation}
where the $w_j$ are all integrated over the same small circle $C_w$ around 1 (not including 0 and small enough so as no to include $qC_w$). The notation $\lambda \vdash k$ means that $\lambda$ partitions $k$ (i.e., if $\lambda=(\lambda_1,\lambda_2,\ldots)$ then $k=\sum \lambda_i$), and the notation $\lambda = 1^{m_1}2^{m_2}\cdots$ means that $i$ shows up $m_i$ times in the partition $\lambda$.

Before stating our first Fredholm determinant, recall its definition. Fix a Hilbert space $L^2(X,\mu)$ where $X$ is a measure space and $\mu$ is a measure on $X$. When $X=\Gamma$, a simple (anticlockwise oriented) smooth contour in $\C$, we write $L^2(\Gamma)$ where for $z\in \Gamma$, $d\mu(z)$ is understood to be $\tfrac{dz}{2\pi \iota}$. When $X$ is the product of a discrete set $D$ and a contour $\Gamma$, $d\mu$ is understood to be the product of the counting measure on $D$ and $\frac{dz}{2\pi \iota}$ on $\Gamma$. Let $K$ be an {\it integral operator} acting on $f(\cdot)\in L^2(X,\mu)$ by $Kf(x) = \int_{X} K(x,y)f(y) d\mu(y)$. $K(x,y)$ is called the {\it kernel} of $K$. One way of defining the Fredholm determinant of $K$, for trace class operators $K$, is via the series
\begin{equation*}
\det(I+K)_{L^2(X)} = 1+\sum_{n=1}^{\infty} \frac{1}{n!} \int_{X} \cdots \int_{X} \det\left[K(x_i,x_j)\right]_{i,j=1}^{n} \prod_{i=1}^{n} d\mu(x_i).
\end{equation*}

By utilizing the small contour formula for the $\mu_k$ given in (\ref{mukintro}) we can evaluate the generating function as a Fredholm determinant (see Proposition~\ref{gendetprop} for a more general statement). Specifically, for all $|\zeta|$ close enough to zero (and by analytic continuation for $\zeta\in \C\setminus q^{\Zleqzero}$),
\begin{equation*}
\left\langle \frac{1}{\left(\zeta q^{\lambda_N};q\right)_{\infty}}\right\rangle_{\MM_{Pl}} = \det(I+K)_{L^2(\Zgzero\times C_{w})}
\end{equation*}
%where $\det(I+K)$ is the Fredholm determinant of
%\begin{equation*}
%K:L^2(\Zgzero\times C_{w}) \to L^2(\Zgzero\times C_{w})
%\end{equation*}
where $K$ is defined in terms of its integral kernel
\begin{equation}\label{abgKernelintro}
K(n_1,w_1;n_2;w_2) = \frac{\zeta^{n_1} f(w_1)f(qw_1)\cdots f(q^{n_1-1}w_1)}{q^{n_1} w_1 - w_2}
\end{equation}
with $f(w) = (1-w)^{-N} \exp\{(q-1)\gamma w\}$ and $C_{w}$ a small circle around 1 (as above).

By using a Mellin-Barnes type integral representation (Lemma~\ref{gammasumlemma}) we can reduce our Fredholm determinant to that of an operator acting on a single contour. The above developments all lead to the following result (Theorem~\ref{PlancherelfredThm}).

\begin{theorem}\label{PlancherelfredThmintro}
Fix $\rho$ a Plancherel nonnegative specialization of the Macdonald symmetric functions (i.e., $\rho$ determined by (\ref{specdefeqn}) with $\gamma>0$ but all $\alpha_i\equiv \beta_i\equiv 0$ for $i\geq 1$). Then for all $\zeta\in \C\setminus\Rplus$
\begin{equation}\label{thmlaplaceeqnintro}
\left\langle \frac{1}{\left(\zeta q^{\lambda_N};q\right)_{\infty}}\right\rangle_{\MM_{t=0}(1,\ldots, 1;\rho)} = \det(I+K_{\zeta})_{L^2(C_{w})}
\end{equation}
%where $\det(I+K_{\zeta})$ is the Fredholm determinant of
%\begin{equation*}
%K_{\zeta}:L^2(C_{w})\to L^2(C_{w})
%\end{equation*}
where $C_w$ is a positively oriented small circle around 1 and the operator $K_{\zeta}$ is defined in terms of its integral kernel
\begin{equation*}
K_{\zeta}(w,w') = \frac{1}{2\pi \iota}\int_{-\iota \infty + 1/2}^{\iota \infty +1/2} \Gamma(-s)\Gamma(1+s)(-\zeta)^s g_{w,w'}(q^s)ds
\end{equation*}
where
\begin{equation}\label{gwwprimeeqnintro}
g_{w,w'}(q^s) = \frac{1}{q^s w - w'} \left(\frac{(q^{s}w;q)_{\infty}}{(w;q)_{\infty}}\right)^N \exp\big(\gamma w(q^{s}-1)\big).
\end{equation}
The operator $K_{\zeta}$ is trace-class for all $\zeta\in \C\setminus\Rplus$.
\end{theorem}

Note also the following (cf. Proposition~\ref{qlaplaceinverse}).
\begin{corollary}\label{lambdadistcorintro}
We also have that
\begin{equation*}
\PP_{\MM_{t=0}(1,\ldots,1;\rho)}(\lambda_N = n) = \frac{-q^n}{2\pi \iota} \int_{C_{n,q}} (q^n \zeta; q)_{\infty} \det(I+K_{\zeta})_{L^2(C_{w})} d\zeta
\end{equation*}
where $C_{n,q}$ is any simple positively oriented contour which encloses the poles $\zeta=q^{-M}$ for $0\leq M\leq n$ and which only intersects $\Rplus$ in finitely many points.
\end{corollary}

A similar result can be found when $\rho$ is given by a pure alpha specialization ($\gamma=0$ and $\beta_i\equiv 0$ for all $i\geq 1$) in the form of Theorem~\ref{purealphafredThm}.

\section{New ``integrable'' interacting particle systems}

The Macdonald processes can be seen as fixed time measures on Gelfand-Tsetlin patterns evolving according to a certain class of dynamics. Discrete and continuous versions of these dynamics are constructed (see Section~\ref{dynamicsSec}) for general parameters $t,q\in (0,1)$ via an approach of Borodin and Ferrari \cite{BF}, that in its turn was based on an idea of Diaconis and Fill \cite{DiaconisFill}. Other examples of such dynamics can be found in \cite{Bor-Gorin,Bor-Sch-Dyn, BorGor, BorodinOlshanski, Bor-Duits, Bor-Kuan, Betea}. Presently we will focus on the continuous time Markov dynamics in the case $t=0$ as this degeneration results in simple, local updates.

The {\it $q$-Whittaker 2d-growth model} is a continuous time Markov process on the space of Gelfand-Tsetlin patterns defined by (\ref{ascseqnintro}). Each of the coordinates $\lambda_k^{(m)}$ has its own independent exponential clock with rate
\begin{equation*}
a_m\frac{(1-q^{\lambda_{k-1}^{(m-1)}-\lambda_k^{(m)}})(1-q^{\lambda_{k}^{(m)}-\lambda_{k+1}^{(m)}+1})}{(1-q^{\lambda_{k}^{(m)}-\lambda_k^{(m-1)}+1})}
\end{equation*}
(the factors above for which the subscript exceeds the superscript, or either is zero must be omitted). When the $\lambda_k^{(m)}$-clock rings we find the longest string $\lambda_k^{(m)}=\lambda_k^{(m+1)}=\dots=\lambda_k^{(m+n)}$ and move all the coordinates in this string to the right by one. Observe that if $\lambda_k^{(m)}=\la_{k-1}^{(m-1)}$ then the jump rate automatically vanishes.

\begin{figure}
\begin{center}
\includegraphics{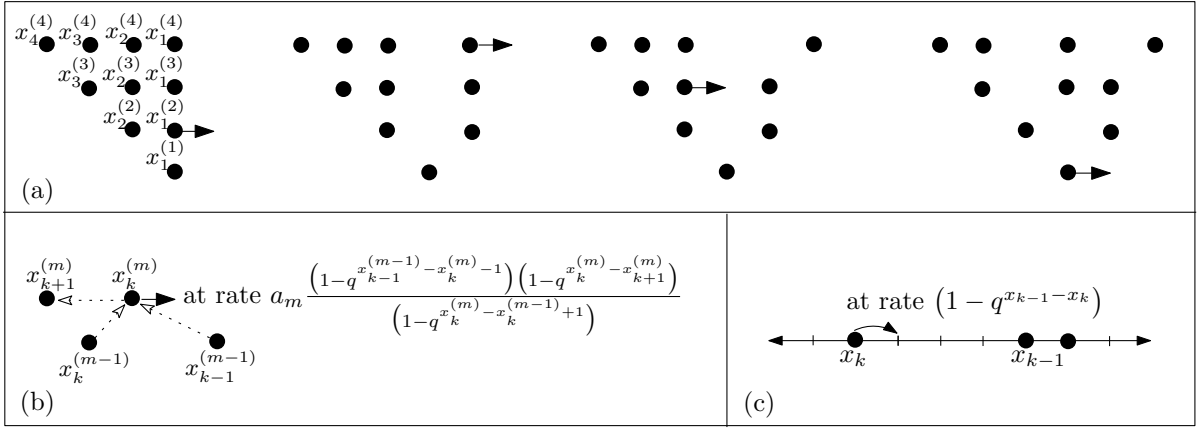}
\end{center}
\caption{(a) First few steps of $q$-Whittaker 2d-growth model written in variables $x_{k}^{(m)} = \lambda_k^{(m)}-k$. (b) The rate at which $x_k^{(m)}$ moves is influenced by three neighbors. The dotted arrows indicate whether the neighbor's influence increases (arrow to the right) or decreases (arrow to the left) the jumping rate. (c) Projecting onto the left edge leads to $q$-TASEP in which particle jump rates are only affected by the number of consecutive empty sites to their right.}\label{qWhitfig}
\end{figure}

A natural initial condition for these dynamics is the zero pattern in which $\lambda_k^{(m)}\equiv 0$. When started from the zero pattern, and run for time $\gamma$ (with $a_1=a_2=\cdots=1$), the marginal distribution of the entire GT pattern is given by the Plancherel specialization of the $t=0$ Macdonald (a.k.a. $q$-Whittaker) process $\M_{Pl}$ (this follows from the more general result of Proposition~\ref{prop16}).

When $k=m$, the rates given above simplify to $(1-q^{\lambda_{k-1}^{(k-1)}-\lambda_k^{(k)}})$. This implies that the edge of the GT pattern $\{\lambda_k^{(k)}\}_{1\leq k\leq N}$ evolves marginally as a Markov process. If we set $x_k=\la_k^{(k)}-k$ for $k=1,2,\ldots$, then the state space consists of ordered sequences of integers $x_1>x_2>x_3>\cdots$. The time evolution is given by a variant on the classical totally asymmetric simple exclusion process (TASEP) which we call {\it $q$-TASEP} (note that $q=0$ reduces $q$-TASEP to TASEP). Now each particle $x_i$ jumps to the right by 1 independently of the others according to an exponential clock with rate $1-q^{x_{i-1}-x_i-1}$ (here $x_{i-1}-x_i-1$ can be interpreted at the number of empty spaces before the next particle to the right). The zero GT pattern corresponds to the step initial condition where $x_n(0)=-n$, $n=1,2,\dots$ (see Section~\ref{qtaseprop} for an interacting particle systems perspective on $q$-TASEP). The gaps of $q$-TASEP evolve according to a certain totally asymmetric zero range process which we will call $q$-TAZRP. The $q$-TAZRP and its quantum version under the name ``$q$-bosons'' were introduced in \cite{SasWad} where their integrability was also noted. A stationary version of $q$-TAZRP was studied in \cite{BKS} and a cube root fluctuation result was shown.

The moment and Fredholm determinant formulas apply to give statistics of the 2d-growth model and $q$-TASEP, and in this sense these models are integrable.
%The Fredholm determinant formula of Theorem~\ref{PlancherelfredThmintro} applies to $\lambda_k^{(k)}$ for any $k$ and hence characterizes the distribution of any given particle $x_k$ at a fixed time $\gamma$ when started in step initial condition. In fact, more generally, particle $x_i$ can move at rate $a_i(1-q^{x_{i-1}-x_i-1})$ and Theorem~\ref{PlancherelfredThm} applies.
For example, Theorem~\ref{PlancherelfredThmintro} should enable a proof that $q$-TASEP, for $q$ fixed, is in the KPZ universality class. For TASEP this was shown by Johansson \cite{KJ} for step initial condition, and for ASEP by Tracy and Widom \cite{TW3}. Additionally, for ASEP, there is a weakly asymmetric scaling limit which converges to the KPZ equation with narrow wedge initial data \cite{ACQ,BG} -- a connection which enables the calculation of one-point exact statistics for its solution \cite{ACQ,SaSp}. For $q$-TASEP there should be such a limit as $q\to 1$ in which one sees the KPZ equation (see Section~\ref{CDRPqTASEP} for a heuristic argument providing scalings). Such direct asymptotics of these models will be pursued elsewhere. Instead, we make an intermediate scaling limit which leads us first to the Whittaker processes and directed polymers, and then to the KPZ equation and universality class.

\section{Crystallization and Whittaker process fluctuations}
The {\it class-one $\mathfrak{gl}_{N}$-Whittaker functions} are basic objects of representation theory and integrable systems \cite{Kostant,Et}. One of their properties is that they are eigenfunctions for the quantum $\mathfrak{gl}_{N}$-Toda chain (see Section~\ref{Whittakerfunctiondefs} for various other properties including orthogonality and completeness relations). As showed by Givental \cite{Giv}, they can also be defined via the following integral representation, which is a degeneration of (\ref{135}) for $q$-Whittaker functions or the combinatorial formula (\ref{7.13},\ref{7.13'}) for Macdonald polynomials:
\begin{equation}\label{giventaldefintro}
\psi_{\lambda}(x_{N,1},\ldots,x_{N,N})=\int_{\R^{N(N-1)/2}} e^{\mathcal{F}_\lambda(x)} \prod_{k=1}^{N-1}\prod_{i=1}^k dx_{k,i},
\end{equation}
where $\lambda=(\lambda_1,\ldots,\lambda_N)$ and
\begin{equation*}
\mathcal{F}_{\lambda}(x)=\iota\sum_{k=1}^{N} \lambda_k\left(\sum_{i=1}^k x_{k,i}-\sum_{i=1}^{k-1} x_{k-1,i}\right)-\sum_{k=1}^{N-1}\sum_{i=1}^k \left(e^{x_{k,i}-x_{k+1,i}}+e^{x_{k+1,i+1}-x_{k,i}}\right).
\end{equation*}

It was observed in \cite{GLOqlim} and proved below in Theorem~\ref{qwhitconvTHM} (along with uniform tail estimates) that Macdonald symmetric functions with $t=0$ (called here $q$-Whittaker functions), when restricted to $N$ variables, converge to class-one $\mathfrak{gl}_{N}$-Whittaker functions as $q\to 1$ (and the various other parameters at play are properly scaled).

It is not just Macdonald symmetric functions which converge to Whittaker functions. The $q$-Whittaker process (Macdonald process as $t=0$) converges to the following measure on $\R^{\frac{N(N+1)}2}$ which was first introduced and studied by O'Connell \cite{OCon} (see Section~\ref{weakconvofan} for a more general definition).

For any $\tau>0$ set
\begin{equation}\label{thetaformulaintro}
\theta_{\tau}(x_1,\ldots,x_N)=\int_{\R^N} \psi_{\nu}(x_1,\ldots,x_{N}) e^{-\tau\sum_{j=1}^N\nu_j^2/2} m_N(\nu)\prod_{j=1}^{N} d\nu_j
\end{equation}
with the Skylanin measure
\begin{equation*}
m_{N}(\nu)=\frac1{(2\pi)^{N} (N)!}\prod_{j\ne k} \frac 1{\Gamma(\iota \nu_k-\iota \nu_j)}\,,
\end{equation*}
and define the following {\it Whittaker process} as a measure on $\R^{\frac{N(N+1)}2}$ with density function (with respect to the Lebesgue measure) given by
\begin{equation}\label{Wdefintro}
\W{\tau}\big(\{T_{k,i}\}_{1\le i\le k\le N}\big)=\exp({\mathcal{F}_{0}(T)})\,{\theta_{\tau}(T_{N,1},\dots,T_{N,N})}.
\end{equation}

The nonnegativity of the density follows from definitions. Integrating over the variables $T_{k,i}$ with $k<N$ yields the following {\it Whittaker measure} with density function given by
\begin{equation}\label{WMdefintro}
\WM{\tau}\big(\{T_{N,i}\}_{1\le i\le N}\big)=\psi_{0}(T_{N,1},\dots,T_{N,N})\,{\theta_{\tau}(T_{N,1},\dots,T_{N,N})}.
\end{equation}
The fact that this measure integrates to one follows from the orthogonality relation for Whittaker functions given in Section~\ref{Whittakerfunctiondefs}. Note that the particles under this measure are no longer restricted to form GT patterns (i.e., lie on $\Zgeqzero$ and interlace).

Let us return to the $q$-Whittaker process. Introduce a scaling parameter $\e$ and set $q=e^{-\e}$. Then for time $\gamma= \e^{-2} \tau$ one finds that as $\e\to 0$ (hence $q\to 1$ and $\gamma\to \infty$) the $q$-Whittaker process on GT patterns crystalizes onto a regular lattice. The fluctuations of the pattern around this limiting lattice converge under appropriate scaling to the above defined Whittaker process (see Section~\ref{weakconvstateproofSec} for a more general statement and proof of this result).

\begin{theorem}\label{theorem26intro}
Fix $\rho$ a Plancherel nonnegative specialization of the Macdonald symmetric functions (i.e., $\rho$ determined by (\ref{specdefeqn}) with $\gamma>0$ but all $\alpha_i\equiv \beta_i\equiv 0$ for $i\geq 1$) and write $\M_{Pl}$ for the $q$-Whittaker process $\M_{asc,t=0}(1,\dots,1;\rho)$. In the limit regime
\begin{equation}\label{scalingsintro}
q=e^{-\epsilon},\qquad \gamma=\tau \epsilon^{-2}, \qquad \lambda^{(k)}_j=\tau\epsilon^{-2}-(k+1-2j)\epsilon^{-1}\log\epsilon+T_{k,j}\epsilon^{-1}, \quad 1\le j\le k\le N,
\end{equation}
$\M_{Pl}$ converges weakly, as $\e\to 0$, to the Whittaker process $\W{\tau}\bigl(\{T_{k,i}\}_{1\le i\le k\le N}\bigr)$.
\end{theorem}

The dynamics of the $q$-Whittaker 2d-growth model also have a limit as the following Whittaker 2d-growth model (see Section~\ref{contlimit2dgrowth} for a general definition). This is a continuous time ($\tau\geq 0$) Markov diffusion process $T(\tau)=\{T_{k,j}(\tau)\}_{1\leq j\leq k\leq N}$ with state space $\R^{\frac{N(N+1)}2}$ which is given by the following system of stochastic (ordinary) differential equations: Let $\{W_{k,j}\}_{1\leq j\leq k \leq N}$ be a collection of independent standard one-dimensional Brownian motions. The evolution of $T$ is defined recursively by $dT_{1,1} = dW_{1,1}$ and for $k=2,\ldots, N$,
\begin{eqnarray*}
dT_{k,1} &=& d W_{k,1} + \left(e^{T_{k-1,1}-T_{k,1}}\right)d\tau\\
dT_{k,2} &=& dW_{k,2} + \left(e^{T_{k-1,2}-T_{k,2}} - e^{T_{k,2}-T_{k-1,1}}\right)d\tau\\
&\vdots&\\
dT_{k,k-1} &=& dW_{k,k-1} + \left(e^{T_{k-1,k-1}-T_{k,k-1}} - e^{T_{k,k-1}-T_{k-1,k-2}}\right)d\tau\\
dT_{k,k} &=& dW_{k,k} - \left(e^{T_{k,k}-T_{k-1,k-1}}\right)d\tau.
\end{eqnarray*}

It follows from Theorem~\ref{theorem26intro} and standard methods of stochastic analysis that the above $q$-Whittaker 2d-growth model initialized with zero initial GT pattern converges (under scalings as in (\ref{scalingsintro})) to the Whittaker 2d-growth model with entrance law for $\{T_{k,j}(\delta)\}_{1\le j\le k\le N}$ given by the density $\W{\delta}\left(\{T_{k,j}\}_{1\le j\le k\le N}\right)$ for $\delta>0$. Let us briefly describe this limit result for the $N=2$ dynamics. Under the specified $q$-Whittaker dynamics, the particle $\lambda_1^{(1)}$ evolves as a rate one Poisson jump process. In time $\tau \e^{-2}$, $\lambda_1^{(1)}$ is macroscopically at $\tau \e^{-2}$. In an $\e^{-1}$ scale, the particle's dynamics converge (as $\e\to 0$) to that of a standard Brownian motion $W_{1,1}$. Turning to $\lambda_2^{(2)}$, the entrance law provided by Theorem~\ref{theorem26intro} shows that
\begin{equation*}
\lambda_1^{(1)}(\e^{-2}\tau)-\lambda_2^{(2)}(\e^{-2}\tau) = \e^{-1} \log \e^{-1} + \left[T_{1,1}(\tau)-T_{2,2}(\tau)\right]\e^{-1}.
\end{equation*}
Thus the jump rate for $\lambda_2^{(2)}$ is given by
\begin{equation*}
1-q^{\lambda_1^{(1)}(\tau)-\lambda_2^{(2)}(\tau)} = 1 -\e e^{T_{2,2}(\tau) - T_{1,1}(\tau)}.
\end{equation*}
In the time scale $\e^{-2}$, $T_{2,2}$ behaves like a Brownian motion $W_{2,2}$ plus a drift due to this perturbation of $-\e e^{T_{2,2}(\tau) - T_{1,1}(\tau)}$ -- exactly as given by the Whittaker 2d-growth model. The argument for $\lambda_2^{(1)}$ is similar.

O'Connell proved (\cite{OCon} Section 9) that the projection of the Whittaker 2d-growth model onto $\{T_{N,j}(\tau)\}_{1\leq j\leq N}$ is itself a Markov diffusion process with respect to its own filtration with entrance law given by density $\WM{\tau}\left(\{T_{N,j}\}_{1\leq j\leq N}\right)$ and infinitesimal generator given by
\begin{equation}\label{semigpW2dintro}
\mathcal{L} = \tfrac{1}{2} \psi_{0}^{-1}H\psi_{0},
\end{equation}
where $H$ is the quantum $\mathfrak{gl}_{N}$-Toda lattice Hamiltonian
\begin{equation*}
H= \Delta - 2 \sum_{j=1}^{N-1} e^{x_{i+1}-x_{i}}.
\end{equation*}

The $q$-Laplace transform generating function and Fredholm determinant of Theorem~\ref{PlancherelfredThmintro} has a limit under the $q\to 1$ scalings (see Section~\ref{WhitFredDetSec} for a more general statement and proof)

\begin{theorem}\label{NeilPolymerFredDetThmintro}
For any $u\geq 0$,
\begin{equation}\label{doubleExpintro}
\left\langle e^{-u e^{-T_{N,N}}}  \right\rangle_{\WM{\tau}} = \det(I+ K_{u})_{L^2(C_\delta)}
\end{equation}
%where $\det(I+ K_{u})$ is the Fredholm determinant of
%\begin{equation*}
%K_{u}: L^2(C_\delta)\to L^2(C_\delta)
%\end{equation*}
where $C_\delta$ is a positively oriented contour containing the origin of radius less than $\delta/2$ with any $0< \delta<1$. The operator $K_u$ is defined in terms of its integral kernel
\begin{equation*}
K_{u}(v,v') = \frac{1}{2\pi \iota}\int_{-\iota \infty + \delta}^{\iota \infty +\delta}ds \Gamma(-s)\Gamma(1+s) \left(\frac{\Gamma(v-1)}{\Gamma(s+v-1)}\right)^N \frac{ u^s e^{v\tau s+\tau s^2/2}}{v+s-v'}.
\end{equation*}
\end{theorem}

\section{Tracy-Widom asymptotics for polymer free energy}

O'Connell \cite{OCon} introduced the Whittaker process to describe the free energy of a semi-discrete directed polymer in a random media (see Section~\ref{genback} for general background on directed polymers or Section~\ref{OConmodel} for more on this particular model). We refer to this as the {\it O'Connell-Yor semi-discrete directed polymer}, as it was introduced in \cite{OCon-Yor}. Define an {\it up/right path} in $\R\times \Z$ as an increasing path which either proceeds to the right (along a copy of $\R$) or jumps up (in $\Z$) by one unit. For each sequence $0<s_1<\cdots<s_{N-1}<\tau$ we can associate an up/right path $\phi$ from $(0,1)$ to $(\tau,N)$ which jumps between the points $(s_i,i)$ and $(s_{i},i+1)$, for $i=1,\ldots, N-1$, and is continuous otherwise. The polymer paths will be such up/right paths, and the random environment will be a collection of independent standard Brownian motions $B(s) = (B_1(s),\ldots, B_N(s))$ for $s\geq 0$ (see Figure \ref{semidiscrete} for an illustration of this polymer). The energy of a path $\phi$ is given by
\begin{equation*}
E(\phi) = B_1(s_1)+\left(B_2(s_2)-B_2(s_1)\right)+ \cdots + \left(B_N(t) - B_{N}(s_{N-1})\right).
\end{equation*}
The (quenched) {\it partition function} $\Zsd^{N}(\tau)$ is given by averaging over the possible paths:
\begin{equation*}
\Zsd^{N}(\tau) = \int e^{E(\phi)} d\phi,
\end{equation*}
where the integral is with respect to Lebesgue measure on the Euclidean set of all up/right paths $\phi$ (i.e., the simplex of jumping times $0<s_1<\cdots<s_{N-1}<\tau$). One can similarly introduce a {\it hierarchy of partition functions} $\Zsd^{N}_{n}(\tau)$ for $0\leq n\leq N$ by setting $\Zsd^{N}_{0}(\tau)=1$ and for $n\geq 1$,
\begin{equation}\label{partfuncthierIntro}
\Zsd^{N}_{n}(\tau) = \int_{D_{n}(\tau)} e^{\sum_{i=1}^{n} E(\phi_i)} d\phi_1 \cdots d\phi_n,
\end{equation}
where the integral is now with respect to the Lebesgue measure on the Euclidean set $D_n(\tau)$ of all $n$-tuples of non-intersecting (disjoint) up/right paths with initial points $(0,1),\ldots, (0,n)$ and endpoints $(\tau,N-n+1),\ldots, (\tau,N)$.

\begin{figure}
\begin{center}
\includegraphics[scale=.6]{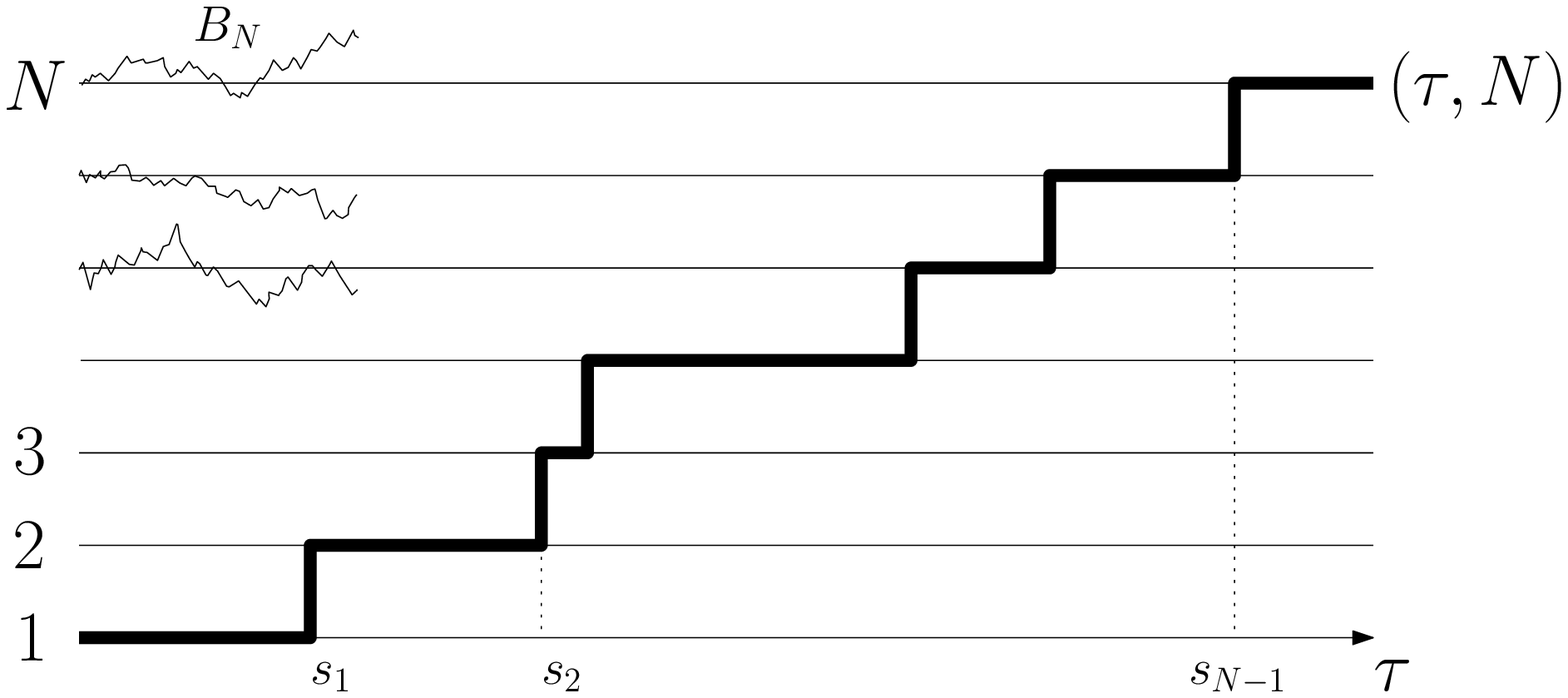}
\end{center}
\caption{The O'Connell-Yor semi-discrete directed polymer is a reweighting of Poisson jump process space-time trajectories (in bold) with respect to Boltzmann weights given by exponentiating a Hamiltonian givne by the path integral along the trajectory through a field of independent 1-dimensional white noises. The (quenched) partition function is the normalizing constant to make this reweighting into a probability measure.}\label{semidiscrete}
\end{figure}

The {\it hierarchy of free energies} $\Fsd^{N}_{n}(\tau)$ for $1\leq n\leq N$ if defined as
\begin{equation*}
\Fsd^{N}_{n}(\tau) = \log\left(\frac{ \Zsd^{N}_{n}(t)}{ \Zsd^{N}_{n-1}(\tau)}\right).
\end{equation*}
Among these, $\Fsd^{N}_{1}(\tau)$ is the {\it directed polymer free energy}.

It is shown in \cite{OCon} Theorem 3.1 that as a process of $\tau$, $(-\Fsd^{N}_{N}(\tau),\ldots, -\Fsd^{N}_{1}(\tau))$ is given by the Markov diffusion process with entrance law given by density $\WM{\tau}\left(-\Fsd^{N}_{N}(\tau),\ldots, -\Fsd^{N}_{1}(\tau)\right)$ and infinitesimal generator given by (\ref{semigpW2dintro}). Note that this means that $\Fsd_{1}^{N}(\tau)$ and $-T_{N,N}(\tau)$ are equal in law and thus we may apply Theorem~\ref{NeilPolymerFredDetThmintro} to characterize the distribution of the free energy $\Fsd_{1}^{N}(\tau)$ of the O'Connell-Yor polymer. The form of the Fredholm determinant is such that we may also calculate a limit theorem for the free energy fluctuations as $\tau$ and $N$ go to infinity.

Before taking these asymptotics, let us motivate the limit theorem we will prove. Since the noise is Brownian, scaling time is the same as introducing a prefactor $\beta$ in front of $E(\phi)$ above. Generally, $\beta$ is called the {\it inverse temperature}, and taking time to infinity is like taking $\beta$ to infinity (or temperature to zero). The limit of the free energy hierarchy (divided by $\beta$) as $\beta$ goes to infinity and time is fixed, is described by a coupled set of maximization problems. In particular, regarding $\Fsd^{N}_{1}(\tau)$,
\begin{eqnarray*}
M^N_1(\tau) &:=& \lim_{\beta\to \infty} \frac{1}{\beta} \log \int e^{\beta E(\phi)} d\phi\\
&=& \max_{0<s_1<\cdots<s_{N-1}<\tau} B_1(s_1)+\left(B_2(s_2)-B_2(s_1)\right)+ \cdots + \left(B_N(\tau) - B_{N}(s_{N-1})\right).
\end{eqnarray*}

As recorded in \cite{OCon} Theorem 1.1 (see the references stated therein), $M^N_1(1)$ is distributed as the largest eigenvalue of an $N\times N$ GUE random matrix (in fact the limit of the entire free energy hierarchy has the law of the standard Dyson Brownian motion). It follows from the original work of Tracy and Widom \cite{TW} and also \cite{For,NagWad} that
\begin{equation}\label{TWGUEpolygroundintro}
\lim_{N \to \infty} \PP\left( \frac{ M^N_1(N) - 2N}{N^{1/3}}\leq r\right) = F_{{\rm GUE}}(r).
\end{equation}

The polymer universality belief (see \cite{CQ2}) holds that all polymers in one space and one time dimension (with homogeneous and light-tailed noise) should have the same asymptotic scaling exponent ($1/3$) and limit law ($F_{{\rm GUE}})$.

In order to state our result here, recall the Digamma function $\Psi(z) = [\log \Gamma]'(z)$. Define for $\kappa>0$
\begin{equation*}
\bfk = \inf_{t>0} (\kappa t - \Psi(t)),
\end{equation*}
and let $\btk$ denote the unique value of $t$ at which the minimum is achieved. Finally, define the positive number (scaling parameter) $\bgk= -\Psi''(\btk)$.

Then by taking rigorous asymptotics of the Fredholm determinant of Theorem~\ref{NeilPolymerFredDetThmintro} (for a proof see Theorem~\ref{TWasymptoticskappaTHM}) we have
\begin{theorem}\label{Thmkappalarge}
For all $N$, set $t=N\kappa$. For $\kappa>0$ large enough,
\begin{equation*}
\lim_{N\to \infty} \PP\left( \frac{ \Fsd^{N}_{1}(t) - N \bfk}{N^{1/3}}\leq r\right) = F_{{\rm GUE}}\left((\bgk / 2)^{-1/3}r\right).
\end{equation*}
\end{theorem}

The condition on $\kappa$ being large is artificial and presently comes out of technicalities in performing the steepest descent analysis of the Fredholm determinant of Theorem~\ref{NeilPolymerFredDetThmintro}. In \cite{BorCorFer} the above technical issues are overcome and the theorem is extended to all $\kappa>0$. Note that the double exponential on the left-hand side of (\ref{doubleExpintro}) becomes an indicator function (and hence its expectation becomes a probability) since $e^{-e^{c(x-a)}}\to {\bf 1}_{x<a}$ as $c\to \infty$. The law of large numbers with the constant $\bfk$ was conjectured in \cite{OCon-Yor} and proved in \cite{OConnellMoriarty}. A tight one-sided bound on the fluctuation exponent for $\Fsd^{N}_1(t)$ was determined in \cite{SeppValko}. Taking $\kappa\to \infty$ one recovers (\ref{TWGUEpolygroundintro}). Spohn \cite{SpohnPolymer} has recently described how the above theorem fits into a general physical KPZ scaling theory.

There exists a similar story to that presented above when the O'Connell-Yor polymer is replaced by the {\it log-gamma discrete directed polymer} introduced by Sepp\"{a}l\"{a}inen \cite{SeppLog}. Using A.N. Kirillov's \cite{Kir,NY} {\it tropical RSK correspondence}, Corwin, O'Connell, Sepp\"{a}l\"{a}inen and Zygouras \cite{COSZ} introduce a different Whittaker process to describe the free energy hierarchy for this polymer (see Section~\ref{logsec} for more on this). Presently we call this the {\it $\alpha$-Whittaker process} (see Section~\ref{alphawhitprocessSec}) and it arises as the limit of the Macdonald process when $\rho$ is a {\it pure alpha specialization} (some $\alpha_i>0$, $\beta_i=0$ for all $i\geq 1$ and $\gamma=0$).

\section{Solvability of the KPZ equation}

The {\it continuum directed random polymer} (CDRP) is the universal scaling limit for discrete and semi-discrete polymers when the inverse-temperature $\beta$ scales to zero in a critical manner (called {\it intermediate disorder scaling}) as the system size goes to infinity (this was observed independently by Calabrese, Le Doussal and Rosso \cite{CDR} and by Alberts, Khanin and Quastel \cite{AKQ} and is proved in \cite{AKQ2} for discrete polymers and in \cite{QM} for the O'Connell-Yor semi-discrete polymer). In the CDRP, the polymer path measure is the law of a Brownian bridge and the random media is given by Gaussian space time white noise.

The CDRP partition function is written as (see \cite{ACQ,ICreview})
\begin{equation*}
\mathcal{Z}(T,X)= p(T,X) \EE\left[:{\rm exp}:\, \left\{ \int_{0}^{T} \whitenoise(t,b(t)) dt \right\} \right]
\end{equation*}
where $\EE$ is the expectation of the law of the Brownian bridge $b(\cdot)$ starting at $X$ at time 0 and ending at $0$ at time $T$. The Gaussian heat kernel is written as $p(T,X)=(2\pi)^{-1/2}e^{-X^2/T}$. The expression $:{\rm exp}:$ is the {\it Wick exponential} and can be defined either through a limiting smoothing procedure as in \cite{BC} or via Weiner-It\^{o} chaos series \cite{ACQ}. Via a version of the Feynman-Kac formula, $\mathcal{Z}(T,X)$ solves the well-posed {\it stochastic heat equation} with multiplicative noise and delta function initial data:
\begin{equation*}
\partial_T \mathcal{Z}= \tfrac{1}{2}\partial_X^2 \mathcal{Z} -\mathcal{Z}\dot{\mathscr{W}}, \qquad \mathcal{Z}(0,X)=\delta_{X=0}.
\end{equation*}
Due to a result of Mueller \cite{Mueller}, almost surely $\mathcal{Z}(T,X)$ is positive for all $T>0$ and $X\in \R$, hence we can take its logarithm:
\begin{equation*}
\mathcal{H}(T,X) = \log \mathcal{Z}(T,X).
\end{equation*}
This is called the Hopf-Cole solution to the {\it Kardar-Parisi-Zhang} (KPZ) equation with {\it narrow wedge} initial data \cite{KPZ,BG,ACQ}. Formally (though it is ill-posed due to the non-linearity) the KPZ equation is written as
\begin{equation*}
\partial_T \mathcal{H} = \tfrac{1}{2}\partial_X^2 \mathcal{H} + \tfrac{1}{2}(\partial_X \mathcal{H})^2 +\whitenoise.
\end{equation*}

We may perform an intermediate disorder scaling limit of the Fredholm determinant of Theorem~\ref{NeilPolymerFredDetThmintro}. Under this scaling, the double exponential in the left-hand side of (\ref{doubleExpintro}) is preserved, giving us a Laplace transform of $\mathcal{Z}$. For $u=s \exp\{-N-\tfrac{1}{2}\sqrt{TN}-\tfrac{1}{2}N\log(T/N)\}$,
\begin{equation}\label{LMRintro}
\langle e^{-s \mathcal{Z}(T,0)} \rangle = \lim_{N\to\infty} \langle e^{-u\Zsd^{N}_{1}(\sqrt{TN})}\rangle  =\det(I- K_{se^{-T/24}})_{L^2(0,\infty)}
\end{equation}
where the operator $K_{s}$ is defined in terms of its integral kernel
\begin{equation*}
K_s(r,r') = \int_{-\infty}^{\infty} \frac{s}{s+e^{-\kappa_T t}} \Ai(t+r) \Ai(t+r') dt
\end{equation*}
where $\kappa_T= 2^{-1/3}T^{1/3}$ and $\Ai$ is the Airy function.

This is shown in Section~\ref{CDRPfreddetSec} by expanding around the critical point of the kernel in Theorem~\ref{NeilPolymerFredDetThmintro} (under the above scalings). A rigorous proof of this result (which requires more than just a critical point analysis) is provided in \cite{BorCorFer}.

The equality of the left-hand and right-hand sides of (\ref{LMRintro}) is already known. It follows from the exact formulas for the probability distribution for the solution to the KPZ equation which was simultaneously and independently discovered in \cite{ACQ,SaSp} and proved rigorously in \cite{ACQ}. Those works took asymptotics of Tracy and Widom's ASEP formulas \cite{TW1,TW2,TW3} and rely on the fact that weakly scaled ASEP converges to the KPZ equation \cite{BG,ACQ}. It has been a challenge to expand upon the solvability of ASEP as necessary for asymptotics (see \cite{TW4} for the one such extension to half-Bernoulli initial conditions and \cite{CQ} for the resulting formula for the KPZ equation with half-Brownian initial data). The many parameters (the $a_i$'s, $\alpha_i$'s, and $\beta_i$',s) we have suppressed in this introduction can be used to access statistics for the KPZ equation with other types of initial data, which will be a subject of subsequent work (cf. \cite{BorCorFer}).

Soon after, Calabrese, Le Doussal and Rosso \cite{CDR} and Dotsenko \cite{Dot} derived the above Laplace transform through a non-rigorous replica approach. There, the Laplace transform is written as a generating function in the moments of $\mathcal{Z}(T,0)$. The $N$-th such moment grows like $e^{cN^{3}}$ and hence this generating function is widely divergent. By rearranging terms and analytically continuing certain functions off integers, this divergent series is summed and the Laplace transform formula results. These manipulations can actually be seen as shadows of the rigorous manipulations performed at the Macdonald process level.

\section{Moment formulas for polymers}

We previously saw that due to the large family of difference operators diagonalized by the Macdonald symmetric functions, we can derive contour integral formulas for a rich family of observables for Macdonald processes (in this introduction, Propositions \ref{prop5intro} and \ref{prop6intro} as well as more general statements in Sections \ref{difopSEC} and \ref{expmomform26}). The limit which takes Macdonald processes (and $q$-Whittaker processes) to Whittaker processes turns these observables into moments of the partition function hierarchy (e.g., (\ref{partfuncthierIntro})).

Let us focus on Proposition~\ref{prop6intro} and its limit. This proposition calculated the Macdonald process expectation of $q^{k\lambda_N^N}$. In the scalings of Theorem~\ref{theorem26intro}, this becomes the Whittaker process expectation of $e^{-k T_{N,N}}$ (see Section~\ref{whitintform} for general precise statements of this sort). It follows from the above discussed connection between Whittaker processes and the O'Connell-Yor semi-discrete polymer that (see Section~\ref{restatedIntForm} for more general statements and a proof of this result which appears as Proposition~\ref{Proposition28OCon})
\begin{proposition}
For any $N\geq 1$, $k\geq 1$ and $\tau\geq 0$,
\begin{equation*}
\left\langle  \left(\Zsd^{N}_{1}(\tau)\right)^k \right\rangle = \frac{e^{k\tau/2}}{(2\pi \iota)^k} \oint\cdots\oint \prod_{1\le A<B\le k} \frac{w_A-w_B}{w_A-w_B-1}\prod_{j=1}^k e^{\tau w_j}\frac{dw_j}{w_j}\,,
\end{equation*}
where the $w_j$-contour contains $\{w_{j+1}+1,\cdots,w_k+1,0\}$ and no other singularities for $j=1,\dots,k$.
\end{proposition}

An intermediate disorder scaling limit of this result leads to contour integral formulas for moments of the CDRP (see Section~\ref{CDRPintSec} for more general statements).
\begin{proposition}\label{Zkformpropintro}
For $k\geq 1$ and $T>0$,
\begin{equation}\label{Zkformintro}
\langle \mathcal{Z}(T,0)^k\rangle =
\frac{1}{(2\pi \iota)^k} \oint\cdots\oint \prod_{1\le A<B\le k} \frac{z_A-z_B}{z_A-z_B-1}\prod_{j=1}^{k} e^{\frac{T}{2}z_j^2}dz_j,
\end{equation}
where the $z_A$-contour is along $C_A+\iota \R$ for any $C_1>C_2+1>C_3+2>\cdots >C_k+(k-1)$.
\end{proposition}
Note that since $\mathcal{Z}(T,X)/p(T,X)$ is a stationary process in $X$ \cite{ACQ}, this likewise gives formulas for all values of $X\in \R$.

Let us work out the $k=1$ and $k=2$ formulas explicitly. For $k=1$, the above formula gives $\langle \mathcal{Z}_1(T,X)^k\rangle = (2\pi T)^{-1/2}$ which matches $p(T,0)$ as one expects. When $k=2$ we have (see Remark \ref{BCrem})
\begin{equation*}
\langle \mathcal{Z}_1(T,0)^2\rangle = \frac{1}{2\pi T}\left(1+ \sqrt{\pi T} e^{\frac{T}{4}} \Phi(\sqrt{T/2})\right),
\end{equation*}
where
\begin{equation*}
\Phi(s) = \frac{1}{\sqrt{2\pi}} \int_{-\infty}^{s} e^{-t^2/2}dt.
\end{equation*}
This formula for $k=2$ matches formula (2.27) of \cite{BC} where it was rigorously derived via local time calculations.

\section{A contour integral ansatz for some quantum many body systems}

Moments of the exactly solvable polymer models studied above solve certain quantum many body systems with delta interactions. This fact is the basis for the replica approach employed in this area since the work of Kardar \cite{K}. Let us focus on the moments of the CDRP (details and discussion are in Section~\ref{CDRPreplica}, whereas Section~\ref{replicaSemidisc} deals with the O'Connell-Yor semi-discrete polymer).

Let $\Weyl{N}=\{X_1<X_2<\cdots<X_N\}$ be the Weyl chamber. We say that a function $u:\Weyl{N}\times \Rplus\to \R$ solves the delta Bose gas with coupling constant $\kappa\in \R$ and delta initial data if:
\begin{itemize}
\item For $X\in  \Weyl{N}$,
\begin{equation*}
\partial_T u = \tfrac{1}{2}\Delta u,
\end{equation*}
\item On the boundary of $\Weyl{N}$,
\begin{equation*}
(\partial_{X_{i}}-\partial_{X_{i+1}}-\kappa)u \big\vert_{X_{i+1}=X_{i}+0} = 0,
\end{equation*}
\item and for any $f\in L^{2}(\Weyl{N})\cap C_b(\overline{\Weyl{N}})$, as $t\to 0$
\begin{equation*}
N! \int_{\Weyl{N}} f(x) u(X;t) dX \to f(0).
\end{equation*}
\end{itemize}
When $\kappa>0$ this is called the {\it attractive} case, whereas when $\kappa<0$ this is the {\it repulsive} case. %By scaling one can assume $\kappa=\pm 1$.

Note that the boundary condition is often included in PDE so as to appear as
\begin{equation*}
\partial_T u = \tfrac{1}{2}\Delta u + \tfrac{1}{2} \kappa \sum_{i\neq j} \delta(X_i-X_j)u.
\end{equation*}

The relevance of the delta Bose gas for the CDRP is that
\begin{equation*}
u(X;T) = \left\langle \prod_{i=1}^{N} \mathcal{Z}(T,X_i) \right\rangle
\end{equation*}
solves the attractive delta Bose gas with coupling constant $\kappa=1$ (attractive) and delta initial data.

Inspired by the simplicity of Proposition~\ref{Zkformpropintro} which gives $u(X;T)$ when $X_i\equiv 0$, we propose and verify the following contour integral ansatz for the solution to this many body problem.
\begin{proposition}\label{ADBGpropintro}
Fix $N\geq 1$. The solution to the delta Bose gas with coupling constant $\kappa\in \R$ and delta initial data can be written as
\begin{equation}\label{ADBGeqnintro}
u(X;T) = \frac{1}{(2\pi \iota)^N} \oint\cdots\oint \prod_{1\leq A<B\leq N} \frac{z_A-z_B}{z_A-z_B-\kappa}e^{\frac{T}{2}\sum_{j=1}^{N} z_j^2 + \sum_{j=1}^{N}X_j z_j}\prod_{j=1}^{N} dz_j,
\end{equation}
where the $z_j$-contour is along $\alpha_j+\iota \R$ for any $\alpha_1>\alpha_2+\kappa>\alpha_3+2\kappa>\cdots >\alpha_N+(N-1)\kappa$.
\end{proposition}
The proof of this result (given in Section~\ref{CDRPreplica}) is straightforward. In particular, it is easy to check that as $T\to 0$ this provides the correct delta initial data.

One should note that the above proposition deals with both the $\kappa>0$ (attractive) and $\kappa<0$ (repulsive) delta Bose gas. Looking in the arguments of Section 4 of \cite{HeckOp} (where they were proving the Plancherel theory for the delta Bose gas) it is possible to extract the formula of the above proposition. This type of formula reveals a symmetry between the two cases which is not present in the eigenfunctions.

An alternative and much earlier taken approach to solving the delta Bose gas is by demonstrating a complete basis of eigenfunctions (and normalizations) which diagonalize the Hamiltonian and respect the boundary condition. The eigenfunctions were written down in 1963 for the repulsive delta interaction by Lieb and Liniger \cite{LL} by Bethe ansatz. Completeness was proved by Dorlas \cite{Dorlas} on $[0,1]$ and by Heckman and Opdam \cite{HeckOp} and then recently by Prolhac and Spohn (using formulas of Tracy and Widom \cite{TWBose}) on $\R$ (as we are considering presently). For the attractive case, McGuire \cite{McGuire} wrote the eigenfunctions in terms of {\it string states} in 1964. As opposed to the repulsive case, the attractive case eigenfunctions are much more involved and are not limited to bound state eigenfunctions (hence a lack of symmetry with respect to the eigenfunctions). The norms of these states were not derived until 2007 in \cite{CalCaux} using ideas of algebraic Bethe ansatz. Dotsenko \cite{Dot} later worked these norms out very explicitly through combinatorial means. Completeness in the attractive case was shown by Oxford \cite{Oxford}, and then by Heckman and Opdam \cite{HeckOp}, and recently by Prolhac and Spohn \cite{ProSpoComp}.

The work \cite{LL,TWBose,Dot,CDR,ProSpoComp} provide formulas for the propagators (i.e., transition probabilities) for the delta Bose gas with general initial data. These formulas involve either summations over eigenstates or over the permutation group. In the repulsive case it is fairly easy to see how the formula of Proposition~\ref{ADBGpropintro} is recovered from these formulas.

For the attractive case we can use a degeneration of the identity given in (\ref{mukintro}) to turn the moment formulas of Proposition~\ref{ADBGpropintro} into the formulas given explicitly in Dotsenko's work \cite{Dot}. The reason why the symmetry, which is apparent in Proposition~\ref{ADBGpropintro}, is lost in the eigenfunction expansion is due to the constraint on the contours. In the repulsive case $\kappa<0$ and the contours are given by having the $z_j$-contour along $\alpha_j+\iota \R$ for any $\alpha_1>\alpha_2+\kappa>\alpha_3+2\kappa>\cdots >\alpha_N+(N-1)\kappa$. The contours, therefore, can be taken to be all the same. It is an easy calculation to turn the Bethe ansatz eigenfunction expansion into the contour integral formula we provide. The attractive case leads to contours which are not the same. In making the contours coincide we encounter a sizable number of poles which introduces many new terms which agrees with the fact that there are many other eigenfunctions coming from the Bethe ansatz in this case.

The ansatz applies in greater generality (revealing the role of each part of the contour integrals, see Remark \ref{endansatzcomment}) and is useful in solving many body systems which arise in the study of other polymers (e.g., for semi-discrete polymers such as the O'Connell-Yor or discrete parabolic Anderson models see Section~\ref{replicaSemidisc}).

\begin{figure}
\begin{center}
\includegraphics{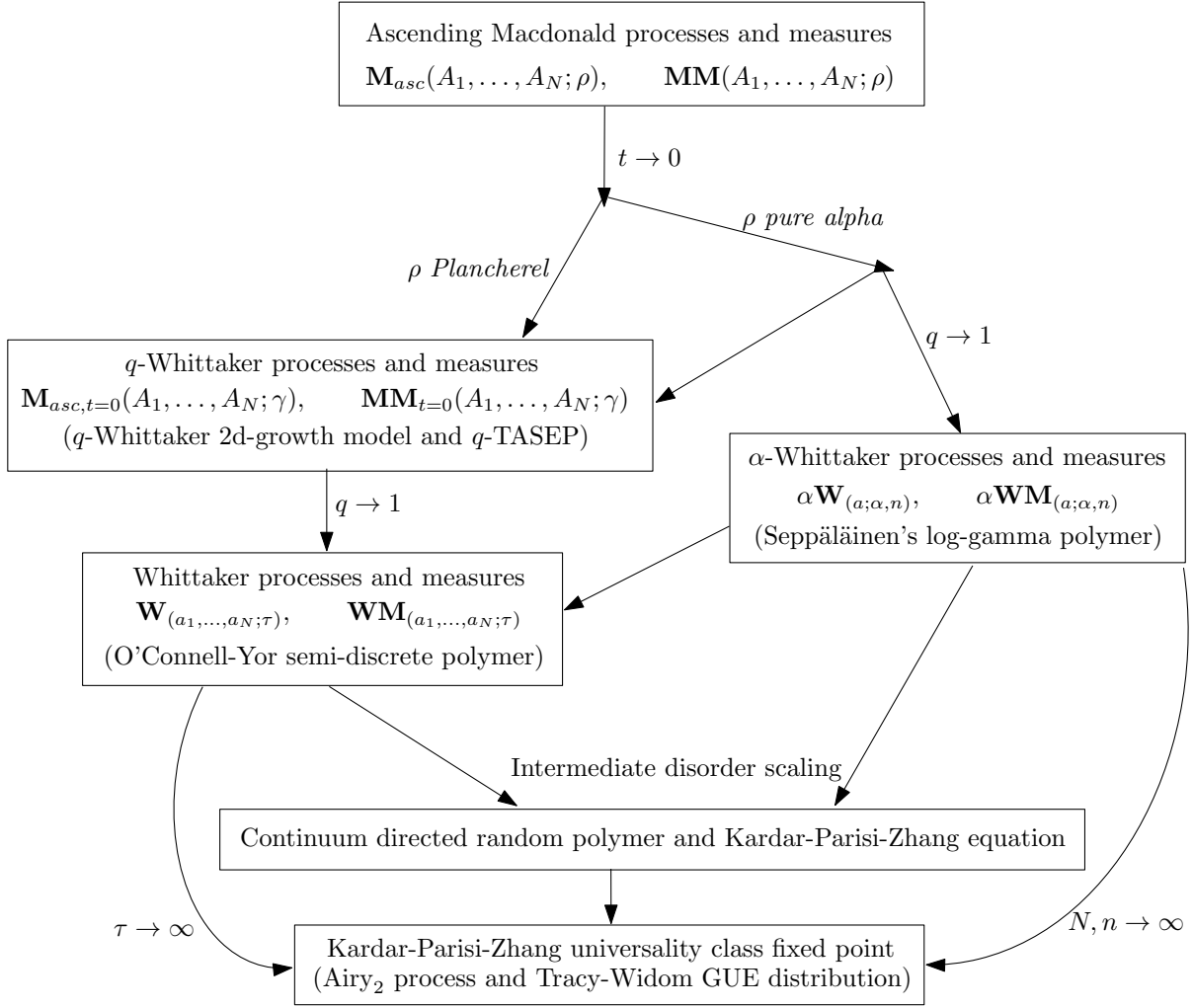}
\end{center}
\caption{A flowchart for Macdonald processes and some of their specializations and limits. Taking $t\to 0$ we focus on two specializations: {\it Plancherel} (primarily) and {\it pure alpha}. The pure alpha case degenerates to the Plancherel as the number of alpha variables grows to infinity. In the Plancherel case we define $q$-Whittaker process and measures. Natural dynamics on Gelfand Tsetlin patterns which preserve these processes are given by the $q$-Whittaker 2d-growth model. A marginal of these dynamics is the $q$-TASEP. Taking $q\to 0$ yields the Whittaker processes (from the Plancherel specialization) and the $\alpha$-Whittaker processes (from the pure alpha specialization). Again, taking $n\to \infty$ takes the $\alpha$-Whittaker processes to the Whittaker processes. These processes now encode the partition functions of directed polymers -- the O'Connell-Yor semi-discrete polymer and Sepp\"{a}l\"{a}inen's log-gamma polymer (respectively). There are two natural scaling limits for these polymers. The first is the intermediate disorder scaling in which the polymer inverse temperature is taken to zero as the other parameters go to infinity. In both cases, the polymer converge to the continuum directed random polymer, whose free energy is the solution to the Kardar-Parisi-Zhang stochastic PDE. The other scaling limit is the strong disorder scaling in which inverse temperature is fixed and positive and the other parameters are taken to infinity. Now the polymer free energy converges to the Tracy-Widom GUE statistics of the Kardar-Parisi-Zhang universality class fixed point.}\label{flowchart}
\end{figure}

\section{Further developments}
Since this work was posted we have developed a few of the directions of research coming from the theory of Macdonald processes. We briefly recount the main points of these developments below.

In joint work with Patrik Ferrari \cite{BorCorFer} we consider two models for directed polymers in space-time independent random media (the O'Connell-Yor semi-discrete directed polymer and the continuum directed random polymer) at positive temperature and prove their KPZ universality via asymptotic analysis of exact Fredholm determinant formulas for the Laplace transform of their partition functions. In particular, we show that for large time $\tau$, the probability distributions for the free energy fluctuations, when rescaled by $\tau^{1/3}$, converges to the GUE Tracy-Widom distribution. This completes Theorem \ref{TWasymptoticskappaTHM} to all $\kappa>0$.

We also consider the effect of boundary perturbations to the quenched random media on the limiting free energy statistics. For the semi-discrete directed polymer, when the drifts of a finite number of the Brownian motions forming the quenched random media are critically tuned, the statistics are instead governed by the limiting Baik-Ben Arous-P\'{e}ch\'{e} distributions \cite{BBP} from spiked random matrix theory. For the continuum polymer, the boundary perturbations correspond to choosing the initial data for the stochastic heat equation from a particular class, and likewise for its logarithm -- the KPZ equation. The Laplace transform formula we prove can be inverted to give the one-point probability distribution of the solution to these stochastic PDEs for the class of initial data.

In joint work with Patrik Ferrari and Balint Veto \cite{BorCorFerVet} we further extend the class of initial data for which we can compute exact statistics so as to include the equilibrium initial data for the KPZ equation. That is to say, we are able to exactly characterize the distribution of the solution to the KPZ equation when started with a two-sided Brownian motion and prove that as $t$ goes to infinity and under $t^{1/3}$ scaling, the one-point probability distribution converges to the $F_0$ distribution of Baik and Rains. The $t^{1/3}$ scaling for fluctuations associated with this initial data was previously shown in \cite{BQS}.

In joint work with Daniel Remenik \cite{BorCorRem} we prove that under $n^{1/3}$ scaling, the limiting distribution as $n\to \infty$ of the free energy of Sepp\"{a}l\"{a}inen's log-Gamma discrete directed polymer is GUE Tracy-Widom. The main technical innovation we provide is a general identity between a class of $n$-fold contour integrals and a class of Fredholm determinants. Applying this identity to the integral formula proved in \cite{COSZ} for the Laplace transform of the log-Gamma polymer partition function, we arrive at a Fredholm determinant which lends itself to asymptotic analysis (and thus yields the free energy limit theorem). The Fredholm determinant was anticipated via the formalism of Macdonald processes yet its rigorous proof was so far lacking because of the nontriviality of certain decay estimates required by that approach (see Theorem \ref{alphalimitthm}).

In joint work with Tomohiro Sasamoto \cite{BorCorSas} we prove duality relations for two interacting particle systems: $q$-TASEP and ASEP. Expectations of the duality functionals correspond to certain joint moments of particle locations or integrated currents, respectively. Duality implies that they solve systems of ODEs. These systems are integrable and for particular step and half stationary initial data we use a nested contour integral ansatz to provide explicit formulas for the systems' solutions and hence also the moments.

We form Laplace transform like generating functions of these moments and via residue calculus we compute two different types of Fredholm determinant formulas for such generating functions. For ASEP, the first type of formula is new and readily lends itself to asymptotic analysis (as necessary to reprove GUE Tracy-Widom distribution fluctuations for ASEP), while the second type of formula is recognizable as closely related to Tracy and Widom's ASEP formula \cite{TW1}-\cite{TW4}. For $q$-TASEP both formulas coincide with those computed via the theory of Macdonald processes (in particular see Section \ref{qWhitSecFormulas} and \ref{qTASEPandthe} below).

Both $q$-TASEP and ASEP have limit transitions to the free energy of the continuum directed polymer, the logarithm of the solution of the stochastic heat equation, or the Hopf-Cole solution to the Kardar-Parisi-Zhang equation. Thus, $q$-TASEP and ASEP are integrable discretizations of these continuum objects; the systems of ODEs associated to their dualities are deformed discrete quantum delta Bose gases; and the procedure through which we pass from expectations of their duality functionals to characterizing generating functions is a rigorous version of the replica trick in physics.

In joint work with Vadim Gorin and Shamil Shakirov \cite{BorCorGorSha} we extend many of the algebraic results developed here. In particular, using a certain operator diagonalized by the Macdonald polynomials, we generalize Theorem \ref{PlancherelfredThmintro} by showing that even when $t\neq 0$, the expectation of a certain observable of the Macdonald process is written in terms of a similar type of Fredholm determinant formula. We also extend the moment computation of Section \ref{difopSEC} to multilevel moment formulas. As a consequence we can compute joint moments for $\{q^{\lambda^{N_i}_{N_i}}\}$ as $i$ varies (when the Macdonald parameter $t=0$). Via the connection to $q$-TASEP, this provides a purely algebraic derivation of the multipoint formulas for $q$-TASEP found in \cite{BorCorSas} via duality.

\chapter{Macdonald processes}

\section{Definition and properties of Macdonald symmetric functions}\label{MacDefSEC}

So as to make our work self-contained we provide a brief review of all of the properties and concepts related to Macdonald symmetric functions which we will be utilizing. Our main reference for this material is the book of Macdonald \cite{M}.

\subsection{Partitions, Young diagrams and tableaux}
A {\it partition} \index{partition} is a sequence $\lambda = (\lambda_1,\lambda_2,\ldots)$ \glossary{$\lambda$} of nonnegative integers such that $\lambda_1\geq~\lambda_2\geq~\cdots$. The {\it length} \index{partition!length} $\ell(\lambda)$ is the number of non-zero $\lambda_i$ and the {\it weight} \index{partition!weight} $|\lambda|=\lambda_1+\lambda_2+\cdots$. If $|\lambda|=n$ then {\it $\lambda$ partitions $n$}. An alternative notation is $\lambda = 1^{m_1}2^{m_2}\cdots$ where $m_i$ represents the multiplicity \index{partition!multiplicity} of $i$ in the partition $\lambda$. The natural (partial) ordering on the space of partitions is called the {\it dominance order} \index{partition!partial ordering} and is given by $\lambda\geq \mu$ if and only if
\begin{equation*}
\lambda_1+\cdots +\lambda_i \geq \mu_1+\cdots +\mu_i\qquad \textrm{ for all } i\geq 1.
\end{equation*}

A partition $\lambda$ can be graphically represented as a {\it Young diagram} \index{Young diagram} with $\lambda_1$ left justified boxes in the top row, $\lambda_2$ in the second row, and so on. Thus $m_i$ represents the number of rows with exactly $i$ boxes. Define $\Y$ \glossary{$\Y$} to be the set of all partitions. The transpose of a Young diagram is denoted $\lambda'$ and defined by the property $\lambda_i' = |\{j:\lambda_j\geq i\}|$. Some of these concepts are illustrated in Figure \ref{diagramfig}.

Given two diagrams $\lambda$ and $\mu$ such that $\lambda\supset \mu$ (as a set of boxes), we call the set-difference $\theta = \lambda -\mu$ a {\it skew Young diagram}\index{Young diagram!skew}. A skew Young diagram $\theta$ is a {\it horizontal $m$ strip}\index{Young diagram!horizontal $m$ strip} if $|\theta|=m$ and if in each column, $\theta$ has at most one box.
A {\it column-strict (skew) Young tableaux}\index{Young tableaux!column-strict} $T$ is a sequence of partitions $\lambda^{(i)}$:
\begin{equation}\label{tableaux}
\mu = \lambda^{(0)} \subset \lambda^{(1)} \subset \cdots \subset \lambda^{(r)} = \lambda
\end{equation}
such that $\theta^{(i)} = \lambda^{(i)}-\lambda^{(i-1)}$ are all horizontal strips (For such a pair of partitions we also write $\lambda^{(i-1)}\prec \lambda^{(i)}$ \glossary{$\prec$}). If $\mu=\varnothing$ then this is also known of as a {\it semi-standard Young tableaux}\index{Young tableaux!semi-standard} and if one fills each skew diagram $\theta^{(i)}$ with the label $i$, then these numbers must be strictly increasing in each column and weakly increasing in each row. The {\it shape}\index{Young tableaux!shape} of $T$ is $\lambda-\mu$, and $(|\theta^{(1)}|,|\theta^{(2)}|,\ldots, |\theta^{(r)}|)$ is the {\it weight}\index{Young tableaux!weight} of $T$. One may likewise define a {\it vertical $m$ strip}\index{Young diagram!vertical $m$ strip} and {\it row-strict (skew) Young tableaux}\index{Young tableaux!row-strict} by transposing the role of rows and columns.

We occasionally will use the notation $\lambda \cup \square_k$ \glossary{$\lambda \cup \square_k$} which is the Young diagram formed by appending an additional square to row $k$ of $\lambda$.

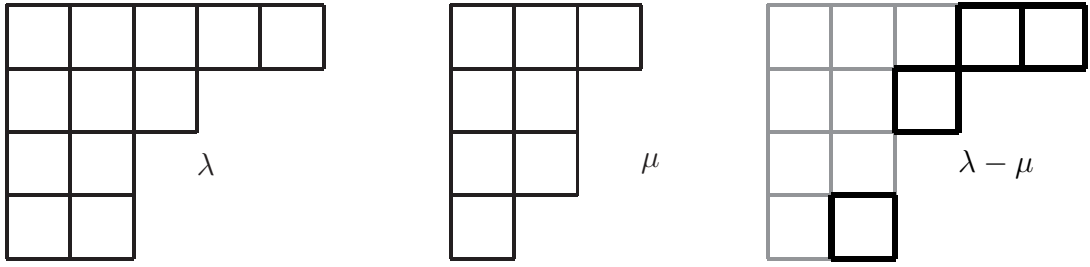
\begin{figure}

\setlength{\unitlength}{1.2pt}
\begin{picture}(200,100)(-30,0)
\linethickness{1pt}

\put(0,10){\line(1,0){40}}
\put(0,30){\line(1,0){40}}
\put(0,50){\line(1,0){60}}
\put(0,70){\line(1,0){100}}
\put(0,90){\line(1,0){100}}

\put(0,10){\line(0,1){80}}
\put(20,10){\line(0,1){80}}
\put(40,10){\line(0,1){80}}
\put(60,50){\line(0,1){40}}
\put(80,70){\line(0,1){20}}
\put(100,70){\line(0,1){20}}

\put(60,40){\makebox(0,0)[l]{$\lambda$}}

\put(140,10){\line(1,0){20}}
\put(140,30){\line(1,0){40}}
\put(140,50){\line(1,0){40}}
\put(140,70){\line(1,0){60}}
\put(140,90){\line(1,0){60}}

\put(140,10){\line(0,1){80}}
\put(160,10){\line(0,1){80}}
\put(180,30){\line(0,1){60}}
\put(200,70){\line(0,1){20}}

\put(200,40){\makebox(0,0)[l]{$\mu$}}

\color{Gray}
\linethickness{1pt}
\put(240,10){\line(1,0){40}}
\put(240,30){\line(1,0){40}}
\put(240,50){\line(1,0){60}}
\put(240,70){\line(1,0){100}}
\put(240,90){\line(1,0){100}}

\put(240,10){\line(0,1){80}}
\put(260,10){\line(0,1){80}}
\put(280,10){\line(0,1){80}}
\put(300,50){\line(0,1){40}}
\put(320,70){\line(0,1){20}}
\put(340,70){\line(0,1){20}}

\color{black}
\linethickness{2pt}
\put(260,10){\line(1,0){20}}
\put(260,10){\line(0,1){20}}
\put(260,30){\line(1,0){20}}
\put(280,10){\line(0,1){20}}

\put(280,50){\line(1,0){20}}
\put(280,50){\line(0,1){20}}
\put(280,70){\line(1,0){20}}
\put(300,50){\line(0,1){20}}

\put(300,70){\line(1,0){40}}
\put(300,70){\line(0,1){20}}
\put(300,90){\line(1,0){40}}
\put(320,70){\line(0,1){20}}
\put(340,70){\line(0,1){20}}

\put(300,40){\makebox(0,0)[l]{$\lambda-\mu$}}

\end{picture}
\caption{The Young diagram $\lambda=(5,3,2,2)$ and its transpose (not shown) $\lambda'=(4,4,2,1,1)$. The length $\ell(\lambda)=4$ and weight $|\lambda|=12$. The Young diagram $\mu=(3,2,2,1)$ is such that $\lambda\supset \mu$. The skew Young diagram $\lambda - \mu$ is shown in black thick lines and is a horizontal 4-strip.}\label{diagramfig}
\end{figure}

\subsection{Symmetric functions}
The ring of polynomials in $n$ independent variables with rational coefficients is denoted $\Q[x_1,\ldots, x_n]$. The symmetric group $S_n$ \glossary{$S_n$} acts on such polynomials by permuting the variable labels. A {\it symmetric polynomial} \index{symmetric polynomial} is any polynomial in $\Q[x_1,\ldots, x_n]$ which is invariant under this action. The symmetric polynomials form a subring and additionally have the structure of a graded ring
\begin{equation*}
\Sym_n = \Q[x_1,\ldots, x_n]^{S_n} = \bigoplus_{k\geq 0} \Sym_n^k
\end{equation*}
where $\Sym_n^k$ consists of all homogeneous symmetric polynomials of degree $k$.

For $\alpha = (\alpha_1,\ldots,\alpha_n)\in \Zgeqzero^n$ we write $x^{\alpha} = x_1^{\alpha_1}\cdots x_n^{\alpha_n}$. A permutation $\pi\in S_n$ acts on a vector $\alpha$ by permuting its indices. For any partition $\lambda$ such that $\ell(\lambda)\leq n$, the {\it monomial symmetric polynomial}\index{symmetric polynomial!monomial} $m_{\lambda}$\glossary{$m_{\lambda}$} is defined as
\begin{equation*}
m_{\lambda}(x_1,\ldots, x_n) = \sum x^{\pi (\lambda)},
\end{equation*}
where the summation is over all $\pi\in S_n$ yielding unique monomials $x^{\pi(\lambda)}$. The collection of $m_{\lambda}$, for $\lambda$ running over all partitions of length less than or equal to $n$, form a $\Q$-basis for $\Sym_n$. Restricting to $\lambda$ such that $|\lambda|=k$, the $m_{\lambda}$ form a $\Q$-basis of $\Sym_n^k$.

It is often more convenient to work with an infinite number of independent variables. This is achieved \cite[I.2]{M} in the following way: for $m\geq n$, $\rho_{m,n}$ is a homomorphism from $\Sym_m$ to $\Sym_n$ (here $m\geq n$) defined by taking a symmetric polynomial $f(x_1,\ldots,x_n,\ldots, x_m)$ to $f(x_1,\ldots,x_n,0,\ldots,0)$. Restricted to the monomial symmetric polynomial basis, $\rho_{m,n}$ takes $m_{\lambda}(x_1,\ldots,x_m)$ to $m_{\lambda}(x_1,\ldots,x_n)$ if $\ell(\lambda)\leq n$ and to 0 otherwise. The homomorphism is surjective, as is its restriction $\rho^{k}_{m,n}$ from $\Sym^k_m$ to $\Sym^k_{n}$.

We form the inverse limit $\Sym^k = \lim_{\leftarrow} \Sym^k_n$ of the $\Q$-modules $\Sym^k_n$ relative to the homomorphisms $\rho^k_{m,n}$. $\Sym^k$ has a $\Q$-basis consisting of the {\it monomial symmetric functions} $m_{\lambda}$ (for all $\lambda$ such that $|\lambda|=k$) defined by $\rho^k_{n}(m_{\lambda}) = m_{\lambda}(x_1,\ldots,x_n)$ for all $n\geq k$.

We define $\Sym = \bigoplus_{k\geq 0} \Sym^k$\glossary{$\Sym$} to be the {\it ring of symmetric functions}\index{symmetric functions!ring} in countably many independent variables $x_1,x_2,\ldots$. There exist surjective homomorphisms $\rho_n$ from $\Sym$ to $\Sym_n$ which act by restricting $m_{\lambda}$ to $m_{\lambda}(x_1,\ldots, x_n)$ if $|\lambda|\leq n$, or otherwise zero. These are ring homomorphisms and thus give $\Sym$ the structure of a graded ring. We will use the term {\it symmetric function} when dealing with elements of $\Sym$ (as they are not, in fact, polynomials but rather formal infinite sums of monomials) and {\it polynomial} when dealing with elements of $\Sym_n$ for some $n$.

Besides the monomial symmetric functions \index{symmetric functions!monomial}, there exist a number of other bases for $\Sym$. The most simple among these are:
\begin{itemize}
\item The {\it elementary symmetric functions}\index{symmetric functions!elementary} $e_r:= m_{1^{r}}$ (recall that $1^r$ represents the partition $(1,1,\ldots, 1)$ with $r$ entries). Then $\Sym=\textrm{span}(e_{\lambda}|\lambda\in \Y)$ with $e_{\lambda} = e_{\lambda_1}e_{\lambda_2} \cdots$\glossary{$e_\lambda$}.
\item The {\it complete homogeneous symmetric functions}\index{symmetric functions!complete homogeneous} $h_r:=\sum_{|\lambda|=r} m_{\lambda}$. Then setting $h_{\lambda} = h_{\lambda_1}h_{\lambda_2} \cdots$ \glossary{$h_\lambda$}, $\Sym=\textrm{span}(h_{\lambda}|\lambda\in \Y)$.
\item The {\it power sum symmetric functions}\index{symmetric functions!power sum} $p_r:=m_{r}$. Since $r$ represents the partition $\lambda=(r)$ this translates formally into $p_r=\sum x_i^r$. From power sums we form $p_{\lambda}=\prod_{i=1}^{\ell(\lambda)} p_{\lambda_i}$\glossary{$p_\lambda$}. Then $\Sym=\textrm{span}(p_{\lambda}|\lambda\in \Y)$.
\end{itemize}

Less obvious symmetric functions include those of {\it Schur}\index{symmetric functions!Schur}, {\it Hall-Littlewood}\index{symmetric functions!Hall-Littlewood}, and {\it Jack}\index{symmetric functions!Jack}. They are all indexed by partitions and share the characteristic that they can be uniquely defined (via Gram-Schmidt -- although the existence of the needed basis is nontrivial as we use partial order) from the following two properties: (a) They can be expressed in terms of the monomial symmetric functions via a strictly upper unitriangular transition matrix (for instance the Schur functions $s_{\lambda}=m_{\lambda} + \sum_{\mu<\lambda \in\Y} K_{\lambda \mu}m_{\mu}$ where $K_{\lambda \mu}$ are the Kostka numbers); (b) They are pairwise orthogonal with respect to a scalar product under which the power sum symmetric functions are orthogonal. The specific scalar product differs between the functions. For the Schur functions it is defined by $\langle p_{\lambda}, p_{\mu}\rangle = \delta_{\lambda\mu} z_{\lambda}$ where $z_{\lambda} = \prod_{i\geq 1}i^{m_{i}}(m_i)!$ (here $\lambda = 1^{m_1}2^{m_2}\cdots$) and where $\delta_{\lambda \mu}$ is the indicator  function that $\lambda=\mu$. For the Hall-Littlewood functions $z_{\lambda}$ is replaced by $z_{\lambda}(t) = z_{\lambda} \prod_{i=1}^{\ell(\lambda)} (1-t^{\lambda_i-1})$, where $t\in (0,1)$ -- and now the coefficients are in $\Q[t]$. For Jack's symmetric functions $z_{\lambda}$ is replaced by $z_{\lambda}\alpha^{\ell(\lambda)}$ for $\alpha>0$ -- and now the coefficients are in $\Q[\alpha]$.

In fact, the symmetric functions of Schur, Hall-Littlewood, and Jack are all generalized by the symmetric functions of Macdonald \cite[Section VI]{M}.

\subsection{Macdonald symmetric functions}\label{MacSymFnSec}

We now introduce the Macdonald symmetric functions\index{Macdonald symmetric functions}\index{symmetric functions!Macdonald} and identify a number of relevant properties which we will appeal to in our development of the Macdonald process. Macdonald symmetric functions are indexed by partitions and denoted $P_{\lambda}(x;q,t)$ where $q,t\in (0,1)$. The coefficients for the symmetric functions are now in $\Q[q,t]$. We will generally suppress the $q$ and $t$ and write $P_{\lambda}(x)$\glossary{$P_{\lambda}(x)$, $P_{\lambda}$} or $P_{\lambda}$ unless their presence is pertinent. $P_{\lambda}$ can be uniquely defined from the following two properties \cite[VI,(4.7)]{M}: (a) They can be expressed in terms of the monomial symmetric functions via a strictly upper unitriangular transition matrix:
\begin{equation*}
P_{\lambda}=m_{\lambda} + \sum_{\mu<\lambda \in \Y} R_{\lambda \mu}P_{\mu},
\end{equation*}
where $R_{\lambda \mu}$ are functions of $q,t$. (b) They are pairwise orthogonal with respect to a scalar product which can be defined on the power sum symmetric functions (since they form a linear basis for the symmetric functions) via
\begin{equation*}
\langle p_{\lambda},p_{\mu}\rangle = \langle p_{\lambda},p_{\mu}\rangle_{q,t} = \delta_{\lambda \mu}z_{\lambda}(q,t),\qquad z_{\lambda}(q,t)= z_{\lambda} \prod_{i=1}^{\ell(\lambda)}\frac{1-q^{\lambda_i}}{1-t^{\lambda_i}}, \qquad z_{\lambda} = \prod_{i\geq 1}i^{m_{i}}(m_i)!,
\end{equation*}
\index{Macdonald symmetric functions!scalar product}
for $\lambda=1^{m_1}2^{m_2}\cdots$.
%We will always use upper-case $P$ for the Macdonald symmetric functions, and lower-case $p$ for power sums (which will be used some, but much less prominently).
Along with $P_{\lambda}$ one defines
\begin{equation*}
Q_{\lambda} = \frac{P_{\lambda}}{\langle P_{\lambda},P_{\lambda}\rangle},
\end{equation*}
\glossary{$Q_{\lambda}$}so that $P_{\lambda}$ and $Q_{\mu}$ are orthonormal.

Specializing $q=t$ recovers the Schur symmetric functions, $q=0$ recovers the Hall-Littlewood symmetric functions with parameter $t$, and taking $q=t^{\alpha}$ with $t\rightarrow 1$ recovers the Jack symmetric functions with parameter $\alpha$.

The complete homogeneous symmetric function $h_{r}$ has a $(q,t)$-analog\index{symmetric functions!complete homogeneous $(q,t)$ analog} which is denoted $g_r=Q_{(r)}$\glossary{$g_r$} and can be expressed as
$g_r = \sum_{|\lambda|=r} z_{\lambda}(q,t)^{-1}p_{\lambda}$ (this is analogous in the sense that $h_{r} = s_{(r)}$). These also form an algebraic basis for $\Sym$ \cite[VI,(2.19)]{M}.

The Macdonald symmetric polynomial\index{Macdonald symmetric functions!polynomial} is defined as the restriction of $P_{\lambda}$ onto a finite number of variables $x_1,\ldots, x_m$. Formally, since $P_{\lambda}\in \Sym^{|\lambda|}$, for any $m\geq 0$, the Macdonald polynomial in $m$ variables is written as $P_{\lambda}(x_1,\ldots, x_m)$ and is the projection of $P_{\lambda}$ into $\Sym^{|\lambda|}_{m}$. If $m<\ell(\lambda)$ then $P_{\lambda}(x_1,\ldots, x_m)=0$.

Macdonald symmetric polynomials demonstrate the following index shift property \cite[VI,(4.17)]{M}:
\begin{equation}\label{macLaurent}
P_{\lambda}(x_1,\ldots, x_{\ell(\lambda)}) = x_1\cdots x_n P_{\mu}(x_1,\ldots,x_{\ell(\lambda)}), \qquad \mu=(\lambda_1-1,\cdots \lambda_{\ell(\lambda)}-1).
\end{equation}
Owing to this property we can extend Macdonald symmetric polynomials so as to be defined for arbitrary integer values of $\lambda_1\geq \lambda_2\geq \cdots \lambda_\ell(\lambda)$ via iterating the above relation. For negative values of $\lambda_j$'s, $P_{\lambda}$ becomes a Laurent polynomial in the $x$'s.

For $u,v\in(0,1)$ define a $\Q[q,t]$-algebra endomorphism  $\omega_{u,v}$\glossary{$\omega_{u,v}$} on $\Sym$ in terms of its action on power sums \cite[VI,(2.14)]{M}:
\begin{equation}\label{omegainv}
\omega_{u,v}(p_r) = (-1)^{r-1} \frac{1-u^r}{1-v^r} p_{r}.
\end{equation}
Clearly $\omega_{v,u}w_{u,v}$ is the identity \index{Macdonald symmetric functions!endomorphism} on $\Sym$. Moreover, $\omega_{u,v}$ acts on Macdonald symmetric functions as \cite[VI,(5.1)]{M}
\begin{equation}\label{omegaonPQ}
\omega_{q,t} P_{\lambda}(x;q,t) = Q_{\lambda'}(x;t,q), \qquad \omega_{q,t} Q_{\lambda}(x;q,t) = P_{\lambda'}(x;t,q).
\end{equation}
This endomorphism takes $g_r$ to $e_r$, the (usual) elementary symmetric functions.

\subsection{Cauchy identity}\label{PiSec}
\index{Macdonald symmetric functions!Cauchy identity}
For any two sequences of independent variables $x_1,x_2,\ldots$ and $y_1,y_2,\ldots$ define
\begin{equation}\label{PiDef1}
\Pi(x;y)=\sum_{\lambda\in \Y} P_{\lambda}(x) Q_{\lambda}(y).
\end{equation}
\glossary{$\Pi(x;y)$}
Then \cite[VI,(2.5)]{M},
\begin{equation}\label{eqn12}
\Pi(x;y)= \prod_{i,j} \frac{(tx_i y_j;q)_{\infty}}{(x_i y_j;q)_{\infty}},
\end{equation}
where $(a;q)_{\infty}=(1-a)(1-aq)(1-aq^2)\cdots$ is the $q$-Pochhammer symbol (see Section~\ref{qSec}). Additionally,
\begin{equation*}
\Pi(x;y)= \exp\left(\sum_{n\geq 1} \frac{1}{n} \frac{1-t^n}{1-q^n} p_n(x)p_n(y)\right).
\end{equation*}
If the $P_{\lambda}$ and $Q_{\lambda}$ are considered as symmetric functions variables $x_1,x_2,\ldots$ and $y_1,y_2,\ldots$ (respectively) then the above identities are as formal powers.  If both sides can be evaluated as absolutely convergent series, then the identities are as numeric equalities. In particular if all but a finite number of the variables are finite (i.e., the case of Macdonald symmetric polynomials), the identities are necessarily numeric equalities. For symmetric functions there exists an extension to the concept of evaluating at an infinite sequence of variables. This is called {\it specializing} the functions and in Section~\ref{NNspecSEC} we will show how the Cauchy identity extends to this context. This will necessitate extending the notation $\Pi(x;y)$ to $\Pi(\rho_1,\rho_2)$ where $\rho_1$ and $\rho_2$ are specializations. When the specialization reduces to evaluation as here, the definition likewise reduces. Thus there should be no ambiguity in our usage of the symbol $\Pi$ in this context.

Note that applying the endomorphism $\omega_{q,t}$ to $\Pi(x;y)$ gives
\begin{equation}\label{invPI}
\omega_{q,t} \Pi(x;y) = \prod_{i,j\geq 1} (1+x_i y_i).
\end{equation}

\subsection{Torus scalar product}\label{torusSec}
\index{Macdonald symmetric functions!torus scalar product}
The Macdonald symmetric polynomials in $N$ independent variables are orthogonal under another scalar product which is called the {\it torus scalar product} and denoted by $\langle \cdot,\cdot \rangle_N'$. It is defined \cite[VI,(9.10)]{M} as
\begin{equation*}
\langle f,g\rangle_N' := \frac{1}{(2\pi \iota)^N N!} \int_{\T^N} f(z)\overline{g(z)} \prod_{i\neq j=1}^{N} \frac{(z_i z_j^{-1};q)_{\infty}}{(tz_i z_j^{-1};q)_{\infty}} \prod_{i=1}^{N}\frac{dz_i}{z_i},
\end{equation*}
\glossary{$\langle f,g\rangle_N'$}
where $\T^N$ is the $N$-fold product of the torus $\{z=e^{2\pi \iota \theta}\}$.

Under this scalar product we can compute $\langle P_{\lambda}, P_{\lambda} \rangle_N'$ where we interpret $P_{\lambda}$ as the Macdonald symmetric polynomial with $N$ variables \cite[VI,(9, Ex 1d)]{M}
\begin{equation}\label{9ex1d}
\langle P_{\lambda}, P_{\lambda} \rangle_N' = \prod_{1\leq i<j} \frac{ (q^{\lambda_i-\lambda_j}t^{j-i};q)_{\infty} (q^{\lambda_i-\lambda_j+1}t^{j-i};q)_{\infty}}{ (q^{\lambda_i-\lambda_j}t^{j-i+1};q)_{\infty} (q^{\lambda_i-\lambda_j+1}t^{j-i-1};q)_{\infty}}.
\end{equation}

It follows from equation (\ref{PiDef1}) that
\begin{equation*}
Q_{\lambda}(x) = \frac{1}{\langle P_{\lambda},P_{\lambda}\rangle_N'} \langle \Pi(\cdot;x),P_{\lambda}(\cdot)\rangle_N'
\end{equation*}
where the dot represents that the inner product is applied to the function which takes $z\mapsto \Pi(z;x)$. Note that above $z$ represents a $N$-vector $(z_1,\ldots,z_N)$.
 %while $x$ represents a finite (though arbitrary) length vector.

\subsection{Pieri formulas for Macdonald symmetric functions}\label{pieresec}

Recall that $g_r=Q_{(r)}$ is the $(q,t)$-analog of the complete homogeneous symmetric function, while $e_r$ is the (usual) elementary symmetric function. Since the Macdonald symmetric functions form a basis for $\Sym$ it follows that for each $\mu$ and $r$ there exist constants (depending on $q,t$) $\varphi_{\lambda/\mu}$, $\psi_{\lambda/\mu}$, $\varphi_{\lambda/\mu}'$ and $\psi_{\lambda/\mu}'$ such that:\index{Macdonald symmetric functions!Piere formulas}

\begin{equation}\label{piereEqn}
P_{\mu}g_r =\sum_{\lambda\in \Y} \varphi_{\lambda/\mu}P_{\lambda},\quad
Q_{\mu}g_r =\sum_{\lambda\in \Y} \psi_{\lambda/\mu}Q_{\lambda},\quad
Q_{\mu}e_r =\sum_{\lambda\in \Y} \varphi_{\lambda/\mu}'Q_{\lambda},\quad
P_{\mu}e_r =\sum_{\lambda\in \Y} \psi_{\lambda/\mu}'P_{\lambda}.
\end{equation}

Define $f(u) = (tu;q)_{\infty}/(qu;q)_{\infty}$. The coefficients above have exact formulas as follows \cite[VI,(6.24)]{M}: If $\lambda - \mu$ is a horizontal $r$-strip then
\begin{eqnarray}
\label{piereFormPhi}\varphi_{\lambda/\mu}&=&\prod_{1\le i\le j\le \ell(\lambda)}\frac{f(q^{\lambda_i-\lambda_j}t^{j-i})f(q^{\mu_i-\mu_{j+1}}t^{j-i})}
{f(q^{\lambda_i-\mu_j}t^{j-i})f(q^{\mu_i-\lambda_{j+1}}t^{j-i})},\\
\label{piereFormPsi}\psi_{\lambda/\mu}&=&\prod_{1\le i\le j\le \ell(\mu)}
\frac{f(q^{\mu_i-\mu_j}t^{j-i})f(q^{\lambda_i-\lambda_{j+1}}t^{j-i})}
{f(q^{\lambda_i-\mu_j}t^{j-i})f(q^{\mu_i-\lambda_{j+1}}t^{j-i})},
\end{eqnarray}
\glossary{$\varphi_{\lambda/\mu}$}\glossary{$\psi_{\lambda/\mu}$}\glossary{$\varphi'_{\lambda/\mu}$}\glossary{$\psi'_{\lambda/\mu}$}
otherwise the coefficient is zero; Applying the endomorphism $\omega_{q,t}$ implies that the coefficients $\varphi_{\lambda/\mu}'(q,t) = \varphi_{\lambda/\mu}(t,q)$ and $\psi_{\lambda/\mu}'(q,t) = \psi_{\lambda/\mu}(t,q)$. These coefficients are zero unless $\lambda-\mu$ is a vertical $r$-strip.

The expressions above can be reduced significantly. For example, given $\lambda-\mu$ a vertical $r$-strip,
\begin{equation}\label{piereFormPrime}
\psi'_{\lambda/\mu}=\prod_{\substack{i<j\\ \lambda_i=\mu_i,\lambda_j=\mu_{j}+1}}
\frac{(1-q^{\mu_i-\mu_j}t^{j-i-1})(1-q^{\lambda_i-\lambda_j}t^{j-i+1})}
{(1-q^{\mu_i-\mu_j} t^{j-i})(1-q^{\lambda_i-\lambda_j}t^{j-i})}.
\end{equation}

\subsection{Skew Macdonald symmetric functions}

Similar to equation (\ref{piereEqn}), for two partitions $\mu,\nu$ one can expand the product $P_{\mu}P_{\nu} = \sum_{\lambda\in \Y} f_{\mu\nu}^{\lambda} P_{\lambda}$. By consideration of degree, $f_{\mu\nu}^{\lambda}$ may only be non-zero if $|\lambda|=|\mu|-|\nu|$, $\lambda\supset \mu$ and $\lambda\supset \nu$. When $q=t$ these are called Littlewood-Richardson coefficients\index{Littlewood-Richardson coefficients}; when $q=0$ these are Hall polynomials \index{Hall polynomials} in $t$ \cite[VI,(7.2)]{M}. Alternatively, one can extract these coefficients via $f_{\mu\nu}^{\lambda} = \langle Q_{\lambda},P_{\mu}P_{\nu}\rangle$.

A {\it skew Macdonald symmetric function}\index{Macdonald symmetric functions!skew} is defined as \cite[VI,(7.5,7.6)]{M}
\begin{equation*}
Q_{\lambda/\mu} := \sum_{\nu\in \Y} f_{\mu\nu}^{\lambda} Q_{\nu}, \quad \textrm{so that } \langle Q_{\lambda/\mu},P_{\nu}\rangle = \langle Q_{\lambda},P_{\mu}P_{\nu}\rangle.
\end{equation*}
\glossary{$Q_{\lambda/\mu}$}
By linearity this implies that $\langle Q_{\lambda/\mu},f\rangle = \langle Q_{\lambda},P_{\mu} f\rangle$ for all $f\in \Sym$. These functions are zero unless $\lambda \supset \mu$, in which case $Q_{\lambda/\mu}$ is homogeneous of degree $|\lambda|-|\mu|$.

Likewise, we can define $P_{\lambda/\mu}$ so that $\langle P_{\lambda/\mu},Q_{\nu}\rangle = \langle P_{\lambda},Q_{\mu}Q_{\nu}\rangle$. Owing to the relationship $Q_{\lambda}=\langle~P_{\lambda},P_{\lambda}\rangle^{-1}P_{\lambda}$ it follows that
\begin{equation}\label{PQSKEW}
P_{\lambda/\mu} = \frac{\langle P_{\lambda},P_{\lambda}\rangle}{\langle P_{\mu},P_{\mu}\rangle} Q_{\lambda/\mu}.
\end{equation}
\glossary{$P_{\lambda/\mu}$}
It is possible to express skew Macdonald symmetric functions in terms of the $\varphi$ and $\psi$ of equations (\ref{piereFormPhi}) and (\ref{piereFormPsi}). Specifically let $T$ represent a column-strict (skew) tableaux of shape $\lambda-\mu$ given by a sequence of $\lambda^{(i)}$ as in equation (\ref{tableaux}). Let $\alpha$ be the weight of $T$ (i.e., $\alpha_i = |\lambda^{(i)}-\lambda^{(i-1)}|$) and define $x^T$ as $x^{\alpha}= x_1^{\alpha_1}x_2^{\alpha_2}\cdots$. Then we have the following {\it combinatorial formula}\index{Macdonald symmetric functions!combinatorial formula} expansion \cite[VI,(7.13)]{M}
\begin{equation}\label{7.13}
Q_{\lambda/\mu}(x) = \sum_{T} \varphi_T x^T,\qquad \textrm{where } \varphi_T = \prod_{i\geq 1} \varphi_{\lambda^{(i)}/\lambda^{(i-1)}},
\end{equation}
and the summation is over all $T$ which are column-strict (skew) tableaux of shape $\lambda-\mu$.

Likewise
\begin{equation}\label{7.13'}
P_{\lambda/\mu}(x) = \sum_{T} \psi_T x^T,\qquad \textrm{where } \psi_T = \prod_{i\geq 1} \psi_{\lambda^{(i)}/\lambda^{(i-1)}},
\end{equation}
and the summation is over all $T$ which are column-strict (skew) tableaux of shape $\lambda-\mu$.

If we restrict the Macdonald symmetric functions to Macdonald polynomials in a single variable $x_1$, then the above expansions imply
\begin{equation}\label{7.14}
Q_{\lambda/\mu}(x_1) = \varphi_{\lambda/\mu} x_1^{|\lambda-\mu|},
\end{equation}
if $\lambda-\mu$ is a horizontal strip, and zero otherwise. Likewise
\begin{equation}\label{7.14'}
P_{\lambda/\mu}(x_1) = \psi_{\lambda/\mu} x_1^{|\lambda-\mu|},
\end{equation}
if $\lambda-\mu$ is a horizontal strip, and zero otherwise.

Finally, we recount a few pertinent formulas involving these skew functions \cite[VI.7]{M}:
\begin{eqnarray}
\label{skewformulas1}\sum_{\lambda\in \Y} Q_{\lambda/\mu}(x) P_{\lambda}(y) &=& P_{\mu}(y) \Pi(x;y),\\
\label{skewformulas2}\sum_{\kappa\in \Y}P_{\kappa/\nu}(x)Q_{\kappa/\hat{\nu}}(y)  &=& \Pi(x;y) \sum_{\tau\in \Y} Q_{\nu/\tau}(y)P_{\hat{\nu}/\tau}(x),\\
\label{skewformulas3}\sum_{\nu\in \Y} Q_{\kappa/\nu}(x)Q_{\nu/\tau}(y) &=& Q_{\kappa/\tau}(x,y),\\
\label{skewformulas4}\sum_{\nu\in \Y} P_{\kappa/\nu}(x)P_{\nu/\tau}(y) &=& P_{\kappa/\tau}(x,y).
\end{eqnarray}
With regards to the last equations, the argument $(x,y)$ simply means the union of the two sets of variables. The function therefore is symmetric with respect to any permutation of the variables in this union. A consequence of this is that to evaluate a symmetric function $f$ at $(x,y)$ one can alternatively expand $f$ in terms of power sum symmetric functions and set $p_n(x,y) = p_n(x)+p_n(y)$. In Section~\ref{NNspecSEC} below we consider general specializations of symmetric functions for which one must take this power sum equality as the definition of the union of two specializations.

Similar to equation (\ref{omegaonPQ}), the endomorphism $\omega_{q,t}$ acts on skew Macdonald symmetric functions as \cite[VI,(7.16)]{M}
\begin{equation}\label{omegaSkew}
\omega_{q,t}P_{\lambda/\mu}(x;q,t) = Q_{\lambda'/\mu'}(x;t,q),\qquad \omega_{q,t}Q_{\lambda/\mu}(x;q,t) = P_{\lambda'/\mu'}(x;t,q).
\end{equation}

\section{The Macdonald processes}

\subsection{Macdonald nonnegative specializations of symmetric functions}\label{NNspecSEC}

A {\it specialization}\index{specialization} $\rho$ of $\Sym$ is an algebra homomorphism of $\Sym$ to $\C$. We denote the application of $\rho$ to $f\in\Sym$ as $f(\rho)$. The {\it trivial} specialization $\varnothing$\index{specialization!trivial} takes value 1 at the constant function $1\in\Sym$ and takes value $0$ at any homogeneous $f\in\Sym$ of degree $\ge 1$. For two specializations $\rho_1$ and $\rho_2$ we define their union $\rho=(\rho_1,\rho_2)$ as the specialization defined on power sum symmetric functions via
\begin{equation*}
p_n(\rho_1,\rho_2)=p_n(\rho_1)+p_n(\rho_2), \qquad n\ge 1,
\end{equation*}
and extended to $\Sym$ by linearity. Also, for $a>0$ define $a\cdot \rho$ as the specialization which takes homogeneous functions $f\in \Sym$ to $a^{{\rm degree}(f)} f(\rho)$, which we write as $f(a\cdot \rho)$.

An example of a specialization is the homomorphism which can be written as $f(x_1,\ldots, x_n)$. This represents restricting $f$ to $\Sym_n$ and then evaluating the resulting polynomial at the values $x_1,\ldots ,x_n$. We call this a {\it finite length} specialization\index{specialization!finite length}. Not all specializations are finite length. The class with which we work will contain more general specializations which can be thought of as unions of limits of such finite length specializations as well as limits of finite length {\it dual} specializations\index{specialization!dual}. A finite length dual specialization is obtained by a finite length specialization composed with the endomorphism $\omega_{q,t}$. All of the formulas involving finite length specializations from Section~\ref{MacDefSEC} likewise hold for general specializations.

Let $t$ and $q$ be two parameters in $(0,1)$, and let $\{P_\lambda\}$ be the corresponding Macdonald symmetric functions (see Section~\ref{MacSymFnSec}).

\begin{definition}
We say that a specialization $\rho$ of $\Sym$ is {\it Macdonald nonnegative} (or just `nonnegative') if it takes nonnegative values on the skew Macdonald symmetric functions: $P_{\lambda/\mu}(\rho)\ge 0$ for any partitions $\la$ and $\mu$.\index{specialization!Macdonald nonnegative}\index{specialization!nonnegative}
\end{definition}

There is no known classification of the nonnegative specializations. The classification is known in the case of nonnegative specializations of Jack's symmetric functions \cite{KOO} and in the subcase of Schur symmetric functions this is a classical statement known as ``Thoma's theorem'' (see \cite{Ker} and references therein). In the Macdonald case, however, it is not hard to come up with a class of examples. In fact, Kerov conjectured that the following class completely classifies all nonnegative specializations (\cite{Ker}, section II.9) -- though this has not been proved.

Let $\{\alpha_i\}_{i\ge 1}$, $\{\beta_i\}_{i\ge 1}$, and $\gamma$ be nonnegative numbers, and $\sum_{i=1}^\infty(\alpha_i+\beta_i)<\infty$. Let $\rho$ be a specialization of $\Sym$ defined by
\begin{equation}\label{tag1}
\sum_{n\ge 0} g_n(\rho) u^n= \exp(\gamma u) \prod_{i\ge 1} \frac{(t\alpha_iu;q)_\infty}{(\alpha_i u;q)_\infty}\,(1+\beta_i u)=: \Pi(u;\rho).
\end{equation}
\glossary{$\Pi(u;\rho)$}
Here $u$ is a formal variable and $g_n=Q_{(n)}$ is the $(q,t)$-analog of the complete homogeneous symmetric function $h_n$. Since $g_n$ forms a $\Q[q,t]$ basis of $\Sym$, this uniquely defines the specialization $\rho$. The expression in equation (\ref{tag1}) which we write as $\Pi(u;\rho)$ is a special case of equation (\ref{PIeqn}) for $\rho_1$ equal to the finite length specialization to a single variable $u$.

Notice also that if $\rho$ is specified by nonnegative numbers $\{\alpha_i\}_{i\ge 1}$, $\{\beta_i\}_{i\ge 1}$, and $\gamma$, \glossary{$\alpha_i$} \glossary{$\beta_i$}\glossary{$\gamma$} then $a\cdot \rho$ is likewise specified by $\{a\alpha_i\}_{i\ge 1}$, $\{a\beta_i\}_{i\ge 1}$, and $a\gamma$ as follows from the calculation
\begin{equation*}
\sum_{n\geq 0} g_{n}(a\cdot \rho) u^n = \sum_{n\geq 0} g_{n}(\rho) (au)^n.
\end{equation*}

\begin{proposition}
For any nonnegative $\{\alpha_i\}_{i\ge 1}$, $\{\beta_i\}_{i\ge 1}$, and $\gamma$  such that $\sum_{i=1}^\infty(\alpha_i+\beta_i)~<~\infty$, the specialization $\rho$ is Macdonald nonnegative.
\end{proposition}

\begin{proof}It suffices to verify the statement for finitely many nonzero $\alpha_i$'s and $\beta_i$'s. One then obtains infinitely many ones by a limit transition, and also produces a nontrivial $\gamma$ by taking $M$ of the $\beta_i$'s equal to
$\gamma/M$ and sending $M$ to infinity.

By using the fact that
\begin{equation}\label{tag1a}
P_{\lambda/\mu}(\rho_1,\rho_2)=\sum_{\nu\in \Y}
P_{\lambda/\nu}(\rho_1)P_{\nu/\mu}(\rho_2),
\end{equation}
which follows from equation (\ref{skewformulas4}), we reduce the statement to $\rho$'s with finitely many nontrivial $\alpha_i$'s or finitely many nontrivial $\beta_i$'s. Applying the endomorphism $\omega_{q,t}$ given in equation (\ref{omegainv}) we see that by virtue of equations (\ref{invPI}), (\ref{PQSKEW}) and (\ref{omegaSkew}), we only need to consider one of these two cases.

On the other hand, if we only have finitely many positive $\alpha_j$'s, $P_{\lambda/\mu}$ are simply skew Macdonald polynomials in variables $\alpha_1,\alpha_2,\dots$. The tableaux expansion of equation (\ref{7.13'}) then allows us to conclude the proof.
\end{proof}

We name a few particularly useful specializations:
\begin{definition}\label{specializationtype}
A Macdonald nonnegative specialization $\rho$ is called
\begin{itemize}
\item {\it Plancherel}\index{specialization!Plancherel} if $\alpha_i=\beta_i=0$ for all $i$ and $\gamma>0$;
\item {\it Pure alpha}\index{specialization!pure alpha} if $\beta_i=0$ for all $i$, $\gamma=0$ and at least one $\alpha_i>0$;
\item {\it Pure beta}\index{specialization!pure beta} if $\alpha_i=0$ for all $i$, $\gamma=0$  and at least one $\beta_i>0$.
\end{itemize}
\end{definition}

For a nonnegative specialization $\rho$, denote by $\Y(\rho)$\glossary{$\Y(\rho)$} the set of partitions (or Young diagrams) $\lambda$ such that
$P_\lambda(\rho)>0$. We also call $\Y(\rho)$ the {\it support}\index{specialization!support} of  $\rho$ (recall the set of all partitions is denoted as $\Y$).

Using the combinatorial formula for the Macdonald symmetric functions in equation (\ref{7.13}) and the endomorphism $\omega_{q,t}$, it is not hard to show that if a nonnegative specialization $\rho$ is defined as in equation (\ref{tag1}) with $\gamma=0$, $p<\infty$ nonzero $\alpha_j$'s and $q<\infty$ nonzero $\beta_j$'s, then $\Y(\rho)$
consists of the Young diagrams that fit into the $\Gamma$-shaped figure with $p$ rows and $q$ columns. Otherwise it is easy to see that $\Y(\rho)=\Y$.

In particular, if in equation (\ref{tag1}) all $\beta_j$'s and $\gamma$ vanish, and there are $p$ nonzero $\alpha_j$'s, then $\Y(\rho)$ consists of Young diagrams with no more than $p$ rows. Such a finite length specialization consists in assigning values $\alpha_j$ to $p$ of the symmetric variables used to define $\Sym$, and $0$'s to all the other symmetric variables.

\subsection{Definition of Macdonald processes}\label{MacProcessSec}
Fix a natural number $N$ and nonnegative specializations $\rho_0^+,\dots,\rho_{N-1}^+$, $\rho_1^-,\dots,\rho_N^-$ of $\Sym$.
For any sequences $\la=(\la^{(1)},\dots,\la^{(N)})$ and $\mu=(\mu^{(1)},\dots,\mu^{(N-1)})$ of partitions satisfying
\begin{equation}\label{tag4}
\varnothing\subset \la^{(1)}\supset \mu^{(1)}\subset
\la^{(2)}\supset \mu^{(2)} \subset \dots \supset \mu^{(N-1)}\subset
\la^{(N)}\supset \varnothing
\end{equation}
define their {\it weight} as \glossary{$\mathcal{W}(\la,\mu)$}
\begin{equation}\label{tag5}
\mathcal{W}(\la,\mu):=P_{\la^{(1)}}(\rho_0^+)\,Q_{\la^{(1)}/\mu^{(1)}}(\rho_1^-)
P_{\la^{(2)}/\mu^{(1)}}(\rho_1^+)\,\cdots
P_{\la^{(N)}/\mu^{(N-1)}}(\rho_{N-1}^+)\, Q_{\la^{(N)}}(\rho_N^-).
\end{equation}

There is one factor for any two neighboring partitions in equation (\ref{tag4}). The fact that all the specializations are nonnegative implies that all the weights are nonnegative.

For any two specializations $\rho_1,\rho_2$ set
\begin{equation}\label{PIeqn}
\Pi(\rho_1;\rho_2)=\sum_{\la\in\Y}P_\la(\rho_1)Q_\la(\rho_2)= \exp\left(\sum_{n\ge 1}\frac{1}{n}\,\frac{1-t^n}{1-q^n}\,p_n(\rho_1)p_n(\rho_2)\right)
\end{equation}
\glossary{$\Pi(\rho_1;\rho_2)$}
provided that the series converge. This extends the definition of $\Pi(x;y)$ given in Section~\ref{PiSec} and of $\Pi(u;\rho)$ in Section~\ref{NNspecSEC}.

\begin{proposition}
Assuming $\Pi(\rho_i^+;\rho_j^-)<\infty$ for all $i,j$, we have
\begin{equation*}
\sum_{\la,\mu\in \Y} \mathcal{W}(\la,\mu)=\prod_{0\le i<j\le N}\Pi(\rho_i^+;\rho_j^-).
\end{equation*}
\end{proposition}
\begin{proof}
The proposition follows from repeated use of the identity in equation (\ref{tag1a}) and
\begin{eqnarray}
\sum_{\kappa\in\Y} P_{\kappa/\nu}(\rho_1)Q_{\kappa/\hat\nu}(\rho_2) &=&\Pi(\rho_1;\rho_2)\sum_{\tau\in\Y}Q_{\nu/\tau}(\rho_2)P_{\hat\nu/\tau}(\rho_1),\label{tag6}\\
\sum_{\nu\in\Y}Q_{\kappa/\nu}(\rho_1)Q_{\nu/\tau}(\rho_2) &=& Q_{\kappa/\tau}(\rho_1
,\rho_2),\label{tag7}
\end{eqnarray}
which extend equations (\ref{skewformulas2}) and (\ref{skewformulas3}) to general specializations.
\end{proof}

\begin{definition}\label{macprocessdef}
The {\it Macdonald process} \index{Macdonald process}
$\M(\rho_0^+,\dots,\rho_{N-1}^+;\rho_1^-,\dots,\rho_N^-)$ \glossary{$\M(\rho_0^+,\dots,\rho_{N-1}^+;\rho_1^-,\dots,\rho_N^-)$} is the probability measure on sequences $(\la,\mu)$ as in equation (\ref{tag4}) with
\begin{equation*}
\M(\rho_0^+,\dots,\rho_{N-1}^+;\rho_1^-,\dots,\rho_N^-)(\la,\mu)=\frac{\mathcal{W}(\la,\mu)}
{\prod_{0\le i<j\le N}\Pi(\rho_i^+;\rho_j^-)}\,.
\end{equation*}
The Macdonald process with $N=1$ is called the {\it Macdonald measure}\index{Macdonald measure} and written as $\MM(\rho^+;\rho^-)$ \glossary{$\MM(\rho^+;\rho^-)$}.

We write the probability distribution and expectation with respect to the Macdonald process (measure) as $\PP_{\M(\rho_0^+,\dots,\rho_{N-1}^+;\rho_1^-,\dots,\rho_N^-)}$ \glossary{$\PP_{\M(\rho_0^+,\dots,\rho_{N-1}^+;\rho_1^-,\dots,\rho_N^-)}$}, $\EE_{\M(\rho_0^+,\dots,\rho_{N-1}^+;\rho_1^-,\dots,\rho_N^-)}$\glossary{$\EE_{\M(\rho_0^+,\dots,\rho_{N-1}^+;\rho_1^-,\dots,\rho_N^-)}$} or $\langle \cdot \rangle_{\M(\rho_0^+,\dots,\rho_{N-1}^+;\rho_1^-,\dots,\rho_N^-)}$\glossary{$\langle \cdot \rangle_{\M(\rho_0^+,\dots,\rho_{N-1}^+;\rho_1^-,\dots,\rho_N^-)}$} ($\PP_{\MM(\rho^+;\rho^-)}$\glossary{$\PP_{\MM(\rho^+;\rho^-)}$}, $\EE_{\MM(\rho^+;\rho^-)}$\glossary{$\EE_{\MM(\rho^+;\rho^-)}$} or $\langle \cdot \rangle_{\MM(\rho^+;\rho^-)}$\glossary{$\langle \cdot \rangle_{\MM(\rho^+;\rho^-)}$}).
\end{definition}

Using equations (\ref{tag1a}), (\ref{tag6}) and (\ref{tag7}) it is not difficult to show that a projection of the Macdonald process to any subsequences of $(\la,\mu)$ is a also a Macdonald process. In particular, the projection of $\M\bigl(\rho_0^+,\dots,\rho_{N-1}^+;\rho_1^-,\dots,\rho_N^-\bigr)$ to $\lambda^{(j)}$ is the Macdonald measure
$\MM\bigl(\rho^+_{[0,j-1]};\rho^-_{[j,N]}\bigr)$, and its projection to $\mu^{(k)}$ is a slightly different Macdonald measure
$\MM\bigl(\rho^+_{[0,k-1]};\rho^-_{[k+1,N]}\bigr)$. Here we used the notation $\rho^\pm_{[a,b]}$ to denote the union of specializations $\rho^\pm_m$, $m=a,\dots,b$.

\begin{proposition}\label{sumsprop}
Under the probability distribution given by the Macdonald process, the random variables $|\lambda^{(1)}|, |\lambda^{(2)}-\mu^{(1)}|, \ldots, |\lambda^{(N)}-\mu^{(N-1)}|$ are independent. Likewise $|\lambda^{(1)}-\mu^{(1)}|,|\lambda^{(2)}-\mu^{(2)}|,\ldots, |\lambda^{(N)}|$ are independent. Moreover these random variables have generating functions given by
\begin{eqnarray*}
\langle u^{|\lambda^{(k)}-\mu^{(k-1)}|} \rangle_{\M(\rho_0^+,\dots,\rho_{N-1}^+;\rho_1^-,\dots,\rho_N^-)} &=& \prod_{j>k} \frac{\Pi(u\rho_k^+;\rho_j^{-})}{\Pi(\rho_k^+;\rho_j^{-})},\\
\langle u^{|\lambda^{(k)}-\mu^{(k)}|} \rangle_{\M(\rho_0^+,\dots,\rho_{N-1}^+;\rho_1^-,\dots,\rho_N^-)} &=& \prod_{i<k} \frac{\Pi(\rho_i^+;u\rho_k^{-})}{\Pi(\rho_i^+;\rho_k^{-})}.\\
\end{eqnarray*}
\end{proposition}
\begin{proof}
Observe that for $a_0^+,a_1^+,\ldots, a_{N-1}^+$ and $a_1^{-},\ldots,a_N^{-}$ positive numbers, (suppressing the Macdonald process subscript in the expectation), $\EE\left[(a_{0}^+)^{|\lambda^{(1)}|}(a_1^{-})^{|\lambda^{(1)}-\mu^{(1)}|} (a_1^+)^{|\lambda^{(2)}-\mu^{(1)}|}\cdots \right]$ is given by
\begin{eqnarray*}
 \frac{1}{Z} \sum_{\lambda,\mu} (a_{0}^+)^{|\lambda^{(1)}|}P_{\lambda^{(1)}}(\rho_0^+)(a_1^{-})^{|\lambda^{(1)}-\mu^{(1)}|}Q_{\lambda^{(1)}/\mu^{(1)}}(\rho_1^-) (a_1^+)^{|\lambda^{(2)}-\mu^{(1)}|}P_{\lambda^{(2)}/\mu^{(1)}}(\rho_1^+)\cdots \\
=\frac{1}{Z} \sum_{\lambda,\mu} P_{\lambda^{(1)}}(a_{0}^+\rho_0^+)Q_{\lambda^{(1)}/\mu^{(1)}}(a_1^{-}\rho_1^-) P_{\lambda^{(2)}/\mu^{(1)}}(a_1^+\rho_1^+)\cdots = \frac{\prod_{0\le i<j\le N}\Pi(a_i^+\rho_i^+;a_j^-\rho_j^-)}{\prod_{0\le i<j\le N}\Pi(\rho_i^+;\rho_j^-)},
\end{eqnarray*}
where $Z=\prod_{0\le i<j\le N}\Pi(\rho_i^+;\rho_j^-)$. Setting $a_i^+\equiv 1$ or $a_j^-\equiv 1$ and observing the way the generating function factors gives the desired independence, and likewise setting all variables to 1 except one gives the generation function formula.
\end{proof}

We will mainly focus on a special case of the Macdonald process.

\begin{definition}
The {\it ascending Macdonald process}\index{Macdonald process!ascending} $\M_{asc}(a_1,\dots,a_N;\rho)$\glossary{$\M_{asc}(a_1,\dots,a_N;\rho)$} is the probability measure on sequences
\begin{equation}\label{ascseqn}
\varnothing \prec \lambda^{(1)}\prec \lambda^{(2)}\prec\cdots \prec\lambda^{(N)}
\end{equation}
(or equivalently column-strict Young tableaux) indexed by positive variables $a_1,\ldots, a_N$ and a single nonnegative specialization $\rho$, with
\begin{equation}\label{ascmeasure}
\M_{asc}(a_1,\dots,a_N;\rho)(\la^{(1)},\dots,\la^{(N)})=
\frac{P_{\la^{(1)}}(a_1)P_{\la^{(2)}/\la^{(1)}}(a_2)\cdots
P_{\la^{(N)}/\la^{(N-1)}}(a_N)
Q_{\la^{(N)}}(\rho)}{\Pi(a_1,\dots,a_N;\rho)}\,.
\end{equation}
Here $\Pi(a_1,\dots,a_N;\rho) = \Pi(a_1;\rho)\cdots \Pi(a_N;\rho)$ with the terms $\Pi(u;\rho)$ as defined in (\ref{tag1}).
\end{definition}

The ascending Macdonald process is a special case of Macdonald processes with $\rho_j^+$ being specializations into a single positive variable $a_{j+1}$, $j=1,\dots,N$,
the specializations $\rho_1^-,\dots,\rho_{N-1}^-$ being trivial, and the only remaining free specialization $\rho_N^-$ being written simply as $\rho$. Then $\mu^{(k)}=\lambda^{(k)}$ for all $k$, and the process lives on sequences as in equation (\ref{ascseqn}) and is given by equation (\ref{ascmeasure}).

If $\rho$ corresponds to parameters $\{\alpha_i\}$, $\{\beta_i\}$, $\gamma$ as in the previous section, it is not hard to see that the condition of the partition function $\Pi(a_1,\dots,a_N;\rho)$ being finite is equivalent to $a_i\alpha_j<1$ for all
$i,j$.

Observe that the projection of $\M_{asc}$ to $\lambda^{(k)}$, $k=1,\dots,N$, is the Macdonald measure
\begin{equation*}
\MM(a_1,\ldots,a_k;\rho)(\la^{(k)})= \frac{P_{\lambda^{(k)}}(a_1,\ldots,a_k) Q_{\lambda^{(k)}}(\rho)} {\Pi(a_1,\ldots,a_k;\rho)}\,.
\end{equation*}
%Here $a_1,\ldots,a_k$ represents union of the single variable the specializations $a_i$.

\subsection{Difference operators and integral formulas}\label{difopSEC}

The relevance of this section to the study of Macdonald processes is explained by the following observation. Assume we have a linear operator $\mathcal{D}$ in the space of functions in $n$ variables whose restriction to the space of symmetric polynomials diagonalizes in the basis of Macdonald polynomials: $\mathcal{D} P_\la=d_\lambda P_\la$ for any partition $\la$ with $\ell(\la)\le n$. Then we can apply $\mathcal{D}$ to both sides of the identity
\begin{equation*}
\sum_{\la:\ell(\la)\le n} P_{\lambda}(a_1,\dots,a_n) Q_{\lambda}(\rho)=\Pi(a_1,\dots,a_n;\rho).
\end{equation*}
Dividing the result by $\Pi(a_1,\dots,a_n;\rho)$ we obtain
\begin{equation}\label{tag8}
\langle d_\lambda \rangle_{\MM(a_1,\dots,a_n;\rho)}=\frac{\mathcal{D}\Pi(a_1,\dots,a_n;\rho)} {\Pi(a_1,\dots,a_n;\rho)}\,,
\end{equation}
where $\langle \cdot\rangle_{\MM(a_1,\dots,a_n;\rho)}$ represents averaging $\cdot$ over the specified Macdonald measure.
If we apply $\mathcal{D}$ several times we obtain
\begin{equation*}
\langle d_\lambda^k \rangle_{\MM(a_1,\dots,a_n;\rho)}=\frac{\mathcal{D}^k \Pi(a_1,\dots,a_n;\rho)} {\Pi(a_1,\dots,a_n;\rho)}\,.
\end{equation*}
If we have several possibilities for $\mathcal{D}$ we can obtain formulas for averages of the observables equal to products of powers of the corresponding eigenvalues. The goal of this section is to provide a few variants for such operators.

In what follows we fix the number of independent variables to be $n\in \Zgzero$.

\begin{definition}
For any $u\in\R$ and $1\le i\le n$, define the {\it shift operator}\index{difference operator!shift} $T_{u,x_i}$ by
\begin{equation*}
(T_{u,x_i}F)(x_1,\dots,x_n)=F(x_1,\dots,ux_i,\dots,x_n).
\end{equation*}
For any subset $I\subset\{1,\dots,n\}$, define
\begin{equation*}
A_I(x;t)=t^{\frac{r(r-1)}2}\prod_{i\in I,\,j\notin I}\frac{tx_i-x_j}{x_i-x_j}\,.
\end{equation*}

Finally, for any $r=1,2,\ldots,n$, define the {\it Macdonald difference operator}\index{difference operator!Macdonald}\index{Macdonald symmetric functions!difference operator}
\begin{equation*}
D_n^r=\sum_{\substack{I\subset\{1,\ldots,n\}\\|I|=r}} A_I(x;t)\prod_{i\in I} T_{q,x_i}.
\end{equation*}
\glossary{$D_n^r$}
\end{definition}

\begin{proposition}\cite[VI(,4.15)]{M}\label{prop5}
For any partition $\la$ with $\ell(\la)\le n$
\begin{equation*}
D_n^r P_\lambda(x_1,\dots,x_n)=e_r(q^{\la_1}t^{n-1},q^{\la_2}t^{n-2},\dots,q^{\la_n}) P_\lambda(x_1,\dots,x_n).
\end{equation*}
Here $e_r$ is the elementary symmetric function, $e_r(x_1,\ldots,x_n) = \sum_{1\leq i_1<\cdots<i_r\leq n} x_{i_1}\cdots x_{i_r}$.
\end{proposition}

Although the operators $D_n^r$ do not look particularly simple, they can be represented by contour integrals, which will later be helpful for evaluating the right-hand side of equation (\ref{tag8}).

\begin{remark}
In the below propositions the desired contours always exist under the hypothesis that the parameter $t$ is sufficiently small and the ratios $x_i/x_j$ are sufficiently close to 1. This will suffice for our purposes. Extending validity of these formulas to the full set of parameters may require more work.
\end{remark}

\begin{proposition}\label{prop6}
Assume that $F(u_1,\dots,u_n)=f(u_1)\cdots f(u_n)$. Take $x_1,\ldots,x_n>0$ and assume that $f(u)$ is holomorphic and nonzero in a complex neighborhood of an interval in $\R$ that contains $\{x_j,qx_j\}_{j=1}^n$. Then for $r=1,2,\dots,n$
\begin{equation} \label{tag9}
(D_n^r F)(x)= F(x)\cdot \frac {1}{(2\pi \iota)^r r!}\oint\cdots\oint
\det\left[\frac 1{tz_k-z_\ell}\right]_{k,\ell=1}^r
\prod_{j=1}^r  \left(\prod_{m=1}^n \frac{tz_j-x_m}{z_j-x_m}\right)
\frac{f(qz_j)}{f(z_j)}\,dz_j,
\end{equation}
where each of $r$ integrals is over positively oriented contour encircling $\{x_1,\ldots,x_n\}$ and no other singularities of the integrand.
\end{proposition}

\begin{proof}
First note that since $t$ and $q$ are assumed to be in $(0,1)$ it is always possible to find contours of integration as desired by the statement of the proposition we are proving. We assume now that the $x_i$ are pairwise disjoint. By continuity of both sides of the formula in the $x$ variables, this suffices to prove the general result. Now use the Cauchy determinant identity
\begin{equation*}
\det\left[\frac 1{tz_k-z_\ell}\right]_{k,\ell=1}^r=\frac{t^{\frac{r(r-1)}2}
\prod_{1\le k<\ell\le r}(z_k-z_\ell)(z_\ell-z_k)}{\prod_{k,\ell=1}^r(tz_k-z_\ell)}
\end{equation*}
and compute the residues at $z_j=x_{m_j}$. Thanks to the Vandermonde determinants in the numerator, one gets zero contributions if $m_j$'s are not pairwise distinct, and if they are distinct one easily verifies that the contribution corresponds to the summand in $D_n^r$ with $I=\{x_{m_1},\ldots,x_{m_r}\}$. The $r!$ factor is responsible for permutations of the elements of $I$.
\end{proof}

%\begin{comment}
%\begin{remark}\label{rem7}
%When $t=q$, it is possible to derive equation (\ref{tag9}) from the determinantal formula for the correlation functions of the Schur measure, cf.
%\cite{OkWedge},\cite{JohECM}, \cite{BorodinRains}. Indeed,
%\begin{equation}
%\langle e_r(q^{\la_1+n-1},\dots,q^{\la_n})\rangle =\frac 1{r!}
%\sum_{k_1,\dots,k_r\ge 0}
%q^{k_1+\dots+k_r} \rho_r(k_1,\dots,k_r),
%\end{equation}
%where $\rho_r$ is the $r$th correlation function corresponding to the point process formed by $\{\la_1+n-1,\dots,\la_n\}$.
%\end{remark}
%}
%\end{comment}

\begin{remark}\label{firstlambdaremark}
We define a useful analog to the Macdonald operator $D_n^r$ to be the operator
\begin{equation*}
\tilde{D}_n^r = t^{-\frac{n(n-1)}{2}} D_n^{n-r} T_{q^{-1}}
\end{equation*}
\glossary{$\tilde{D}_n^r$}
where the operator $T_{q^{-1}}$ multiplies all variables by $q^{-1}$. Observe that
\begin{equation*}
T_{q^{-1}}P_{\lambda}(x_1,\ldots, x_n) = q^{-|\lambda|} P_{\lambda}(x_1,\ldots,x_n)
\end{equation*}
and
\begin{eqnarray*}
\tilde{D}_n^r P_{\lambda}(x_1,\ldots,x_n) &=& t^{-\frac{n(n-1)}{2}} q^{-|\lambda|} e_{n-r}(q^{\lambda_1}t^{n-1},\ldots, q^{\lambda_n})P_{\lambda}(x_1,\ldots,x_n)\\
 &=& e_r(q^{-\lambda_1}t^{1-n},q^{-\lambda_2}t^{2-n},\ldots,q^{-\lambda_n})P_{\lambda}(x_1,\ldots, x_n).
\end{eqnarray*}

Observe that
\begin{eqnarray*}
t^{-\frac{n(n-1)}{2}} D_{n}^{n-r} T_{q^{-1}} &=& t^{-\frac{n(n-1)}{2}} \sum_{\substack{I\subset \{1,\ldots, n\}\\|I|=n-r}} \prod_{i\in I, j\notin I} \frac{tx_i-x_j}{x_i-x_j} t^{\frac{(n-r)(n-r-1)}{2}}\prod_{i\in I} T_{q,x_i}T_{q^{-1}}\\
&=& t^{\frac{r(r+1)}{2}-rn} \sum_{\substack{J\subset \{1,\ldots, n\}\\|J|=r}} \prod_{j\in J,i\notin J} \frac{ x_j-tx_i}{x_j-x_i} \prod_{j\in J} T_{q^{-1},x_j}.
\end{eqnarray*}
Applying this to a function $F(u_1,\ldots, u_n)=f(u_1)\cdots f(u_n)$ as in Proposition~\ref{prop6} we obtain an analogous integral formula
\begin{equation*}
(\tilde{D}_n^r F)(x) = F(x)\cdot \frac{t^{r(1-n)}}{(2\pi \iota)^r r!}\oint\cdots\oint
\det\left[\frac{1}{z_\ell-tz_k}\right]_{k,\ell=1}^r
\prod_{j=1}^r  \left(\prod_{m=1}^n \frac{z_j-tx_m}{z_j-x_m}\right)
\frac{f(q^{-1}z_j)}{f(z_j)}\,dz_j,
\end{equation*}
where the contours include the poles $z_j=x_m$ and no other singularities. Under the change of integration variables $z_j=1/w_j$ for $j=1,\ldots, r$ we have
\begin{equation*}
(\tilde{D}_n^r F)(x) = F(x)\cdot \frac{t^{r(1-n)}}{(2\pi \iota)^r r!}\oint\cdots\oint
\det\left[\frac{1}{tw_k-w_\ell}\right]_{k,\ell=1}^r
\prod_{j=1}^r  \left(\prod_{m=1}^n \frac{tw_j-x_m^{-1}}{w_j-x_m^{-1}}\right)
\frac{f((qw_j)^{-1})}{f\big(w_j^{-1}\big)}\,dw_j,
\end{equation*}
where the contours now include the poles $w_j=x_m^{-1}$ and no other singularities. The formula is exactly as in Proposition~\ref{prop6} with an extra factor of $t^{r(n-1)}$, the $x_m$'s replaced by their inverses, and $f(u)$ replaced by $f(u^{-1})$. Iterating these operators will lead to very similar formulas to the ones we now obtain for powers of $D_n^r$.
\end{remark}

\begin{remark}
Recall that $\Pi(a_1,\dots,a_n;\rho)=\prod_{i=1}^n \Pi(a_i;\rho)$. Hence, Proposition~\ref{prop6} is suitable for evaluating the right-hand side of equation (\ref{tag8}). We will not rewrite these formulas as they are immediate.
\end{remark}

Observe also that for any fixed $z_1,\dots,z_r$, the right-hand side of equation (\ref{tag9}) is a multiplicative function of $x_j$'s. This makes it possible to iterate the procedure. Let us do that in the simplest case first, then the full case below in Proposition~\ref{prop8FULL}.

\begin{proposition}\label{prop8}
Fix $k\ge 1$. Assume that $F(u_1,\dots,u_n)=f(u_1)\cdots f(u_n)$. Take $x_1,\dots,x_n >0$ and assume that $f(u)$ is holomorphic and nonzero in a complex neighborhood of an interval in $\R$ that contains $\{q^ix_j\mid i=0,\ldots,k,j=1\ldots,n\}$.
Then
\begin{equation*}
\frac{\bigl({(D_n^1)}^k F\bigr)(x)}{F(x)}=\frac {(t-1)^{-k}}{(2\pi \iota)^k} \oint\cdots\oint \prod_{1\le a<b\le k} \frac{(tz_a-qz_b)(z_a-z_b)}{(z_a-qz_b)(tz_a-z_b)} \prod_{c=1}^k\left(\prod_{m=1}^n \frac{tz_c-x_m}{z_c-x_m}\right)
\frac{f(qz_c)}{f(z_c)}\,\frac{dz_c}{z_c}\,,
\end{equation*}
where the $z_c$-contour contains $\{qz_{c+1},\ldots,qz_k,x_1,\ldots,x_n\}$ and no other singularities for $c=1,\dots,k$.
\end{proposition}
\begin{proof}
This follows from  sequential application of Proposition~\ref{prop6}. For $k=1$ the statement coincides with Proposition~\ref{prop6}. We call the integration variable $z_1$ and deform the integration contour to contain $\{x_j,qx_j\}_{j=1}^n$. Then
the $x$-dependent part of the integrand satisfies the assumption of Proposition~\ref{prop6} with $\tilde{f}(u)=f(u)(tz_1-u)/(z_1-u)$, and we apply Proposition~\ref{prop6} to the integrand. Call the new integration variable $z_2$ and iterate the procedure.
\end{proof}
\begin{remark}\label{rem9}
In the Schur measure case, one can apply products of operators $D_n^1$ with {\it different} values of $q=t$. This would allow to derive a formula for the correlation functions of the Schur measure. Specifically observe that when $q=t$
\begin{equation*}
D_n^1 P_{\lambda} = e_1(q) P_{\lambda}, \qquad e_1 = q^{\lambda_1+n-1} + q^{\lambda_2+n-2}+\cdots q^{\lambda_n}.
\end{equation*}
Calling $(\lambda_1+n-1,\ldots, \lambda_n) = (\ell_1,\ldots, \ell_n)$ we thus have
\begin{equation*}
\langle e_1 \rangle  = \sum_{\ell\in \Zgeqzero}\rho_1(\ell) q^\ell
\end{equation*}
where $\rho_1$ is the first correlation function. We likewise find that
\begin{equation*}
\langle e_1(q_1)e_1(q_2)\cdots e_1(q_s)\rangle = \left\langle \prod_{j=1}^{s}(q_j^{\ell_1}+q_j^{\ell_2}+\cdots +q_j^{\ell_n})\right\rangle,
\end{equation*}
%\begin{equation*}
%\langle e_1(q_1)e_1(q_2)\cdots e_1(q_s)\rangle = \left\langle \prod_{j=1}^{s}(q_j^{\ell_1}+q_j^{\ell_2}+\cdots +q_j^{\ell_n})\right\rangle = \sum_{j=1}^{s} c_j \sum_{\ell_1,\ldots, \ell_j\in \Z} \rho_j(\ell_1,\ldots,\ell_j) q_1^{\ell_1}q_2^{\ell_2}\cdots q_j^{\ell_j},
%\end{equation*}
is easily expressible via the first $s$ correlation. From this one might recover all of the correlation functions of the Schur measure.
\end{remark}

\begin{proposition}\label{prop8FULL}
Fix $k\ge 1$ and $r_{\alpha}\geq 1$ for $1\leq \alpha\leq k$. Assume that $F(u_1,\dots,u_n)=f(u_1)\cdots f(u_n)$. Take $x_1,\dots,x_n >0$ and assume that $f(u)$ is holomorphic and nonzero in a complex neighborhood of an interval in $\R$ that contains $\{q^i x_j\mid i=0,\ldots,k;\,j=1\ldots,n\}$. Then
\begin{eqnarray*}
\left(\prod_{\alpha=1}^{k}D_n^{r_{\alpha}} F\right)(x) &=& F(x) \cdot \prod_{\alpha=1}^{k} \frac{ (t-1)^{-r_{\alpha}}}{(2\pi \iota)^{r_{\alpha}} r_{\alpha}!} \oint\cdots \oint \prod_{1\leq \alpha<\beta\leq k} \left(\prod_{i=1}^{r_{\alpha}}\prod_{j=1}^{r_{\beta}} \frac{(tz_{\alpha,i}-q z_{\beta,j})(z_{\alpha,i}-z_{\beta,j})}{(z_{\alpha,i}-qz_{\beta,j})(tz_{\alpha,i}-z_{\beta,j})}\right)\\
& & \times  \prod_{\alpha=1}^{k}\left( \prod_{i\neq j =1}^{r_\alpha} \frac{z_{\alpha,i}-z_{\alpha,j}}{tz_{\alpha,i}-z_{\alpha,j}} \left(\prod_{j=1}^{r_{\alpha}}\frac{(tz_{\alpha,j}-x_1)\cdots(tz_{\alpha,j}-x_{n})}{(z_{\alpha,j}-x_1)\cdots(z_{\alpha,j}-x_{n})} \,\frac{f(qz_{\alpha,j})}{f(z_{\alpha,j})}\,\frac{dz_{\alpha,j}}{z_{\alpha,j}}\right)\right)
\end{eqnarray*}
where the $z_{\alpha,j}$-contour contains $\{qz_{\beta,i}\}$ for all $i\in\{1,\ldots, r_{\beta}\}$ and $\beta>\alpha$, as well as $\{x_1,\ldots,x_n\}$ and no other singularities.
\end{proposition}
The proof is similar to those of Propositions \ref{prop6} and \ref{prop8}.
%\note{give a proof of this for $k=2$ at least}

Let us describe another family of difference operators that are diagonalized by Macdonald polynomials.

\begin{proposition}\label{prop10}
With the notation $y_m=x_mq^{\eta_m}$ for $m=1,\ldots,n$, $(a)_{\infty} = (a;q)_{\infty}$ and $x=(x_1,\ldots, x_n)$, $y=(y_1,\ldots, y_n)$ we have
\begin{eqnarray*}
\nonumber&&\sum_{\eta_1,\ldots,\eta_n=0}^{\infty} \prod_{i=1}^n \left(z^{\eta_i}t^{(i-1)\eta_i}\frac{(t)_{\infty}(q^{\eta_i+1})_{\infty}}{(tq^{\eta_i})_{\infty}(q)_{\infty}}\right) \prod_{1\le j<k\le n}\frac{({y_j}y_k^{-1})_\infty}{({x_j}y_k^{-1})_{\infty}}\frac{(\tfrac{q}{t}\,{x_j}y_k^{-1})_\infty}{(\tfrac{q}{t}\,{x_j}x_k^{-1})_{\infty}}
\frac{(t{x_j}x_k^{-1})_\infty}{(t{y_j}x_k^{-1})_{\infty}}\frac{(q{y_j}x_k^{-1})_\infty}{(q{y_j}y_k^{-1})_{\infty}}P_\lambda(y)\\
\nonumber&&=\prod_{i=1}^n \frac{(q^{\lambda_i}t^{n-i+1}z;q)_\infty}{(q^{\lambda_i}t^{n-i}z;q)_\infty}\, P_\lambda(x)\,.
\end{eqnarray*}
Here $z$ is a formal variable, and the formula can be viewed as an identity of formal power series in $z$.
\end{proposition}
The coefficient of $z^r$ in the right-hand side is the polynomial $g_r=Q_{(r)}$ evaluated at the variables $\left\{q^{\lambda_i}t^{n-i}\right\}_{i=1}^n$. The formula is dual to the first Pieri formula of equation (\ref{piereEqn}) in the same way as Proposition~\ref{prop5} is dual to the last Pieri formula of equation (\ref{piereEqn}) (see Section~\ref{pieresec} or \cite[VI,(6.7)]{M} for detailed explanations). While the left-hand side can be viewed as an application of a $q$-integral operator to $P_\la$, the coefficient of $z^r$ constitutes a difference operator of order $r$.

\begin{remark}
We were unable to find this formula in the literature. It is similar to Theorem I in \cite{OkShifted}, however, that theorem increases the number of variables in the polynomial. It is plausible that one can recover Proposition~\ref{prop10} from Okounkov's raising operator, but we chose to give a direct proof.
\end{remark}

\begin{proof}[Proof of Proposition~\ref{prop10}]
Following \cite[VI.6]{M} for each partition $\mu$ of length less than or equal to $n$, define a homomorphism (i.e. specialization)
\begin{equation*}
u_{\mu}:\C[x_1,\ldots, x_n] \to \C, \qquad \textrm{by} \quad u_{\mu}(x_i) = q^{\mu_i} t^{n-i},\quad 1\leq i\leq n.
\end{equation*}
We extend $u_\mu$ to rational functions in $x_1,\ldots, x_n$ for which the specialized denominator does not vanish. It follows from equation (\ref{eqn12}) that
\begin{equation*}
\prod_{i=1}^{n} \frac{(q^{\lambda_i}t^{n-i+1}z)_{\infty}}{(q^{\lambda_i}t^{n-i}z)_{\infty}} = u_{\lambda}\left(\sum_{m\geq 0} g_m(x;q,t)z^m\right),
\end{equation*}
hence the coefficient of $z^m$ on the right-hand side of our desired identity is $u_{\lambda}(g_m)P_{\lambda}(x_1,\ldots, x_n)$.

It suffices to prove the identity under application of $u_{\mu}$ to both sides for arbitrary $\mu$. For the right-hand side we obtain $u_{\lambda}(g_m)u_{\mu}(P_{\lambda})$ as the coefficient of $z^m$.

Using \cite[VI,(6.6)]{M} we have the index-variable duality relation \index{Macdonald symmetric functions!index-variable duality}
\begin{equation*}
u_{\mu}(P_{\lambda}) = \frac{u_0(P_{\lambda})}{u_0(P_{\mu})}u_{\lambda}(P_{\mu}).
\end{equation*}
Hence using the first of the Pieri formulas (\ref{skewformulas1})
\begin{equation*}
u_{\lambda}(g_m)u_{\mu}(P_{\lambda}) = \frac{u_0(P_{\lambda})}{u_0(P_{\mu})}u_{\lambda}(g_m P_{\mu}) =  \frac{u_0(P_{\lambda})}{u_0(P_{\mu})} \sum_{\nu\succ\mu: |\nu/\mu|=m} \varphi_{\nu/\mu} u_{\lambda}(P_{\nu}) = \sum_{\nu\succ\mu: |\nu/\mu|=m} \frac{u_0(P_{\nu})}{u_0(P_{\mu})}\varphi_{\nu/\mu} u_{\nu}(P_{\lambda}).
\end{equation*}
The sum over $\nu$ can be restricted to those for which $\nu/\mu$ is a horizontal strip.

Set $n(\lambda) = \sum_{i=1}^{\ell(\lambda)} (i-1)\lambda_i$. What remains then is to show that for $\eta_i = \nu_i-\mu_i$ we have the following identity
\begin{equation*}
t^{n(\eta)}\prod_{i=1}^{n}\frac{(t)_{\infty}(q^{\eta_i+1})_{\infty}}{(tq^{\eta_i})_{\infty}(q)_{\infty}} u_{\mu}\left(\prod_{1\le j<k\le n}\frac{({y_j}y_k^{-1})_\infty}{({x_j}y_k^{-1})_{\infty}}\frac{(\tfrac{q}{t}\,{x_j}y_k^{-1})_\infty}{(\tfrac{q}{t}\,{x_j}x_k^{-1})_{\infty}}
\frac{(t{x_j}x_k^{-1})_\infty}{(t{y_j}x_k^{-1})_{\infty}}\frac{(q{y_j}x_k^{-1})_\infty}{(q{y_j}y_k^{-1})_{\infty}}\right) =\frac{u_0(P_{\nu})}{u_0(P_{\mu})}\varphi_{\nu/\mu}.
\end{equation*}

Following \cite[VI,(6.11)]{M} we find that if we set
\begin{equation*}
\Delta^+ = \prod_{1\leq i<j\leq n} \frac{(x_i/x_j)_{\infty}}{(tx_i/x_j)_{\infty}}
\end{equation*}
then
\begin{equation*}
\frac{u_0(P_{\nu})}{u_0(P_{\mu})} = t^{n(\nu)-n(\mu)} \frac{ u_{\nu}(\Delta^+)}{ u_{\mu}(\Delta^+)} = t^{n(\nu)-n(\mu)} u_{\mu}\left(\prod_{1\leq i<j\leq n }\frac{(q^{\eta_i-\eta_j}x_i/x_j)_{\infty} }{(tq^{\eta_i-\eta_j}x_i/x_j)_{\infty}}\,\frac{(tx_i/x_j)_{\infty}}{(x_i/x_j)_{\infty}}\right).
\end{equation*}
Recall the formula for $\varphi_{\nu/\mu}$ given in (\ref{piereFormPhi}). By evaluating the first ratio of $f$'s when $i=j$ and then shifting indices for the second ratio we find that
\begin{equation*}
\varphi_{\nu/\mu} = \prod_{i=1}^{n} \frac{f(1)}{f(q^{\eta_i})} \prod_{1\leq i<j\leq n} \frac{f(q^{\nu_i-\nu_j}t^{j-i})}{f(q^{\nu_i-\mu_j}t^{j-i})}\frac{f(q^{\mu_i-\mu_j}t^{j-i-1})}{f(q^{\mu_i-\nu_j}t^{j-i-1})}
\end{equation*}
where $f(u) = (tu)_{\infty} / (qu)_{\infty}$.

We can, however, recognize that the product of $i<j$ can be written as
\begin{equation*}
u_{\mu}\left(\prod_{1\leq i<j\leq n} \frac{f(q^{\eta_i-\eta_j}x_i/x_j)}{f(q^{\eta_i}x_i/x_j)}\frac{f(t^{-1}x_i/x_j)}{f(q^{-\eta_j}t^{-1}x_i/x_j)} \right)
\end{equation*}
One now immediately sees that the required identity follows by substituting $y_m=x_m q^{\eta_m}$ as necessary.
\end{proof}

%\begin{remark} 1. It is unclear if this formula is anywhere in the literature. It is similar to Theorem I in \cite{Okounkov, (Shifted) Macdonald...}, however, that theorem increases the number of variables in the polynomial. Eric Rains says that there
%is an analytic continuation argument that would derive the formula above. From his letter:
%``If he (Okounkov) only gives the ``raising'' operator, your operator can be obtained by taking the $(n+m)$-variable version, setting $m$ of the variables to 0, then observing that everything's a rational function of $t^m$.  (Hmmm.  Possibly you should
%instead set $m$ of the variables to $1,t,t^2,\dots,t^{m-1}$)'' There is also an elliptic generalization of the identities,
%see \cite{Rains, Sections 2 and 7 of ``Transformations...''}.
%\end{remark}

\section{Dynamics on Macdonald processes}\label{dynamicsSec}

\subsection{Commuting Markov operators}\label{commMarkovOps}
Let $y,z,v$ be Macdonald nonnegative specializations of $\Sym$. Set
\begin{align*}
p_{\lambda\mu}^{\uparrow}(y;z)&:=\frac 1{\Pi(y;z)}\frac{P_{\mu}(y)}{P_{\la}(y)}\,Q_{\mu/\la}(z),& \la,\mu\in \Y(y),\qquad\qquad&\\
p_{\lambda\nu}^{\downarrow}(y;v)&:=\frac{P_{\nu}(y)}{P_{\la}(y,v)}\,P_{\la/\nu}(v), &\la\in\Y(y,v),\ \nu\in\Y(y),\qquad\qquad&
\end{align*}
\glossary{$p_{\lambda\mu}^{\uparrow}$}\glossary{$p_{\lambda\nu}^{\downarrow}$}
where recall $\Y(\rho)=\{\kappa\in\Y\mid P_\kappa(\rho)>0\}$, and for the first definition we assume that $\Pi(y;z)=\sum_{\kappa\in\Y} P_\kappa(y)Q_\kappa(z)<\infty$.

Equations (\ref{tag1a}) and (\ref{tag6}) imply that the matrices
\begin{equation*}
p^\uparrow(y;z)=\bigl[p_{\lambda\mu}^{\uparrow}(y;z)\bigr]_{\la,\mu\in\Y(y)}\quad
\text{and}\quad p^\downarrow(y;v)= \bigl[p_{\lambda\nu}^{\downarrow}(y;v)\bigr]_{\la\in\Y(y,v),\nu\in\Y(y)}
\end{equation*}
are stochastic:
\begin{equation*}
\sum_{\mu\in\Y(y)} p_{\lambda\mu}^{\uparrow}(y;z)=\sum_{\nu\in\Y(y)} p_{\lambda\nu}^{\downarrow}(y;v)=1.
\end{equation*}

It also follows from equations (\ref{tag1a}), (\ref{tag6}) and (\ref{tag7}) that $p^\uparrow$ and $p^\downarrow$ act well
on the Macdonald measures:
\begin{equation}\label{tag9a}
\MM(x;y)p^{\uparrow}(y;z)=\MM(x,z;y),\qquad \MM(x;y,v)p^{\downarrow}(y;v)=\MM(x;y).
\end{equation}

Observe that $\MM(\rho_1;\rho_2)=\MM(\rho_2;\rho_1)$, so the parameters of the Macdonald measures in these relations can also be permuted.

\begin{proposition}\label{prop11}
Let $y,z,z_1,z_2,v_1,v_2$ be nonnegative specializations of $\Sym$. Then we have the commutativity relations
\begin{eqnarray*}
p^{\uparrow}(y;z_1)p^{\uparrow}(y;z_2)&=&
p^{\uparrow}(y;z_2)p^{\uparrow}(y;z_1),\\
p^{\downarrow}(y,v_2;v_1)p^{\downarrow}(y;v_2)&=&
p^{\downarrow}(y,v_1;v_2)p^{\downarrow}(y;v_1),\\
p^{\uparrow}(y,v;z)p^{\downarrow}(y;v)&=&
p^{\downarrow}(y;v)p^{\uparrow}(y;z),
\end{eqnarray*}
where for the first relation we assume $\Pi(y;z_1,z_2)<\infty$, and for the third relation we assume $\Pi(y,v;z)<\infty$.
\end{proposition}

\begin{proof} The arguments for all three identities are similar; we only give a proof of the third one (which is in a way the hardest and will be used later on). We have
\begin{eqnarray*}
\sum_{\mu}
p_{\la\mu}^{\uparrow}(y,v;z)p^{\downarrow}_{\mu\nu}(y;v)=
\frac{1}{\Pi(y,v;z)}\sum_{\mu\in\Y(y,v)}
\frac{P_\mu(y,v)}{P_\la(y,v)}Q_{\mu/\la}(z)\,\frac{P_\nu(y)}{P_\mu(y,v)}P_{\mu/\nu}(v)\\=
\frac{1}{\Pi(y,v;z)}\frac{P_\nu(y)}{P_\la(y,v)}\sum_{\mu\in\Y}
Q_{\mu/\la}(z)P_{\mu/\nu}(v)=\frac{\Pi(v;z)}{\Pi(y,v;z)}
\frac{P_\nu(y)}{P_\la(y,v)} \sum_{\kappa\in\Y}
P_{\la/\kappa}(v)Q_{\nu/\kappa}(z)\\=
\frac{1}{\Pi(y;z)}\sum_{\kappa\in\Y(y)}
\frac{P_\kappa(y)}{P_\la(y,v)}P_{\la/\kappa}(v)\,\frac{P_\nu(y)}{P_\kappa(y)}Q_{\nu/\kappa}(z)=
\sum_{\kappa\in\Y(y)}
p_{\la\kappa}^{\downarrow}(y;v)p^{\uparrow}_{\kappa\nu}(y;z),
\end{eqnarray*}
where along the way we extended the summation in $\mu$ from $\Y(y,v)$ to $\Y$ because $P_\nu(y)P_{\mu/\nu}(v)>0$ implies
$P_\mu(y,v)>0$ by equation (\ref{tag1a}); we used equation (\ref{tag6}) to switch from $\mu$ to $\kappa$, and finally we restricted the summation in $\kappa$ from $\Y$ to $\Y(y)$ because $P_{\nu}(y)Q_{\nu/\kappa}(z)>0$ implies $\kappa\subset\nu$ and $P_{\kappa}(y)>0$.
\end{proof}

For a Macdonald nonnegative specialization $y$ define a matrix $q^\uparrow=[q^\uparrow_{\la\mu}]_{\la,\mu\in \Y(y)}$ \glossary{$q^\uparrow$} by
\begin{equation*}
q_{\la\mu}^\uparrow(y)=\begin{cases}
\dfrac{P_{\mu}(y)}{P_{\la}(y)}\,\varphi_{\mu/\la},& \mu/\la
\text{ is a single box},\\
0,& \mu\ne \la \text{ and }\mu\ne \la\cup\{\square\},\\
-\sum_{\nu\ne \la} q_{\la\nu}^\uparrow(y),&\mu=\la,
\end{cases}
\end{equation*}
where $\varphi_{\la/\mu}$ is defined in Section~\ref{pieresec}. Then the off-diagonal entries of this matrix are nonnegative and $\sum_{\mu}q_{\la\mu}^\uparrow(y)=0$, hence $q^\uparrow$ could serve as a matrix of transition rates for a continuous time Markov process.

\begin{corollary}\label{cor12}
Let $y,v$ be nonnegative specializations of $\Sym$. Then
\begin{equation*}
q^{\uparrow}(y,v)p^{\downarrow}(y;v)= p^{\downarrow}(y;v)q^{\uparrow}(y).
\end{equation*}
\end{corollary}
\begin{proof}
Consider the third relation of Proposition~\ref{prop11} and take $z$ to be the specialization into a single variable $\epsilon$. Using equation (\ref{7.14}) to collect the linear terms in $\epsilon$ we obtain the desired equality.
\end{proof}

\subsection{A general construction of multivariate Markov chains}\label{genconstructionSec}

Let $(\SS_1,\dots,\SS_n)$ be an $n$-tuple of discrete countable sets, and $P_1,\dots,P_n$ be stochastic matrices defining Markov chains $\SS_j\to\SS_j$. Also let $\La_1^2,\dots,$ $\La_{n-1}^n$ be {\it stochastic links} \index{stochastic links} between these sets:
\begin{eqnarray*}
P_k:\SS_k\times\SS_k\to [0,1], &
\sum_{y\in\SS_k}P_k(x,y)=1, & x\in \SS_k,\quad k=1,\ldots,n;\\
\La_{k-1}^k:\SS_k\times\SS_{k-1}\to [0,1], &
\sum_{y\in\SS_{k-1}}\La_{k-1}^k(x,y)=1, & x\in \SS_k,\quad
k=2,\ldots,n.
\end{eqnarray*}

Assume that these matrices satisfy the commutation relations
\begin{equation}\label{tag10}
\Delta^k_{k-1}:=\La^k_{k-1}P_{k-1}=P_k\La^k_{k-1},\qquad k=2,\dots,n.
\end{equation}

We will define a {\it multivariate\/} Markov chain \index{multivariate Markov chain} with transition matrix $P^{(n)}$ on the state space
\begin{equation}\label{statespace}
\SS^{(n)}=\Bigl\{(x_1,\dots,x_n)\in\SS_1\times\cdots\times\SS_n\mid \prod_{k=2}^n\La_{k-1}^k(x_k,x_{k-1})\ne 0\Bigr\}.
\end{equation}
The transition probabilities for the Markov chain with transition matrix $P^{(n)}$ are defined as (we use the notation $X_n=(x_1,\dots,x_n)$,
$Y_n=(y_1,\dots,y_n)$)
\begin{equation}\label{tag11}
P^{(n)}(X_n,Y_n)= P_1(x_1,y_1)\prod\limits_{k=2}^n\dfrac{P_k(x_k,y_k)\La_{k-1}^k(y_k,y_{k-1})}
{\Delta^k_{k-1}(x_k,y_{k-1})}
\end{equation}
if $\prod_{k=2}^n\Delta^k_{k-1}(x_k,y_{k-1})>0$, and $0$ otherwise.

One way to think of $P^{(n)}$ is as follows: Starting from $X=(x_1,\dots,x_n)$, we first choose $y_1$ according to the transition matrix $P_1(x_1,y_1)$, then choose $y_2$ using $\frac{P_2(x_2,y_2)\La_{1}^2(y_2,y_1)}{\Delta^2_{1}(x_2,y_1)}$,
which is the conditional distribution of the middle point in the successive application of $P_2$ and $\La^2_1$ provided that we start at $x_2$ and finish at $y_1$, after that we choose $y_3$ using the conditional distribution of the middle point in the successive application of $P_3$ and $\La^3_2$ provided that we start at $x_3$ and finish at $y_2$, and so on. Thus, one could say that $Y$ is obtained from $X$ by {\it sequential update} \index{multivariate Markov chain!sequential update}.

\begin{proposition}\label{prop13}
Let $m_n$ be a probability measure on $\SS_n$. Let $m^{(n)}$ be a probability measure on $\SS^{(n)}$ defined by
\begin{equation*}
m^{(n)}(X_n)=m_n(x_n)\La^n_{n-1}(x_n,x_{n-1})\cdots \La^2_1(x_2,x_1),\qquad X_n=(x_1,\ldots,x_n)\in\SS^{(n)}.
\end{equation*}
Set $\tilde m_n=m_nP_n$ and
\begin{equation*}
\tilde m^{(n)}(X_n) = \tilde m_n(x_n)\La^n_{n-1}(x_n,x_{n-1})\cdots \La^2_1(x_2,x_1),\qquad X_n=(x_1,\dots,x_n)\in\SS^{(n)}.
\end{equation*}
Then $m^{(n)}P^{(n)}= \tilde m^{(n)}$.
\end{proposition}
\begin{proof}
The argument is straightforward. Indeed,
\begin{equation*}
m^{(n)}P^{(n)}(Y_n)=\sum_{X_n\in\SS^{(n)}}m_n(x_n)\La^n_{n-1}(x_n,x_{n-1})\cdots
\La^2_1(x_2,x_1) P_1(x_1,y_1)\prod\limits_{k=2}^n\dfrac{P_k(x_k,y_k)\La_{k-1}^k(y_k,y_{k-1})}
{\Delta^k_{k-1}(x_k,y_{k-1})}\,.
\end{equation*}
Extending the sum to $x_1\in\SS_1$ adds 0 to the right-hand side. Then we can use equation (\ref{tag10}) to compute the sum over $x_1$, removing $\La^2_1(x_2,x_1)$, $P_1(x_1,y_1)$ and $\Delta^2_1(x_2,y_1)$ from the expression. Similarly, we sum
consecutively over $x_2,\dots,x_n$, and this gives the needed result.
\end{proof}

A generalization of Proposition~\ref{prop13} can be found in (\cite{BF}, Proposition 2.7). A continuous time variant of the above construction, which is not as straightforward, can be found in (\cite{BorodinOlshanski} Section 8).

\subsection{Markov chains preserving the ascending Macdonald processes}\label{MarkovChainAscSEC}

For any $k=1,2,\dots$, denote by $\Y(k)$ \glossary{$\Y(k)$} the set of Young diagrams with at most $k$ rows.

For a nonnegative specialization $\rho$ and a positive number $a$ such that $\Pi(\rho;a)<\infty$, and for two partitions
$\la\in\Y(n-1)$ and $\mu\in \Y(n)$, we define a probability distribution on $\Y(n)$ via

\begin{equation*}
P_{a,\rho}(\nu \,\Vert\,\la,\mu)=const\cdot P_{\nu/\la}(a)
Q_{\nu/\mu}(\rho),\qquad \nu\in\Y(n).
\end{equation*}
\glossary{$P_{a,\rho}(\nu \,\Vert\,\la,\mu)$}

Here we assume that the set of $\nu$'s giving nonzero contributions to the right-hand side is nonempty. Then equation (\ref{tag6}) implies the existence of the normalizing constant.

\begin{example}\label{ex14}
\mbox{}\newline
\noindent{\bf 1.} If $\rho$ is a specialization into a single variable $\alpha>0$, then equations (\ref{7.14}) and (\ref{7.14'}) enable us to rewrite $P_{a,\rho}$ in the form
\begin{equation*}
P_{a,\alpha}(\nu \,\Vert\,\la,\mu)=\begin{cases} const\cdot  \varphi_{\nu/\mu} \psi_{\nu/\la}
(\alpha a)^{|\nu|},& \nu/\la \text{ and } \nu/\mu \text{ are horizontal strips},\\
0& \text{otherwise},
\end{cases}
\end{equation*}
where $\varphi$ and $\psi$ are defined in Section~\ref{pieresec} and explicit formulas for them are given in equations (\ref{piereFormPhi}) and (\ref{piereFormPsi}). Note that the constant is independent of $\nu$ but does depend on all other variables.

\noindent{\bf 2.} If $\rho$ is a specialization into a single dual variable $\beta>0$, i.e. the right-hand side of (\ref{tag1}) has the form $1+\beta u$, then equation (\ref{omegaSkew}) implies
\begin{equation*}
P_{a,\hat\beta}(\nu \,\Vert\,\la,\mu)=\begin{cases} const\cdot \psi'_{\nu/\mu} \psi_{\nu/\la}
(\beta a)^{|\nu|},& \nu/\la \text{ and } \nu'/\mu' \text{ are horizontal strips},\\
0& \text{otherwise},
\end{cases}
\end{equation*}
where $\psi'$ is defined in Section~\ref{pieresec} and given by an explicit formula in equation (\ref{piereFormPrime}). We have used the notation $\hat\beta$ above to distinguish the $\beta$ specialization from the $\alpha$ specialization in the previous example.

\noindent{\bf 3.} The first two terms in Taylor expansions of $P_{a,\alpha}$ and $P_{a,\hat\beta}$ as $\alpha,\beta\to0$ are
\begin{equation*}
P_{a,\alpha}(\nu \,\Vert\,\la,\mu)\sim\bfone_{\nu=\mu}+\alpha a\begin{cases}
\frac{\psi_{\nu/\la}}{ \psi_{\mu/\la}}\,\varphi_{\nu/\mu} & \nu/\mu
\text{ is a single box},\\
0,& \nu\ne \mu \text{ and }\nu\ne \mu\cup\{\square\},\\
-\sum_{\kappa=\mu\cup\{\square\}} \frac{\psi_{\kappa/\la}}{\psi_{\mu/\la}}\,\varphi_{\kappa/\mu},&\nu=\mu,
\end{cases}
\end{equation*}
and
\begin{equation*}
P_{a,\hat\beta}(\nu \,\Vert\,\la,\mu)\sim\bfone_{\nu=\mu}+\beta a\begin{cases}
\frac{\psi_{\nu/\la}}{ \psi_{\mu/\la}}\,\psi'_{\nu/\mu} & \nu/\mu
\text{ is a single box},\\
0,& \nu\ne \mu \text{ and }\nu\ne \mu\cup\{\square\},\\
-\sum_{\kappa=\mu\cup\{\square\}} \frac{\psi_{\kappa/\la}}{\psi_{\mu/\la}}\,\psi'_{\kappa/\mu},&\nu=\mu.
\end{cases}
\end{equation*}
From definitions of $\varphi$ and $\psi'$ one immediately sees that the coefficient of $\alpha$ in
$P_{a,\alpha}$ is $\frac{1-t}{1-q}$ times the coefficient of $\beta$ in $P_{a,\hat\beta}$.
\end{example}

Fix $N\ge 1$, and denote \glossary{$\ss^{(N)}$}
\begin{equation*}
\ss^{(N)}=\left\{(\la^{(1)},\dots,\la^{(N)})\in \Y(1)\times\dots\times\Y(N)\mid \la^{(1)}\prec
\la^{(2)}\prec \dots\prec\la^{(N)}\right\}.
\end{equation*}
The ascending Macdonald process $\M_{asc}(a_1,\dots,a_N;\rho)$ defined in Section~\ref{MacProcessSec} is a probability measure supported on $\ss^{(N)}$.

Let $\sigma$ be a nonnegative specialization with $\Pi(\sigma;a_j)<\infty$ for $j=1,\dots,N$.
Define a matrix $\frakP_\sigma$ with rows and columns parameterized by $\ss^{(N)}$
via \glossary{$\frakP_\sigma$}
\begin{equation*}
\frakP_\sigma\left((\la^{(1)},\dots,\la^{(N)}),({\mu}^{(1)},\dots,{\mu}^{(N)})\right)=
 \prod_{k=1}^N
P_{a_k,\sigma}\left(\mu^{(k)}\,\Vert
\, \mu^{(k-1)},\lambda^{(k)}\right)
\end{equation*}
The structure of $\frakP_\sigma$ is such that to compute the entry with row indexed by $(\la^{(1)},\dots,\la^{(N)})$ and column $(\mu^{(1)},\ldots, \mu^{(N)})$, one first computes the probability of $\mu^{(1)}$ given $\la^{(1)}$, then $\mu^{(2)}$ given $\la^{(2)}$ and $\mu^{(1)}$, then $\mu^{(3)}$ given $\la^{(3)}$ and $\mu^{(2)}$, and so on.

\begin{proposition}\label{prop15}
The matrix $\frakP_\sigma$ is well-defined and it is stochastic. Moreover,
\begin{equation*}
\M_{asc}(a_1,\dots,a_N;\rho)\, \frakP_\sigma=\M_{asc}(a_1,\dots,a_N;\rho,\sigma).
\end{equation*}
\end{proposition}
\begin{proof}
This is a special case of Proposition~\ref{prop13}. We specialize the notation of Section~\ref{genconstructionSec}
as follows: $n=N$, $\SS_k=\Y(k)$, $k=1,\dots,N$,
\begin{eqnarray*}
P_k(\la,\mu)&=&p^\uparrow_{\la\mu}(a_1,\dots,a_k;\sigma),  k=1,\dots,n,\\
\La^{k}_{k-1}(\la,\nu)&=&p_{\lambda\nu}^\downarrow(a_1,\dots,a_{k-1};a_k),  k=1,\dots,n.
\end{eqnarray*}
The commutation relation of equation (\ref{tag10}) follows from the third identity in Proposition~\ref{prop11}.
One further takes $m_n$ to be the Macdonald measure $\MM(a_1,\dots,a_n;\rho)$. This immediately
implies
\begin{equation*}
\tilde m_n=\MM(a_1,\dots,a_n;\rho)p^\uparrow(a_1,\dots,a_n;\sigma)=\MM(a_1,\dots,a_n;\rho,\sigma),
\end{equation*}
cf. equation (\ref{tag9a}), and the statement follows.
\end{proof}

\subsubsection{A continuous time analog}\label{ctnsMarkovChainsec}%of Proposition~\ref{prop15}
\index{multivariate Markov chain!continuous time}
Define a matrix $\frakq$ with rows and columns parameterized by $\ss^{(N)}$ as follows. For the off-diagonal entries, for any triple of integers $(A,B,C)$ such that
\begin{equation*}
1\le A\le B\le N,\qquad  0\le C\le N-B,
\end{equation*}
set \glossary{$\frakq$}
\begin{equation*}
\frakq\left(\{\la^{(1)},\dots,\la^{(N)}\},\{\mu^{(1)},\dots,\mu^{(N)}\}\right)=
a_B\frac{\psi_{\mu^{(B)}/\lambda^{(B-1)}}}{\psi_{\lambda^{(B)}/\lambda^{(B-1)}}}
\,\psi'_{\mu^{(B)}/\lambda^{(B)}}
\end{equation*}
if (we use the notation $\lambda^{(j)} = (\lambda^{(j)}_1 \geq \cdots \geq \lambda^{(j)}_{j})$)
\begin{eqnarray*}
&\la_A^{(B)}=\la_{A}^{(B+1)}=\cdots=\la_{A}^{(B+C)}=x,\\
&\mu_A^{(B)}=\mu_{A}^{(B+1)}=\cdots=\mu_{A}^{(B+C)}=x+1,
\end{eqnarray*}
$\mu_k^{(m)}=\la_k^{(m)}$ for all other values of $(k,m)$, and
\begin{equation*}
\frakq\left(\{\la^{(1)},\dots,\la^{(N)}\},\{\mu^{(1)},\dots,\mu^{(N)}\}\right)=0
\end{equation*}
if a suitable triple does not exist. The diagonal entries of $\frakq$ are then defined so that the sum of entries in every row is zero.

Less formally, this continuous time Markov chain can be described as follows. Each of the coordinates $\la_k^{(m)}$ has its own exponential clock with rate
\begin{equation*}
a_m\frac{\psi_{(\la^{(m)}\cup\square_k)/\lambda^{(m-1)}}}{\psi_{\lambda^{(m)}/\lambda^{(m-1)}}}
\,\psi'_{(\la^{(m)}\cup\square_k)/\lambda^{(m)}}\,,
\end{equation*}
where all clocks are independent. Here $\square_k$ denotes a box of a Young diagram that is located in the $k$th row. When the $\la_A^{(B)}$-clock rings, the coordinate checks if its jump by one to the right would violate the interlacing condition. If no violation happens, that is, if
\begin{equation*}
\la_A^{(B)}<\la_{A-1}^{(B-1)} \quad \text{and} \quad \la_A^{(B)}<\la_{A}^{(B+1)},
\end{equation*}
then this jump takes place. This means that we find the longest string $\la_A^{(B)}=\la_A^{(B+1)}=\dots=\la_A^{(B+C)}$ and move all the coordinates in this string to the right by one. If a jump would violate the interlacing condition, then no action is taken.

The reason to define $\frakq$ in this way is that it is the linear term in $\epsilon$ of the matrix $\frakP_\sigma$ defined earlier, when $\sigma$ is the specialization into a single dual variable $\epsilon$, cf. Example \ref{ex14}.3 above.

From the definition of $\psi,\psi'$ via products over boxes in equations (\ref{piereFormPsi}) and (\ref{piereFormPrime}), it is obvious that the off-diagonal matrix elements of $\frakq$ are nonnegative and uniformly bounded.
Also, the definition of $\frakq$ implies that in each row, at most $N(N+1)/2$ entries of $\frakq$ are nonzero. Thus, $\frakq$ uniquely defines a Feller Markov process on $\ss^{(N)}$ that has $\frakq$ as its generator, cf. \cite{LiggettMarkov}. For any $\tau\ge 0$, the matrix of transition probabilities for this process after time $\tau$ is $\frakQ_\tau=\exp(\tau\frakq)$. \glossary{$\frakQ_\tau$}

\begin{proposition}\label{prop16}
For any $\tau\ge 0$ we have
\begin{equation*}
\M_{asc}(a_1,\dots,a_N;\rho)\, \frakQ_\tau=\M_{asc}(a_1,\dots,a_N;\rho,\rho_\tau),
\end{equation*}
where $\rho_\tau$ is the Plancherel specialization afforded by equation (\ref{tag1}) with $\gamma=\tau$ and all other parameters equal to zero.
\end{proposition}
\begin{proof}
We start with Proposition~\ref{prop15} and take $\sigma=\hat{\epsilon}$ to be the specialization into a single dual variable $\epsilon$. Noting that $\frakP_{\hat \epsilon}$ is a triangular matrix whose entries are polynomials in $\epsilon$ of degree at most $N(N+1)/2$, we conclude that
\begin{equation*}
\lim_{\epsilon\to 0} (\frakP_{\hat \epsilon})^{[\tau\epsilon^{-1}]}=\exp(\tau\frakq)
\end{equation*}
entry-wise (recall that $\frakq$ is the coefficient of $\epsilon$ in $\frakP_{\hat \epsilon}$). On the other hand, Proposition~\ref{prop15} implies that
\begin{equation*}
\M_{asc}(a_1,\dots,a_N;\rho)\, (\frakP_{\hat \epsilon})^{[\tau\epsilon^{-1}]}
=\M_{asc}(a_1,\dots,a_N;\rho,\sigma_{\tau,\epsilon}),
\end{equation*}
where $\sigma_{\tau,\epsilon}$ is the specialization into $[\tau\epsilon^{-1}]$ dual variables equal to $\epsilon$. Since for any $\tau\ge 0$, we have that  $\lim_{\epsilon\to 0}~f(\sigma_{\tau,\epsilon})~=f(\rho_\tau)$ for any symmetric function $f$, and since also $\lim_{N\to \infty} \Pi(a_1,\ldots, a_N;\rho,\sigma_{\tau,\e}) = \Pi(a_1,\ldots, a_N;\rho,\rho_{\tau})$, it follows that $\M_{asc}(a_1,\dots,a_N;\rho,\sigma_{\tau,\epsilon})$ weakly converges to $\M_{asc}(a_1,\dots,a_N;\rho,\rho_\tau)$, and the proof is complete.
\end{proof}

\begin{remark}\label{rem17}
One could construct $\frakq$ starting from matrices $q^\uparrow$ introduced at the end of Section~\ref{commMarkovOps} by using the formalism of \cite{BorodinOlshanski}, Section 8. The needed commutativity relations follow from Corollary \ref{cor12}.
\end{remark}

\chapter{q-Whittaker processes}%: $q\in (0,1)$, $t=0$

\section{The q-Whittaker processes}\label{qWhitSec}

\subsection{Useful q-deformations}\label{qSec}

We record some $q$-deformations of classical functions and transforms. Section 10 of  \cite{AAR} is a good references for many these definitions and statements. We assume throughout that $|q|<1$. The classical functions are recovered in all cases in the $q\to 1$ limit, though the exact nature of this convergence will be relevant (and discussed) later.

\subsubsection{q-deformations of classical functions}\label{classicalqfunctions}
The {\it $q$-Pochhammer symbol}\index{q-deformed functions!Pochhammer symbol} is written as $(a;q)_{n}$ \glossary{$(a;q)_{n}$} and defined via the product (infinite convergent product for $n=\infty$ \glossary{$(a;q)_{\infty}$})
\begin{equation*}
(a;q)_{n}=(1-a)(1-aq)(1-aq^2)\cdots (1-aq^{n-1}), \qquad (a;q)_{\infty}=(1-a)(1-aq)(1-aq^2)\cdots.
\end{equation*}
The {\it $q$-factorial}\index{q-deformed functions!factorial} is written as either $[n]_{q}!$\glossary{$[n]_{q}!$} or just $n_q!$\glossary{$n_q!$} and is defined as
\begin{equation*}
n_q! = \frac{(q;q)_n}{(1-q)^n} = \frac{(1-q)(1-q^2)\cdots (1-q^n)}{(1-q)(1-q)\cdots (1-q)}.
\end{equation*}
The {\it $q$-binomial coefficients}\index{q-deformed functions!binomial coefficients} are defined in terms of $q$-factorials as \glossary{${n\choose k}_q$}
\begin{equation*}
{n\choose k}_q = \frac{n_q!}{k_q!(n-k)_q!} = \frac{(q;q)_{n}}{(q;q)_{k}(q;q)_{n-k}}.
\end{equation*}
The {\it $q$-binomial theorem}\index{q-deformed functions!binomial theorem} (\cite{AAR} Theorem 10.2.1) says that for all $|x|<1$ and $|q|<1$,
\begin{equation*}
\sum_{k=0}^{\infty} \frac{(a;q)_k}{(q;q)_k} x^k = \frac{(ax;q)_{\infty}}{(x;q)_{\infty}}.
\end{equation*}
Two corollaries of this theorem (\cite{AAR} Corollary 10.2.2a/b) which will be used later is that under the same hypothesis on $x$ and $q$,
\begin{equation}\label{qLaplace}
\sum_{k=0}^{\infty} \frac {x^k}{k_q!} = \frac{1}{\big((1-q)x;q\big)_{\infty}}, \qquad\qquad \sum_{k=0}^{\infty} \frac{(-1)^k q^{\frac{k(k-1)}{2}} x^k}{k_q!} = \big((1-q) x;q\big)_{\infty}.
\end{equation}
For any $x$ and $q$ we also have (\cite{AAR} Corollary 10.2.2.c)
\begin{equation}\label{finqBinExp}
\sum_{k=0}^{n} {n\choose k}_{q} (-1)^k q^{\frac{k(k-1)}{2}}x^k = (x;q)_n.
\end{equation}

%\begin{comment}
%The {\it $q$-integral} (see \cite{AAR} 10.1 for background) is defined as
%\begin{equation*}
%\int_0^{a} f(x) d_qx = \sum_{n=0}^{\infty} f(aq^n)(aq^n-aq^{n+1}), \qquad \int_0^{\infty} f(x) d_qx = (1-q)\sum_{n=-\infty}^{\infty} f(q^n)q^n.
%\end{equation*}
%\end{comment}

There are two different {\it $q$-exponential functions}\index{q-deformed functions!exponential functions}. The first (which we will use extensively) is denoted $e_q(x)$ and defined as
\begin{equation*}
e_q(x) = \frac{1}{\big((1-q)x;q\big)_{\infty}},
\end{equation*}
\glossary{$e_q(x)$}
while the second is defined as
\begin{equation*}
E_q(x) = \big(-(1-q)x;q\big)_{\infty}.
\end{equation*}
\glossary{$E_q(x)$}
For compact sets of $x$, both $e_q(x)$ and $E_q(x)$ converge uniformly to $e^{x}$ as $q\to 1$. In fact, the convergence of $e_q(x)\to e^{x}$ is uniform over $x\in (-\infty,0)$ as well.

The {\it $q$-gamma function}\index{q-deformed functions!gamma function} is defined as
\begin{equation}\label{qGamma}
\Gamma_q(x) = \frac{(q;q)_{\infty}}{(q^x;q)_{\infty}} (1-q)^{1-x}.
\end{equation}
\glossary{$\Gamma_q(x)$}
For $x$ in compact subsets of $\C\setminus\{0,-1,\cdots\}$, $\Gamma_q(x)$ converges uniformly to $\Gamma(x)$ as $q\to 1$.
%\note{should be mention the Ramanujan sum formula -- even though we don't use it here, it arises in TW work}

\subsubsection{q-Laplace transform}\label{qlaplaceinvsec}
\index{q-deformed functions!Laplace transform}
%\note{We should define the two different q laplace transforms and mention where the $E_q$ one is inverted, but that ours is not.. I had trouble accessing the Hahn paper so I have not written this yet}.

Define the following transform of a function $f\in \ell^1(\{0,1,\ldots\})$:
\begin{equation}\label{qlaplacedef}
\qhat{f}(z):= \sum_{n\geq 0} \frac{f(n)}{(zq^n;q)_{\infty}},
\end{equation}
where $z\in \C$.

\begin{proposition}\label{qlaplaceinverse}
One may recover a function $f\in \ell^1(\{0,1,\ldots\})$ from its transform $\qhat{f}(z)$ with $z\in \C\setminus\{q^{-k}\}_{k\geq 0}$ via the inversion formula
\begin{equation}\label{qlaplaceinverseEQN}
f(n) = -q^{n} \frac{1}{2\pi \iota} \int_{C_n} (q^{n+1}z;q)_{\infty} \qhat{f}(z) dz,
\end{equation}
where $C_n$ is any positively oriented contour which encircles the poles $z=q^{-M}$ for $0\leq M \leq n$.
\end{proposition}

\begin{remark}
This statement can be viewed as an inversion formula for the Laplace transform with $e_q(x)$ replacing the exponential function. An inversion of the Laplace transform with $E_q(x)$ replacing the exponential function goes back to \cite{Hahn}. However, for the $e_q(x)$ Laplace transform, it appears that the recent manuscript of Bangerezako \cite{Banger} contains the first inversion formula (as well as many other properties of the transform and worked out examples). We were initially unaware of this manuscript and thus produced our own inversion formula and the proof below.
\end{remark}

\begin{proof}
Observe that the residue of $\qhat{f}$ at $z=q^{-M}$ can be easily calculated for any $M\geq 0$ with the outcome
\begin{equation*}
\Res{z=q^{-M}} \qhat{f}(z)  = \sum_{n=0}^{M} \frac{ - f(n)}{(1-q^{n-M})\cdots (1-q^{-1})} q^{-M} \frac{1}{(q;q)_{\infty}}.
\end{equation*}
Alternatively this can be written in terms of matrix multiplication. Let $A=[A_{k,\ell}]_{k,\ell\geq 0}$ be an upper triangular matrix defined via its entries
\begin{equation*}
A_{k,\ell} = \bfone_{k\leq \ell} \frac{q^{-\ell}}{(1-q^{k-\ell})(1-q^{k-\ell+1})\cdots (1-q^{-1})}.
\end{equation*}

Then we have
\begin{equation*}
\Res{z=q^{-M}} \qhat{f}(z) = -\frac{1}{(q;q)_{\infty}} \vec{f} A = -\frac{1}{(q;q)_{\infty}} \sum_{n\geq 0} f(n) A_{n,M},
\end{equation*}
where $\vec{f}$ is the row vector $\vec{f}=\{f(0),f(1),\ldots\}$. Since $A$ is upper triangular, it has a unique inverse which is also upper triangular and can be readily calculated to find that $A^{-1} = B = [B_{k,\ell}]_{k,\ell\geq 0}$ is given by
\begin{equation*}
B_{k,\ell} = \bfone_{k\leq \ell} \frac{q^\ell}{(1-q^{\ell-k})(1-q^{\ell-k-1})\cdots (1-q)}.
\end{equation*}
Therefore
\begin{equation*}
f(n) =-(q;q)_{\infty} \sum_{M\geq 0}B_{M,n}\Res{z=q^{-M}} \qhat{f}(z) = -\sum_{M\geq 0} q^{n} (1-q^{n-M+1})(1-q^{n-M+2})\cdots \Res{z=q^{-M}} \qhat{f}(z).
\end{equation*}
Expressing this in terms of a contour integral yields the desired result of equation (\ref{qlaplaceinverseEQN}).
\end{proof}

\subsection{Definition and properties of q-Whittaker functions}\label{qwhittakersec}
\subsubsection{q-deformed Givental integral and recursive formula}
Macdonald polynomials in $\ell+1$ variables with $t=0$ are also known of as {\it $q$-deformed $\mathfrak{gl}_{\ell+1}$ Whittaker functions}\index{q-Whittaker functions} \cite{GLOqlim}. We denote the $P$ version of the $q$-Whittaker function as $\qWhitP$ and define it by
\begin{equation}\label{134}
\qWhitP_{x_1,\ldots, x_{\ell+1}}(\ul{p}{\ell+1}) = P_{\ul{p}{\ell+1}}(x_1,\ldots, x_{\ell+1}),
\end{equation}
\glossary{$\qWhitP_{x_1,\ldots, x_{\ell+1}}(\ul{p}{\ell+1})$}
where $\ul{p}{\ell+1} = \{p_{\ell+1,1},\ldots, p_{\ell+1,\ell+1}\}$. In \cite{GLOqlim} the $Q$ version of the $q$-Whittaker function is considered. It is denoted by $\qWhitQ$ and defined by \begin{equation*}
\qWhitQ_{x_1,\ldots, x_{\ell+1}}(\ul{p}{\ell+1}) = Q_{\ul{p}{\ell+1}}(x_1,\ldots, x_{\ell+1}).
\end{equation*}
The two functions are related by
\begin{equation*}
\qWhitQ_{x_1,\ldots, x_{\ell+1}}(\ul{p}{\ell+1}) = \Delta(\ul{p}{\ell+1}) \qWhitP_{x_1,\ldots, x_{\ell+1}}(\ul{p}{\ell+1})
\end{equation*}
where $\Delta(\ul{p}({ell+1})$ is defined in (\ref{deltadefeqn}). We will focus on the $\qWhitP$ function here which accounts for slight changes between what we write and what one finds in \cite{GLOqlim}. Note that the Pieri formulas for Section~\ref{pieresec} become, in this limit, the Hamiltonians for the quantum $q$-deformed $\mathfrak{gl}_{\ell+1}$-Toda chain (see e.g. \cite{Ru}, \cite{Et} or \cite{GLOqlim}).
%\note{This last statement needs to be revised... probably much more should be said about the ways the hamiltonians degenerate...}

These $q$-Whittaker functions can be expressed in terms of a combinatorial formula which follows from the combinatorial formula (\ref{7.13'}) for Macdonald polynomials. To state the combinatorial formula let us denote $\GT^{(\ell+1)}(\ul{p}{\ell+1})$ \glossary{$\GT^{(\ell+1)}(\ul{p}{\ell+1})$} to be the set of triangular arrays of integers $p_{k,i}$ for $1\leq i\leq k\leq \ell$ satisfying the {\it interlacing condition}\index{interlacing condition} that $p_{k+1,i}\geq p_{k,i}\geq~p_{k+1,i+1}$. One can also think of this set as those interlacing triangular arrays of height $\ell+1$ with a fixed top row given by $\ul{p}{\ell+1}$. Also let $\GT_{\ell+1,\ell}(\ul{p}{\ell+1})$ be a set $\ul{p}{\ell} = \{p_{\ell,1},\ldots, p_{\ell,\ell}\}$ of integers satisfying the interlacing condition $p_{\ell+1,i}\geq p_{\ell,i}\geq p_{\ell+1,i+1}$. Then \index{q-Whittaker functions!Givental integral representation}
\begin{equation}\label{135}
\qWhitP_{x_{1},\ldots, x_{\ell+1}}(\ul{p}{\ell+1}) = \sum_{\GT^{(\ell+1)}
(\ul{p}{\ell+1})} \prod_{k=1}^{\ell+1} x_k^{\sum_{i=1}^k p_{k,i} -\sum_{i=1}^{k-1} p_{k-1,i}} \frac{ \prod_{k=2}^{\ell+1}\prod_{i=1}^{k-1} \qq{p_{k,i}-p_{k,i+1}}}{\prod_{k=1}^{\ell}\prod_{i=1}^{k} \qq{p_{k+1,i}-p_{k,i}}\qq{p_{k,i}-p_{k+1,i+1}}}.
\end{equation}
For $\ul{p}{\ell+1}$ which is not ordered, we define the $q$-Whittaker function as zero (see Example 1.1 of \cite{GLOqlim}). Equation (\ref{135}) follows from (\ref{134}) and the combinatorial formula (\ref{7.13'}) for the Macdonald polynomials.

It follows from the combinatorial formula that the $q$-Whittaker functions satisfy a defining recursive relation: \index{q-Whittaker functions!Givental recursion relation}
\begin{equation}\label{receqn}
\qWhitP_{x_1,\ldots, x_{\ell+1}}(\ul{p}{\ell+1})= \sum_{\ul{p}{\ell}\in \GT_{\ell+1,\ell}(\ul{p}{\ell+1})} \Delta(\ul{p}{\ell+1}) \link_{\ell+1,\ell}(\ul{p}{\ell+1},\ul{p}{\ell};q)\qWhitP_{x_1,\ldots,x_{\ell}}(\ul{p}{\ell})
\end{equation}
where
\begin{eqnarray}\label{deltadefeqn}
\Delta(\ul{p}{\ell})  &=& \prod_{i=1}^{\ell-1} \qq{p_{\ell,i} - p_{\ell,i+1}},\\
\nonumber \link_{\ell+1,\ell}(\ul{p}{\ell+1},\ul{p}{\ell};q) &=& \prod_{i=1}^{\ell} \qq{p_{\ell+1,i}-p_{\ell,i}}^{-1}\qq{p_{\ell,i}-p_{\ell+1,i+1}}^{-1}.
\end{eqnarray}

%Even at the level of Macdonald polynomials (before taking $t=0$) there exist similar (though less explicit) combinatorial and recursive formulas (\cite{M} VI.9).

This recursive formula will play a prominent role later when we prove that the $t=0$ ascending Macdonald process converges to the Whittaker process. It is easy to see that a similar formula exists for $\qWhitQ$, as one finds in \cite{GLOqlim}.

\subsection{Difference operators and integral formulas}\label{expmomform26}
The present goal is to take limits $t\to +0$ in some of the previously stated results.

The nonnegative specializations in this case are described by the degeneration of equation (\ref{tag1}):
\begin{equation*}
\sum_{n\ge 0} g_n(\rho) u^n= \exp(\gamma u) \prod_{i\ge 1} \frac{(1+\beta_i u)}{(\alpha_i u;q)_\infty}\,.
\end{equation*}

When $t=0$ we call the ascending Macdonald process $\M_{asc,t=0}(a_1,\dots,a_N;\rho)$ \glossary{$\M_{asc,t=0}(a_1,\dots,a_N;\rho)$} the {\it $q$-Whittaker process}\index{q-Whittaker process} and the Macdonald measure $\MM_{t=0}(a_1,\dots,a_N;\rho)$ \glossary{$\MM_{t=0}(a_1,\dots,a_N;\rho)$} the {\it $q$-Whittaker measure} \index{q-Whittaker measure}.

The partition function for the corresponding $q$-Whittaker measure $\MM_{t=0}(a_1,\dots,a_N;\rho)$ is given by
\begin{equation*}
\sum_{\lambda\in \Y(N)} P_\la(a_1,\dots,a_N)Q_\la(\rho)=\Pi(a_1,\dots,a_N;\rho)=\prod_{j=1}^N \exp(\gamma a_j) \prod_{i\ge 1} \frac{(1+\beta_i a_j)}{(\alpha_i a_j;q)_\infty}\,,
\end{equation*}
where $\rho$ is determined by $\{\alpha_j\}$, $\{\beta_j\}$, and $\gamma$ as before, and we assume $\alpha_i a_j<1$ for all $i,j$ so that the series converge.

Taking the limit of Proposition~\ref{prop6} yields
\begin{proposition}\label{qsumlambdaprop}
For any $1\le r\le N$, and assuming $\alpha_i a_j<1$ for all $i,j$,
\begin{eqnarray*}
\lefteqn{\left\langle q^{\la_N+\la_{N-1}+\cdots+\la_{N-r+1}}\right\rangle_{\MM_{t=0}(a_1,\dots,a_N;\rho)}=}  \\
&& \frac{(-1)^{\frac{r(r+1)}2}}{(2\pi \iota)^rr!} \oint\cdots\oint
\prod_{1\le k<\ell\le r} (z_k-z_\ell)^2 \\
&& \times \prod_{j=1}^r \left(\prod_{m=1}^N \frac{a_m}{a_m-z_j}\right) \left(\prod_{i\ge 1}(1-\alpha_i z_j)\,\frac{1+q\beta_iz_j}{1+\beta_iz_j}\right)
\exp\big((q-1)\gamma z_j\big)\frac{dz_j}{z_j^r}\,,
\end{eqnarray*}
where the $z_j$-contours contain $\{a_1,\ldots, a_N\}$ and no other poles.
\end{proposition}
\begin{proof}
We have
\begin{eqnarray*}
\lim_{t\to 0}t^{-\frac {r(r-1)}2}e_r(q^{\la_1}t^{N-1},q^{\la_2}t^{N-1},\dots, q^{\la_N}) &=&
q^{\la_N+\la_{N-1}+\dots+\la_{N-r+1}},\\
\lim_{t\to 0}t^{-\frac {r(r-1)}2}\det\left[\frac 1{tz_k-z_l}\right]_{k,\ell=1}^r &=&(-1)^{\frac{r(r+1)}2}
\frac{\prod_{1\le k<\ell\le r} (z_k-z_\ell)^2}{\prod_{j=1}^r z_j^r}\,.
\end{eqnarray*}
Proposition~\ref{prop6} and relation (\ref{tag8}) conclude the proof.
\end{proof}

\begin{remark}
One may also take the $t\to 0$ limit of the operator $\tilde{D}_N^r$ defined in Remark \ref{firstlambdaremark}. Doing so we see that $\tilde{D}_n^r t^{rn-r(r+1)/2}$ converges to an operator whose eigenvalue at $P_{\lambda}$ is equal to $q^{-\lambda_1-\cdots - \lambda_r}$. Just as above, we may likewise take a limit of the contour integral formula for expectations.
\end{remark}

Let us take the limit $t\to 0$ of Proposition~\ref{prop8} along the same lines.

\begin{proposition}\label{prop8tzero}
For any $k \ge 1$,  and assuming $\alpha_i a_j<1$ for all $i,j$,
\begin{eqnarray*}
\lefteqn{\left\langle q^{k\la_N}\right\rangle_{\MM_{t=0}(a_1,\dots,a_N;\rho)} =}\\
&&\frac{(-1)^{k}q^{\frac{k(k-1)}2}}{(2\pi \iota)^k} \oint\cdots\oint
\prod_{1\le \kappa_1<\kappa_2\le k} \frac{z_{\kappa_1}-z_{\kappa_2}}{z_{\kappa_1}-qz_{\kappa_2}}\\
&&\times \prod_{j=1}^k
\left(\prod_{m=1}^N \frac{a_m}{a_m-z_j}\right) \left(\prod_{i\ge 1}(1-\alpha_i z_j)\,\frac{1+q\beta_iz_j}{1+\beta_iz_j}\right)
\exp((q-1)\gamma z_j)\frac{dz_j}{z_j}\,,
\end{eqnarray*}
where $z_j$-contour contains $\{qz_{j+1},\ldots,qz_k,a_1,\ldots,a_N\}$ and no other singularities for $j=1,\dots,k$.
\end{proposition}

\begin{proof} Straightforward limit $t\to 0$ of Proposition~\ref{prop8}.
\end{proof}

%\example{Remark 20} Similar contour integral formulas can be obtained for
%$$
%\bigl\langle q^{k(\la_N+\la_{N-1}+\dots+\la_{N-r+1})}\bigr\rangle_{\M(a_1,\dots,a_N;\rho)}
%$$
%for any $k\ge 1$ and $1\le r\le N$ by iterating Proposition 6. One would need $kr$ integration
%variables for this expectation.
%\endexample

Likewise we take the limit $t\to 0$ of Proposition~\ref{prop8FULL}.
\begin{proposition}\label{prop8FULLtzero}
Fix $k\ge 1$ and $r_{\kappa}\geq 1$ for $1\leq \kappa \leq k$. Then assuming $\alpha_i a_j<1$ for all $i,j$,
\begin{eqnarray*}
\lefteqn{\left\langle \prod_{\kappa=1}^{k} q^{\lambda_N+\cdots \lambda_{N-r_{\kappa}+1}} \right\rangle_{\MM_{t=0}(a_1,\dots,a_N;\rho)}=}\\
&&\prod_{\kappa=1}^{k} \frac{(-1)^{-r_{\kappa}}}{(2\pi \iota)^{r_{\kappa}} r_{\kappa}!} \oint\cdots \oint \left(\prod_{1\leq \kappa_1<\kappa_2\leq k}\prod_{i=1}^{r_{\kappa_1}}\prod_{j=1}^{r_{\kappa_2}} q\frac{z_{\kappa_1,i}-z_{\kappa_2,j}}{z_{\kappa_1,i}-qz_{\kappa_2,j}}\right)\left(\prod_{\kappa=1}^{k}\prod_{i\neq j=1}^{ r_\kappa}\frac{z_{\kappa,i}-z_{\kappa,j}}{-z_{\kappa,j}}\right)\\
&&\times  \prod_{\kappa=1}^{k}\prod_{j=1}^{r_{\kappa}}\frac{(-a_1)\cdots(-a_{n})}{(z_{\kappa,j}-a_1)\cdots(z_{\kappa,j}-a_{N})} \,\left(\prod_{i\ge 1}(1-\alpha_i z_{\kappa,j})\,\frac{1+q\beta_iz_{\kappa,j}}{1+\beta_iz_{\kappa,j}}\right)
e^{(q-1)\gamma z_{\kappa,j}}\frac{dz_{\kappa,j}}{z_{\kappa,j}}
\end{eqnarray*}

where the $z_{\kappa_1,j}$-contour contains $\{qz_{\kappa_2,i}\}$ for all $i\in\{1,\ldots, r_{\kappa_2}\}$ and $\kappa_2>\kappa_1$, as well as $\{a_1,\ldots,a_N\}$ and no other singularities.
\end{proposition}
\begin{proof} Straightforward limit $t\to 0$ of Proposition~\ref{prop8FULL}.
\end{proof}

Finally, let us take the limit $t\to 0$ in Proposition~\ref{prop10}.

\begin{proposition}\label{qlaplaceintegralform}
With the notation $b_m=a_mq^{\eta_m}$ for $m=1,\ldots,N$, we have
\begin{eqnarray*}
\lefteqn{\left\langle \frac {1}{(q^{\lambda_N}u;q)_\infty}
\right\rangle_{\MM_{t=0}(a_1,\dots,a_N;\rho)}=}\\
&&\frac{1}{(q;q)_{\infty}}\sum_{\eta_1,\dots,\eta_N=0}^\infty \prod_{i=1}^N u^{\eta_i}a_i^{\eta_{i+1}+\dots+\eta_N-(i-1)\eta_i}q^{-(i-1)\frac{\eta_i(\eta_i-1)}2}{(q^{\eta_i+1};q)_\infty} \\
&&\times \prod_{1\le j<k\le N}\left((1-b_jb_k^{-1})\frac{(q{b_j}a_k^{-1};q)_\infty} {({a_j}b_k^{-1};q)_\infty}\,\right) \frac{\Pi(b_1,\dots,b_N;\rho)}{\Pi(a_1,\dots,a_N;\rho)}.
\end{eqnarray*}
Here $u$ is a formal variable, and the formula can be viewed as an identity of formal power series in $u$.
\end{proposition}
\begin{proof} The only part in the identity of Proposition~\ref{prop10} that does not have an obvious limit is
\begin{equation*}
\prod_{i=1}^N t^{(i-1)\eta_i}\prod_{1\le j<k\le N} \frac{(\frac qt\,{x_j}y_k^{-1};q)_\infty}{(\frac qt\,{x_j}x_k^{-1};q)_\infty}\,.
\end{equation*}
We have
\begin{equation*}
t^{\eta_k}\frac{(\frac qt\,{x_j}y_k^{-1};q)_\infty}{(\frac qt\,{x_j}x_k^{-1};q)_\infty}=t^{\eta_k}\left(1-\frac qtx_jx_k^{-1}q^{-\eta_k}\right)\cdots \left(1-\frac qtx_jx_k^{-1}q^{-1}\right) \to (-x_jx_k^{-1})^{\eta_k} q^{-\frac{\eta_k(\eta_k-1)}2}
\end{equation*}
as $t\to 0$. Multiplying over all $1\le j<k\le N$ we obtain the result.
\end{proof}

\section{Moment and Fredholm determinant formulas}\label{qWhitSecFormulas}

\subsection{Moment formulas}

Recall the $q$-factorial $n_q! = \frac{(1-q)(1-q^2)\cdots (1-q^n)}{(1-q)(1-q)\cdots (1-q)}$ (see also Section~\ref{qSec}).

\begin{proposition}\label{mukprop}
For a meromorphic function $f(z)$ and $k\geq 1$ set $\poles$ to be a fixed set of singularities of $f$ (not including 0) and assume that $q^m\poles$ is disjoint from $\poles$ for all $m\geq 1$. Then setting
\begin{equation}\label{mukdef}
\mu_k:=\frac{(-1)^k q^{\frac{k(k-1)}{2}}}{(2\pi \iota)^k} \oint \cdots \oint \prod_{1\leq A<B\leq k} \frac{z_A-z_B}{z_A-qz_B} \frac{f(z_1)\cdots f(z_k)}{z_1\cdots z_k} dz_1\cdots dz_k,
\end{equation}
we have
\begin{equation}\label{muk}
\mu_k= k_q! \sum_{\substack{\lambda\vdash k\\ \lambda=1^{m_1}2^{m_{2}}\cdots}} \frac{1}{m_1!m_2!\cdots} \, \frac{(1-q)^{k}}{(2\pi \iota)^{\ell(\lambda)}} \oint \cdots \oint \det\left[\frac{1}{w_i q^{\lambda_i}-w_j}\right]_{i,j=1}^{\ell(\lambda)} \prod_{j=1}^{\ell(\lambda)}  f(w_j)f(qw_j)\cdots f(q^{\lambda_j-1}w_j) dw_j,
\end{equation}
where the $z_p$-contours contain $\{q z_j\}_{j> p}$, the fixed set of singularities $\poles$ of $f(z)$ but not 0, and the $w_j$ contours contain the same fixed set of singularities $\poles$ of $f$ and no other poles.
\end{proposition}

As a quick example consider $f(z)$ which has a pole at $z=1$. Then the $z_k$-contour is a small circle around 1, the $z_{k-1}$-contour goes around $1$ and $q$, and so on until the $z_1$-contour encircles $\{1,q,\ldots, q^{k-1}\}$ (see Figure \ref{circontours} for example). All the $w$ contours are small circles around 1 and can be chosen to be the same.

\begin{figure}
\begin{center}
\includegraphics[scale=1.3]{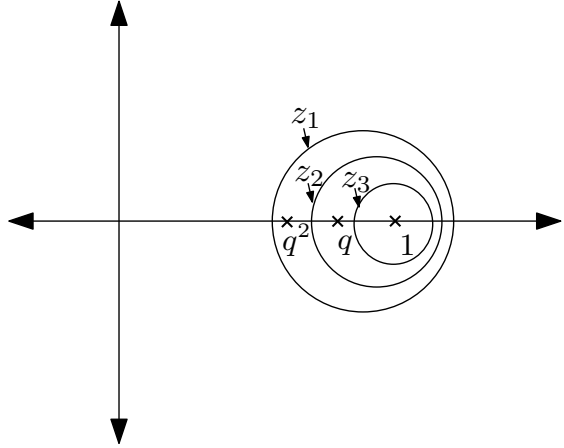}
\end{center}
\caption{Possible contours when $k=3$ for the $z_j$ contour integrals in Proposition~\ref{mukprop}}\label{circontours}
\end{figure}

\begin{proof}
This proposition amounts to book-keeping of the residues of $\prod \frac{z_A-z_B}{z_A - qz_B}$ and can be proved by considering a rich enough class of ``dummy'' functions $f(z)$. Assume that $f(z)$ has poles at $\poles:=\{a_j\}_{j\in J}$ (where $J=\{1,\ldots, |\poles|\}$ is an index set) inside the integration contour.
%$F_j:= \Res{z=a_j} F(z)$ and

When we evaluate $\mu_k$ as the sum of residues in (\ref{muk}), for each $\lambda\vdash k$ we have to sum over $\{a_{j_i}\}_{i=1}^{\ell(\lambda)}$ with $a_{j_i}\in \poles$ referring to the pole of $f(w_i)$. Because of the determinant, we get nonzero contributions only when all $a_{j_i}$ are mutually distinct. Also, if $\lambda_{i}=\lambda_{i'}$ then permuting $j_i$ and $j_{i'}$ does not affect the contribution, This symmetry allows us to cancel the prefactor $(m_1!m_2!\cdots)^{-1}$ and replace the summation over $\lambda$ by a summation over disjoint subsets of $J$ of size $m_1,m_2,$ and so on. Hence we can evaluate expression (\ref{muk}) for $\mu_k$ as a sum over sets $S=(S_1,S_2,\ldots,)$ of size $m_1,m_2,\ldots$ such that $S_i\subset \poles$, $|S_i|=m_i$, $m_1+2m_2+3m_3+\cdots =k$ and the $S_i$ are disjoint. Call $\mathcal{S}$ the set of all such sets $S=(S_1,S_2,\ldots)$. It is also convenient to index the elements of $S$ as
\begin{equation*}
S=\{b_{1},b_{2},\ldots,b_{m_1},b_{m_1+1},\ldots, b_{m_1+m_{2}},\ldots, b_{m_1+m_2+\cdots}\}
\end{equation*}
%$$S=\{a_{j_1},a_{j_2},\ldots,a_{j_{m_1}},a_{j_{m_1+1}},\ldots, a_{j_{m_1+m_{2}}},\ldots a_{j_{m_1+m_2+\cdots}}\}$$
where the first $m_1$ elements are in $S_1$, the next $m_2$ are in $S_2$ and so on. For a given $b\in S$ we denote by $\lambda(b)$ the index of the $S_i$ such that $b\in S_i$. Then by expanding into residues equation (\ref{muk}) takes the form
\begin{equation}\label{reseval}
\mu_k = \sum_{S\in \mathcal{S}}  k_q! (1-q)^{k} \prod_{j=1}^{m_1+m_2+\cdots} \Res{w=b_j}f(w) f(qb_j)\cdots f(q^{\lambda(b_j)-1}b_j) \det\left[\frac{1}{b_i q^{\lambda(b_i)}-b_j}\right]_{i,j=1}^{m_1+m_2+\cdots}.
\end{equation}

We will prove our identity (the equivalence of (\ref{mukdef}) and (\ref{muk})) via induction on $k$. The case $k=1$ is immediately checked. Let $k$ be general, and assume we already have a proof for $k-1$. In (\ref{mukdef}) we can evaluate the integral over $z_k$ as a sum of residues involving $(k-1)$-fold integral:
\begin{eqnarray}\label{eqnSS}
(\ref{mukdef}) &=& (-q^{k-1})\sum_{j\in J} \frac{\Resfrac{z=a_j} f(z)}{a_j} (-1)^{k-1} \frac{q^{\frac{(k-1)(k-2)}{2}}}{(2\pi \iota)^{k-1}} \\
\nonumber &&\times \oint \cdots \oint \prod_{1\leq A<B\leq k-1} \frac{z_A-z_B}{z_A-qz_B} \frac{\left(f(z_1)\frac{z_1-a_j}{z_1-qa_j}\right)\cdots \left(f(z_{k-1})\frac{z_{k-1}-a_j}{z_{k-1}-qa_j}\right)}{z_1\cdots z_{k-1}} dz_1\cdots dz_{k-1}.
\end{eqnarray}

Now apply the induction hypothesis to the $(k-1)$-fold integral in the right-hand side above with $\tilde{f}_j(z) = f(z)\frac{z-a_j}{z-qa_j}$ and the set of poles $\tilde{\poles}_j = (\poles\backslash\{a_j\})\cup \{qa_j\}$. We note two useful facts: For $\ell\neq j$
\begin{eqnarray}\label{eqnSSS}
\nonumber\tilde{f}_j(a_{\ell})\tilde{f}_j(q a_{\ell})\cdots \tilde{f}_j(q^{\lambda_{\ell}-1}a_{\ell}) &=& \frac{a_{\ell}-a_k}{a_{\ell}-qa_j}\,  \frac{qa_{\ell}-a_k}{qa_{\ell}-qa_j}\cdots  \frac{q^{\lambda_{\ell}-1}a_{\ell}-a_k}{q^{\lambda_{\ell}-1}a_{\ell}-qa_j}\, f(a_{\ell})\cdots f(q^{\lambda_{j}-1}a_{\ell})\\
&=& \frac{q^{\lambda_{\ell}-1}a_{\ell}-a_j}{a_{\ell}-qa_j} \, q^{1-\lambda_{\ell}} f(a_{\ell})\cdots f(q^{\lambda_{j}-1}a_{\ell}),
\end{eqnarray}
and
\begin{eqnarray}\label{eqnSSSS}
\nonumber\Res{z=qa_j} \tilde{f}_j(z) \tilde{f}_{j}(q^2a_j)\cdots \tilde{f}_j(q^{\lambda_j -1}a_j) &=& (q-1)a_j \frac{q^2 a_j-a_j}{q^2a_j-qa_j}\cdots \frac{q^{\lambda_j -1}a_j - a_j}{q^{\lambda_j -1}a_j - qa_j} f(qa_j)\cdots f(q^{\lambda_j-1}a_j)\\
&=& (q^{\lambda_{j}-1}a_j-a_j)q^{2-\lambda_j}f(qa_j)\cdots f(q^{\lambda_j-1}a_j).
\end{eqnarray}

By induction, the integral in the right-hand side of (\ref{eqnSS}) can be written (according to the discussion at the beginning of this proof) as a sum over non-intersecting subsets $\tilde{S}_j =(\tilde{S}_{j,1},\tilde{S}_{j,2}\ldots)$ with  $\tilde{S}_{j,\ell}\subset \tilde{\poles}_j$, $|\tilde{S}_{j,\ell}|=\tilde{m}_{\ell}$ and $\tilde{m}_1+2\tilde{m}_2+3\tilde{m}_3+\cdots = k-1$ (let $\tilde{\mathcal{S}}_j$ be the set of all such $\tilde{S}_j$). Thus from (\ref{reseval}) the right-hand side of (\ref{eqnSS}) can be expanded as
\begin{eqnarray}\label{resexptilde}
\nonumber&&(k-1)_q! (1-q)^{k-1} \sum_{j\in J} \sum_{\tilde{S}\in \tilde{\mathcal{S}}_j} -q^{k-1} \frac{1}{a_j} \det\left[\frac{1}{\tilde b_i q^{\tilde\lambda(\tilde b_i)}-\tilde b_\ell}\right]_{\tilde b_i,\tilde b_j\in \tilde{S}_j}\\
&&\times \prod_{\tilde b\in \tilde{S}\backslash \{qa_j\}} \Res{w=b} \tilde{f}_j(w) \cdot \tilde{f}_j(q\tilde b)\cdots \tilde{f}_{j}(q^{\tilde{\lambda}(\tilde b)-1}\tilde b) \\
\nonumber &&\times \Res{w=a_j} f(w)\cdot  \Res{w=qa_j} \tilde{f}_{j}(w)\cdot \tilde{f}_j(q^2 a_j)\cdots \tilde{f}_j(q^{\tilde{\lambda}(qa_j)}a_j),
\end{eqnarray}
where $\tilde{\lambda}(\tilde b)$ is the index of the set $\tilde{S}_{i,\cdot}$ that contains $\tilde b$.

It will be more convenient to map these subsets and $a_j$ onto a new set of subsets. From $a_j\in \poles$ and the collection $(\tilde{S}_{j,1},\tilde{S}_{j,2},\ldots)$ we now construct a collection $(S_1,S_2,\ldots)$ of disjoint subsets of $\poles$ with $|S_i|=m_i$ and $m_1+2m_2+3m_3+\cdots = k$. If $qa_j\in \tilde{S}_{j,\ell}$ for some $\ell$ then
\begin{equation*}
S_{\ell}=\tilde{S}_{j,\ell}\backslash \{qa_j\}; \qquad S_{\ell+1}=\tilde{S}_{j,\ell+1}\cup\{a_j\}; \qquad S_m=\tilde{S}_{j,m}, \quad m\neq \ell,\ell+1.
\end{equation*}
If $qa_j\notin \cup_{\ell\geq 1} \tilde{S}_{j,\ell}$ then
\begin{equation*}
S_{1} = \tilde{S}_{j,1}\cup\{a_j\}; \qquad S_{m}=\tilde{S}_{j,m}, \quad m>1.
\end{equation*}

Then using (\ref{eqnSSS}) for the second line of (\ref{resexptilde}) and (\ref{eqnSSSS}) for the third line we find that for a given collection $(S_1,S_2,\ldots)$, the associated term in equation (\ref{resexptilde}) is given by

\begin{eqnarray*}
&&(k-1)_q!(1-q)^{k-1}\sum_{j=1}^{m_1+m_2+\cdots} -q^{k-1} \frac{1}{b_j} b_j(q^{\lambda(b_j)-1}-1)q^{2-\lambda(b_j)} \prod_{\ell\neq j} \frac{q^{\lambda(b_\ell)-1}b_\ell - b_j}{b_{\ell}-qb_j} q^{1-\lambda(b_\ell)}\\
&& \times\det\left[\frac{1}{\tilde{b}_i q^{\lambda(b_i)-\delta_{ij}}-\tilde{b}_{\ell}}\right]_{i,\ell=1}^{m_1+m_2+\cdots} \prod_{b\in S_1\cup S_2\cup \cdots } \Res{w=b} f(w) f(qb)f(q^2b)\cdots f(q^{\lambda(b)-1}b),
\end{eqnarray*}
where $\tilde{b}_i=b_i q^{\delta_{ij}}$ and $\delta_{ij}=1$ if $i=j$ and zero otherwise. Note that although the above expression makes sense literally only for $\lambda(b_j)>1$, the natural limit $\lambda(b_j)\to 1$ gives the correct contribution for those $b\in S_1$. In that case the factor $b_j(q^{\lambda(b_j)-1}-1)q^{2-\lambda(b_j)}$ disappears and the $j^{th}$ row and column of the matrix whose determinant we compute is removed.

To prove the induction step, we need to show that the sum above equals the analogous term in (\ref{reseval}), which is
\begin{equation*}
k_q! (1-q)^{k} \det\left[\frac{1}{b_i q^{\lambda(b_i)}-b_\ell}\right]_{i,\ell=1}^{m_1+m_2+\cdots} \prod_{b\in S_1\cup S_2\cup\cdots} \Res{w=b}f(w) f(qb)\cdots f(q^{\lambda(b)-1}b).
\end{equation*}
The contribution of $f$ can be canceled right away. Recalling the $q$-factorial definition and gathering factors of $q$ we are thus lead to prove the following identity
\begin{eqnarray*}
&&\sum_{j=1}^{m_1+m_2+\cdots} (q-q^{\lambda(b_j)})\prod_{\ell:\ell\neq j} \frac{q b_j -q^{\lambda(b_\ell)}b_\ell}{qb_j - b_{\ell}} \det\left[\frac{1}{b_i q^{\lambda(b_i)} - b_{\ell}q^{\delta_{\ell j}}}\right]_{i,\ell=1}^{m_1+m_2+\cdots} \\
&&= (1-q^{\sum\lambda(b_i)}) \det\left[\frac{1}{b_i q^{\lambda(b_i)} - b_{\ell}}\right]_{i,\ell=1}^{m_1+m_2+\cdots}.
\end{eqnarray*}

Abbreviating $\lambda(b_i)=:\lambda_i$ and evaluating the Cauchy determinants via
\begin{equation*}
\det\left[\frac{1}{x_i+y_j}\right] = \frac{\prod_{i<j} (x_i-x_j)(y_i-y_j)}{\prod_{i,j}(x_i+y_j)} = \frac{V(x_i)V(y_i)}{\prod_{i,j}(x_i+y_j)}
\end{equation*}
(where $V(x_i)= \prod_{i<j} (x_i-x_j)$ is the Vandermonde determinant) we obtain
\begin{equation*}
\sum_{j=1}^{m_1+m_2+\cdots} (q-q^{\lambda_j}) \prod_{\ell\neq j} \frac{qb_j - q^{\lambda_\ell}b_{\ell}}{qb_j-b_\ell} \, \frac{V(b_iq^{\lambda_i})V(b_\ell q^{\delta_{\ell j}})}{\prod_{i,\ell} (b_{\ell}q^{\delta_{\ell j}} - b_i q^{\lambda_i})} = (1-q^{\sum \lambda_i}) \frac{V(b_iq^{\lambda_i})V(b_\ell)}{\prod_{i,\ell} (b_{\ell} - b_i q^{\lambda_i})}.
\end{equation*}

Equivalently, the desired identity reduces to showing that
\begin{equation*}
\sum_{j\geq 1} \frac{\prod_{\ell\geq 1} (b_j-b_{\ell}q^{\lambda_\ell})}{\prod_{\ell\neq j}(b_j-b_{\ell})} \frac{1}{b_j} = 1-q^{\sum\lambda_i}.
\end{equation*}

The left-hand side of the above formula is the sum of the residues of the function
\begin{equation*}
\prod_{\ell\geq 1} \frac{z-b_{\ell}q^{\lambda_\ell}}{z-b_{\ell}} \frac{1}{z}
\end{equation*}
at the points $z=b_{\ell}$ for $\ell\geq 1$. On the other hand $-(1-q^{\sum \lambda_\ell})$ is the difference of residues of the same function at $z=\infty$ and $z=0$. This completes the proof.
\end{proof}

We state an analogous formula to the one contained in Proposition~\ref{mukprop} which now involves integrating around zero and over large contours.

\begin{proposition}\label{mukproplarge}
For any continuous function $f(z)$ and $k\geq 1$,
\begin{eqnarray}\label{mukproplargeeqn}
&& \frac{(-1)^k q^{\frac{k(k-1)}{2}}}{(2\pi \iota)^k} \oint \cdots \oint \prod_{1\leq A<B\leq k} \frac{z_A-z_B}{z_A-qz_B} \frac{f(z_1)\cdots f(z_k)}{z_1\cdots z_k} dz_1\cdots dz_k\\
&&= \frac{k_q!}{k!} \frac{(1-q^{-1})^{k}}{(2\pi \iota)^{k}} \oint \cdots \oint \det\left[\frac{1}{w_i q^{-1}-w_j}\right]_{i,j=1}^{k} \prod_{j=1}^{k} f(w_j) dw_j,
\end{eqnarray}
where the $z_j$-contours and $w_j$-contours are all the same.
\end{proposition}

\begin{remark}\label{starstar31}
Assuming $\poles$ is inside a small enough neighborhood of $z=1$, the integration contours of Proposition~\ref{mukprop} with additional small circles around $z_j=0$, $j=1,\ldots, k$ (cf. Proposition~\ref{muandmutildeprop} below) can be deformed without passing through singularities to the same circle of the form $|z_j|=R>1$.
\end{remark}

%\begin{proposition}\label{mukproplarge}
%For a meromorphic function $f(z)$ and $k\geq 1$ set $\poles$ to be a fixed set of singularities of $f$ and assume that $q^m\poles$ is disjoint from $\poles$ for all $m\geq 1$. Then %setting
%\begin{equation}\label{muktildedef}
%\tilde{\mu}_k=(-1)^k \frac{q^{\frac{k(k-1)}{2}}}{(2\pi \iota)^k} \oint \cdots \oint \prod_{1\leq A<B\leq k} \frac{z_A-z_B}{z_A-qz_B} \frac{f(z_1)\cdots f(z_k)}{z_1\cdots z_k} dz_1\cdots dz_k.
%\end{equation}
%Then it follows that $\tilde{\mu}_k$ is also given by
%\begin{equation}\label{muklarge}
%\frac{k_q!}{k!} \frac{(1-q^{-1})^{k}}{(2\pi \iota)^{k}} \oint \cdots \oint \prod_{j=1}^{k} dw_j f(w_j) \det\left[\frac{1}{w_i q^{-1}-w_j}\right]_{i,j=1}^{k},
%\end{equation}
%where, as before, $z_p$ contours contain $\{q z_j\}_{j> p}$, a fixed set of singularities $\poles$ of $f(z)$ and 0, and the $w_j$ contour contains $\{qw_i\}_{i\neq j}$, the same fixed set of singularities $\poles$ of $f$ and 0 (for example all $z$ and $w$ contours could be the same circle $|w_j|=R\gg 1$).
%\end{proposition}

\begin{proof}
The following is a useful combinatorial identity from which the result readily follows. This identity can be derived using Proposition~\ref{mukprop} by choosing $f$ with exactly $k$ poles and then evaluating both sides via residues. It can also be shown through induction.

\begin{lemma}\cite[III,(1.4)]{M}
Set $c_{q,k} = (q^{-1}-1)\cdots (q^{-k}-1)$ then
\begin{equation*}
\sum_{\sigma\in S_k} \prod_{1\leq A<B\leq k} \frac{z_{\sigma(A)}-z_{\sigma(B)}}{z_{\sigma(A)}-qz_{\sigma(B)}} = c_{q,k} z_{1}\cdots z_{k} \det\left[\frac{1}{z_i q^{-1} -z_j}\right]_{i,j=1}^{k}.
\end{equation*}
\end{lemma}

We may use this to symmetrize the integrand in the right-hand side of (\ref{mukproplargeeqn}). The only non-symmetric term is the product over $A<B$. However, using the above combinatorial identity we find that (recall the $k!$ came from symmetrizing)
\begin{equation*}
(\ref{mukproplargeeqn})
= \frac{1}{k!}(-1)^k c_{q,k} \frac{q^{\frac{k(k-1)}{2}}}{(2\pi \iota)^k} \oint \cdots \oint \det\left[\frac{1}{w_i q^{-\lambda_i}-w_j}\right]_{i,j=1}^{k} \prod_{j=1}^{k} f(w_j) dw_j  .
\end{equation*}
It is now straightforward to check that the prefactors of the integral are as desired.
\end{proof}

%\note{About big contours. Please check carefully here, I've been back and forth a few times... First, why we need to include 0. In our main example, $z_k$-contour includes 1, $z_{k-1}$-contour includes 1 and q; ...;  $z_1$-contour includes $1,q,q^2,...,q^{k-1}$. This sequence accumulates to 0, where we have an additional pole. If we want contours that are k-independent, which we need for a Fredholm det., we have to include 0.}

\begin{proposition}\label{muandmutildeprop}
For simplicity assume $f(0)=1$. Recall $\mu_j$ from (\ref{mukdef}) (with $\mu_0=1$ for convenience) and define $\tilde{\mu}_j$ by adding a small circle around 0 to each of the $z_j$ contours in (\ref{mukdef}). Then
%and define $\tilde{\mu}_k$ and $\mu_j$ (for $0\leq j\leq k$ -- $\mu_0=1$) with respect to the same set of poles $\poles$. Then

\begin{equation*}
\tilde{\mu}_k = (-1)^k q^{\frac{k(k-1)}{2}} \sum_{j=0}^{k} {k\choose j}_{q^{-1}}(-1)^{j} q^{-\frac{j(j-1)}{2}} \mu_j.
\end{equation*}
\end{proposition}
\begin{proof}
Define a triangular array of integrals $\nu_{j,k}$ by
\begin{equation*}
\nu_{j,k} = \frac{1}{(2\pi \iota)^{k+j}} \oint\cdots \oint \prod_{1\leq A<B\leq k+j} \frac{z_A-z_B}{z_A-qz_B} f(z_1)\cdots f(z_{k+j})\frac{dz_1\cdots dz_{k+j}}{z_1\cdots z_{k+j} }
\end{equation*}
where the contours for $z_{1},\ldots, z_{j}$ do not contain 0, whereas the contours for $z_{j+1},\ldots, z_{k+j}$ do contain 0.

By shrinking the contour $z_{j+1}$ without crossing through 0 we can establish the recurrence relation
\begin{equation*}
\nu_{j,k} = \nu_{j+1,k-1} + q^{-k+1} \nu_{j,k-1}.
\end{equation*}
This can be solved to find that
\begin{equation*}
\nu_{0,k} = \sum_{j=0}^{k}  {k\choose j}_{q^{-1}} \nu_{j,0},
\end{equation*}
from which the proposition follows by relating $\nu_{j,0}$ to $\mu_j$ and $\nu_{0,k}$ to $\tilde{\mu}_k$.
\end{proof}

\subsection{General Fredholm determinant formula}
In this section we state certain formal identities between power series and Fredholm determinant expansions. Later we specialize our choice of functions and show that these formal identities then hold numerically as well. We should first, however, specify what we mean by a formal identity versus a numerical identity, and also what is a Fredholm determinant (in both contexts).

\subsubsection{Background and definitions for Fredholm determinants}
\index{Fredholm determinant}

For a general overview of the theory of Fredholm determinants, the reader is referred to \cite{Lax, BSimon, GKrein}. However, for our purposes, the below definition will suffice and no further properties of the determinants will be needed.

\begin{definition}
Fix a Hilbert space $L^2(X,\mu)$ where $X$ is a measure space and $\mu$ is a measure on $X$. When $X=\Gamma$, a simple (anticlockwise oriented) smooth contour in $\C$, we write $L^2(\Gamma)$ where for $z\in \Gamma$, $d\mu(z)$ is understood to be $\tfrac{dz}{2\pi \iota}$. When $X$ is the product of a discrete set $D$ and a contour $\Gamma$, $d\mu$ is understood to be the product of the counting measure on $D$ and $\frac{dz}{2\pi \iota}$ on $\Gamma$.

Let $K$ be an {\it integral operator} acting on $f(\cdot)\in L^2(X,\mu)$ by $Kf(x) = \int_{X} K(x,y)f(y) d\mu(y)$. $K(x,y)$ is called the {\it kernel} of $K$ and we will assume throughout that $K(x,y)$ is continuous in both $x$ and $y$. A {\it formal Fredholm determinant expansion} of $I+K$ is a formal series written as
\begin{equation*}
\det(I+K)_{L^2(X)} = 1+\sum_{n=1}^{\infty} \frac{1}{n!} \int_{X} \cdots \int_{X} \det\left[K(x_i,x_j)\right]_{i,j=1}^{n} \prod_{i=1}^{n} d\mu(x_i).
\end{equation*}
In other words, it is the formal power series in $z$ of $\det(I+zK)_{L^2(X)}$ with $z$ set to 1.

If $K$ is a {\it trace-class} operator then the terms in the above expansion can be evaluated and form a convergent series. This is called the {\it numerical Fredholm determinant expansion} as it actually takes a numerical value. Note that it is not necessary that $K$ be trace-class for the series to be absolutely convergent.

The following is a useful criteria for an operator to be trace-class (see \cite{Lax} page 345 or \cite{Bornemann}).
\begin{lemma}\label{traceclasscrit}
An operator $K$ acting on $L^2(\Gamma)$ for a simple smooth contour $\Gamma$ in $\C$ with integral kernel $K(x,y)$ is trace-class if $K(x,y):\Gamma^2\to \R$ is continuous as well as $K_2(x,y)$ is continuous in $y$. Here $K_2(x,y)$ is the derivative of $K(x,y)$ along the contour $\Gamma$ in the second entry.
\end{lemma}

A formal power series in a variable $\zeta$ and a formal Fredholm determinant expansion whose kernel $K(x,y)$ depends implicitly on $\zeta$ are {\it formally equal} (or their equality is a {\it formal identity}) if they are equal as formal power series in $\zeta$.
\end{definition}

\subsubsection{Fredholm determinant formulas}

\begin{proposition}\label{gendetprop}
Consider $\mu_k$ as in equation (\ref{muk}) defined with respect to the same set of poles $\poles$ for $k=1,2,\ldots$, and set the $w_j$-contours all to be $C_{w}$ such that $C_{w}$ contains the set of poles $\poles$ of the meromorphic function $f(w)$. Then the following formal equality holds
\begin{equation}\label{deteqn}
\sum_{k\geq 0}\mu_k \frac{\zeta^k}{k_q!} = \det(I+K)_{L^2(\Zgzero\times C_{w})}
\end{equation}
where $\det(I+K)$ is a formal Fredholm determinant expansion
%of
%\begin{equation*}
%K:L^2(\Zgzero\times C_{w}) \to L^2(\Zgzero\times C_{w})
%\end{equation*}
and the operator $K$ is defined in terms of its integral kernel
\begin{equation*}
K(n_1,w_1;n_2;w_2) = \frac{(1-q)^{n_1}\zeta^{n_1} f(w_1)f(qw_1)\cdots f(q^{n_1-1}w_1)}{q^{n_1} w_1 - w_2}.
\end{equation*}
The above identity is formal, but also holds numerically if the following is true:
There exists a positive constant $M$ such that for all $w\in C_{w}$ and all $n\geq 0$, $|f(q^n w)|\leq M$ and $\zeta$ is such that $|(1-q)\zeta| <M^{-1}$; and $1/|q^n w -w'|$ is uniformly bounded from zero for all $w,w'\in C_{w}$ and $n\geq 1$.
\end{proposition}
\begin{proof} All contour integrals in this proof are along $C_{w}$. Observe that we can rewrite the summation in the definition of $\mu_k$ so that
\begin{equation*}
\mu_k \frac{\zeta^k}{k_q!}= \sum_{L\geq 0}\sum_{\substack{m_1,m_2,\ldots\\ \sum m_i =L \\ \sum i m_i = k}} \frac{1}{(m_1+m_2+\cdots)!}\cdot \frac{(m_1+m_2+\cdots)!}{m_1! m_2! \cdots}\oint\cdots \oint \prod_{j=1}^{L}  I_L(\lambda;w;\zeta) dw_j,
\end{equation*}
where $w=(w_1,\ldots, w_L)$ , $\lambda=(\lambda_1,\ldots, \lambda_L)$ and is specified by $\lambda=1^{m_1}2^{m_2}\cdots$, and where the integrand is
\begin{equation*}
I_L(\lambda;w;\zeta) = \frac{1}{(2\pi \iota)^L} \det\left[\frac{1}{w_i q^{\lambda_i}-w_j}\right]_{i,j=1}^{L} \prod_{j=1}^{L}(1-q)^{\lambda_j}\zeta^{\lambda_j} f(w_j)f(qw_j)\cdots f(q^{\lambda_j-1}w_j).
\end{equation*}

The term $\tfrac{(m_1+m_2+\cdots)!}{m_1! m_2! \cdots}$ is a multinomial coefficient and can be removed by replacing the inner summation by
\begin{equation*}
\sum_{n_1,\ldots,n_L\in \mathcal{L}_{k,m_1,m_2,\ldots}} \oint\cdots \oint \prod_{j=1}^{L}  I_L(n;w;\zeta)dw_j,
\end{equation*}
with $n=(n_1,\ldots,n_L)$ and where $\mathcal{L}_{k,m_1,m_2,\ldots} = \{n_1,\ldots,n_L\geq 1: \sum n_i = k \textrm{ and for each } j\geq 1, m_j \textrm{ of the } n_i \textrm{ equal } j\}$.
This gives

\begin{equation*}
\mu_k \frac{\zeta^k}{k_q!} = \sum_{L\geq 0}\frac{1}{L!} \sum_{\substack{n_1,\ldots,n_L\geq 1\\ \sum n_i=k}} \oint\cdots \oint \prod_{j=1}^{L} I_L(n;w;\zeta) dw_j.
\end{equation*}

Now we may sum over $k$ which removes the requirement that $\sum n_i = k$. This yields that the left-hand side of equation (\ref{deteqn}) can be expressed as
\begin{equation}\label{fredexpabove}
\sum_{L\geq 0} \frac{1}{L!} \sum_{n_1,\ldots,n_L\geq 1} \oint \cdots \oint \det\left[\frac{1}{q^{n_i}w_i-w_j}\right]_{i,j=1}^{L} \prod_{j=1}^{L} (1-q)^{n_j}\zeta^{n_j} f(w_j)f(qw_j)\cdots f(q^{n_j-1}w_j) \frac{dw_j}{2\pi \iota } .
\end{equation}
This is the definition of the formal Fredholm determinant expansion $\det(I+K)_{L^2(\Zgzero\times C_{w})}$, as desired.

We turn now to the additional conditions for numerical equality. In order to justify the rearrangements in the above argument, we must know that the double summation of (\ref{fredexpabove}) is absolutely convergent. By the conditions stated and Hadamard's bound we find that:
\begin{equation*}
\left|   \oint \cdots \oint  \det\left[\frac{1}{q^{n_i}w_i-w_j}\right]_{i,j=1}^{L} \prod_{j=1}^{L} (1-q)^{n_j}\zeta^{n_j} f(w_j)f(qw_j)\cdots f(q^{n_j-1}w_j)\frac{dw_j}{2\pi \iota} \right| \leq \prod_{j=1}^{L} r^{n_j} B^{L} L^{L/2}
\end{equation*}
for some constant $B$ large enough and a positive constant $r<1$ (such that $|(1-q)\zeta C|\leq~r~<1$). This bound may be plugged into the summation in $L$ and in $n_1,\ldots, n_{L}$ and evaluating the geometric sums we arrive at
\begin{equation*}
\sum_{L\geq 0} \frac{1}{L!} \left(\frac{B}{(1-r)}\right)^L L^{L/2}
\end{equation*}
which is finite due to the $L!$ in the denominator. This absolute convergence means we can permute terms any way we want, and hence justifies the earlier calculations as being numerical identities (and not just formal manipulations).
\end{proof}

\begin{proposition}\label{tildemuDetProp}
For
\begin{equation}\label{muktildedef}
\tilde{\mu}_k= \frac{(-1)^k q^{\frac{k(k-1)}{2}}}{(2\pi \iota)^k} \oint \cdots \oint \prod_{1\leq A<B\leq k} \frac{z_A-z_B}{z_A-qz_B} \frac{f(z_1)\cdots f(z_k)}{z_1\cdots z_k} dz_1\cdots dz_k,
\end{equation}
cf. Proposition~\ref{mukproplarge} and Remark \ref{starstar31}, with $k\geq 1$ and integrals along a fixed contour $\tilde{C}_w$, the following formal identity holds
\begin{equation*}
\sum_{k\geq 0}\tilde{\mu}_k \frac{\zeta^k}{k_q!} = \det(I+\tilde{K})_{L^2(\tilde{C}_{w})}
\end{equation*}
where $\det(I+\tilde{K})$ is a formal Fredholm determinant expansion
% of
%\begin{equation*}
%\tilde K:L^2(\tilde{C}_{w}) \to L^2(\tilde{C}_{w})
%\end{equation*}
and the operator $\tilde K$ is defined in terms of its integral kernel
\begin{equation*}
\tilde{K}(w_1,w_2) = (1-q^{-1})\zeta \frac{f(w_1)}{w_1 q^{-1}-w_2}.
\end{equation*}
The above identity is formal, but also holds numerically for $\zeta$ such that the left-hand side converges absolutely and the right-hand side operator $\tilde K$ is trace-class.
\end{proposition}

%Consider $\tilde{\mu}_k$ as in equation (\ref{muktildedef}) defined with respect to the same set of poles $\poles$. Set the $w_j$ contours all to be $\tilde{C}_{w}$ such that $\tilde{C}_{w}$ contains the set of poles $\poles$ of the meromorphic function $f(w)$ (though does not cross through any) and is star-shaped with respect to the origin. Then the following formal equality holds
%\begin{equation*}
%\sum_{k\geq 0}\tilde{\mu_k} \frac{\zeta^k}{k_q!} = \det(I+\tilde{K})
%\end{equation*}
%where $\det(I+\tilde{K})$ is the formal Fredholm determinant expansion of
%\begin{equation*}
%K:L^2(\tilde{C}_{w}) \to L^2(\tilde{C}_{w})
%\end{equation*}
%defined in terms of its integral kernel (the $2\pi \iota$ is absorbed into $dw$)
%\begin{equation*}
%\tilde{K}(w_1,w_2) = (1-q^{-1})\zeta \frac{f(w_1)}{w_1 q^{-1}-w_2}.
%\end{equation*}
%The above identity is formal, but also holds numerically for those $\zeta$ such that the left-hand side converges absolutely and the right-hand side operator $K$ is trace-class.

\begin{proof}
Observe that $\tilde{\mu}_k$ coincides with the left-hand side of Proposition \ref{mukproplarge}. Rewriting it as in the right-hand side and plugging this formula into the power series gives (absorbing the $2\pi \iota$ into $dw$):
\begin{equation*}
\sum_{k\geq 0}\tilde{\mu}_k \frac{\zeta^k}{k_q!} = \sum_{k\geq 0}\frac{1}{k!} \oint_{\tilde{C}_w} \cdots \oint_{\tilde{C}_w} \det\left[(1-q^{-1})\zeta\frac{f(w_i)}{w_i q^{-1}-w_j}\right]_{i,j=1}^{k} \prod_{i=1}^{k} \frac{dw_i}{2\pi \iota} = \det(I+\tilde{K})_{L^2(\tilde{C}_{w})}.
\end{equation*}
Given the additional convergence, the numerical equality is likewise clear.
\end{proof}

\subsection{q-Whittaker process Fredholm determinant formulas}\label{Macdetformsec}
We may apply the general theory above to the special case given by Proposition~\ref{prop8tzero} to compute the $q$-Laplace transform of $\lambda_N$ under the measure $\MM_{t=0}(a_1,\ldots,a_N;\rho)$.

\subsubsection{Small contour formulas}

\begin{corollary}\label{abgCor}
There exists a positive constant $C\geq 1$ such that for all $|\zeta|<C^{-1}$ the following numerical equality holds
\begin{equation*}
\left\langle \frac{1}{\left(\zeta q^{\lambda_N};q\right)_{\infty}}\right\rangle_{\MM_{t=0}(a_1,\ldots,a_N;\rho)} = \det(I+K)_{L^2(\Zgzero\times C_{a,\rho})}
\end{equation*}
%where $\det(I+K)$ is the Fredholm determinant of
%\begin{equation*}
%K:L^2(\Zgzero\times C_{a,\rho}) \to L^2(\Zgzero\times C_{a,\rho})
%\end{equation*}
where the operator $K$ is defined in terms of its integral kernel
\begin{equation}\label{abgKernel}
K(n_1,w_1;n_2,w_2) = \frac{\zeta^{n_1} f(w_1)f(qw_1)\cdots f(q^{n_1-1}w_1)}{q^{n_1} w_1 - w_2}
\end{equation}
with
\begin{equation}\label{abgf}
f(w) = \left(\prod_{m=1}^{N} \frac{a_m}{a_m-w}\right) \left(\prod_{i\geq 1} (1-\alpha_i w)\frac{1+q \beta_iw}{1+\beta_i w}\right) \exp\{(q-1)\gamma w\}
\end{equation}
and $C_{a,\rho}$ is a positively oriented contour containing $a_1,\ldots, a_N$ and no other singularities of $f$.
%The constant $C$ is such that for all $w\in C_{a,\rho}$ and for all $n\geq 0$, $|f(q^n w)|<C$.
\end{corollary}
\begin{proof}
Observe that using the corollary of the $q$-binomial theorem recorded in equation (\ref{qLaplace}), for $|\zeta|<1$ we may expand
\begin{equation*}
\left\langle \frac{1}{\left(\zeta q^{\lambda_N};q\right)_{\infty}}\right\rangle_{\MM_{t=0}(a_1,\ldots,a_N;\rho)} = \sum_{k\geq 0} \frac{(\zeta/(1-q))^k}{k_q!} \left\langle q^{k\lambda_N}\right\rangle_{\MM_{t=0}(a_1,\ldots,a_N;\rho)}.
\end{equation*}
The convergence of the right-hand side above follows from the fact that $q^{k\lambda_N}\leq 1$ and
\begin{equation*}
\frac{(\zeta/(1-q))^k}{k_q!}  = \frac{\zeta^k}{(1-q)\cdots (1-q^k)},
\end{equation*}
which shows geometric decay for large enough $k$.

In Propositions \ref{mukprop} and \ref{gendetprop} we may choose $f$ as in equation (\ref{abgf}) and choose contours so as to enclose only the $a_1,\ldots, a_m$ singularities of $f$. The result is the kernel of equation (\ref{abgKernel}). Proposition~\ref{gendetprop} now applies (since the constant $C$ in the statement of the result can be chosen such that for all $w\in C_{a,\rho}$ and for all $n\geq 0$, $|f(q^n w)|<C$, the equality in that proposition is numerical). Finally Proposition~\ref{prop8tzero} shows that $\left\langle q^{k\lambda_N}\right\rangle_{\MM_{t=0}(a_1,\ldots,a_N;\rho)}=\mu_k$ as defined by the $f$ above, and thus the corollary follows.
\end{proof}

\begin{theorem}\label{PlancherelfredThm}
Fix $\rho$ a Plancherel (see Definition~\ref{specializationtype}) Macdonald nonnegative specialization. Fix $0<\delta<1$ and $a_1,\ldots, a_N$ such that $|a_i -1|\leq d$ for some constant $d <\frac{1-q^{\delta}}{1+q^{\delta}}$. Then for all $\zeta\in \C\setminus \Rplus$
\begin{equation}\label{thmlaplaceeqn}
\left\langle \frac{1}{\left(\zeta q^{\lambda_N};q\right)_{\infty}}\right\rangle_{\MM_{t=0}(a_1,\ldots,a_N;\rho)} = \det(I+K_{\zeta})_{L^2(C_{a})}
\end{equation}
%where $\det(I+K_{\zeta})$ is the Fredholm determinant of
%\begin{equation*}
%K_{\zeta}:L^2(C_{a})\to L^2(C_{a})
%\end{equation*}
where $C_a$ is a positively oriented circle $|w-1|=d$ and the operator $K_{\zeta}$ is defined in terms of its integral kernel
\begin{equation*}
K_{\zeta}(w,w') = \frac{1}{2\pi \iota}\int_{-\iota \infty + \delta}^{\iota \infty +\delta} \Gamma(-s)\Gamma(1+s)(-\zeta)^s g_{w,w'}(q^s)ds
\end{equation*}
where
\begin{equation}\label{gwwprimeeqn}
g_{w,w'}(q^s) = \frac{1}{q^s w - w'} \prod_{m=1}^{N} \frac{(q^{s}w/a_m;q)_{\infty}}{(w/a_m;q)_{\infty}} \exp\big(\gamma w(q^{s}-1)\big).
\end{equation}
The operator $K_{\zeta}$ is trace-class for all $\zeta\in \C \setminus\Rplus$.
\end{theorem}

The above theorem should be contrasted with Proposition~\ref{qlaplaceintegralform} which gives an $N$-fold summation formula for the left-hand side of (\ref{thmlaplaceeqn}). Note also the following (cf. Proposition~\ref{qlaplaceinverse}).
\begin{corollary}\label{lambdadistcor}
We also have that
\begin{equation*}
\PP_{\MM_{t=0}(a_1,\ldots,a_N;\rho)}(\lambda_N = n) = \frac{-q^n}{2\pi \iota} \int_{C_{n,q}} (q^n \zeta; q)_{\infty} \det(I+K_{\zeta})_{L^2(C_{a})} d\zeta
\end{equation*}
where $C_{n,q}$ is any simple positively oriented contour which encloses the poles $\zeta=q^{-M}$ for $0\leq M\leq n$ and which only intersects $\Rplus$ in finitely many points.
\end{corollary}

\begin{proof}[Proof of Theorem~\ref{PlancherelfredThm}]
The starting point for this proof is Corollary \ref{abgCor}. There are, however, two issues we must deal with. First, the operator in the corollary acts on a different $L^2$ space; second, the equality is only proved for $|\zeta|<C^{-1}$. We split the proof into three steps. Step 1: We present a general lemma which provides an integral representation for an infinite sum. Step 2: Assuming $\zeta\in \{\zeta:|\zeta|<C^{-1}, \zeta\notin\Rplus\}$ we derive equation (\ref{thmlaplaceeqn}). Step 3: A direct inspection of the left-hand side of that equation shows that for all $\zeta\neq q^{-M}$ for $M\geq 0$ the expression is well-defined and analytic. The right-hand side expression can be analytically extended to all $\zeta\notin \Rplus$ and thus by uniqueness of the analytic continuation, we have a valid formula on all of $\C\setminus\Rplus$.

{\bf Step 1:}
The purpose of the next lemma is to change that $L^2$ space we are considering and to replace the summation in Corollary \ref{abgCor} by a contour integral.
\begin{lemma}\label{gammasumlemma}
For all functions $g$ which satisfy the conditions below, we have the identity that for $\zeta\in \{\zeta:|\zeta|<1, \zeta\notin\Rplus\}$:
\begin{equation}\label{gammares}
\sum_{n=1}^{\infty} g(q^n) (\zeta)^n  = \frac{1}{2\pi \iota} \int_{C_{1,2,\ldots}} \Gamma(-s)\Gamma(1+s)(-\zeta)^s g(q^s) ds,
\end{equation}
where the infinite contour $C_{1,2,\ldots}$ is a negatively oriented contour which encloses $1,2,\ldots$ and no singularities of $g(q^s)$ (e.g. $C_{1,2,\ldots}=\iota \R+\delta$ oriented from $-\iota \infty+\delta$ to $\iota \infty +\delta$), and $z^s$ is defined with respect to a branch cut along $z\in \R^{-}$. For the above equality to be valid the left-hand-side must converge, and the right-hand-side integral must be able to be approximated by integrals over a sequence of contours $C_{k}$ which enclose the singularities at $1,2,\ldots, k$ and which partly coincide with $C_{1,2,\ldots}$ in such a way that the integral along the symmetric difference of the contours $C_{1,2,\ldots}$ and $C_{k}$ goes to zero as $k$ goes to infinity.
\end{lemma}
\begin{proof}
The identity follows from $\Res{s=k}\Gamma(-s)\Gamma(1+s) = (-1)^{k+1}$.
\end{proof}
\begin{remark}\label{gammasumlemmaremark}
Let us briefly illustrate the application of Lemma~\ref{gammasumlemma} which is similar to Mellin-Barnes type integral representations for hypergeometric functions. Consider
\begin{equation*}
\frac{1}{(\zeta;q)_{\infty}} = \sum_{k=0}^{\infty} \frac{\zeta^k}{(q;q)_k} = \frac{1}{(q;q)_{\infty}} \sum_{k=0}^{\infty} \zeta^k (q^k;q)_{\infty}.
\end{equation*}
By including the pole at 0 in Lemma~\ref{gammasumlemma} one readily checks that for $\zeta\in \{\zeta:|\zeta|<1, \zeta\notin\Rplus\}$
\begin{equation*}
\sum_{k=0}^{\infty} \zeta^k (q^k;q)_{\infty} = \int_{C_{0,1,2,\ldots}} \Gamma(-s)\Gamma(1+s) (-\zeta)^s(q^s;q)_{\infty} ds,
\end{equation*}
and thus for $\zeta\in \{\zeta:|\zeta|<1, \zeta\notin\Rplus\}$
\begin{equation*}
\frac{(q;q)_{\infty}}{(\zeta;q)_{\infty}} = \int_{C_{0,1,2,\ldots}} \Gamma(-s)\Gamma(1+s) (-\zeta)^s(q^s;q)_{\infty} ds.
\end{equation*}
The left-hand side is analytic for all $\zeta\neq q^{-M}$, $M\geq 0$, while the right-hand side is analytic for $\zeta \notin \Rplus$. By analytic continuation, the above equality holds for all $\zeta\notin \Rplus$. Analytic continuation of the right-hand side to the whole domain of analyticity would require additional means.
\end{remark}

{\bf Step 2:} For this step let us assume that $\zeta\in \{\zeta:|\zeta|<C^{-1}, \zeta\notin\Rplus\}$. We may rewrite equation (\ref{abgKernel}) as
\begin{equation*}
K(n_1,w_1;n_2,w_2) = \zeta^{n_1} g_{w_1,w_2}(q^{n_1})
\end{equation*}
where $g$ is given in equation (\ref{gwwprimeeqn}).

Writing out the $M^{th}$ term in the Fredholm expansion we have ($C_{a,\rho}$ of Corollary \ref{abgCor} are chosen to be $C_a$ of the statement of Theorem~\ref{PlancherelfredThm})
\begin{equation*}
\frac{1}{M!} \sum_{\sigma\in S_M} sgn(\sigma)\prod_{j=1}^{M} \oint_{C_{a}} dw_j \sum_{n_j=1}^{\infty} \zeta^{n_j} g_{w_j,w_{\sigma(j)}}(q^{n_j}).
\end{equation*}

In order to apply Lemma~\ref{gammasumlemma} to the inner summations consider contours $\{C_{k}\}_{k\geq 1}$ which are negatively oriented semi-circles that contain the right half of a circle centered at $\delta$ and of radius $k$. The infinite contour $C_{1,2,\ldots} = \delta + \iota\R$. By the condition on the contour $C_{a}$ we are assured that these finite semi-circles $C_k$ do not contain any poles beyond those of the Gamma function $\Gamma(-s)$. The only other possible poles as $s$ varies is when $q^s w_j -w_{\sigma(j)} =0$. However, the conditions on $C_{a}$ and on the contours on which $s$ varies show that these poles are not present in the semi-circles $C_{k}$.

In order to apply the above lemma we must also estimate the integral along the symmetric difference. Identify the part of the symmetric difference given by the semi-circle arc as $C^{arc}_k$ and the part given by $\{\iota y + \delta: |y|>k\}$ as $C^{seg}_k$. Since $\Real(s)\geq \delta$ it is straight-forward to see that $g_{w_1,w_2}(q^s)$ stays uniformly bounded as $s$ varies along these contours. Consider then $(-\zeta)^s$. This is defined by writing $-\zeta = r e^{i\theta}$ for $\theta\in (-\pi,\pi)$ and $r>0$, and setting $(-\zeta)^s =r^s e^{\iota s\theta}$. Writing $s=x+\iota y$ we have $|(-\zeta)^s| = r^{x}e^{-y\theta}$. Note that our assumption on $\zeta$ corresponds to $r<1$ and $\theta\in(-\pi,\pi)$. Concerning the product of Gamma functions, recall Euler's Gamma reflection formula
\begin{equation}\label{gammareflection}
\Gamma(-s)\Gamma(1+s) = \frac{\pi}{\sin (-\pi s)}.
\end{equation}
One readily confirms that for all $s$: $\dist(s,\Z)>c$ for some $c>0$ fixed,
\begin{equation*}
\left| \frac{\pi}{\sin (-\pi s)} \right| \leq \frac{c'}{e^{\pi \Imag(s)}}
\end{equation*}
for a fixed constant $c'>0$ which depends on $c$. Therefore, along the $C^{seg}_k$ contour,
\begin{equation*}
|(-\zeta)^s\Gamma(-s)\Gamma(1+s)|\sim e^{-y\theta}e^{-\pi|y|},
\end{equation*}
and since $\theta\in(-\pi,\pi)$ is fixed, this product decays exponentially in $|y|$ and the integral goes to zero as $k$ goes to infinity. Along the $C^{arc}_k$ contour, the product of Gamma functions still behaves like $c'e^{-\pi |y|}$ for some fixed $c'>0$. Thus along this contour
\begin{equation*}
|(-\zeta)^s\Gamma(-s)\Gamma(1+s)| \sim e^{-y\theta}r^x e^{-\pi|y|}.
\end{equation*}
Since $r<1$ and $-(\pi+\theta)<0$ these terms behave like $e^{-c''(x+|y|)}$ ($c''>0$ fixed) along the semi-circle arc. Clearly, as $k$ goes to infinity, the integrand decays exponentially in $k$ (versus the linear growth of the length of the contour) and the conditions of the lemma are met.

Applying Lemma~\ref{gammasumlemma} we find that the $M^{th}$ term in the Fredholm expansion can be written as
\begin{equation*}
\frac{1}{M!} \sum_{\sigma\in S_M} sgn(\sigma)\prod_{j=1}^{M} \oint_{C_{a}} dw_j \frac{1}{2\pi \iota}\int_{-\iota \infty +\delta}^{\iota \infty +\delta}ds_j \Gamma(-s)\Gamma(1+s)(-\zeta)^{s} g_{w_j,w_{\sigma(j)}}(q^{s}).
\end{equation*}
Therefore the determinant can be written as $\det(I+K_{\zeta})_{L^2(C_{a})}$ as desired.

{\bf Step 3:} In order to analytically extend our formula we must prove two facts. First, that the left-hand side of equation (\ref{thmlaplaceeqn}) is analytic for all $\zeta\notin\Rplus$; and second, that the right-hand side determinant is likewise analytic for all $\zeta\notin\Rplus$ (and in doing so that the operator is trace class on that set of $\zeta$).

Expand the left-hand side of equation (\ref{thmlaplaceeqn}) as
\begin{equation}
\sum_{n=0}^{\infty} \frac{ \PP(\lambda_N = n) } {(1-\zeta q^n)(1-\zeta q^{n+1})\cdots},
\end{equation}
where $\PP=\PP_{\MM_{t=0}(a_1,\ldots,a_N;\rho)}$.

Observe that for any $\zeta\notin \{q^{-M}\}_{M=0,1,\ldots}$, within a neighborhood of $\zeta$ the infinite products are uniformly convergent and bounded away from zero. As a result the series is uniformly convergent in a neighborhood of any such $\zeta$ which implies that its limit is analytic, as desired.

Turning to the Fredholm determinant, let us first check that we are, indeed, justified in writing the infinite series as a Fredholm determinant (i.e., the $K_{\zeta}$ is trace-class). We use the conditions given in Lemma~\ref{traceclasscrit} which requires that we check that $K_{\zeta}(w,w')$ is continuous simultaneously in both $w$ and $w'$, as well as that the derivative of $K_{\zeta}(w,w')$ in the second entry is continuous in $w'$. Both of these facts are readily checked. In fact, since we have a uniformly and exponentially convergent elementary integral, it and all its derivatives in $w$ and $w'$ are continuous. This is because differentiation does not affect the exponential decay.

%For the first fact, consider $K_{\zeta}(w,w')$ and $K_{\zeta}(z,z')$. The difference is given by
%\begin{equation}
%K_{\zeta}(w,w') - K_{\zeta}(z,z') = \frac{1}{2\pi \iota}\int_{-\iota \infty + 1/2}^{\iota \infty +1/2}ds \Gamma(-s)\Gamma(1+s) \left(\frac{-\zeta}{(1-q)^N}\right)^{s} \left(\frac{h_{w}(q^{s})}{q^s w - w'} - \frac{h_{z}(q^{s})}{q^s z - z'}\right) ds.
%\end{equation}
%So long as $\zeta\notin \Rplus$, the term $\Gamma(-s)\Gamma(1+s) \left(\frac{-\zeta}{(1-q)^N}\right)^{s}$ decays exponentially fast in $|Im(s)|$. Observe now that
%\begin{equation}
%\frac{h_{w}(q^{s})}{q^s w - w'} - \frac{h_{z}(q^{s})}{q^s z - z'}
%\end{equation}
%is bounded in $s$ along the contour of integration and is uniformly converging to 0 as $(z,z')\to (w,w')$ for any compact set of $s$. The continuity of $K_{\zeta}$ then follows. The continuity of the derivative of $K_{\zeta}$ in its second entry is similarly easy to establish. We may differentiate $K_{\zeta}(w,w')$ in $w'$ along the tangent to $C_{a,\rho}$ at $w'$. This just results in the denominator term $q^s w- w'$ being squared and does not affect the continuity argument given above.

In order to establish that the Fredholm determinant $\det(I+K_{\zeta})_{L^2(C_{a})}$ is an analytic function of $\zeta$ away from $\Rplus$ we appeal to the Fredholm expansion and the same criteria that limits of uniformly absolutely convergent series of analytic functions are themselves analytic. There exist theorems (such as in \cite{GKrein} Section IV.1, Property 8, or in the appendix of \cite{ACQ}) which show that if a trace-class operator is an analytic function of a parameter $\zeta$, then so if the determinant. However, the analyticity of the operator must be shown in the trace-class norm which tends to be involved. Since we are already dealing directly with Fredholm series expansions, we find it more natural to prove analyticity directly from the series rather than try to appeal to these results.

Returning to the issue at hand, recall that
\begin{equation*}
\det(I+K_{\zeta})_{L^2(C_{a})} = 1 + \sum_{n=1}^{\infty} \frac{1}{n!} \int_{C_{a}}dw_1 \cdots  \int_{C_{a}}dw_n \det(K_{\zeta}(w_i,w_j))_{i,j=1}^n.
\end{equation*}
It is clear from the definition of $K_{\zeta}$ that $\det(K_{\zeta}(w_i,w_j))_{i,j=1}^{n}$ is analytic in $\zeta$ away from $\Rplus$. Thus any partial sum of the above series is analytic in the same domain. What remains is to show that the series is uniformly absolutely convergent on any fixed neighborhood of $\zeta$ not including $\Rplus$.  Towards this end consider the $n^{th}$ term in the Fredholm expansion:
\begin{eqnarray}\label{fredexpbound}
F_n(\zeta) &=& \frac{1}{n!} \int_{C_{a}}dw_1\cdots  \int_{C_{a}}dw_n \int_{-\iota \infty +\delta}^{\iota \infty +\delta} ds_1 \cdots \int_{-\iota \infty + \delta}^{\iota \infty +\delta} ds_n \det\left(\frac{1}{q^{s_i}w_i -w_j}\right)_{i,j=1}^{n}\\
\nonumber&&\times \prod_{j=1}^{n}  \left(\Gamma(-s_j)\Gamma(1+s_j) (-\zeta)^{s_j} \prod_{m=1}^{N}\frac{(q^{s_j}w_j/a_m;q)_{\infty}}{(w_j/a_m;q)_{\infty}} \exp\big(\gamma w_j(q^{s_j}-1)\big)\right).
\end{eqnarray}
We wish to bound the absolute value of this. Observe that uniformly over $s_j$ and $w_j$ the term
\begin{equation}\label{pochbound}
\left| \prod_{m=1}^{N} \frac{(q^{s_j}w_j/a_m;q)_{\infty}}{(w_j/a_m;q)_{\infty}} \exp\big(\gamma w_j(q^{s_j}-1)\big) \right| \leq const_1
\end{equation}
for some fixed constant $const_1>0$. Now consider the determinant. Along the contours for $s$ and $w$ the ratio $\frac{1}{q^{s_i}w_i -w_j}$ is uniformly bounded in absolute value by a constant $const_2$. By Hadamard's bound
\begin{equation}\label{detbound}
\left| \det\left(\frac{1}{q^{s_i}w_i -w_j}\right)_{i,j=1}^{n} \right| \leq (const_2)^n n^{n/2}.
\end{equation}

%  , whose absolute value can be rewritten (as it is a Cauchy determinant) as
%\begin{equation}
%\left| \det\left(\frac{1}{q^{s_i}w_i -w_j}\right)_{i,j=1}^{n} \right|= \frac{ \prod_{i<j}^n |q^{s_i}w_i -q^{s_j}w_j| \prod_{i<j}^n |w_i-w_j|}{\prod_{i,j}^n |q^{s_i}w_i - w_j|}.
%\end{equation}
%By the triangle inequality
%\begin{equation}
%|q^{s_i}w_i -q^{s_j}w_j| \leq |q^{s_i}w_i - w_j| + |w_j-w_i| + |w_i-q^{s_j}w_j|.
%\end{equation}
%Now observe that given that $\Real(s_i) = 1/2$ and the choice of contour $C_{a,\rho}$ for the $w_j$ it follows that
%\begin{equation}
%\frac{|w_i-w_j|}{|q^{s_i}w_i - w_j|} \leq C_2
%\end{equation}
%for some constant $C_2>0$. Combining the above two facts we have
%\begin{equation}
%\left| \det\left(\frac{1}{q^{s_i}w_i -w_j}\right)_{i,j=1}^{n} \right| \leq \frac{ (C_2)^n+ (C_2)^{2n} + (C_2)^n}{\prod_{i=j}^n |(q^{s_i}-1)w_i|}  \leq (C_3)^n
%\end{equation}
%for some other constant $C_3>0$.

Taking the absolute value of (\ref{fredexpbound}) and bringing the absolute value all the way inside the integrals, we find that after plugging in the results of (\ref{pochbound}) and (\ref{detbound})
\begin{eqnarray*}
|F_n(\zeta)| &\leq&  \frac{1}{n!} \int_{C_{a}}dw_1\cdots  \int_{C_{a}}dw_n \int_{-\iota \infty + \delta}^{\iota \infty +\delta} ds_1 \cdots \int_{-\iota \infty +\delta}^{\iota \infty +\delta} ds_n \\ && \times \prod_{j=1}^{n}  \left|\Gamma(-s_j)\Gamma(1+s_j) (-\zeta)^{s_j}\right| (const_1)^n (const_2)^n n^{n/2}.
\end{eqnarray*}
For $\zeta\notin \Rplus$ it is then easy to check (using decay properties) that in a neighborhood of $\zeta$ which does not touch $\Rplus$, the expression
\begin{equation*}
 \int_{-\iota \infty + \delta}^{\iota \infty +\delta} ds |\Gamma(-s)\Gamma(1+s) (-\zeta)^{s}|
\end{equation*}
is uniformly bounded by a fixed constant $const_3>0$. Thus we have
\begin{equation*}
|F_n(\zeta)| \leq  \frac{1}{n!} \int_{C_{a}}dw_1\cdots  \int_{C_{a}}dw_n (const_1)^n (const_2)^n (const_3)^n n^{n/2}\leq \frac{(const_4)^n n^{n/2}}{n!}
\end{equation*}
which is uniformly convergent as a series in $n$, just as was desired to prove analyticity in $\zeta$.
\end{proof}

\begin{proof}[Proof of Corollary \ref{lambdadistcor}]
By virtue of the fact (given in Step 3 of the proof of Theorem~\ref{PlancherelfredThm}) that $\langle \frac{1}{\left(\zeta q^{\lambda_N};q\right)_{\infty}}\rangle$ is analytic away from $\zeta = q^{-M}$, for integers $M\geq 0$, it follows that even though the Fredholm determinant $\det(I+K_{\zeta})_{L^2(C_{a})}$ is not defined for $\zeta\in\Rplus$ we may still integrate along a contour which intersects $\Rplus$ only a finite number of times. We may then apply the inversion formula of Proposition~\ref{qlaplaceinverse} to conclude our corollary.
\end{proof}

\subsubsection{Large contour formulas}\label{largecontourformulas}
\begin{definition}
A closed simple contour $C$ is {\it star-shaped} with respect to a point $p$ if it strictly contains $p$ and every ray from $p$ crosses $C$ exactly once.
\end{definition}

\begin{theorem}\label{largeconThm}
Fix $\rho$ a Macdonald nonnegative specialization, defined as in equation (\ref{tag1}) in terms of nonnegative numbers $\{\alpha_i\}_{i\geq 1}$, $\{\beta_i\}_{i\geq 1}$ and $\gamma$. Also fix $a_1,\ldots, a_N$ positive numbers such that $\alpha_i a_j<1$ for all $i,j$. Then for all $\zeta\in \C$
\begin{equation}\label{thmlaplaceeqnLARGE}
\left\langle (\zeta;q)_{\lambda_N} \right\rangle_{\MM_{t=0}(a_1,\ldots,a_N;\rho)} =
(\zeta;q)_{\infty} \left\langle \frac{1}{\left(\zeta q^{\lambda_N};q\right)_{\infty}}\right\rangle_{\MM_{t=0}(a_1,\ldots,a_N;\rho)} = \det(I+\zeta \tilde{K})_{L^2(\tilde{C}_{a,\rho})}
\end{equation}
where $\det(I+\zeta\tilde{K})$ is an entire function of $\zeta$ and
% is the Fredholm determinant of
%\begin{equation*}
%\tilde{K}:L^2(\tilde{C}_{a,\rho}) \to L^2(\tilde{C}_{a,\rho})
%\end{equation*}
the operator $\tilde{K}$ is defined in terms of its integral kernel
\begin{equation*}
\tilde{K}(w_1,w_2) =\frac{f(w_1)}{q w_2 - w_1}
\end{equation*}
with
\begin{equation*}
f(w) = \left(\prod_{m=1}^{N} \frac{a_m}{a_m-w}\right) \left(\prod_{i\geq 1} (1-\alpha_i w)\frac{1+q \beta_iw}{1+\beta_i w}\right) \exp\{(q-1)\gamma w\}
\end{equation*}
and $\tilde{C}_{a,\rho}$ a positively oriented star-shaped (with respect to 0) contour containing $a_1,\ldots, a_N$ and no other singularities of $f$.
\end{theorem}

\begin{proof}
Observe first that both sides of equation (\ref{thmlaplaceeqnLARGE}) are analytic functions of $\zeta$. The right-hand side is clear since $\det(I+\zeta K)_{L^2(\tilde{C}_{a,\rho})}$ is always an entire function of $\zeta$ for $K$ trace-class. The left-hand side is analytic as can be seen from considering its series expansion:
\begin{equation*}
\sum_{n=0}^{\infty} \PP(\lambda_N = n)(1-\zeta)(1-\zeta q)\cdots (1-\zeta q^{n-1}),
\end{equation*}
where $\PP=\PP_{\MM_{t=0}(a_1,\ldots,a_N;\rho)}$.
The products above are clearly converging uniformly to non-trivial limits in any fixed neighborhood of $\zeta$ and hence the series is likewise uniformly convergent in any neighborhood defining an analytic function. One similarly shows the (real) analyticity of both sides of (\ref{thmlaplaceeqnLARGE}) in $a_1,\ldots, a_N$.

It remains to prove equation (\ref{thmlaplaceeqnLARGE}) for $|\zeta|<1$ and $a_i$'s in a small enough neighborhood of $1$. Define $\tilde\mu_k$ by the left-hand side of (\ref{muktildedef}) with $f$ as defined in the statement of the theorem and $\tilde{C}_w$ equal to $\tilde{C}_{a,\rho}$ likewise defined in the statement of the theorem. By Remark \ref{starstar31} and Proposition~\ref{muandmutildeprop} we can rewrite $\tilde{\mu}_k$ via an expansion into $\mu_j$ as $j$ varies between $0$ and $k$ and where $\mu_j$ is defined as in (\ref{mukdef}) with $f$ as in the statement of the theorem and with nested contours not including 0. By Proposition~\ref{prop8tzero}, we can recognize that $\mu_k = \left\langle q^{k\lambda_N}\right\rangle_{\MM_{t=0}(a_1,\ldots,a_N;\rho)}$.

Summing up these deductions we have found that
\begin{equation*}
\tilde{\mu}_k = (-1)^k q^{\frac{k(k-1)}{2}} \sum_{j=0}^{k} {k\choose j}_{q^{-1}}(-1)^{j} q^{-\frac{j(j-1)}{2}} \left\langle q^{k\lambda_N}\right\rangle_{\MM_{t=0}(a_1,\ldots,a_N;\rho)}.
\end{equation*}
By (\ref{finqBinExp}) we get
\begin{equation*}
\tilde{\mu}_k = (-1)^k q^{\frac{k(k-1)}{2}}   \left\langle (1-q^{\lambda_N})\cdots (1-q^{\lambda_N-k})\right\rangle_{\MM_{t=0}(a_1,\ldots,a_N;\rho)} =  \left\langle q^{k\lambda_N}(q^{-\lambda_N};q)_{k}\right\rangle_{\MM_{t=0}(a_1,\ldots,a_N;\rho)}.
\end{equation*}
From the above it is clear that $|\tilde{\mu}_k|\leq 1$. Observe that due to the $q$-Binomial theorem and this a priori bound, for $|\zeta|<1$
\begin{equation*}
\sum_{k\geq 0} \tilde{\mu}_k \frac{(\zeta/(1-q))^k}{k_q!} = \left\langle \sum_{k\geq 0} \frac{ (q^{-\lambda_N};q)_k}{(q;q)_{k}} \left(\zeta q^{\lambda_N}\right)^k \right\rangle_{\MM_{t=0}(a_1,\ldots,a_N;\rho)}= \left\langle\frac{(\zeta;q)_{\infty}}{(\zeta q^{\lambda_N};q)_{\infty}} \right\rangle_{\MM_{t=0}(a_1,\ldots,a_N;\rho)}.
\end{equation*}
On the other hand, replacing $\zeta$ by $\zeta/(1-q)$ in Proposition~\ref{tildemuDetProp}, we find that the left-hand side above is equal to the Fredholm determinant in equation (\ref{thmlaplaceeqnLARGE}). Thus for $|\zeta|<1$ and $a_i$'s in a small enough neighborhood of $1$ we have proved the desired result. The full result follows by the earlier mentioned analyticity.
\end{proof}

Note also the following (cf. Proposition~\ref{qlaplaceinverse}),
\begin{corollary} Under the assumptions of Theorem~\ref{largeconThm} one has
\begin{equation*}
\PP_{\MM_{t=0}(a_1,\ldots,a_N;\rho)}(\lambda_N = n) = \frac{-q^n}{2\pi \iota} \int \frac{\det(I+\zeta \tilde{K})_{L^2(\tilde{C}_{a,\rho})}}{(\zeta; q)_{n+1}}d\zeta
\end{equation*}
where the integration is along any positively oriented contour which encloses the poles $\zeta=q^{-M}$ for $0\leq M\leq n$.
\end{corollary}
%\begin{proof}
%From Theorem~\ref{largeconThm}
%\begin{equation*}
%\left\langle \frac{1}{\left(\zeta q^{\lambda_N};q\right)_{\infty}}\right\rangle_{\MM_{t=0}(a_1,\ldots,a_N;\rho)} = \frac{\det(I+\zeta \tilde{K})}{(\zeta;q)_{\infty}}.
%\end{equation*}
%This transform of $\PP_{\MM_{t=0}(a_1,\ldots,a_N;\rho)}(\lambda_N = n)$ can be inverted via Proposition~\ref{qlaplaceinverse} to yield the desired expression.
%\end{proof}

\begin{remark}
The above, simple, formula for the probability density is highly reminiscent of the initial Fredholm determinant expression of Tracy and Widom \cite{TW3} in the context of the ASEP. While simple, it is not obvious how to take asymptotics of this formula. The approach of \cite{TW3} may serve as a guide and one might try to factor the kernel so as to cancel the denominator product and then re-express the resolvent as a kernel written in terms of a separate integration. We, however, make no attempt at this presently, as we already have a different Fredholm determinant expression which is more readily studied asymptotically.
\end{remark}

\section{q-TASEP and the q-Whittaker 2d-growth model}\label{qTASEPandthe}

%\note{include the pure alpha model as well (is it local?), include $q$-TASEP basic properties, equilibrium, flux expression, sum of $\lambda_N^{(N)}+\cdots +\lambda_{1}^{(N)}$ evolves as poisson process, and the more general statement of this for pure alpha. also put in (possibly here) the asymptotics of $\langle q^{\lambda_N}\rangle$ for $N/\gamma$ fixed as $\gamma\to\infty$.}

\subsection{The q-Whittaker 2d-growth model}\label{2dgrowthmodel}

Let us now look at the limits as $t\to 0$ of the transition rates of the continuous time Markov process from Section~\ref{ctnsMarkovChainsec}.
\begin{lemma}\label{lemma22}
For $t=0$ we have
\begin{eqnarray*}
\varphi_{\la/\mu}&=&(q;q)_\infty^{-\ell(\la)}\prod_{i=1}^{\ell(\la)}\frac{(q^{\la_i-\mu_i+1};q)_\infty
(q^{\mu_i-\lambda_{i+1}+1};q)_\infty}{(q^{\mu_i-\mu_{i+1}+1};q)_\infty}\,,\\
\psi_{\la/\mu}&=&(q;q)_\infty^{-\ell(\mu)}\prod_{i=1}^{\ell(\mu)}\frac{(q^{\la_i-\mu_i+1};q)_\infty
(q^{\mu_i-\lambda_{i+1}+1};q)_\infty}{(q^{\la_i-\la_{i+1}+1};q)_\infty}\,,\\
\psi'_{\la/\mu}&=&\prod_{i\ge 1: \la_i=\mu_i, \la_{i+1}=\mu_{i+1}+1} (1-q^{\mu_i-\mu_{i+1}}).
\end{eqnarray*}
\end{lemma}
\begin{proof}
Substitution $t=0$ into the formulas of Section~\ref{pieresec}.
\end{proof}

Lemma~\ref{lemma22} enables us to compute the limits of the matrix elements of the generator $\frak q$. The off-diagonal entries are either 0 or
\begin{equation*}
a_m\frac{\psi_{(\la^{(m)}\cup\square_k)/\lambda^{(m-1)}}}{\psi_{\lambda^{(m)}/\lambda^{(m-1)}}}
\,\psi'_{(\la^{(m)}\cup\square_k)/\lambda^{(m)}}\,,
\end{equation*}
for suitable $1\le k\le m$.
\begin{lemma}\label{lemma23}
For any $1\le k\le m$, $\lambda^{(m-1)}\in\Y(m-1)$, $\lambda^{(m)}\in\Y(m)$, $\lambda^{(m-1)}\prec \lambda^{(m)}$, we have
\begin{equation*}
\lim_{t\to 0} \frac{\psi_{(\la^{(m)}\cup\square_k)/\lambda^{(m-1)}}}{\psi_{\lambda^{(m)}/\lambda^{(m-1)}}}
\,\psi'_{(\la^{(m)}\cup\square_k)/\lambda^{(m)}}\,=
\frac{(1-q^{\lambda_{k-1}^{(m-1)}-\lambda_k^{(m)}})(1-q^{\lambda_{k}^{(m)}-\lambda_{k+1}^{(m)}+1})}
{(1-q^{\lambda_{k}^{(m)}-\lambda_k^{(m-1)}+1})}
\end{equation*}
If $k=m=1$, the right-hand side is 1. If $k=1$ and $m>1$ the right-hand side should be read as ${(1-q^{\lambda_{1}^{(m)}-\lambda_{2}^{(m)}+1})}/{(1-q^{\lambda_{1}^{(m)}-\lambda_1^{(m-1)}+1})}$. Finally, if $k=m$ the right-hand side is equal to $(1-q^{\lambda_{m-1}^{(m-1)}-\lambda_m^{(m)}})$.
\end{lemma}
\begin{proof}
We will do the computation for $m>1$ and $1<k<m$. The three cases listed separately are obtained by omitting suitable factors in the expressions below. From the formula of Lemma~\ref{lemma22} we have
\begin{eqnarray*}
\frac{\psi_{(\la\cup\square_k)/\mu}}{\psi_{\lambda/\mu}} &=& \frac{(q^{\la_k-\mu_k+2};q)_\infty}
{(q^{\la_k-\mu_k+1};q)_\infty}\,\frac{(q^{\mu_{k-1}-\la_k};q)_\infty}{(q^{\mu_{k-1}-\la_k+1};q)_\infty}
\frac{(q^{\la_{k-1}-\la_k+1};q)_{\infty} (q^{\la_k-\la_{k+1}+1};q)_\infty}
{(q^{\la_{k-1}-\la_k};q)_{\infty}(q^{\la_k-\la_{k+1}+2};q)_\infty}\\
&=& \frac{(1-q^{\mu_{k-1}-\la_k})(1-q^{\la_k-\la_{k+1}+1})}{(1-q^{\la_k-\mu_k+1})(1-q^{\la_{k-1}-\la_k})}\,.
\end{eqnarray*}
Since $\psi'_{(\la\cup\square_k)/\lambda}=(1-q^{\la_{k-1}-\la_k})$, we arrive at the result.
\end{proof}

We may now make the following definition.

\begin{definition}\label{2dgrowth}
The {\it $q$-Whittaker 2d-growth model}\index{q-Whittaker 2d-growth model} with drift $a=(a_1,\ldots, a_N)\in \R^N$ is a continuous time Markov process on the space $\ss^{(N)}$ of Section~\ref{MarkovChainAscSEC}, whose transition rates depend on $q$ and the set of parameters $a_1\ldots, a_N$ and are as follows: Each of the coordinates $\la_k^{(m)}$ has its own independent exponential clock with rate
\begin{equation*}
a_m\frac{(1-q^{\lambda_{k-1}^{(m-1)}-\lambda_k^{(m)}})(1-q^{\lambda_{k}^{(m)}-\lambda_{k+1}^{(m)}+1})}{(1-q^{\lambda_{k}^{(m)}-\lambda_k^{(m-1)}+1})}
\end{equation*}
(as in Lemma~\ref{lemma23} the factors that do not make sense must be omitted). When the $\la_A^{(B)}$-clock rings we find the longest string $\la_A^{(B)}=\la_A^{(B+1)}=\dots=\la_A^{(B+C)}$ and move all the coordinates in this string to the right by one. Observe that if $\la_A^{(B)}=\la_{A-1}^{(B-1)}$ then the jump rate automatically vanishes.

We will sometimes call the above dynamics the q-Whittaker 2d-dynamics. To visualize this process it is convenient to set $x_{k}^{(m)} = \lambda_k^{(m)}-k$. These dynamics are illustrated in Figure \ref{qWhitfig} (a) and (b).
\end{definition}

\begin{remark}
It follows immediately from Proposition~\ref{prop16} that if the $q$-Whittaker 2d-growth model is initialized according to $\MM_{t=0}(a_1,\ldots,a_N;\rho)$ with $\rho$ a Plancherel specialization with $\gamma=0$ (i.e., initially $\lambda_k^{(m)}\equiv 0$), then after time $\tau$, the entire GT pattern is distributed according to $\MM_{t=0}(a_1,\ldots,a_N;\rho)$ with $\rho$ a Plancherel specialization with $\gamma=\tau$.
\end{remark}

\begin{remark}
Even though the above 2d-growth model is initially defined on the state space $\ss^{(N)}$ for a fixed number of layers $N$, it is possible to extend these dynamics to the space of infinite Gelfand-Tsetlin patterns $\ss^{(\infty)}$ -- one can find constructions of such Markov dynamics on infinite schemes in the Schur case in \cite{BorodinOlshanski,BorGor}.
\end{remark}

\begin{remark}
Proposition~\ref{sumsprop} implies that for $k$ fixed, the sum $\sum_{i=1}^{k} \lambda^{(k)}_{i}$ evolve under the $q$-Whittaker 2d-dynamics as a Poisson processes, and at a fixed time moment, these quantities for different $k$ are distributed jointly as independent Poisson random variables.
\end{remark}

\subsection{q-TASEP formulas}

The projection of the $q$-Whittaker 2d-growth model dynamics to the set of smallest coordinates $\{\la_k^{(k)}\}_{k\ge1}$ is also Markovian. If we set $x_k=\la_k^{(k)}-k$ for $k=1,2,\ldots$, then the state space consists of ordered sequences of integers $x_1>x_2>x_3>\cdots$. The time evolution is given by a variant on the classical totally asymmetric simple exclusion process (TASEP) which we call $q$-TASEP and define as follows (see Figure \ref{qWhitfig} (c)):

\begin{definition}
The {\it $q$-TASEP} \index{q-TASEP} is a continuous time interacting particle system on $\Z$ where each particle $x_i$ jumps to the right by 1 independently of the others according to an exponential clock with rate $a_i(1-q^{x_{i-1}-x_i-1})$, where $a_i$'s are positive and $0<q<1$. A particular initial condition of interest is that of the {\it step} \index{q-TASEP!step initial condition} initial condition where $x_n(0)=-n$, $n=1,2,\dots$.
\end{definition}
As long as there is a finite number of particles to the right of the origin initially, this process is well-defined since only particles to the right affect a particle's evolution.  In Section~\ref{qtaseprop} below we briefly discuss two approaches which may be used to construct $q$-TASEP for any initial condition.

The exponential moment formulas of Section~\ref{expmomform26} and the Fredholm determinant formula of Section~\ref{Macdetformsec} translate immediately into formulas for $q$-TASEP since $x_n+n = \lambda^{(n)}_{n}$. In particular, Theorem~\ref{PlancherelfredThm} gives the distribution of $x_n$ in terms of a Fredholm determinant, from which one could extract asymptotics.

\subsection{Properties of the q-TASEP}\label{qtaseprop}

We now briefly study the properties of the $q$-TASEP from the perspective of interacting particle systems. We assume here that the jump rate parameters $a_m\equiv 1$. We also indicate a few results one might hope to extend from the usual TASEP setting. The totally asymmetric simple exclusion process (TASEP) corresponds to the $q=0$ limit of $q$-TASEP in which particles jump right at rate 1 if the destination is unoccupied (this is the exclusion rule).

When initialized with a finite number of particles to the right of the origin it is easy to construct the $q$-TASEP, as we did in the previous section. This approach fails when there are an infinite number of particles because we can not simply construct the process from the right-most particle on. There are two approaches one employs to prove the existence of such infinite volume dynamics: Markov generators, and graphical constructions. Both of these approaches work for the $q$-TASEP though we will mainly present the graphical construction.

\begin{proposition}
For any (possibly random) initial condition $x^0=\{x_{i}(0)\}$ of particles, there exists a $q$-TASEP process $x(t)$ for $t\geq 0$ with $x(0)=x^0$ almost surely.
\end{proposition}
\begin{proof}
Assume there is an infinite number of particles to the right of the origin (otherwise the dynamics have already been constructed). The {\it graphical construction} (and idea going back to Harris in 1978 \cite{Harris}) \index{q-TASEP!graphical construction} enables us to construct the dynamics simultaneous for all initial conditions (and hence provides the {\it basic coupling}\index{q-TASEP!basic coupling} of $q$-TASEP). We only consider a finite time interval $[0,T]$ but because of the Markov property, we can extend our construction to all $t\geq 0$. The key idea is to prove that the dynamics split into finite (random) pieces on which the construction is trivially that of a continuous Markov chain on a finite state space.

Above every site of $\Z$ draw half-infinite time lines and throw independent Poisson point processes of intensity 1 along these lines. To every point associate a random variable uniformly distributed on $[0,1]$. Particles move up along these time lines until they encounter a Poisson point at which time they compare the point's random variable $r$ with the quantity $r'=1-q^{d-1}$ where $d$ is the distance to the next particle to the right. If $r<r'$ then the particle jumps one space to the right, otherwise it makes no action.

This process exists for the following reason: For any fixed time $T>0$ there almost surely exists a site $n$ (which depends on the initial data and the Poissonian environment) such that there is initially a particle at site $n$ and such that in the time interval $[0,T]$ there are no Poisson points above $n$ (i.e., the particle at $n$ never moves). This ensures that the process to the left of $n$ evolves independently of the process to the right and hence ensures the existence of the process to the left. There will be infinitely many such special separation points to the right of the origin and hence the process can be constructed in a consistent way in every compact set and hence the infinite volume limit dynamics exists.
\end{proof}

In particular it is worth noting that $q$-TASEP falls into a category of exclusion processes which are known as speed-changed\index{q-TASEP!speed-changed process} simple exclusion processes ($q$-TASEP, however, does not satisfy the gradient condition and has a speed function which is not a local function of the occupation variables -- see for instance \cite{DPSW} for more on speed-changed exclusion).

\begin{remark}
Without getting into the technicalities, the generator construction deals not with particles but rather with occupation variables $\eta\in \{0,1\}^{\Z}$ where $\eta(z)=1$ corresponds to having a particle at $z$ and $\eta(z)=0$ corresponds to having a hole there. The generator for $q$-TASEP is denoted by $L^{q-TASEP}$ and is defined by its action on local (i.e., depending only on a finite number of occupation variables) functions $f:\{0,1\}^{\Z}\to \R$ by
\begin{equation*}
L^{q-TASEP}f(\eta) = \sum_{x\in \Z} (1-q^{g_x(\eta)+1}) \eta(x) (f(\eta^{x,x+1})-f(\eta))
\end{equation*}
where $g_x(\eta)$ is the number of 0's in $\eta$ to the right of $x$ until the first 1, and where $\eta^{x,x+1}$ is the configuration in which the occupation variables at $x$ and $x+1$ have been switched. One readily checks the conditions for Theorem 1.3.9 of \cite{LiggettIPSBOOK} are satisfied, hence the above generator gives rise to a Feller Markov process as desired (condition 1.3.3 is clearly satisfied, whereas 1.3.8 is satisfied since $\sum_{u} \gamma(x,u)$ is bounded by a geometric series with parameter $q$ for $u>x$).
\end{remark}

Instead of keeping track of particle locations or occupation variables, one could alternatively consider the gaps between consecutive particles. Let $y_i=x_i-x_{i+1} -1$ represent the number of holes between particle $i$ and particle $i+1$ (which recall is to its left). The evolution of the collection of $\{y_i\}$ also has a nice description in terms of a zero range process.

\begin{definition}
The totally asymmetric (nearest neighbor) zero range process (TAZRP) on $\Z$ with rate function $g:\Zgzero\to [0,\infty)$ is a Markov process with state space $\{0,1,\ldots\}^{\Z}$. For a state $y\in \{0,1,\ldots\}^{\Z}$ one says there are $y_i$ particles at site $i$. The dynamics of TAZRP are given as follows: in each site $i$, $y_i$ decreases by 1 and $y_{i+1}$ increase by 1 (simultaneously) in continuous time at rate given by $g(y_i)$. We fix throughout that $g(0)=0$.
\end{definition}

We record a few pertinent facts about TAZRP proved in the literature on the subject (see the book of Kipnis and Landim \cite{KipLan} for additional discussion).

\begin{proposition}
\mbox{}
\begin{enumerate}
\item Assume $g(k)>0$ for all $k\geq 1$ and $\sup_{k\geq 0} | g(k+1)-g(k)|<\infty$ then TAZRP is a well-defined Markov process.
\item The set of invariant distributions for TAZRP which are also translation-invariant (with respect to translations of $\Z$) are generated by the following extremal invariant distributions $\mu_\alpha$ on $y(0)$, for $\alpha\in [0,\sup g(k))$: $\mu_\alpha$ is the product measure with each $y_i(0)$ distributed according to
    \begin{equation*}
    \mu_{\alpha}(y(0): y_i(0)=k) = \begin{cases} Z_{\alpha} \frac{ \alpha^k}{g(1)\cdots g(k)}, & \textrm{if } k>0,\\ Z_\alpha & \textrm{if } k=0.\end{cases}
    \end{equation*}
\item If $g(k+1)>g(k)$ for $k\geq 1$ (the {\it attractive} case) then for two initial conditions $y(0)\leq y'(0)$ (i.e. $y_i(0)\leq y'(0)$ for all $i\in \Z$) there exists a monotone coupling of the dynamics such that almost surely under the coupled measure, $y(t)\leq y'(t)$ for all $t$.
\end{enumerate}
\end{proposition}
\begin{proof}
The existence of the dynamics is proved in \cite{LiggettZRP}. The invariant distributions and monotone coupling is shown in \cite{Andjel}. Recall that a measure $\mu$ is invariant for a Markov chain given by a semi-group $S(t)$ if the push-forward of $\mu$ under $S(t)$, written as $\mu S(t)$ is equal to $\mu$ for all $t$.
\end{proof}

From the above theorem we conclude the following immediate corollary.
\begin{corollary}
The TAZRP associated to $q$-TASEP started with an infinite number of particles to the left and right of the origin is given by a rate function $g(k)=1-q^k$ is a well-defined Markov process. Its translation invariant extremal invariant distributions are given by product measure $\mu_\alpha$ for $\alpha\in [0,1)$ with each $y_i(0)$ distributed according to the marginal distribution
\begin{equation*}
\mu_\alpha ( y_i(0)) =  (\alpha;q)_{\infty} \frac{\alpha^k}{(q;q)_{k}}.
\end{equation*}
We call this distribution the {\it $q$-Geometric distribution of parameter $\alpha$}\index{q-Geometric distribution}. In the $q=0$ (TASEP) limit, the $q$-Geometric distributions converge to standard geometric distributions with probability of $k$ given by $(1-\alpha)\alpha^k$.
\end{corollary}

It is an easy application of the $q$-Binomial theorem (see Section~\ref{qSec}) that for a random variable $r$ given by a $q$-Geometric distribution of parameter $\alpha$
\begin{equation*}
\EE[q^{r}] = \sum_{k\geq 0} (\alpha;q)_{\infty} \frac{ (q\alpha)^k}{(q;q)_k} = (1-\alpha).
\end{equation*}

In the case of $q$-TASEP with only a finite number of particles to the left of or right of the origin, the TAZRP is no longer on $\Z$ but rather an interval (possibly half-infinite) of $\Z$. At the boundaries one must put a source and a sink, but let us not go into that in detail.

\begin{remark}
In the case $a_i\equiv 1$, the TAZRP associated with $q$-TASEP and its quantum version under the name ``$q$-bosons'' were introduced in \cite{SasWad} where their integrability was also noted. A stationary version of $q$-TAZRP was studied in \cite{BKS} and a cube root fluctuation result was shown as a consequence of a more general theory of cube root fluctuations for TAZRPs.
\end{remark}

Second class particles play an important role in the study of ASEP and can be studied here as well. For ASEP started with invariant distribution initial data, second class particles macroscopically follow characteristic lines for the conservative Hamilton-Jacobi PDE which governs the limiting shape of the height function associated to this particle system. Using couplings, this approach was employed to prove general hydrodynamic theories \cite{SeppHydro}. Finer scale properties of second class particles can be related to two-point functions for ASEP as well as to the current of particles crossing a characteristic. These relations have been useful in establishing tight bounds on scaling exponents for the height function evolution of the ASEP \cite{BalSepp}. It would be certainly very interesting to see if these approaches still apply for $q$-TASEP.

One of the key pieces in determining the associated hydrodynamic PDE for a particle system is to develop an expression for the {\it flux} in equilibrium. We now provide a heuristic computation for what the flux through the bond between 0 and 1 should be per unit time (i.e., $\EE[N_t]/t$ where $N_t$ is defined as the number of particles to cross from site 0 to site 1 in time $t$). In order for a particle to move from site 0 to 1, there must be a particle at site 0 at time $t$ and hole at 1 at time $t$. First consider the probability of having site 1 occupied. If we can determine the overall density of particles for a given parameter $\alpha$, then the probability of a particle at 0 should equal that density. This density is given by $(1+\EE[r])^{-1}$ where $r$ is given by a $q$-Geometric distribution of parameter $\alpha$. In order to calculate $\EE[r]$ observe that
\begin{equation*}
\EE[r] = (\alpha;q)_{\infty} \alpha \frac{\partial}{\partial \alpha} \left(\sum_{k=0}^{\infty} \frac{\alpha^k}{(q;q)_{k}} \right) = (\alpha;q)_{\infty} \alpha \frac{\partial}{\partial \alpha} \left(\prod_{m=0}^{\infty}\frac{1}{1-\alpha q^m}\right) = \alpha \sum_{m=0}^{\infty} \frac{q^m}{1-\alpha q^m},
\end{equation*}
where we used the $q$-Binomial theorem in the second equality.

Now given a particle at 0, the probability of a jump occurring in a unit time $dt$ is given by
\begin{equation*}
\sum_{k=0}^{\infty} \PP(r=k) (1-q^k) = \alpha \sum_{k=1}^{\infty} (\alpha;q)_{\infty} \frac{\alpha^{k-1}}{(q;q)_{k-1}} = \alpha
\end{equation*}
where the last equality is from shifting the indices and recognizing the sum as $\sum_k \PP(r=k)$. Putting these two pieces together we find that (at least in this heuristic way)
\begin{equation*}
\frac{\EE[N_t]}{t} = \frac{1}{\frac{1}{\alpha} + \sum_{m=0}^{\infty} \frac{ q^m}{1-\alpha q^m}},
\end{equation*}
where $N_t$ is defined as the number of particles to cross from 0 to 1 in time $t$. We do not have a nicer form than this for the flux, though when $q=0$ this reduces to $\alpha(1-\alpha)$ as expected.

\chapter{Whittaker processes} %: $q\to 1$, $t=0$

\section{The Whittaker process limit}\label{whittakerSec}

\subsection{Definition and properties of Whittaker functions}\label{Whittakerfunctiondefs}

\subsubsection{Givental integral and recursive formula}
The {\it class-one $\mathfrak{gl}_{\ell+1}$-Whittaker functions} \index{Whittaker functions} are basic objects of representation theory and integrable systems \cite{Kostant,Et}. One of their properties is that they are eigenfunctions for the quantum $\mathfrak{gl}_{\ell+1}$-Toda chain. As showed by Givental \cite{Giv}, they can also be defined via the following integral representation \index{Whittaker functions!Givental integral representation}
\begin{equation}\label{giventaldef}
\psi_{\ul{\lambda}{\ell+1}}(\ul{x}{\ell+1})=\int_{\R^{\ell(\ell+1)/2}} e^{\mathcal{F}_\lambda(x)} \prod_{k=1}^{\ell}\prod_{i=1}^k dx_{k,i},
\end{equation}
\glossary{$\psi_{\ul{\lambda}{\ell+1}}(\ul{x}{\ell+1})$}
$\ul{x}{\ell+1} = \{x_{\ell+1,1},\ldots, x_{\ell+1,\ell+1}\}\in \R^{\ell+1}$, $\ul{\lambda}{\ell+1}=\{\lambda_{1},\ldots,\lambda_{\ell+1}\}\in \C^{\ell+1}$ and
\begin{equation*}
\mathcal{F}_{\lambda}(x)=\iota\sum_{k=1}^{\ell+1} \lambda_k\left(\sum_{i=1}^k x_{k,i}-\sum_{i=1}^{k-1} x_{k-1,i}\right)-\sum_{k=1}^{\ell}\sum_{i=1}^k \left(e^{x_{k,i}-x_{k+1,i}}+e^{x_{k+1,i+1}-x_{k,i}}\right).
\end{equation*}
\glossary{$\mathcal{F}_{\lambda}(x)$}
Note $e^{-e^{const\cdot M}}$-type decay of the integrand outside a box of size $M$.
\begin{remark}
Despite notational similarity, these are {\it different} than the $\psi_{\lambda/\mu}$ functions which arise in the Pieri formulas. The letter $\psi$ is used for both functions so as to be consistent with the literature. The Whittaker functions will always be written with both a subscript and an argument, while the $\psi$ from the Pieri formula appears with only a subscript.
\end{remark}

The Givental integral representation has a recursive structure. The resulting recursive formula for the Whittaker functions plays a significant role in estimating bounds on the growth of Whittaker functions in the $\ul{x}{\ell+1}$ variables. It says that \cite{GKLO} \index{Whittaker functions!Givental recursion}
\begin{equation}\label{recWhit}
\psi_{\ul{\lambda}{\ell+1}}(\ul{x}{\ell+1}) = \int_{\R^\ell} d\ul{x}{\ell} \BaxKer{\ell+1}{\ell}{\lambda_{\ell+1}}{\ul{x}{\ell+1}}{\ul{x}{\ell}} \psi_{\ul{\lambda}{\ell}}(\ul{x}{\ell}),
\end{equation}
where we have a base case $\BaxKerfour{1}{0}{\lambda_1}{x_{1,1}}=e^{\iota \lambda_1 x_{1,1}}$ and for $\ell>0$
\begin{equation}\label{baxterQ}
\BaxKer{\ell+1}{\ell}{\lambda_{\ell+1}}{\ul{x}{\ell+1}}{\ul{x}{\ell}} = e^{\iota\lambda_{\ell+1}\left(\sum_{i=1}^{\ell+1} x_{\ell+1,i} - \sum_{i=1}^{\ell} x_{\ell,i}\right)} \prod_{i=1}^{\ell}e^{-e^{-(x_{\ell+1,i}-x_{\ell,i})}-e^{-(x_{\ell,i}-x_{\ell+1,i+1})}}.
\end{equation}
\glossary{$\BaxKer{\ell+1}{\ell}{\lambda_{\ell+1}}{\ul{x}{\ell+1}}{\ul{x}{\ell}}$}
The operator $\BaxOp{\ell+1}{\ell}{\lambda_{\ell+1}}$ is known of as a {\it Baxter operator} \index{Whittaker functions!Baxter operator}.
%Baxter operator (2.32) of GLOBax

\subsubsection{Orthogonality and completeness relations}\label{orthcom}
The Whittaker functions have many remarkable properties. Let us note the orthogonality and completeness relations which can be found in \cite{STS,Wal2} or \cite{GLOBax} Theorem 2.1: For any $\ell\ge 0$ and $\ul{\lambda}{\ell+1},\ul{\tilde\lambda}{\ell+1},\ul{x}{\ell+1},\ul{y}{\ell+1}\in\R^{\ell+1}$,\index{Whittaker functions!orthogonality and completeness}.
\begin{eqnarray*}
&&\int_{\R^{\ell+1}} \overline{\psi_{\ul{\lambda}{\ell+1}}}(\ul{x}{\ell+1})\,\psi_{\ul{\tilde{\lambda}}{\ell+1}}(\ul{x}{\ell+1})d\ul{x}{\ell+1} = \frac{1}{(\ell+1)!\, m_{\ell+1}(\ul{\lambda}{\ell+1})} \sum_{\sigma\in S_{\ell+1}} \delta(\ul{\lambda}{\ell+1}-\sigma(\ul{\tilde{\lambda}}{\ell+1})),\\
&&\int_{\R^{\ell+1}} \overline{\psi_{\ul{\lambda}{\ell+1}}}(\ul{x}{\ell+1})\, \psi_{\ul{\lambda}{\ell+1}}(\ul{y}{\ell+1})\, m_{\ell+1}(\ul{\lambda}{\ell+1})d\ul{\lambda}{\ell+1} =\delta(\ul{x}{\ell+1}-\ul{y}{\ell+1}),
\end{eqnarray*}
where the Skylanin measure $m_{\ell+1}(\ul{\lambda}{\ell+1})$ \glossary{$m_{\ell+1}(\ul{\lambda}{\ell+1})$} is defined by \index{Whittaker functions!Skylanin measure}
\begin{equation*}
m_{\ell+1}(\ul{\lambda}{\ell+1})=\frac1{(2\pi)^{\ell+1} (\ell+1)!}\prod_{j\ne k} \frac 1{\Gamma(\iota\la_k-\iota\la_j)}\,.
\end{equation*}
These identities should be understood in a weak sense and show that the transform defined by integration against Whittaker functions is an isometry from $L^2(\R^{\ell+1},d\ul{x}{\ell+1})$ to $L^{2,sym}(\R^{\ell+1},m_{\ell+1} d\ul{\lambda}{\ell+1})$ where the {\it sym} implies that functions are invariant in permuting their entries.

\subsubsection{Mellin Barnes integral representation and recursion}\label{WMellinBarnes}

There exists a dual integral representation to that of Givental which was discovered in \cite{KL} and goes by the name of a Mellin Barnes representation\index{Whittaker functions!Mellin Barnes representation}.  Following \cite{GKLO} ($q$-Whittaker functions also have an analogous representation \cite{GLOq3}, as do Macdonald polynomials \cite{AOS}) we have
\begin{equation}\label{MBrep}
\psi_{\ul{\lambda}{\ell+1}}(\ul{x}{\ell+1}) = \int_{S} \prod_{n=1}^{\ell} \frac{\prod_{k=1}^{n}\prod_{m=1}^{n+1} \Gamma(\iota \lambda_{n,k}-\iota \lambda_{n+1,m})}{(2\pi \iota )^n n! \prod_{s\neq p}^{\ell} \Gamma(\iota \lambda_{n,s}-\iota \lambda_{n,p})} e^{-\iota \sum_{n=1}^{\ell+1}\sum_{j=1}^{\ell+1} (\lambda_{n,j}-\lambda_{n-1,j})x_n} \prod_{1\leq j\leq n\leq \ell} d\lambda_{n,j},
\end{equation}
where, just for the purpose of this representation we write $\ul{\lambda}{n} = \{\lambda_{n,1},\ldots, \lambda_{n,n}\}\in \C^{n}$ and where the domain of integration $S$ is defined by the conditions $\max_j\{\Imag \lambda_{k,j}\} <\min_m\{\Imag \lambda_{k+1,m}\}$ for all $k=1,\ldots,\ell$. Recall that we assumed $\lambda_{n,j}=0$ for $j>n$.

The Mellin Barnes integral representation has a recursive structure. The resulting recursive formula for the Whittaker functions plays a significant role in estimating bounds for the growth of Whittaker functions in the $\ul{\lambda}{\ell+1}$ indices:
\begin{equation}\label{MBrecformula}
\psi_{\ul{\lambda}{\ell+1}}(\ul{x}{\ell+1}) = \int_{S_{\ell}} \DBaxKer{\ell+1}{\ell}{x_{\ell+1}}{\ul{\lambda}{\ell+1}}{\ul{\lambda}{\ell}} \psi_{\ul{\lambda}{\ell}}(\ul{x}{\ell})m_{\ell}(\ul{\lambda}{\ell})\prod_{j=1}^{\ell}d\lambda_{\ell,j},
\end{equation}
where we have for $\ell\geq 1$
\begin{equation*}
\DBaxKer{\ell+1}{\ell}{x_{\ell+1}}{\ul{\lambda}{\ell+1}}{\ul{\lambda}{\ell}}  = e^{-\iota \left(\sum_{j=1}^{\ell+1} \lambda_{\ell+1,j}-\sum_{k=1}^{\ell} \lambda_{\ell,k}\right)x_{\ell+1} } \prod_{k=1}^{\ell}\prod_{m=1}^{\ell+1} \Gamma(\iota \lambda_{\ell,k} - \iota \lambda_{\ell+1,m}),
\end{equation*}
and where the domain of integration $S_{\ell}$ is defined by the conditions $\max_j\{\Imag \lambda_{\ell,j}\} <\min_{m}\{\Imag \lambda_{\ell+1,m}\}$.

The operator $\DBaxOp{\ell+1}{\ell}{x_{\ell+1}}$\glossary{$\DBaxOp{\ell+1}{\ell}{x_{\ell+1}}$} is known of as a {\it dual Baxter operator} \index{Whittaker functions!dual Baxter operator}.

%\begin{comment}
%\note{remark on how GLO give a go-between for these two types of Baxter operators}
%\end{comment}
\subsubsection{Controlled decay of Whittaker functions}
\index{Whittaker functions!controlled decay}
\begin{definition}\label{sigmadef}
For $\ul{x}{\ell+1}\in \R^{\ell+1}$ define the set-valued function
\begin{equation*}
\sigma(\ul{x}{\ell+1}) = \{i\in \{1,\ldots, \ell\}: x_{\ell+1,i}-x_{\ell+1,i+1}\leq 0\}.
\end{equation*}
\glossary{$\sigma(\ul{x}{\ell+1})$}
\end{definition}

\begin{proposition}\label{whitProp}
Fix $\ell\geq 0$. Then for each $\sigma\subseteq\{1,\ldots,\ell\}$ there exists a polynomial $\poly_{\ell+1,\sigma}$  of $\ell+1$ variables such that for all $\ul{\lambda}{\ell+1}\in \R^{\ell+1}$ and for all $\ul{x}{\ell+1}$ with $\sigma(\ul{x}{\ell+1})=\sigma$ we have the following estimate:
\begin{equation*}
|\psi_{\ul{\lambda}{\ell+1}}(\ul{x}{\ell+1})| \leq \poly_{\ell+1,\sigma}(\ul{x}{\ell+1}) \prod_{i\in \sigma} \exp\{-e^{-(x_{\ell+1,i}-x_{\ell+1,i+1})/2}\}.
\end{equation*}
%\note{Check the factor of $1/2$ as well!}
\end{proposition}

The idea here is that double exponentials serve as softened versions of indicator functions. If the double exponentials were replaced by indicator functions, and if the $\lambda$ were set to zero, then the integral (\ref{giventaldef}) could easily be evaluated and found equal to a constant times the Vandermonde determinant of the vector $\underline{x}_{\ell+1}$ when the vector entries are ordered, and zero otherwise. This is certainly of the type of bound given above. The purpose of the bound above is to show that the softening does not greatly change such an upper bound.

The proof of this Proposition relies on the following.
\begin{lemma}\label{gumblemma}
For all $m\geq 0$ there exists constants $c=c(m)\in (0,\infty)$ such that for $a\leq 0$
\begin{equation*}
\int_{-\infty}^{0} |x|^m e^{-e^{-(a+x)}} \leq c(m) e^{-e^{-a}}.
\end{equation*}
\end{lemma}

\begin{proof}
Let $x^*$ by the largest $x\leq -1$ such that the following inequality holds: $e^{-x}\geq~m~\log|x-1|-x$. Clearly for $x$ negative enough this will hold, though $x^*$ will depend on $m$. Now split the integral between $-\infty$ and $x^*$ and $x^*$ and 0. Observe that when the inequality holds, so does $e^{-x}\geq e^{a} m \log |x-1| - x$ (since $a<0$ implies $e^{a}<1$ and $\log|x-1|$ is positive on $x<0$) and hence
\begin{equation*}
|x|^m e^{-e^{-(a+x)}} \leq \frac{|x|^m}{|x-1|^m}\left(e^{e^{-a}}\right)^{x} \leq \left(e^{e^{-a}}\right)^{x}.
\end{equation*}
Therefore the first integral (between $-\infty$ and $x^*<-1$) can be majorized by the integral of $\left(e^{e^{-a}}\right)^x$ which is certainly bounded by a constant time $e^{-e^{-a}}$. For the second integral, we can upper-bound $e^{-e^{-(a+x)}}$ by $e^{-e^{-a}}$ so that the integral is bounded by that times the integral of  $|x|^m$ from $x^*$ to 0, which is simply a constant depending on $m$, as desired.
%To prove the second part, observe that we can expand the polynomial into monomials. Then we can perform each integration separately and using the first part of the lemma for the integrations for $i\in I$, and otherwise just evaluating the finite interval integration, we get our desired result.
\end{proof}

\begin{proof}[Proof of Proposition~\ref{whitProp}]
The proof of this Proposition is instructive since it underlies the idea of the proof of the analogous upper bound in Theorem~\ref{qwhitconvTHM}.
We prove this by induction on $\ell$. The base case is $\ell=0$. Then $\psi_{\lambda_{1,1}}(x_{1,1})= e^{\iota \lambda_{1,1}x_{1,1}}$ and it is clear that the proposition's claim is satisfied by letting the polynomial be a constant exceeding one.

%The set $\sigma(\underline{x}_{2})$ is either empty of contains the index $1$. If $\sigma=\emptyset$ then $x_{2,1}-x_{2,2}\geq 0$. We can split our integration into three parts:
%\begin{equation}
%|\psi_{\underline{\lambda}_{\ell+1}}(\underline{x}_{\ell+1})| \leq \left(\int_{-\infty}^{x_{2,2}} +\int_{x_{2,2}}^{x_{2,1}} + \int_{x_{2,1}}^{\infty}\right)dx_{1,1} \exp\{-e^{-(x_{2,1}-x_{1,1})}\}\exp\{-e^{-(x_{1,1}-x_{2,2})}\}.
%\end{equation}
%On the first integration we bound $ \exp\{-e^{-(x_{2,1}-x_{1,1})}\}$ by 1 and apply Lemma~\ref{gumblemma} which bounds the whole integral by a constant. Likewise we can bound the third integration. For the middle integration we bound both integrands by 1 and evaluate the integral as $x_{2,1}-x_{2,2}$. This certainly satisfies the statement of the proposition.

%Turning to the case of  $\sigma(\underline{x}_{2}) = \{1\}$, which is to say $x_{2,1}-x_{2,2}\leq 0$, we split the integration now into two parts:
%\begin{equation}
%|\psi_{\underline{\lambda}_{\ell+1}}(\underline{x}_{\ell+1})| \leq \left(\int_{-\infty}^{(x_{2,1}+x_{2,2})/2} +\int_{(x_{2,1}+x_{2,2})/2}^{\infty} \right)dx_{1,1} \exp\{-e^{-(x_{2,1}-x_{1,1})}\}\exp\{-e^{-(x_{1,1}-x_{2,2})}\}.
%\end{equation}
%On the first integration we bound $\exp\{-e^{-(x_{2,1}-x_{1,1})}\}$ by 1 and apply Lemma~\ref{gumblemma} to the remaining term, allowing us to bound the integral by a constant times $\exp\{-e^{-(x_{2,1}-x_{2,2})/2}\}$. The second %integration is likewise bounded. Again, this satisfies the statement of the proposition, and hence proves the base case for induction.

Now assume that we have proved the induction up to $\ell-1$ and we will prove it for $\ell$. The key here is the recursive formula for the Whittaker functions given in equation (\ref{recWhit}).

Consider $\underline{x}_{\ell+1}\in \R^{\ell+1}$ and call $\sigma=\sigma(\underline{x}_{\ell+1})$. We do not, in fact, need to use the double exponential decay of the Whittaker function (given by the inductive hypothesis), but rather just that its absolute value is bounded by a polynomial of its arguments. Actually, by dropping this double exponential term (which is after all bounded by 1) we can expand the bounding polynomial into monomials and then factor the integration into one-dimensional integrals. These integrals are of the form
\begin{equation}\label{factoredintegralform}
\int_{-\infty}^{\infty} dx_{\ell,i} (x_{\ell,i})^m e^{-e^{-(x_{\ell+1,i}-x_{\ell,i})}-e^{-(x_{\ell,i}-x_{\ell+1,i+1})}}.
\end{equation}
We must consider separately the case $x_{\ell+1,i}-x_{\ell+1,i+1} \geq 0$ and $x_{\ell+1,i}-x_{\ell+1,i+1} \leq 0$. The first case, where $x_{\ell+1,i}-x_{\ell+1,i+1} \geq 0$, corresponds to $i\notin \sigma(\ul{x}{\ell+1})$. We would like to show, therefore, that (\ref{factoredintegralform}) is bounded by a polynomial. We may split the integral in (\ref{factoredintegralform}) into three parts:
\begin{equation}\label{threepartsintegral}
\left(\int_{-\infty}^{x_{\ell+1,i+1}} +\int_{x_{\ell+1,i+1}}^{x_{\ell+1,i}}+\int_{x_{\ell+1,i}}^{\infty}\right) dx_{\ell,i} (x_{\ell,i})^m \exp\{-e^{-(x_{\ell+1,i}-x_{\ell,i})}\}\exp\{-e^{-(x_{\ell,i}-x_{\ell+1,i+1})}\}.
\end{equation}
On the first integration we bound $\exp\{-e^{-(x_{\ell+1,i}-x_{\ell,i})}\}$ by 1. We then perform a change of variables $\tilde{x}_{\ell,i}= x_{\ell,i} - x_{\ell+1,i+1}$ and the bound on this integration becomes
\begin{equation*}
\int_{-\infty}^{0}d\tilde{x}_{\ell,i} (\tilde{x}_{\ell,i}+x_{\ell+1,i+1})^m \exp\{-e^{-\tilde{x}_{\ell,i}}\}.
\end{equation*}
By expanding this as a polynomial in $\tilde{x}_{\ell,i}$ we can apply Lemma~\ref{gumblemma} to each term separately to see that the first part of (\ref{threepartsintegral}) is bounded by a polynomial in $\ul{x}{\ell+1}$ as desired. The bound for the third part goes similarly. For the second part of the integral, we bound both double exponential terms by one and then note that the remaining integral is clearly a polynomial. Summing up, when $x_{\ell+1,i}-x_{\ell+1,i+1} \geq 0$, we find a polynomial bound for (\ref{factoredintegralform}), as desired.

We turn to the case $x_{\ell+1,i}-x_{\ell+1,i+1} \leq 0$, which corresponds to $i\in \sigma(\ul{x}{\ell+1})$. We would like to show that in this case, (\ref{factoredintegralform}) is bounded by a polynomial times a double exponential factor. We can split (\ref{factoredintegralform}) into two parts:
\begin{equation}\label{twopartsintegral}
\left(\int_{-\infty}^{(x_{\ell+1,i}+x_{\ell+1,i+1})/2} + \int_{(x_{\ell+1,i}+x_{\ell+1,i+1})/2}^{\infty} \right) dx_{\ell,i} (x_{\ell,i})^m \exp\{-e^{-(x_{\ell+1,i}-x_{\ell,i})}\}\exp\{-e^{-(x_{\ell,i}-x_{\ell+1,i+1})}\}.
\end{equation}
Let us focus on just the first part of the integral, since the second part is analogously bounded. On the first part of the integration we bound $\exp\{-e^{-(x_{\ell+1,i}-x_{\ell,i})}\}$ by one and perform a change of variables $\tilde{x}_{\ell,i} = x_{\ell,i} - (x_{\ell+1,i}+x_{\ell+1,i+1})/2$. The result is a bound on that part of the integral by
\begin{equation}\label{twopartsintegralbound}
\int_{-\infty}^{0}d\tilde{x}_{\ell,i} \left(\tilde{x}_{\ell,i}+\frac{x_{\ell+1,i}+x_{\ell+1,i+1}}{2}\right)^m \exp\{-e^{-(\tilde{x}_{\ell,i}+(x_{\ell+1,i}-x_{\ell+1,i+1})/2)}\}.
\end{equation}
By expanding the polynomial $\left(\tilde{x}_{\ell,i}+\frac{x_{\ell+1,i}+x_{\ell+1,i+1}}{2}\right)^m$ into monomials we may apply Lemma~\ref{gumblemma} with $a= (x_{\ell+1,i}-x_{\ell+1,i+1})/2$. The result is that we find that (\ref{twopartsintegralbound}) is bounded by a polynomial in $\ul{x}{\ell+1}$ times the double exponential factor $\exp\{-e^a\}$ with $a$ as above. This, however, is exactly the desired bound.

By combining the results for each monomial of the form of (\ref{factoredintegralform}) we get the desired bound for level $\ell+1$, thus completing the inductive step.
\end{proof}

%\begin{comment}
%\note{Now put in the other decay result coming from the Mellin Barnes formula!!!!}
%\end{comment}

\begin{remark}
Fix any compact set $D\subset \R^{\ell+1}$. Then for all $c<\pi/2$ there is a constant $C$ such that
\begin{equation*}
|\psi_{\ul{\lambda}{\ell+1}}(\ul{x}{\ell+1})| \leq C \exp\left\{-c \sum_{i<j}^{\ell+1} |\lambda_{\ell+1,i}-\lambda_{\ell+1,j}|\right\}
\end{equation*}
for all $\ul{x}{\ell+1}\in D$.

Let us briefly sketch how this is shown. This estimate relies on the recursive version of the Mellin Barnes representation for Whittaker functions given in Section~\ref{WMellinBarnes}. By induction on $\ell$, and by replacing Gamma functions by their asymptotic exponential behavior, one can bound $|\psi_{\ul{\lambda}{\ell+1}}(\ul{x}{\ell+1})|$ in terms of the integral
\begin{equation*}
\int \exp\left\{ - c \sum_{i<j}^{\ell} |\lambda_{\ell,i}-\lambda_{\ell,j}| - \pi/2 \left( \sum_{i=1}^{\ell}\sum_{j=1}^{\ell+1} |\lambda_{\ell,i}-\lambda_{\ell+1,j}| - 2\sum_{i< j}^{\ell} |\lambda_{\ell,i}-\lambda_{\ell,j}|\right)\right\} \prod_{i=1}^{\ell} d\lambda_{\ell,i}.
\end{equation*}
By studying the critical points of the function being exponentiated, one sees that the integral is bounded as desired.
\end{remark}

\subsection{Whittaker functions as degenerations of q-Whittaker functions}\label{721}
In order to relate $q$-Whittaker functions to classical Whittaker functions we must take certain rescalings and limits.

\begin{definition}\label{scaleDef}
We introduce the following scalings:
\begin{eqnarray*}
q=e^{-\e}, & z_k = e^{\iota \e \nu_k}, & p_{\ell,k} = (\ell+1-2k)m(\e) + \e^{-1} x_{\ell,k},\\
m(\e) = -\left[\e^{-1}\log \e\right], & \A(\e) = -\tfrac{\pi^2}{6} \tfrac{1}{\e} - \tfrac{1}{2} \log \tfrac{\e}{2\pi}, & f_{\alpha}(y,\e) = (q;q)_{[\e^{-1}y] + \alpha m(\e)},\, \alpha\in\Z^{\geq 1}.
\end{eqnarray*}
\glossary{$m(\e)$} \glossary{$\A(\e)$} \glossary{$f_{\alpha}(y,\e)$}
Furthermore, define the rescaling of the $q$-Whittaker function $\qWhitP_{\ul{z}{\ell+1}}(\ul{p}{\ell+1})$ (defined in Section~\ref{qwhittakersec}) as
\begin{equation*}
\psi^{\e}_{\ul{\nu}{\ell+1}}(\ul{x}{\ell+1}) = \e^{\frac{\ell(\ell+1)}{2}} e^{\frac{\ell(\ell+3)}{2}\A(\e)} \qWhitP_{\ul{z}{\ell+1}}(\ul{p}{\ell+1}).
\end{equation*}
\glossary{$\psi^{\e}_{\ul{\nu}{\ell+1}}(\ul{x}{\ell+1})$}
\end{definition}

%Let us recall equation (\ref{macLaurent}) which allows us to extend the definition of the Macdonald symmetric polynomials $P_\la(z_1,\dots,z_{\ell+1})$ (and thus the $q$-Whittaker functions as well) to arbitrary integral values of $\la_1\ge\la_2\ge\dots\ge \la_{\ell+1}$ via
%\begin{equation*}
%P_{\la_1+1,\dots,\la_{\ell+1}+1}(z_1,\dots,z_{\ell+1})=z_1\cdots z_{\ell+1}\,P_{\la_1,\dots,\la_{\ell+1}}(z_1,\dots,z_{\ell+1})
%\end{equation*}
%(this identity holds for any partition $\la$). For negative values of $\la_j$'s, $P_\la$ becomes a Laurent polynomial in $z$'s. \note{Why are we recalling this presently?}

We are now in a position to state our main result concerning the controlled convergence of $q$-Whittaker functions to Whittaker functions.

\begin{theorem}\label{qwhitconvTHM}
\index{Whittaker functions!degeneration of $q$ Whittaker functions}
For all $\ell\geq 1$ and all $\underline{\nu}_{\ell+1}\in \R^{\ell+1}$ we have the following:
\begin{enumerate}
\item For each $\sigma\subseteq\{1,\ldots,\ell\}$ there exists a polynomial $\poly_{\ell+1,\sigma}$  of $\ell+1$ variables (chosen independently of $\nu$ and $\e$) such that for all $\underline{x}_{\ell+1}$ with $\sigma(\ul{x}{\ell+1})=\sigma$ (recall the function $\sigma(\cdot)$ from equation (\ref{sigmadef})) we have the following estimate: for some $c^*>0$
\begin{equation}\label{qboundthmeqn}
|\psi^{\e}_{\ul{\nu}{\ell+1}}(\ul{x}{\ell+1})| \leq \poly_{\ell+1,\sigma(\ul{x}{\ell+1})}(\ul{x}{\ell+1}) \prod_{i\in \sigma(\ul{x}{\ell+1})} \exp\{-c^* e^{-(x_{\ell+1,i}-x_{\ell+1,i+1})/2}\}.
\end{equation}
\item For $\ul{x}{\ell+1}$ varying in a compact domain, $\psi^{\e}_{\ul{\nu}{\ell+1}}(\ul{x}{\ell+1})$ converges (as $\e$ goes to zero) uniformly to $\psi_{\ul{\nu}{\ell+1}}(\ul{x}{\ell+1})$.
\end{enumerate}
\end{theorem}

This theorem is proved in Section~\ref{indproofsec} and follows a similar, albeit more involved, route as the proof of Proposition~\ref{whitProp} for Whittaker functions. Prior to commencing that proof it is necessary to record and prove a few precise estimates about asymptotics of $q$-factorials.

\begin{remark}
In \cite{GLOqlim} a non-rigorous derivation of a point-wise version of part 2 of the above convergence result is given. Through careful considerations of uniformity and tail bounds we prove the above theorem.
\end{remark}
%, which are given now in Section~\ref{qfactSec}, Proposition~\ref{qfactProp}.

\subsubsection{Asymptotics and bounds on a certain q-Pochhammer symbol}\label{qfactSec}
Recall that we have defined $\A(\e) = -\frac{\pi^2}{6} \frac{1}{\e} -\frac{1}{2}\log\frac{\e}{2\pi}$, $m(\e) = -[\e^{-1}\log \e], q=e^{-\e}$ and
$f_{\alpha}(y,\e) = (q;q)_{\e^{-1}y_{\alpha,\e}}$, where $y_{\alpha,\e}=y+\e \alpha m(\e)$. Throughout this section we will always assume that $y$ is such that $\e^{-1}y_{\alpha,\e}$ is a nonnegative integer.

\begin{proposition}\label{qfactProp}
Assume $\alpha\geq 1$, then for all $y$ such that $\e^{-1}y_{\alpha,\e}$ is a nonnegative integer,
\begin{equation}\label{lwbdd}
\log f_{\alpha}(y,\e) -\A(\e) \geq -c + \e^{-1}e^{-y_{\alpha,\e}},
\end{equation}
where $c=c(\e)>0$ is independent of all variables besides $\e$ and can be taken to go to zero with $\e$.

On the other hand for all $b^*<1$ and for all $y$ such that $\e^{-1}y_{\alpha,\e}$ is a nonnegative integer,
\begin{equation}\label{upbdd}
\log f_{\alpha}(y,\e) -\A(\e) -\e^{-1} e^{-(\e+y_{\alpha,\e})} \leq c+ e^{-(\e+y_{\alpha,\e})} + (b^*)^{-1}\e^{-1} \sum_{r=2}^{\infty} \frac{e^{-r(\e+y_{\alpha,\e})}}{r^2},
\end{equation}
where $c=c(y_{\alpha,\e})<C$ for $C<\infty$ fixed and independent of all variables, and where $c(y_{\alpha,\e})$ can be taken as going to zero as $y_{\alpha,\e}$ increases.
\end{proposition}

The following is then an immediate consequence of combining the above asymptotics:

\begin{corollary}\label{qfactCor}
Assume $\alpha\geq 1$, then for any $M>0$ and any $\delta>0$, there exists $\e_0>0$ such that for all $\e<\e_0$ and all $y\geq -M$ such that $\e^{-1}y_{\alpha,\e}$ is a nonnegative integer,
\begin{equation}\label{unifconveqn}
\log f_{\alpha}(y,\e) -\A(\e) - \e^{\alpha-1}e^{-y}\in [-\delta,\delta].
\end{equation}%\note{check whether we need both sides of the interval...seems like we do}
\end{corollary}

\begin{proof}[Proof of Proposition~\ref{qfactProp}]

Let us first introduce the plan of this proof. There are three steps. In the first step we consider $\log f_{\alpha}(y,\e)$ as $y$ goes to infinity and compute its limit. In the second step we take the derivative of this expression in $y$ and notice that this is a much more manageable series (the logarithms disappear) so that we can approximate it to any degree of accuracy we need. Finally, in the third step, we write out the bounds from below and above and integrate them to obtain the claimed estimates.

{\bf Step 1:} First observe that for any fixed $\e$, $\lim_{y\rightarrow \infty} \log f_{\alpha}(y,\e)$ exists and (by dominated convergence) is given by
\begin{equation}\label{ylimitformula}
\lim_{y\rightarrow \infty} \log f_{\alpha}(y,\e)  = -\sum_{r=1}^{\infty} \frac{1}{r} \left(\frac{e^{-r\e}}{1-e^{-r\e}}\right)= -\sum_{n=1}^{\infty} \sum_{r=1}^{\infty} \frac{1}{r} e^{-nr\e} =\log \prod_{n=1}^{\infty} (1-e^{-n\e}).
\end{equation}
We may say even more.
\begin{lemma}\label{Aepslemma}
For all $\alpha\geq 1$ and all $\delta>0$ there exists $\e_0>0$ such that for all $\e<\e_0$
\begin{equation}
\lim_{y\rightarrow \infty} \log f_{\alpha}(y,\e) - \A(\e) \in [-\delta,0]
\end{equation}
\end{lemma}

\begin{proof}
The Dedekind eta function
\begin{equation*}
\eta(\tau) = e^{\frac{\iota\pi\tau}{12}} \prod_{n=1}^{\infty} (1-e^{2\pi \iota n\tau})
\end{equation*}
has the modular property that
\begin{equation*}
\eta(-\tau^{-1}) = \sqrt{-\iota\tau} \eta(\tau).
\end{equation*}
Taking $\tau = \tfrac{\iota\e}{2\pi}$ and recalling the definition of $f_{\alpha}$ we have
\begin{equation*}
\lim_{y\rightarrow \infty} f_{\alpha}(y,\e) = e^{\e/24} \eta\left(\frac{\iota \e}{2\pi}\right) =e^{\e/24} \sqrt{2\pi \e^{-1}} e^{-\frac{\pi^2}{6} \e^{-1}} \prod_{n=1}^{\infty}(1-e^{-\e^{-1}(2\pi)^2n}).
\end{equation*}
Note that the second equality follows from the modular property of the Dedekind eta function.

This implies that
\begin{equation*}
\lim_{y\rightarrow \infty} \log f_{\alpha}(y,\e) = \tfrac{\e}{24} + \A(\e) + \log\prod_{n=1}^{\infty}(1-e^{-\e^{-1}(2\pi)^2n}).
\end{equation*}
All that remains, therefore, is to show that for any $\delta$ we can choose $\e_0$ such that for all $\e<\e_0$
\begin{equation*}
\log\prod_{n=1}^{\infty}(1-e^{-\e^{-1}(2\pi)^2n}) \in [-\delta,0].
\end{equation*}
The upper bound is clear. For the lower bound we use the following inequality
\begin{equation*}
\log(1-e^{-a}) \geq -\frac{1}{a^2} \quad \textrm{ for } a>0.
\end{equation*}
Applying this inequality shows that
\begin{equation*}
\log\prod_{n=1}^{\infty}(1-e^{-\e^{-1}(2\pi)^2n}) \geq - \sum_{n=1}^{\infty} \frac{1}{(\e^{-1}(2\pi)^2n)^2}
\end{equation*}
which, for $\e$ small enough is certainly above $-\delta$.
\end{proof}

{\bf Step 2:} Let us first note the identity which holds for $q<1$
\begin{equation*}
\log \prod_{n=1}^{N}(1-q^n) = \sum_{n=1}^{N} \log (1-q^n) = -\sum_{n=1}^{N}\sum_{r=1}^{\infty} \frac{1}{r} q^{nr} = -\sum_{r=1}^{\infty} \frac{q^r}{r} \left(\frac{1-q^{Nr}}{1-q^r}\right).
\end{equation*}
Now set $q=e^{-\e}$ and $N = \e^{-1} y_{\alpha,\e}$ to obtain
\begin{equation}\label{logseries}
\log f_{\alpha}(y,\e) = -\sum_{r=1}^{\infty} \frac{e^{-r\e}}{r} \left(\frac{1 - e^{-r y_{\alpha,\e}}}{1-e^{-r\e}}\right).
\end{equation}
Observe that since $y_{\alpha,\e}\geq 0$, each term in the summation above is positive. This expression becomes more manageable if we differentiate both sides.

\begin{lemma}\label{difflemma}
For all $y_{\alpha,\e}\geq 0$
\begin{equation}\label{diffseries}
\frac{\partial}{\partial y} \log f_{\alpha}(y,\e) = -\sum_{r=1}^{\infty} \frac{e^{-r(\e+y_{\alpha,\e})}}{1-e^{-r\e}}.
\end{equation}
\end{lemma}
\begin{proof}
It suffices that for $\e>0$ fixed and $y_{\alpha,\e}\geq 0$, we prove that (a) there exists $\eta>0$ such that the right-hand side of (\ref{logseries}) is uniformly convergent (in $r$) in $[y_{\alpha,\e}-\eta,y_{\alpha,\e}+\eta]$, (b) each term in the series has a continuous derivative, and (c) the term-wise differentiated series given on the right-hand side of equation (\ref{diffseries}) is also uniformly convergent (in $r$) in $[y_{\alpha,\e}-\eta,y_{\alpha,\e}+\eta]$.

It is clear that (b) is true. To prove (a) and (c) consider the tail of the series.  Let us focus on (a) for the moment. As observed already, the summands in equation (\ref{logseries}) are all positive. Since $\e$ is fixed, assume that we only consider tail terms for which $r\geq  \lceil \e^{-1}\log 2 \rceil$. This means that we can bound $1-e^{-r\e}\geq 1/2$. Therefore it suffices to show the following converges uniformly for $x$ varying in $[y_{\alpha,\e}-\eta,y_{\alpha,\e}+\eta]$
\begin{equation*}
\sum_{r=\lceil \e^{-1}\log 2 \rceil}^{\infty} \frac{e^{-r\e}}{r} \left(\frac{1-e^{-rx}}{1/2}\right).
\end{equation*}
Since we could choose $\eta$ small enough so that $\e+x> 0$ for all $x\in[y_{\alpha,\e}-\eta,\eta+y_{\alpha,\e}]$ it is clear that this positive series is bounded above by a convergent geometric series and hence is uniformly convergent in $r$. For (c) a similar logic applies.
\end{proof}

{\bf Step 3:}
We immediately get an upper bound on the derivative of the logarithm
\begin{equation}\label{logupbdd}
\frac{\partial}{\partial y} \log f_{\alpha}(y,\e) \leq -\frac{e^{-(\e+y_{\alpha,\e})}}{\e}.
\end{equation}
This is easy to see since each term of the summation in (\ref{diffseries}) is positive. This bound comes from taking only the $r=1$ term and using the inequality $1-e^{-a}\leq a$ for $a\geq 0$.

In order to get a lower bound on the derivative of the logarithm we must work a little harder. Fix $b^*<1$ and note that there exists an $a^*>0$ such that
\begin{equation}\label{ineq1}
1-e^{-a}\geq b^* a \quad \textrm{ for } 0\leq a\leq a^*.
\end{equation}
\begin{lemma}\label{loglowbdd}
Define $r^*=\lfloor \e^{-1}a^*\rfloor$ with $a^*$ as above. Then
\begin{equation}
\frac{\partial}{\partial y} \log f_{\alpha}(y,\e) +\e^{-1} e^{-(\e + y_{\alpha,\e})} \geq -e^{-(\e+y_{\alpha,\e})} -\sum_{r=2}^{\infty} \frac{e^{-r(\e+y_{\alpha,\e})}}{b^*r\e} - \sum_{r=r^*+1}^{\infty} \frac{e^{-r(\e+y_{\alpha,\e})}}{1-e^{-a^*}}.
\end{equation}
\end{lemma}
\begin{proof}
Start by splitting (\ref{diffseries}) into three parts as
\begin{equation}\label{threeparteqn}
\frac{\partial}{\partial y} \log f_{\alpha}(y,\e) = -\frac{e^{-(\e+y_{\alpha,\e})}}{1-e^{-\e}}- \sum_{r=2}^{r^*} \frac{e^{-r(\e+y_{\alpha,\e})}}{1-e^{-r\e}} - \sum_{r=r^*+1}^{\infty} \frac{e^{-r(\e+y_{\alpha,\e})}}{1-e^{-r\e}}.
\end{equation}
The first term is controlled by the following bound:
\begin{equation*}
\frac{1}{a} -\frac{1}{1-e^{-a}} \geq -1
\end{equation*}
from which it follows that
\begin{equation*}
-\frac{e^{-(\e+y_{\alpha,\e})}}{1-e^{-\e}} \geq -e^{-(\e+y_{\alpha,\e})}-\e^{-1} e^{-(\e+y_{\alpha,\e})}.
\end{equation*}

For the second term we may apply inequality (\ref{ineq1}) which gives
\begin{equation}\label{middleterm}
\sum_{r=2}^{r^*} \frac{e^{-r(\e+y_{\alpha,\e})}}{1-e^{-r\e}} \leq \sum_{r=2}^{r^*} \frac{e^{-r(\e+y_{\alpha,\e})}}{b^*r\e}\leq \sum_{r=2}^{\infty} \frac{e^{-r(\e+y_{\alpha,\e})}}{b^*r\e}.
\end{equation}
For the third term we use the fact that $1-e^{-a}$ is increasing for $a\geq 0$, which gives
\begin{equation}\label{thirdterm}
\sum_{r=r^*+1}^{\infty} \frac{e^{-r(\e+y_{\alpha,\e})}}{1-e^{-r\e}} \leq \sum_{r=r^*+1}^{\infty} \frac{e^{-r(\e+y_{\alpha,\e})}}{1-e^{-a^*}}.
\end{equation}
Combining these three bounds gives the lemma.
\end{proof}

The upper bound of (\ref{logupbdd}) and the lower bound of Lemma~\ref{loglowbdd} we have shown deal with the derivative of $\log f_{\alpha}(y,\e)$. In order to conclude information about $\log f_{\alpha}(y,\e)$ itself we employ the following:
\begin{lemma}\label{diffineq}
Consider $a\in \R$ and $y_0<\infty$, and a continuous differentiable function $g(y)$ such that $\lim_{y\rightarrow \infty}g(y)$ exists and equals $g^{\infty}\in \R$. Then if
\begin{equation*}
(1)\quad \frac{\partial}{\partial y} g(y) \leq -e^{-y} a \qquad \textrm{ or } \qquad (2)\quad \frac{\partial}{\partial y} g(y) \geq -e^{-y} \qquad \textrm{ for } y\geq y_0
\end{equation*}
then (respectively)
\begin{equation*}
(1)\quad g(y) \geq g^{\infty} + e^{-y}a \qquad \textrm{ or } \qquad (2)\quad g(y) \leq g^{\infty} + e^{-y}a \qquad \textrm{ for } y\geq y_0.
\end{equation*}
\end{lemma}
\begin{proof}
First observe that if $h(y)$ has a limit $h^{\infty}$ at $y=\infty$, then
\begin{equation*}
\int_{y}^{\infty} \frac{\partial}{\partial x} h(x) dx  = h^{\infty} - h(y).
\end{equation*}
Thus if $\frac{\partial}{\partial y} h(y)\leq 0$ for all $y\geq y_0$ then also
\begin{equation*}
 h^{\infty} - h(y) \leq 0.
\end{equation*}
Applying this and noting that $e^{-y} = \frac{\partial}{\partial y} \left(- e^{-y}\right)$, we find that by taking $h(y)= g(y) - e^{-y}a$ we get (since $e^{-\infty}=0$) the desired inequality (1). Similarly we derive (2).
\end{proof}
%\note{START}
Now we can complete the proof of Proposition~\ref{qfactProp}.
From equation (\ref{ylimitformula}) of Step 1 we know that a limit exists for $\log f_{\alpha}(y,\e)$ as $y$ goes to infinity. Thus we may apply case (1) of Lemma~\ref{diffineq} and using the lower bound in Lemma~\ref{Aepslemma} along with the upper bound on the derivative given in (\ref{logupbdd}), we immediately conclude the desired lower bound of equation (\ref{lwbdd}).

To get the upper bound of equation (\ref{upbdd}) we apply case (2) of Lemma~\ref{diffineq} and using the upper bound in Lemma~\ref{Aepslemma} along with the lower bound on the derivative given in (\ref{loglowbdd}), we find that
\begin{equation}\label{almostthereeqn}
\log f_{\alpha}(y,\e) -\A(\e) -\e^{-1} e^{-(\e+y_{\alpha,\e})} \leq e^{-(\e+y_{\alpha,\e})} + (b^*)^{-1}\e^{-1} \sum_{r=2}^{\infty} \frac{e^{-r(\e+y_{\alpha,\e})}}{r^2} +\sum_{r=r^*+1}^{\infty} \frac{e^{-r(\e+y_{\alpha,\e})}}{r(1-e^{-a^*})}.
\end{equation}
The last summation above can be bounded using the value of $r^*$ as
\begin{equation*}
\sum_{r=r^*+1}^{\infty} \frac{e^{-r(\e+y_{\alpha,\e})}}{r(1-e^{-a^*})} \leq C\e \frac{1}{1-e^{-(\e + y_{\alpha,\e})}}
\end{equation*}
which is itself bounded by a constant $c$ uniformly in $\e$ and $y_{\alpha,e}\geq 0$. Moreover, this constant can be taken to be going to zero as $y_{\alpha,\e}$ increases. Putting this bound back into (\ref{almostthereeqn}) gives equation (\ref{ylimitformula}) and hence concludes the proof of Proposition~\ref{qfactProp}.
\end{proof}

\subsubsection{Inductive proof of Theorem~\ref{qwhitconvTHM}}\label{indproofsec}

Let us first explain the plan of the proof. We prove the statement of Theorem~\ref{qwhitconvTHM} by induction on $\ell$.  Step 1 is to restate in (\ref{epsreceqn}) the recursion equation (\ref{receqn}) given the scaling we are now considering. The base case $\ell=0$ for the inductive proof we give is trivially satisfied. Step 2 is to record uniform bounds on the terms in (\ref{epsreceqn}). Here we make critical use of Proposition~\ref{qfactProp}. Finally, Step 3 is to prove the inductive hypothesis. To prove the bounds in part 1 of the theorem we can employ the bounds from  Step 2 to reduce the problem to a similar consideration as in the proof of Proposition~\ref{whitProp}. The proof of the convergence result in part 2 of the theorem follows from the inductive hypothesis along with the uniformity of the bounds given in Proposition~\ref{qfactProp}.

{\bf Step 1:} Let us recall the recursion equation (\ref{receqn}), which can be rewritten (see \cite{GLOqlim}) as
\begin{equation}\label{epsreceqn}
\psi^{\e}_{\ul{\nu}{\ell+1}}(\ul{x}{\ell+1})= \sum^{\qquad\e}_{\ul{x}{\ell+1}\in R_{\e}(\ul{x}{\ell+1})} \Delta^{\e}(\ul{x}{\ell}) e^{\iota \nu_{\ell+1}\left(\sum_{i=1}^{\ell+1} x_{\ell+1,i} - \sum_{i=1}^{\ell} x_{\ell,i}\right)} Q^{\e}_{\ell+1,\ell}(\ul{x}{\ell+1},\ul{x}{\ell})\psi^{\e}_{\ul{\nu}{\ell}}(\ul{x}{\ell})
\end{equation}
where
\begin{eqnarray*}
\sum^{\qquad\e}_{\ul{x}{\ell}\in R_{\e}(\ul{x}{\ell+1})} &=& \sum_{x_{\ell,1}\in R^{\e}_1} \e \cdots  \sum_{x_{\ell,\ell}\in R^{\e}_\ell} \e\\
R^{\e}_i &=& \{x_{\ell,i}\in \e\Z : -\e m(\e) +x_{\ell+1,i+1} \leq x_{\ell,i} \leq \e m(\e) +x_{\ell+1,i}\}
\end{eqnarray*}
and where
\begin{eqnarray*}
\Delta^{\e}(\ul{x}{\ell})  &=& \prod_{i=1}^{\ell-1} e^{-\A(\e)}f_2(x_{\ell,i} - x_{\ell,i+1},\e),\\
Q^{\e}_{\ell+1,\ell}(\ul{x}{\ell+1},\ul{x}{\ell}) &=& \prod_{i=1}^{\ell} e^{\A(\e)}f_1(x_{\ell+1,i}-x_{\ell,i},\e)^{-1}e^{\A(\e)}f_1(x_{\ell,i}-x_{\ell+1,i+1},\e)^{-1}.
\end{eqnarray*}

When $\ell=0$, we have
\begin{equation*}
\psi^{\e}_{\nu_1}(x_{1,1}) = e^{\iota \nu_1 x_{1,1}}.
\end{equation*}

This actually is equal to $\psi_{\nu_1}(x_{1,1})$ for all $\e$, hence part 2 of the Theorem is satisfied. Moreover, since $\ell=0$, part 1 is trivially satisfied by taking the polynomial to be a constant greater than 1. Therefore we have established the $\ell=0$ base case of the theorem. The rest of the proof is devoted to proving the induction step.

{\bf Step 2:} Assume now that we have proved the theorem for $\ell-1$ and we will show how the theorem for $\ell$ follows. From the inductive hypothesis we have the following bounds which are valid for $\underline{x}_{\ell}\in R_{\e}(\underline{x}_{\ell+1})$ and for $\e$ sufficiently small (here $c$ is a positive constant which varies between lines and $b^*$ is arbitrary but strictly bounded above by 1):
\begin{eqnarray}\label{psideltaqbdd}
\nonumber |\psi^{\e}_{\ul{\nu}{\ell}}(\ul{x}{\ell})| &\leq& \poly_{\ell,\sigma(\ul{x}{\ell})}(\ul{x}{\ell}) \prod_{i\in \sigma(\ul{x}{\ell})} \exp\{-e^{-(x_{\ell,i}-x_{\ell,i+1})/2}\}\\
\nonumber |\Delta^{\e}(\ul{x}{\ell})| &\leq& \exp\left\{\sum_{i=1}^{\ell-1} \left(c + (b^*)^{-1} \sum_{r=1}^{\infty} \frac{\e^{2r-1}e^{-r(x_{\ell,i}-x_{\ell,i+1})}}{r^2}\right)\right\}\\
|Q^{\e}_{\ell+1,\ell}(\ul{x}{\ell+1},\ul{x}{\ell})| & \leq & e^{c} \exp\left\{\sum_{i=1}^{\ell} \left( -e^{-(x_{\ell+1,i}-x_{\ell,i})} - e^{-(x_{\ell,i}-x_{\ell+1,i+1})}\right)\right\}.
\end{eqnarray}

The first inequality is directly from part 1 of the theorem, for $\ell-1$. The second inequality relies on the bound given in equation (\ref{upbdd}) when $\alpha=2$. The domain $R_{\e}(\ul{x}{\ell+1})$ for $\ul{x}{\ell}$ corresponds to the condition that $y_{2,\e}\geq 0$. The exact statement of the inequality uses the fact that
\begin{equation*}
(\e^{-1}+1)e^{-(\e+y_{\alpha,\e})} \leq (b^*)^{-1} \e^{-1}e^{-(\e+y_{\alpha,\e})}
\end{equation*}
for $b^{*}<1$ fixed and $\e$ small enough; it also uses the fact that $e^{-r\e}<1$. The third inequality comes immediately from equation (\ref{lwbdd}).

Observe that the $\Delta^{\e}$ term has the potential to become large. This occurs when $x_{\ell,i}-x_{\ell,i+1}$ approaches the lower bound of its values on the domain for $\ul{x}{\ell}$ which corresponds to $x_{\ell,i}-x_{\ell,i+1} = -2\log \e^{-1}$. This bound on the domain is due to the fact that in Definition \ref{scaleDef}, $p_{\ell,k}$ satisfied the interlacing condition. The growth of $\Delta^{\e}$ is, however, canceled by the rapid decay of the $Q^{\e}$ term, as we now show:

\begin{lemma}\label{unifupperboundLemma}
Consider $\ul{x}{\ell+1}$ such that $x_{\ell+1,i}-x_{\ell+1,i+1}\geq -2\log \e^{-1}$, $\ul{x}{\ell}\in R_{\e}(\ul{x}{\ell+1})$ and $b^*$ close to (though less than) one. Then there exist constants $c^*\in (0,1)$ and $C>0$ such that
\begin{equation*}
|\Delta^{\e}(\ul{x}{\ell+1}) Q^{\e}_{\ell+1,\ell}(\ul{x}{\ell+1},\ul{x}{\ell})|\leq C|Q^{\e}_{\ell+1,\ell}(\ul{x}{\ell+1},\ul{x}{\ell})|^{c^*}.
\end{equation*}
\end{lemma}
\begin{proof}
Using the bounds of (\ref{psideltaqbdd}) showing this suffices to proving that for each $1\leq i\leq \ell$,
\begin{equation*}%\label{rearrangedlemma}
(b^*)^{-1} \sum_{r=1}^{\infty} \frac{\e^{2r-1}e^{-r(x_{\ell+1,i}-x_{\ell+1,i+1})}}{r^2} \leq (c^*-1) \left(-e^{-(x_{\ell+1,i}-x_{\ell,i})} - e^{-(x_{\ell,i}-x_{\ell+1,i+1})}\right)
\end{equation*}
with $\ul{x}{\ell+1}$ and $\ul{x}{\ell}$ as specified in the statement of the lemma. Call $a=x_{\ell+1,i}-x_{\ell,i} + \log \e^{1}$ and $b=x_{\ell,i}-x_{\ell+1,i+1}+\log\e^{-1}$. Observe that the conditions on $\ul{x}{\ell+1}$ and $\ul{x}{\ell}$ then reduce to the condition that $a,b\geq 0$ and the desired inequality reduces to proving that
\begin{equation*}
(b^*)^{-1} \sum_{r=1}^{\infty} \frac{\e^{2r-1}e^{-ra}e^{-rb}}{r^2} \leq (c^*-1) \left(-e^{-a} - e^{-b}\right)
\end{equation*}
for all $a,b\geq 0$. It is now straightforward to prove the above inequality. For $a=b=0$ this reduces to showing that
\begin{equation*}
(b^*)^{-1} \zeta(2) \leq 2(1-c^*)
\end{equation*}
which follows for $b^*$ close enough to 1 and $c^*$ close enough to 0 (since $\zeta(2)=\pi^2/6<2$).
To extend the proof to all $a,b$ simply observe that the derivative of the left-hand side in any positive direction in the $(a,b)$ plane, is less than the corresponding derivative of the right-hand side. Showing this (which only uses the inequality $\log(1-x)\leq -x$ for $x<1$) completes the proof of the lemma.
\end{proof}

{\bf Step 3:} We may now prove that given the inductive hypothesis for $\ell-1$, the theorem is satisfied for $\ell$. Let us start with part 1 of the theorem. By the recursive formula (\ref{epsreceqn}), the triangle inequality and Lemma~\ref{unifupperboundLemma} we have
\begin{equation}\label{epsreceqnbound}
|\psi^{\e}_{\ul{\nu}{\ell+1}}(\ul{x}{\ell+1})|\leq  \sum^{\qquad\e}_{\ul{x}{\ell}\in R_{\e}(\ul{x}{\ell+1})} |Q^{\e}_{\ell+1,\ell}(\ul{x}{\ell+1},\ul{x}{\ell})|^{c^*} |\psi^{\e}_{\ul{\nu}{\ell}}(\ul{x}{\ell})|
\end{equation}
for the constant $c^*\in (0,1)$ given in Lemma~\ref{unifupperboundLemma}. From the inductive hypothesis, $|\psi^{\e}_{\ul{\nu}{\ell}}(\ul{x}{\ell})|$ is bounded by a polynomial in $\ul{x}{\ell}$ times a collection of double exponential terms. The double exponential terms are bounded by 1, so let us bound $|\psi^{\e}_{\ul{\nu}{\ell}}(\ul{x}{\ell})|$ by just a polynomial. Splitting that polynomial into monomials and using the bound given in (\ref{psideltaqbdd}) for $Q^\e$ we find that the summation in (\ref{epsreceqnbound}) factors into a product of terms like
\begin{equation*}
\sum_{x_{\ell,i}\in R^{\e}_i} \e (x_{\ell,i})^m \exp\{-c^* e^{-(x_{\ell+1,i}-x_{\ell,i})}\}\exp\{- c^* e^{-(x_{\ell,i}-x_{\ell+1,i+1})}\},
\end{equation*}
where $m\in \Zgeqzero$.

This summation can be bounded by the integral it converges to which is
\begin{equation*}
\int_{-\infty}^{\infty} dx_{\ell,i} (x_{\ell,i})^m \exp\{-c^* e^{-(x_{\ell+1,i}-x_{\ell,i})}\}\exp\{- c^* e^{-(x_{\ell,i}-x_{\ell+1,i+1})}\}.
\end{equation*}

Using the exact same argument as in the inductive step in the proof of Proposition~\ref{whitProp}, we can now complete the proof of part 1 of the theorem by bounding this integral, and then recombining all of the monomial terms. The only change between that proof and this is the factor of $c^*$ in the exponential. However, as one readily seems from Lemma~\ref{gumblemma}, this factor just carries through to the final formula and accounts for the $c^*$ in (\ref{qboundthmeqn}).

We turn now to proving part 2 in this inductive step. Fix a compact domain $D_{\ell+1}$ in which $\underline{x}_{\ell+1}$ may vary. We wish to show that on $D_{\ell+1}$, $\psi^{\e}_{\ul{\nu}{\ell+1}}(\ul{x}{\ell+1})$ converges uniformly as $\e$ goes to zero and $\ul{x}{\ell+1}$ is varied in $D$ to $\psi_{\ul{\nu}{\ell+1}}(\ul{x}{\ell+1})$. By Lemma~\ref{unifupperboundLemma}, for any fixed $\eta>0$ we can find a compact domain $D_{\ell}$ for the $\underline{x}_{\ell}$ variables, outside of which the summation in equation (\ref{epsreceqn}) is bounded in absolute value by $\eta$, uniformly over $\underline{x}_{\ell+1}$ varying in $D_{\ell+1}$.

On the domain $D_{\ell}$ we have uniform convergence of the summand of (\ref{epsreceqn}) to the analogous integrand of equation (\ref{recWhit}). This can be seen by combining the inductive hypothesis for $\ell-1$ and equation (\ref{unifconveqn}) of Corollary \ref{qfactCor}. It follows then that the Riemann sum of the summand of (\ref{epsreceqn}) over the domain $D_{\ell}$ converges to the integral over the domain $D_{\ell}$ in the recursive formula for the Whittaker functions, given in equation (\ref{recWhit}). Moreover, this convergence is uniform as $\underline{x}_{\ell+1}\in D_{\ell+1}$.

All that remains to finish the proof is to note that by Proposition~\ref{whitProp}, by taking $D_{\ell}$ large enough we can ensure an upper bound of $\eta$ on the absolute value of the contribution of the integral in (\ref{recWhit}) outside of $D_{\ell}$. Moreover, this upper bound is uniform over $\underline{x}_{\ell+1}\in D_{\ell+1}$. Since $\eta$ was arbitrary, this completes the inductive step for uniform convergence, and hence completes the proof of Theorem~\ref{qwhitconvTHM}.

\subsection{Weak convergence to the Whittaker process}\label{weakconvofan}

We start by defining the Whittaker process and measure as introduced and first studied by O'Connell \cite{OCon} in relation to the image of the O'Connell-Yor semi-discrete directed polymer under a continuous version of the tropical Robinson-Schensted-Knuth correspondence (see Section~\ref{OConmodel}).

\begin{definition}
For any $\tau>0$ set
\begin{equation}\label{thetaformula}
\theta_{\tau}(\ul{x}{N})=\int_{\R^N} \psi_{\ul{\nu}{N}}(\ul{x}{N}) e^{-\tau\sum_{j=1}^N\nu_j^2/2} m_N(\ul{\nu}{N})d\ul{\nu}{N}.
\end{equation}
\glossary{$\theta_{\tau}(\ul{x}{N})$}

For any $N$-tuple $a=(a_1,\dots,a_N)\in \R^N$ and $\tau>0$, define the {\it Whittaker process}\index{Whittaker process} as a measure on $\R^{\frac{N(N+1)}2}$ with density function (with respect to the Lebesgue measure) given by
\begin{equation}\label{Wdef}
\W{a;\tau}\big(\{T_{k,i}\}_{1\le i\le k\le N}\big)=e^{-\tau\sum_{j=1}^N a_j^2/2} \exp({\mathcal{F}_{\iota a}(T)})\,{\theta_{\tau}(T_{N,1},\dots,T_{N,N})}.
\end{equation}
\glossary{$\W{a;\tau}$}
The nonnegativity of the density follows from definitions. Integrating over the variables $T_{k,i}$ with $k<N$ yields the following {\it Whittaker measure}\index{Whittaker measure} with density function given by
\begin{equation}\label{WMdef}
\WM{a;\tau}\big(\{T_{N,i}\}_{1\le i\le N}\big)=e^{-\tau\sum_{j=1}^N a_j^2/2}\psi_{\iota a}(T_{N,1},\dots,T_{N,N})\,{\theta_{\tau}(T_{N,1},\dots,T_{N,N})}.
\end{equation}
\glossary{$\WM{a;\tau}$}
\end{definition}

\begin{remark}\label{diffformofmeasures}
The form of the Whittaker measure and the Macdonald measure may appear different. The Macdonald $P$ polynomial has degenerated into the Whittaker function $\psi$, the normalizing constant has become the Gaussian prefactor. The specialization of the Macdonald $Q$ function which degenerates to the measure we are presently studying is the Plancherel specialization. Whittaker functions only arise as limits of pure alpha Macdonald polynomials. In order to find the limit of the Placherel specialized Macdonald $Q$ function we use the fact that it can be represented via the torus inner product in terms of the Macdonald $P$ polynomial and the normalizing constant. This torus integral limits to the expression $\theta_{\tau}(\ul{x}{N})$ above. The Whittaker measure we are considering is, perhaps, more appropriately called the Plancherel Whittaker measure, as we will introduce in Section \ref{alphawhitprocessSec} the $\alpha$-Whittaker measure.
\end{remark}

\subsubsection{The Whittaker measure is a probability measure}\label{whitprobmeas}
It is not clear, a priori, that these measures integrate to 1, however we have the following:

\begin{proposition}\label{equiv1}
Fix $N\geq 1$. For all $a=(a_1,\ldots, a_N)\in \R^N $,
\begin{equation*}
\int_{\R^N} \WM{a;\tau}(\ul{x}{N}) d\ul{x}{N} =1.
\end{equation*}
\end{proposition}

\begin{proof}
Formally this fact follows immediately by applying the orthogonality relations for Whittaker functions given in Section~\ref{orthcom}. That relation, however, is only necessarily true when the index, here written as $\iota a$ is purely real. Thus, in order to justify this identity for indices $(a_1,\ldots, a_N)$ not all identically 0, we must perform an analytic continuation of the orthogonality relation (when $a_i\equiv 0$ this continuation is not needed and the result follows immediately). This analytic continuation relies on a Paley-Weiner type super-exponential decay estimate for $\theta_{\tau}(\ul{x}{N})$ which is given in Proposition~\ref{supexpthetadecay}.

%Presently, we only confirm this proposition for $N=1$ and $2$ though conjecture it to hold for all $N\geq 1$. This conjecture is very plausible and we may have an approach to it which has not been fully realized.

A less direct proof that the Whittaker measure integrates to 1 comes from the fact that the Whittaker process / measure arise as the law of the image of a directed polymer model (studied in \cite{OCon}) under a continuous version of the tropical Robinson-Schensted-Knuth correspondence (see Remark \ref{probresrem}).

For the direct proof we now give, define a function
\begin{equation*}
f(a_1,\ldots, a_N) = \int_{\R^N} e^{-\tau \sum_{j=1}^{N} a_j^2/2}  \psi_{\iota a}(\ul{x}{N}) \theta_{\tau}(\ul{x}{N}) d\ul{x}{N}
\end{equation*}
for $(a_1,\ldots, a_N)\in \C^N$. When restricted to the complex axes $(a_1,\ldots,a_N)\in \iota\R^N$, the orthogonality of Whittaker functions readily implies that
\begin{equation*}
f(a_1,\ldots, a_N) \equiv 1 \quad \textrm{for }(a_1,\ldots,a_N)\in \iota\R^N.
\end{equation*}
We wish to analytically continue this result to all of $\C^N$ and hence show that $f$ is everywhere identically 1. This requires us to show that $f(a_1,\ldots, a_N)$ is analytic. The Whittaker function $\psi_{\iota a}(\ul{x}{N})$ is known to be entire with respect to its index (see the discussion in \cite{OCon}) as well as continuous in its variable. On account of this, analyticity of $f(a_1,\ldots, a_N)$ will follow if we can show that for any compact region $A\subset \C^N$ and any $\e>0$, there exists a compact region $X\in\R^N$ such that
\begin{equation}\label{epsboundtau}
\int_{\R^N \setminus X} \left| e^{-\tau \sum_{j=1}^{N} a_j^2/2}  \psi_{\iota a}(\ul{x}{N}) \theta_{\tau}(\ul{x}{N})\right| d\ul{x}{N} \leq \e
\end{equation}
for all $(a_1,\ldots,a_N)\in A$.

This claim relies on two pieces. The first is a lemma about the growth of Whittaker functions with imaginary index and the second (a more sizable result) is about the super-exponential decay of $\theta_{\tau}$.
\begin{lemma}
For any positive real numbers $y_1,\ldots, y_N$ there exist $c_0, C_0>0$ such that
\begin{equation*}
|\psi_{\nu_1,\ldots ,\nu_N}(x_1,\ldots,x_N)| \leq C_0 e^{c_0 \sum_{j=1}^{N} |x_j|} \qquad \textrm{ for all } \nu_j: |\Imag(\nu_j)|\leq y_j, j=1,\ldots, N.
\end{equation*}
\end{lemma}
\begin{proof}
This follows easily by induction from the Givental recursive formula (\ref{recWhit}) for the Whittaker functions. \end{proof}
The second estimate is given by the following:
\begin{proposition}\label{supexpthetadecay}
Fix $N\geq 1$. For any constant $R\geq 0$, there exists a constant $C>0$ such that for all $\ul{x}{N}\in \R^N$,
\begin{equation*}
|\theta_{\tau}(\ul{x}{N})| \leq C e^{-R \sum_{j=1}^{N} |x_{j}|}.
\end{equation*}
\end{proposition}
%We view the following conjecture as very plausible and may have an approach to it which has not been fully realized.
%\begin{conj}\label{decayconjecture}
%For all $N\geq 1$, Proposition~\ref{supexpthetadecay} holds.
%\end{conj}

By combining this proposition with the above lemma we see that by choosing $R$ to exceed the $c_0$, we can ensure exponential decay in $(x_1,\ldots, x_N)$ and hence prove the bound (\ref{epsboundtau}). This completes the proof of analyticity of $f(a_1,\ldots, a_N)$ and hence by analytic continuation, $f(a_1,\ldots, a_N)\equiv 1$ for all $(a_1,\ldots, a_N)\in \C^N$ as desired.

All that remains, therefore, is to prove Proposition~\ref{supexpthetadecay}.

\begin{proof}[Proof of Proposition~\ref{supexpthetadecay}]

The case of $N=1$ is standard, yet instructive. In this case, $\psi_{\nu}(x) = e^{-i\nu x}$ and $m_N$ is just $1/(2\pi)$. The idea now is to shift the contour of integration in (\ref{thetaformula}) from $\R$ to $\R+\iota y$. This is justified by Cauchy's theorem and the fact that the integrand decays suitably fast along every horizontal contour. Changing variables to absorb this shift into $x$ yields
\begin{equation*}
\theta_{\tau}(x)=\int_{\R} e^{-\iota \nu x} e^{yx} e^{-\tau \nu^2/2} e^{-\tau \iota y \nu} e^{\tau y^2/2}  \frac{d\nu}{2\pi}=e^{yx}e^{\tau y^2/2} \int_{\R} e^{-\iota\nu x} e^{-\tau \nu^2/2} e^{-\tau \iota y \nu} \frac{d\nu}{2\pi}.
\end{equation*}
Since the final integral above is bounded in absolute value by a constant, this yields the inequality
\begin{equation*}
|\theta_{\tau}(x)| \leq e^{yx} C_y
\end{equation*}
where $C_y>0$ depends on $y$. By taking $y=R$ when $x<0$ and $y=-R$ when $x>0$, we arrive at the claimed inequality of the proposition.

The case of $N\geq 2$ is more subtle and relies on a series expansion for the Whittaker function. We work out the calculation explicitly for $N=2$. For general $N\geq 2$ a similar series expansion into fundamental Whittaker functions exists (combining Theorem 18 of \cite{IStade} and Proposition 3 of \cite{OConmanuscript}) for which the argument presented below can readily be adapted.

In the case $N=2$ the Whittaker functions can be expressed explicitly in terms of the Bessel-K function (sometimes also called the Macdonald function -- not to be confused with the Macdonald symmetric function):
\begin{equation*}
\psi_{\nu_1,\nu_2}(x_1,x_2) = 2 e^{\frac{1}{2} \iota \kappa_1 y_1} K_{\iota\kappa_2}(2 e^{-y_2/2})
\end{equation*}
where
\begin{equation*}
\kappa_1=\nu_1+\nu_2,\qquad \kappa_2 = \nu_1-\nu_2,\qquad y_1= x_1+x_2,\qquad y_2=x_1-x_2,
\end{equation*}
and the Bessel-K function is given by
\begin{equation*}
K_{v}(z) = \frac{1}{2}\int_{-\infty}^{\infty} e^{xv} \exp\left(-\frac{z}{2}(e^x+e^{-x})\right) dx.
\end{equation*}

Thus we can write
\begin{equation}\label{thetaformulanew}
\theta_{\tau}(x_1,x_2)=\int_{\R^2} 2 e^{\frac{1}{2} \iota \kappa_1 y_1} K_{\iota \kappa_2}(2 e^{-y_2/2}) e^{-\frac{\tau}{4}(\kappa_1^2+\kappa_2^2)} \frac{1}{(2\pi)^2 2!} \frac{1}{\Gamma(\iota \kappa_2)\Gamma(-\iota \kappa_2)} \frac{d\kappa_1 d\kappa_2}{2}.
\end{equation}
This double integral factors as
\begin{equation*}
\theta_{\tau}(x_1,x_2)= \left(\int_{\R}  e^{\frac{1}{2} \iota \kappa_1 y_1}  e^{-\frac{\tau}{4}\kappa_1^2}\frac{d\kappa_1}{4\pi} \right) \left(\int_{\R} e^{-\frac{\tau}{4}\kappa_1^2}  K_{\iota\kappa_2}(2 e^{-y_2/2})\frac{\kappa_2 \sinh(\pi \kappa_2)}{\pi} \frac{d\kappa_2}{2\pi}\right).
\end{equation*}

We may compute the first integral as
\begin{equation*}
F_1(y_1)= \int_{\R}  e^{\frac{1}{2} \iota\kappa_1 y_1}  e^{-\frac{\tau}{4}\kappa_1^2}\frac{d\kappa_1}{4\pi} = \frac{1}{2\sqrt{\pi}} e^{-y_1^2/4\tau}.
\end{equation*}

Turning to the second integral, let us call
\begin{equation*}
F_2(y_2) = \int_{\R} e^{-\frac{\tau}{4}\kappa_2^2}  K_{\iota \kappa_2}(2 e^{-y_2/2})\frac{\kappa_2 \sinh(\pi \kappa_2)}{\pi} \frac{d\kappa_2}{2\pi}.
\end{equation*}
If we can show that for any $R>0$, $|F_2(y_2)| \leq \exp(-R |y_2|)$ then Proposition~\ref{supexpthetadecay} will be proved. This is because we already know (from the above computation) that $|F_1(y_1)|\leq \exp(-R|y_1|)$ for all $R>0$. The claimed result of the proposition follows by the triangle inequality ($|y_1|+|y_2| \geq |x_1|+|x_2|$) and the relation $|\theta_{\tau}(x)| = |F_1(y_1)||F_2(y_2)|$.

In order to show $|F_2(y_2)| \leq \exp(-R |y_2|)$ we should consider two cases: (1) when $y_2\leq 0$, and (2) when $y_2\geq 0$. In case (1) we use the uniform asymptotics that as $z\to \infty$, for all $v\in \R$,
\begin{equation*}
|K_{\iota v}(z)| \sim \sqrt{\pi/2z} e^{-z}.
\end{equation*}
The exponential decay is the important part here and we can bound the right-hand side above by $Ce^{-cz}$ for some $C>0$ and any $c<1$. This allows us to show that for $y_2\leq 0$,
\begin{equation*}
|F_2(y_2)| \leq C e^{-2c e^{-y_2/2}} \int_{\R} e^{-\frac{\tau}{4}\kappa_2^2} \frac{\kappa_2 \sinh(\pi \kappa_2)}{\pi} \frac{d\kappa_2}{2\pi}.
\end{equation*}
The remaining integral is independent of $y_2$ and clearly converges due to the Gaussian decay in $\kappa_2$. Thus, for $y_2\leq 0$
\begin{equation*}
|F_2(y_2)| \leq C' e^{-2c e^{-y_2/2}}.
\end{equation*}
It is clear that this double exponential decay can be bounded by exponential decay to any order $R$, at the cost of prefactor constant to account for $y_0$ near 0. This completes the desired estimate in case (1).

Turning to case (2), where $y_2\geq 0$, we will first utilize the fact that the Bessel-K function can be rewritten as a Bessel-I function in such a way as to cancel the $\sinh$ term. This expansion and cancelation is important and a similar phenomena occurs for general $N\geq 2$ (see Proposition 3 of \cite{OConmanuscript}). In particular
\begin{equation*}
K_{v}(z) = \frac{\pi}{2} \frac{I_{-v}(z) - I_{v}(z)}{\sin(\pi v)}
\end{equation*}
where the Bessel-I function has the following convergent Taylor series expansion around $z=0$.
\begin{equation*}
I_v(z)  = (z/2)^v \sum_{\ell=0}^{\infty} \frac{(z^2/4)^\ell}{\ell! \Gamma(v+\ell+1)}.
\end{equation*}
Plugging this in (and absorbing numerical constants as $c$)
\begin{equation*}
F_2(y_2) = c \int_{-\infty}^{\infty} e^{-\frac{\tau}{4} \kappa_2^2} \kappa_2 \left(I_{-\iota \kappa_2}(2 e^{-y_2 /2})-I_{\iota \kappa_2}(2 e^{-y_2 /2})\right) d\kappa_2.
\end{equation*}

We wish to cut the series expansion for Bessel-I at a particular level $\ell^*$ and estimate the remainder term.

Use the Gamma functional equation to rewrite this as
\begin{equation*}
I_v(z)  = \frac{(z/2)^v}{\Gamma(v+1)} \sum_{\ell=1}^{\infty} \frac{(z^2/4)^\ell}{\ell! (v+1)\cdots (v+\ell)},
\end{equation*}
where $(v+1)\cdots (v+\ell)$ for $\ell=0$ is just interpreted as $1$.
For all $z\in \R$ and $v\in \iota\R$ it follows that
\begin{equation*}
\left| \sum_{\ell=\ell^*}^{\infty} \frac{(z^2/4)^\ell}{\ell! (v+1)\cdots (v+\ell)}\right| \leq \sum_{\ell=\ell^*}^{\infty} \frac{(z^2/4)^\ell}{(\ell!)^2} \leq C(\ell^*) (z^2/4)^{\ell^*}
\end{equation*}
for the constant $C(\ell^*)$ depending on $\ell^*$ but independent of $z$ and $v$.

Thus we can show that
\begin{equation*}
|F_2(y_2)| \leq I+II+III,
\end{equation*}
where
\begin{eqnarray*}
I &=& c \sum_{\ell=0}^{\ell^*-1} \left|\int_{-\infty}^{\infty} e^{-\frac{\tau}{4} \kappa_2^2} \kappa_2 e^{-\iota y_2\kappa_2 /2} \frac{e^{-y_2 \ell}}{\ell! \Gamma(\iota \kappa_2+\ell+1)} d\kappa_2 \right| \\
II &=&  c \sum_{\ell=0}^{\ell^*-1} \left|\int_{-\infty}^{\infty} e^{-\frac{\tau}{4} \kappa_2^2} \kappa_2 e^{\iota y_2\kappa_2 /2} \frac{e^{-y_2 \ell}}{\ell! \Gamma(-\iota \kappa_2+\ell+1)} d\kappa_2 \right|\\
III &=& 2 c  e^{-y_2 (\ell^*+1)} C(\ell^*) \int_{-\infty}^{\infty} e^{-\frac{\tau}{4} \kappa_2^2} \kappa_2  d\kappa_2.
\end{eqnarray*}
Note that III contains the remainder term estimate from both the $I_{-\iota \kappa_2}$ and $I_{\iota \kappa_2}$ terms, with $z=2 e^{-y_2/2}$.

It remains to show that each of these terms can be bounded by $Ce^{-R|y_2|}$ for $y_2\geq 0$ and $C$ depending on $R$ but not on $y_2$. As noted before, this will complete the proof of the proposition. Consider III first. By taking $\ell^*+1 \geq R$ the desired bound $III\leq C e^{-Ry_2}$ immediately follows. I and II can be dealt with in the same manner, thus let us just consider I. We will bound each term in I separately and call them $I_{\ell}$. Recall that $1/\Gamma(w)$ is an entire function in $w$. It follows from Cauchy's theorem and the Gaussian decay of the integrand, that for any $b\in \R$,
\begin{equation*}
I_{\ell}= \left|\int_{-\infty}^{\infty} e^{-\frac{\tau}{4} \kappa_2^2} \kappa_2 e^{-\iota y_2\kappa_2 /2} \frac{e^{-y_2 \ell}}{\ell! \Gamma(\iota \kappa_2+\ell+1)} d\kappa_2 \right| = \left|\int_{-\infty+\iota b}^{\infty+\iota b} e^{-\frac{\tau}{4} \kappa_2^2} \kappa_2 e^{-\iota y_2\kappa_2 /2} \frac{e^{-y_2 \ell}}{\ell! \Gamma(\iota \kappa_2+\ell+1)} d\kappa_2 \right|.
\end{equation*}
The shift by $\iota b$ of the contour can be removed via a change of variables which yields
\begin{equation*}
I_{\ell} = e^{y_2 b/2} \left|\int_{-\infty}^{\infty} e^{-\frac{\tau}{4} (\kappa_2+\iota b)^2} (\kappa_2+\iota b) e^{-\iota y_2\kappa_2 /2} \frac{e^{-y_2 \ell}}{\ell! \Gamma(\iota (\kappa_2+\iota b)+\ell+1)} d\kappa_2 \right|.
\end{equation*}
Take $b=-2R$ and bound the integral above as
\begin{equation*}
\left|\int_{-\infty}^{\infty} e^{-\frac{\tau}{4} (\kappa_2+\iota b)^2} (\kappa_2+\iota b) e^{-\iota y_2\kappa_2 /2} \dfrac{e^{-y_2 \ell}}{\ell! \Gamma(\iota (\kappa_2+\iota b)+\ell+1)} d\kappa_2\right|
\end{equation*}
\begin{equation*}
\leq \int_{-\infty}^{\infty} e^{-\frac{\tau}{4} \kappa_2^2}e^{\frac{\tau}{2}b^2} |\kappa_2+b| \dfrac{1}{\ell! |\Gamma(\iota (\kappa_2+\iota b)+\ell+1)|} d\kappa_2 \leq C
\end{equation*}
where we have used $e^{-y_2\ell} \leq 1$ since $y_2\geq 0$, and where the constant $C$ is thus independent of $y_2$.
This means that for $y_2\geq 0$
\begin{equation*}
I_{\ell} \leq C e^{-R y_2}
\end{equation*}
which is exactly as necessary. This estimate can be made for each term in I and likewise for each term in II and hence completes the proof of Proposition~\ref{supexpthetadecay}.
\end{proof}

As we have proved Proposition~\ref{supexpthetadecay} above, the proof of Proposition~\ref{equiv1} is also completed.
\end{proof}

\subsubsection{Weak convergence statement and proof}\label{weakconvstateproofSec}

\begin{theorem}\label{theorem26}
In the limit regime
\begin{eqnarray*}
&&q=e^{-\epsilon},\qquad \gamma=\tau \epsilon^{-2}, \qquad A_k=e^{-\epsilon a_k}, \  1\le k\le N, \\
&&\lambda^{(k)}_j=\tau\epsilon^{-2}-(k+1-2j)\epsilon^{-1}\log\epsilon+T_{k,j}\epsilon^{-1}, \quad 1\le j\le k\le N,
\end{eqnarray*}
the $q$-Whittaker process $\M_{asc,t=0}(A_1,\dots,A_N;\rho)$ with $t=0$ and a Plancherel specialization
$\rho$ determined by $\gamma$ (and $\alpha_i,\beta_i=0$ for all $i\geq 0$) as in equation (\ref{tag1}), weakly converges, as $\e\to 0$, to the Whittaker process $\W{a,\tau}\bigl(\{T_{k,i}\}_{1\le i\le k\le N}\bigr)$.
\end{theorem}
\begin{remark}
The $q$-Whittaker process induces a measure on $\{T_{k,j}\}_{1\leq j\leq k\leq N}$ through the scalings given above. It is this measure which has a limit as $\e$ goes to zero.
\end{remark}

\begin{proof}

Recall definition (\ref{ascmeasure}) which says
\begin{equation}\label{MascDefrecall}
\M_{asc,t=0}(A_1,\ldots, A_N;\rho)(\lambda^{(1)},\ldots,\lambda^{(N)})= \frac{P_{\lambda^{(1)}}(A_1)P_{\lambda^{(2)}/\lambda^{(1)}}(A_2)\cdots P_{\lambda^{(N)}/\lambda^{(N-1)}}(A_N) Q_{\lambda^{(N)}}(\rho)}{\Pi(A_1,\ldots, A_N;\rho)}.
\end{equation}

It suffices to control the convergence for $\{T_{k,i}\}$ varying in any given compact set. This is due to the positivity of the measure and because we can see independently that the limiting expression integrates to 1. However, since we wish to prove weak convergence of this measure to the limiting measure, it is crucial that all of our estimates (with respect to convergence as $\e$ goes to zero) are suitably uniform over the given compact set. As such, for the rest of the proof fix a compact set for the $\{T_{k,i}\}$.

We consider the right-hand side of equation (\ref{MascDefrecall}) in three lemmas which, when combined, prove the theorem: (1) The product of skew Macdonald polynomials $P_{\lambda/\mu}$, (2) The $\Pi$ factor in the denominator, (3) The $Q$ Macdonald symmetric function. The first of these lemmas relies on the asymptotics developed in Proposition~\ref{qfactProp}, the second is simply a matter of Taylor approximation. The third lemma is more involved and employs the torus scalar product. The full strength of Theorem~\ref{qwhitconvTHM} is utilized in the proof of this third lemma.

Before commencing the proof, let us recall the notation introduced in Section~\ref{721}: $\A(\e) = -\frac{\pi^2}{6} \frac{1}{\e} -\frac{1}{2}\log\frac{\e}{2\pi}$ and $f_{\alpha}(y,\e)~=~(q;q)_{\e^{-1}y_{\alpha,\e}}$ where $y_{\alpha,\e}=y+\e \alpha m(\e)$, and $m(\e)= -[\e^{-1}\log\e]$.

Let us start by considering the skew Macdonald polynomials with $t=0$.
\begin{lemma}\label{Skewlemma}
Fix any compact subset $D\subset \R^{N(N+1)/2}$. Then
\begin{equation}\label{skewPprodeqn}
P_{\lambda^{(1)}}(A_1)P_{\lambda^{(2)}/\lambda^{(1)}}(A_2)\cdots P_{\lambda^{(N)}/\lambda^{(N-1)}}(A_N) = \left(e^{-\frac{(N-1)(N-2)}{2} \A(\e)} e^{-\e^{-1}\tau \sum_{k=1}^{N}a_k}\right)\mathcal{F}_{\iota a}(T) e^{o(1)}
\end{equation}
where the $o(1)$ error goes uniformly (with respect to $\{T_{k,i}\}_{1\leq i\leq k\leq N}\in D$) to zero as $\e$ goes to zero.
\end{lemma}
\begin{proof}
From equation (\ref{7.14'}) we have that when restricted to one variable, $P_{\lambda/\mu}(A)=\psi_{\lambda/\mu} A^{|\lambda-\mu|}$,  with $\psi_{\lambda/\mu}$ as in Lemma~\ref{lemma22}. From that we may express
\begin{equation}\label{PlambdaAeqn}
P_{\lambda^{(\ell+1)}/\lambda^{(\ell)}}(A_{\ell+1}) =A_{\ell+1}^{|\lambda^{(\ell+1)}-\lambda^{(\ell)}|} \prod_{i=1}^{\ell} \frac{ \qq{\lambda^{(\ell+1)}_{i} - \lambda^{(\ell+1)}_{i+1}}}{\qq{\lambda^{(\ell+1)}_{i} - \lambda^{(\ell)}_{i}}\qq{\lambda^{(\ell)}_{i} - \lambda^{(\ell+1)}_{i+1}}}.
\end{equation}

Now observe that given the scalings we have chosen, we have
\begin{eqnarray*}
\lambda^{(\ell+1)}_i - \lambda^{(\ell+1)}_{i+1} &=& \e^{-1}(T_{\ell+1,i} - T_{\ell+1,i+1}) + 2m(\e)\\
\lambda^{(\ell+1)}_i - \lambda^{(\ell)}_{i} &=& \e^{-1}(T_{\ell+1,i} - T_{\ell,i}) + m(\e)\\
\lambda^{(\ell)}_i - \lambda^{(\ell+1)}_{i+1} &=& \e^{-1}(T_{\ell,i} - T_{\ell+1,i+1}) + m(\e).
\end{eqnarray*}

Noting that
\begin{eqnarray*}
|\lambda^{(\ell+1)}-\lambda^{(\ell)}| &=& \tau\e^{-2}+ \e^{-1} \left(\sum_{i=1}^{\ell+1}T_{\ell+1,i} - \sum_{i=1}^{\ell}T_{\ell,i}\right),\\
\qq{\lambda^{(\ell+1)}_{i} - \lambda^{(\ell+1)}_{i+1}} &=& f_2(T_{\ell+1,i}-T_{\ell+1,i+1},\e),\\
\qq{\lambda^{(\ell+1)}_{i} - \lambda^{(\ell)}_{i}} &=& f_1(T_{\ell+1,i}-T_{\ell,i},\e),\\
\qq{\lambda^{(\ell)}_{i} - \lambda^{(\ell+1)}_{i+1}} &=& f_1(T_{\ell,i}-T_{\ell+1,i+1},\e),
\end{eqnarray*}
it follows from the estimates of Corollary \ref{qfactCor} and scalings of the $A_i$ that we may rewrite equation (\ref{PlambdaAeqn}) as
\begin{eqnarray*}
P_{\lambda^{(\ell+1)}/\lambda^{(\ell)}}(A_{\ell+1}) &=&  e^{-\ell \A(\e)} e^{-\e^{-1}\tau a_{\ell+1}} \exp\left\{-a_{\ell+1} \left(\sum_{i=1}^{\ell+1} T_{\ell+1,i}-\sum_{i=1}^{\ell} T_{\ell,i}\right)\right\} \\ \nonumber &&\times \exp\left\{-\sum_{i=1}^{\ell}\exp\{T_{\ell,i}-T_{\ell+1,i}\} - \sum_{i=1}^{\ell}\exp\{T_{\ell+1,i+1}-T_{\ell,i}\}\right\} e^{o(1)}
\end{eqnarray*}
where the $o(1)$ error goes uniformly (with respect to $\{T_{k,i}\}_{1\leq i\leq k\leq N}\in D$) to zero as $\e$ goes to zero.

Applying this to each of the skew Macdonald polynomials we arrive at the claimed result.
\end{proof}

Now turn to the $\Pi$ factor.
\begin{lemma}\label{Pilemma}
We have
\begin{equation*}
\Pi(A_1,\ldots,A_N;\rho) = \left(e^{\tau N \e^{-2}} e^{-\e^{-1}\tau \sum_{k=1}^{N} a_k}\right) e^{\tau \sum_{k=1}^{N} a_k^2/2}e^{o(1)}
\end{equation*}
where the $o(1)$ error goes to zero as $\e$ goes to zero.
\end{lemma}
\begin{proof}
Observe that by the definition of $\Pi$ and $A$, and simple Taylor approximation
\begin{equation*}
\Pi(A_1,\ldots,A_N;\rho) = \prod_{k=1}^{N} e^{\gamma A_k} = e^{\sum_{k=1}^{N} \tau \e^{-2} e^{-\e a_k}} = e^{\tau N \e^{-2}} e^{-\e^{-1}\tau \sum_{k=1}^{N} a_k} e^{\tau \sum_{k=1}^{N} a_k^2/2}e^{o(1)}.
\end{equation*}
\end{proof}

Finally turn to the $Q_{\lambda^{(N)}}(\rho)$ factor.
\begin{lemma}\label{Qlemma}
Fix any compact subset $D\subset \R^{N(N+1)/2}$. Then
\begin{equation*}
Q_{\lambda^{(N)}}(\rho) = \left(e^{\frac{(N-1)(N-2)}{2}\A(\e)} e^{\tau N \e^{-2}} \e^{\frac{N(N+1)}{2}}\right)\theta_{\tau}(T_{N,1},\ldots,T_{N,N}) e^{o(1)}
\end{equation*}
where the $o(1)$ error goes uniformly (with respect to $\{T_{k,i}\}_{1\leq i\leq k\leq N}\in D$) to zero as $\e$ goes to zero.
\end{lemma}
\begin{proof}

We employ the torus scalar product (see Section~\ref{torusSec}) with respect to which the Macdonald polynomials are orthogonal (we keep $t=0$):
\begin{equation*}
\langle f,g\rangle'_N = \int_{\T^N} f(z) \overline{g(z)} m_N^{q}(z) \prod_{i=1}^{N} \frac{ dz_i}{z_i}, \qquad m_N^q(z)  = \frac{1}{(2\pi \iota)^{N} N!}\prod_{i\neq j} (z_iz_j^{-1};q)_{\infty}.
\end{equation*}
Note that taking $t=0$ in equation (\ref{9ex1d}) yields
\begin{equation*}
\langle P_\lambda,P_\lambda\rangle'_N = \prod_{i=1}^{N-1} (q^{\lambda_i-\lambda_{i+1}+1};q)_{\infty}^{-1}.
\end{equation*}
Recalling the definition of $\Pi$ from equation (\ref{PIeqn}) we may write
\begin{equation*}
Q_{\lambda}(\rho) = \frac{1}{\langle P_\lambda,P_\lambda\rangle'_N} \langle \Pi(z_1,\ldots, z_N;\rho),P_{\lambda}(z_{1},\ldots, z_{N})\rangle'_N.
\end{equation*}

Let us break up the study of the asymptotics of $Q$ into a few steps.

{\bf Step 1:} We show that
\begin{equation*}
\langle P_{\lambda^{(N)}},P_{\lambda^{(N)}}\rangle'_N = e^{o(1)},
\end{equation*}
where the $o(1)$ error goes (with respect to $\{T_{k,i}\}_{1\leq i\leq k\leq N}\in D$) to zero as $\e$ goes to zero.

Since
\begin{equation*}
\lambda^{(N)}_{i} - \lambda^{(N)}_{i+1} +1= -2\e^{-1}\log \e +(T_{N,i}-T_{N,i+1})\e^{-1} +1
\end{equation*}
and $q=e^{-\e}$ we find that
\begin{equation*}
q^{\lambda^{(N)}_{i} - \lambda^{(N)}_{i+1} +1}  = \e^{2} e^{T_{N,i+1}-T_{N,i} + \e}.
\end{equation*}
Now consider (taking $b=e^{T_{N,i+1}-T_{N,i} + \e}$)
\begin{equation*}
\log(q^{\lambda^{(N)}_{i} - \lambda^{(N)}_{i+1} +1};q)_{\infty} = \log (\e^2 b;e^{-\e})_{\infty} = \sum_{k=0}^{\infty} \log(1-\e^2 b e^{-k\e}) = -\sum_{k=0}^{\infty} \left(\e^2 b e^{-k\e} + O\left(\e^{4}b^2 e^{-2k\e}\right)\right),
\end{equation*}
where the last approximation is valid uniformly for $b$ bounded (as is our case). Thus, summing the geometric series, we find that
\begin{equation*}
\log(q^{\lambda^{(N)}_{i} - \lambda^{(N)}_{i+1} +1};q)_{\infty} = -\e^{2}b \frac{1}{1-e^{-\e}} -O\left(\e^{4}b^2\frac{1}{1-e^{-2\e}}\right) = -\e b + O(\e^2).
\end{equation*}
Since $\log (\langle P_{\lambda^{(N)}},P_{\lambda^{(N)}}\rangle'_N) = \sum_{i=1}^{N-1} \log(q^{\lambda^{(N)}_{i} - \lambda^{(N)}_{i+1} +1};q)_{\infty}$, the desired result immediately follows.
%This implies that
%\begin{equation}
%\langle P_{\lambda^{(N)}},P_{\lambda^{(N)}}\rangle_N = \exp\left\{\sum_{i=1}^{N-1} \e e^{T_{N,i+1}-T_{N,i} + \e}\right\} = e^{o(1)},
%\end{equation}
%and hence has negligible effect on the asymptotics (again due to the compact interval of $T_{k,j}$ we are considering).

The next three steps deal with the convergence of the integrand of the torus scalar product, and then the convergence of the integral itself. We introduce the change of variables $z_j= \exp\{\iota \e \nu_j\}$.

{\bf Step 2:} Fix a compact subset $V\subset \R^N$. Then
\begin{equation*}
\Pi(z_1,\ldots, z_N;\rho) = E_{\Pi} e^{-\tau \sum_{k=1}^{N} \nu_k^2/2} e^{o(1)}, \qquad E_{\Pi}= e^{\tau N \e^{-2}}  e^{\tau  \e^{-1} \iota \sum_{k=1}^{N} \nu_k},
\end{equation*}
and
\begin{equation*}
P_{\lambda^{(N)}}(z_1,\ldots, z_N) = E_{P}\psi_{\nu_1,\ldots, \nu_N}(T_{N,1},\ldots, T_{N,N}) e^{o(1)},\qquad E_{P}=\e^{-\frac{N(N-1)}{2}} e^{-\frac{(N-1)(N+2)}{2}\A(\e)}e^{\tau \e^{-1}\iota\sum_{k=1}^{N} \nu_k},
\end{equation*}
where the $o(1)$ error goes uniformly (with respect to $\{T_{k,i}\}_{1\leq i\leq k\leq N}\in D$ and $\{\nu_i\}_{1\leq i\leq N}\in V$) to zero as $\e$ goes to zero.

The $\Pi$ asymptotics follows just as in the proof of Lemma~\ref{Pilemma}. For $P_{\lambda^{(N)}}(z_1,\ldots, z_N)$, using the shift property of equation (\ref{macLaurent}) for Laurent Macdonald polynomials we may remove the $\gamma$ factor out of each term of $\lambda^{(N)}$ at the cost of introducing the prefactor of $(z_1\cdots z_N)^{\gamma} = e^{\tau\e^{-1}\iota \sum_{k=1}^{N} \nu_k}$. Theorem~\ref{qwhitconvTHM} applies to $P_{\lambda^{(N)}-(\gamma,\ldots, \gamma)}(z_1,\ldots, z_N)$ and allows us to deduce the uniform estimate above (recall Macdonald symmetric functions with $t=0$ are $q$-Whittaker functions with variable and index switched).

{\bf Step 3:} Fix a compact subset $V\subset \R^N$. Then
\begin{equation}\label{meas}
m_N^{q}(z)\prod_{i=1}^{N} \frac{ dz_i}{z_i} = E_m m_N(\ul{\nu}{N}) \prod_{i=1}^{N} d\nu_i e^{o(1)}, \qquad E_{m}=\e^{N^2}e^{N(N-1)\A(\e)},
\end{equation}
where the $o(1)$ error goes uniformly (with respect to $\{\nu_i\}_{1\leq i\leq N}\in V$) to zero as $\e$ goes to zero.

Set $x=\iota(\nu_j-\nu_i)$, so that  $q^x = z_i z_j^{-1}$. Then recalling the $q$-gamma function defined in equation (\ref{qGamma}), we find that the left-hand side of equation (\ref{meas}) equals
\begin{equation*}
\frac{\e^N}{(2\pi)^N N!} \prod_{i\neq j} \frac{(q;q)_{\infty}  (1-q)^{1-\iota(\nu_j-\nu_i)}}{\Gamma_{q}(\iota(\nu_j-\nu_i))} \prod_{i=1}^{N} d\nu_i.
\end{equation*}
Note that the term $\iota(\nu_j-\nu_i)$ can be removed from the exponent of $(1-q)$ due to cancelation over the full product $i\neq j$. Also note that $(q;q)_{\infty} = e^{\A(\e)} e^{o(1)}$. Lastly, recall that $\Gamma_q$ converges to $\Gamma$ as $\e\to 0$ uniformly on compact sets bounded from $\{0,-1,\ldots\}$. Combine these observations we recover the right-hand side of equation (\ref{meas}).

{\bf Step 4:} We have
\begin{equation}\label{PiPscalareqn}
\langle \Pi(z_1,\ldots, z_N;\rho),P_{\lambda}(z_{1},\ldots, z_{N})\rangle'_N = \left(e^{\frac{(N-1)(N-2)}{2}\A(\e)} e^{\tau N \e^{-2}} \e^{\frac{N(N+1)}{2}}\right)\theta_{\tau}(T_{N,1},\ldots,T_{N,N}) e^{o(1)},
\end{equation}
where the $o(1)$ error goes uniformly (with respect to $\{T_{k,i}\}_{1\leq i\leq k\leq N}\in D$) to zero as $\e$ goes to zero.

Steps 2 and 3 readily show that the integrand of the scalar product on the left-hand side of equation (\ref{PiPscalareqn}) converges to the integrand of $\theta_{\tau}$ on the right-hand side times $E_{\Pi}\overline{E_{P}} E_m$ (recall that we must take the complex conjugate of $P_{\lambda}$ and hence also $E_{P}$). These three terms combine to equal the prefactor of $\theta_{\tau}$ above. This convergence, however, is only on compact sets of the integrand variables $\nu$. In order to conclude convergence of the integrals we must show suitable tail decay for large $\nu$.

In particular, we need to show that if we define
\begin{equation*}
V_M = \{(z_1,\ldots, z_N): \forall 1\leq k\leq N, z_k=e^{\iota\e \nu_k} \textrm{ and } |\nu_k|>M\},
\end{equation*} then
\begin{equation}\label{DMzero}
\lim_{M\rightarrow \infty} \lim_{\e\rightarrow 0}\int_{\nu\in V_{M}} \frac{\Pi(z_1,\ldots, z_N;\rho)}{E_{\Pi}} \frac{\overline{P_{\lambda^{(N)}}(z_1,\ldots, z_N)}}{\overline{E_{P}}} \frac{m_N^q(z)}{E_m} \prod_{i=1}^{N} \frac{ dz_i}{z_i}=0
\end{equation}
with the limits uniform with respect to $\{T_{k,i}\}_{1\leq i\leq k\leq N}\in D$.

First consider $\left|\frac{\Pi(z_1,\ldots, z_N;\rho)}{E_{\Pi}}\right|$. For $x\in [-\pi,\pi]$, $\cos(x)-1 \leq -x^2/6$ (this is not sharp but will do). This shows that
\begin{equation*}
\Real(e^{\iota \e \nu_k}-1) = \cos(\e \nu_k) -1 \leq -\e^2 \nu_k^2/6
\end{equation*}
for $\nu_k\in [-\e^{-1}\pi,\e^{-1}\pi]$ (which are the bounds of integration in $\nu_k$). This implies that for all $\nu_k\in [-\e^{-1}\pi,\e^{-1}\pi]$,
\begin{equation}\label{decaynueqn}
\left|\frac{\Pi(z_1,\ldots, z_N;\rho)}{E_{\Pi}}\right| \leq e^{-\tau \sum_{k=1}^{N} \nu_k^2/6}.
\end{equation}

Turning now to $\left|\frac{\overline{P_{\lambda^{(N)}}(z_1,\ldots, z_N)}}{\overline{E_{P}}}\right|$, we see that by applying the first part of Theorem~\ref{qwhitconvTHM}, this term can be bounded by a constant, uniformly with respect to $\{T_{k,i}\}_{1\leq i\leq k\leq N}\in D$.

Finally, consider
\begin{equation*}
\left|\frac{m_N^q(z)}{E_m} \prod_{i=1}^{N} \frac{ dz_i}{z_i}\right| = \left(\prod_{i\neq j} \e^{-1}(1-q)e^{-\A(\e)}(q;q)_{\infty}\right)\frac{1}{(2\pi)^N N!}\prod_{i\neq j} \frac{1}{\Gamma_{q}(\iota(\nu_j-\nu_i))} \prod_{i=1}^{N} d\nu_i.
\end{equation*}
The prefactors are independent of $\nu$ and we have already determined that they go to 1 suitably as $\e$ goes to zero. Thus we can focus entirely on the term $\frac{1}{\Gamma_{q}(\iota(\nu_j-\nu_i))}$. For this we use the fact that there exist constant $c,c'$ such that for all $a\in[-2\pi \e^{-1},2\pi \e^{-1}]$,
\begin{equation}\label{365}
\left|\frac{1}{\Gamma_{q}(\iota a)}\right| \leq c' e^{c|a|}.
\end{equation}
From this inequality, it follows that on the whole domain of integration for the $\nu_k\in [-\e^{-1}\pi,\e^{-1}\pi]$ we may bound this part of the integrand by
\begin{equation*}
\left|\frac{1}{\Gamma_{q}(\iota(\nu_j-\nu_i))}\right| \leq c' e^{c|\nu_j-\nu_i|}.
\end{equation*}

This growth bound can be compared to the decay bound of equation (\ref{decaynueqn}). There exists an $M_0>0$ and $c''>0$ such that for all $M>M_0$, for $\nu\in D_M$
\begin{equation*}
-\tau\sum_{k=1}^{N} \nu_k^2/6 + \sum_{i\neq j} c|\nu_j-\nu_i| \leq -c''\sum_{k=1}^{N} \nu_k^2.
\end{equation*}
Therefore, on this domain, we have been able to bound our integrand by a constant times $e^{-c'\sum_{k=1}^{N} \nu_k^2}$. Moreover, this bound is uniform in $\e$ small enough and in $\{T_{k,i}\}_{1\leq i\leq k\leq N}\in D$. Thus we can take $\e$ to zero and we are left with an integral which has Gaussian decay in the $\nu_k$. Therefore, as $M\to \infty$, the integrand goes to zero as desired, thus proving equation (\ref{DMzero}) and completing this step of the proof and hence the proof of Lemma~\ref{Qlemma}.
\end{proof}

Having proved Lemmas \ref{Skewlemma}, \ref{Pilemma} and \ref{Qlemma} we can combine the three factors along with a Jacobian factor of $\e^{-\frac{N(N+1)}{2}}$ which comes from our rescaling of the $q$-Whittaker process. Combining these factors, and noting both the cancelation of the prefactors and uniformity of the convergence of each term, we find our desired limit and thus conclude the proof of Theorem~\ref{theorem26}.
\end{proof}

\subsubsection{Whittaker 2d-growth model}\label{contlimit2dgrowth}

\begin{definition}
Fix $N\geq 1$, and $a=(a_1,\ldots, a_N)\in \R^N$. The {\it Whittaker 2d-growth model}\index{Whittaker 2d-growth model} with drift $a$ is a continuous time Markov diffusion process $T(\tau)=\{T_{k,j}(\tau)\}_{1\leq j\leq k\leq N}$ \glossary{$T(\tau)$} with state space $\R^{N(N+1)/2}$ and $\tau$ representing time, which is given by the following system of stochastic (ordinary) differential equations: Let $\{W_{k,j}\}_{1\leq j\leq k \leq N}$ be a collection of independent standard one-dimensional Brownian motions. The evolution of $T$ is defined recursively by $dT_{1,1} = dW_{1,1}+a_1 dt$ and for $k=2,\ldots, N$,
\begin{eqnarray*}
dT_{k,1} &=& d W_{k,1} + \left(a_k+ e^{T_{k-1,1}-T_{k,1}}\right)dt\\
dT_{k,2} &=& dW_{k,2} + \left(a_k+ e^{T_{k-1,2}-T_{k,2}} - e^{T_{k,2}-T_{k-1,1}}\right)dt\\
&\vdots&\\
dT_{k,k-1} &=& dW_{k,k-1} + \left(a_k+ e^{T_{k-1,k-1}-T_{k,k-1}} - e^{T_{k,k-1}-T_{k-1,k-2}}\right)dt\\
dT_{k,k} &=& dW_{k,k} + \left(a_k- e^{T_{k,k}-T_{k-1,k-1}}\right)dt.
\end{eqnarray*}
The Whittaker 2d-growth model coincides with ``symmetric dynamics'' of \cite{OCon} Section 9.
\end{definition}

\begin{theorem}\label{thm7.26}
Fix $N\geq 1$, and $a=(a_1,\ldots, a_N)\in \R^N$. Consider the $q$-Whittaker 2d-growth model with drifts $(A_1,\ldots, A_N)$ (Definition~\ref{2dgrowth}) initialized with $\lambda^{(k)}_{j}=0$ for all $1\leq j\leq k\leq N$. Let $\lambda^{(k)}_j(\tau)$ represent the value at time $\tau\geq 0$ of $\lambda^{(k)}_j$. For $\e>0$ set
\begin{equation*}
q=e^{-\e},\qquad A_{k} = e^{-\e a_k}, \space 1\leq k\leq N, \qquad \lambda^{(k)}_j(\tau) = \tau \e^{-2} -(k+1-2j)\e^{-1} \log \e +T_{k,j}(\tau) \e^{-1}, \space 1\leq j\leq k \leq N.
\end{equation*}
Then, for each $0< \delta \leq t$, $T(\cdot) = \{T_{k,j}(\cdot)\}_{1\leq j\leq k\leq N}$ converges in the topology of $C([\delta,t],\R^{N(N+1)/2})$ (i.e., continuous trajectories in the space of triangular arrays $\R^{N(N+1)/2}$ with the uniform topology) to the Whittaker 2d-growth model with drift $(-a_1,\ldots,-a_N)$ and the entrance law \index{Whittaker 2d-growth model!entrance law} for $\{T_{k,j}(\delta)\}_{1\le j\le k\le N}$ given by the density $\W{a;\delta}\left(\{T_{k,j}\}_{1\le j\le k\le N}\right)$.
\end{theorem}

\begin{proof}
The entrance law follows immediately from the weak convergence of the $q$-Whittaker process to the Whittaker process (i.e., the single time convergence of the 2d-growth processes) which is given by Theorem~\ref{theorem26}. The convergence of the dynamics of the Markov processes follows by induction on the level $k$ and standard methods of stochastic analysis.
\end{proof}

\begin{remark}\label{Whittaker2dgrowthprojections}
Given the Whittaker 2d-growth model dynamics $T(\tau)$ with drift $(-a_1,\ldots,-a_N)$ there are two projections which are themselves (i.e., with respect to their own filtration) Markov diffusion processes. It is easy to see that $\{T_{k,k}\}_{1\leq k\leq N}$ evolves autonomously of the other coordinates. This is the continuous space limit of the $q$-TASEP process.

A less obvious second case is proved in Section 9 of \cite{OCon}: The projection onto $\{T_{N,j}(\tau)\}_{1\leq j\leq N}$ is a Markov diffusion process with entrance law given by the density $\WM{a;\tau}\left(\{T_{N,j}\}_{1\leq j\leq N}\right)$, i.e., for $x=(x_1,\ldots,x_N)$
\begin{equation*}
\WM{a;\tau}(x) = e^{-\frac{\tau}{2} \sum_{j=1}^{N} a_j^2} \psi_{\iota a}(x) \int_{\R^N} \overline{\psi_{\lambda}(x)}e^{-\frac{\tau}{2} \sum_{j=1}^{N} \lambda_j^2} m_N(\lambda)d\lambda.
\end{equation*}
The infinitesimal generator for the diffusion is given by
\begin{equation}\label{semigpW2d}
\mathcal{L} = \tfrac{1}{2} \psi_{\iota a}^{-1}\left(H-\sum_{j=1}^{N}a_j^2\right) \psi_{\iota a},
\end{equation}
\glossary{$\mathcal{L}$}
where $H$ is the {\it quantum $\mathfrak{gl}_{n}$-Toda lattice Hamiltonian}\index{quantum $\mathfrak{gl}_{n}$-Toda lattice Hamiltonian}
\begin{equation*}
H= \Delta - 2 \sum_{j=1}^{N-1} e^{x_{i+1}-x_{i}}.
\end{equation*}
\glossary{$H$}
\end{remark}

\subsection{Integral formulas for the Whittaker process}\label{whitintform}
By applying the scalings of Theorem~\ref{theorem26} we may deduce integral formulas for the Whittaker measure which are analogous to those of Propositions \ref{qsumlambdaprop}, \ref{prop8tzero}, \ref{prop8FULLtzero} and \ref{qlaplaceintegralform}.

The limiting form of Proposition~\ref{qsumlambdaprop} is the following:
\begin{proposition}\label{Proposition27}

For any $1\le r\le N$,
\begin{eqnarray*}
\lefteqn{\left\langle e^{-(T_{N,N}+T_{N,N-1}+\cdots+T_{N,N-r+1})}\right\rangle_{\WM{a_1,\dots,a_N;\tau}} =}\\
&&\frac{(-1)^{\frac{r(r+1)}2+Nr}}{(2\pi \iota)^rr!} e^{r\tau/2} \oint\cdots\oint \prod_{1\le k<l\le r} (w_k-w_l)^2\prod_{j=1}^r\left(\prod_{m=1}^N \frac{1}{w_j+a_m}\right)
{e^{-\tau w_j}dw_j}\,,
\end{eqnarray*}
where the contours include all the poles of the integrand.
\end{proposition}
%\note{the $ e^{r\tau/2}$ prefactor comes from the second order term in $1-q = -\e + \e^2/2$}

The limiting form of Proposition~\ref{prop8tzero} is the following:
\begin{proposition}\label{Proposition28}
For any $k \ge 1$,
\begin{eqnarray*}
\lefteqn{\left\langle e^{-kT_{N,N}}\right\rangle_{\WM{a_1,\dots,a_N;\tau}} =}  \\
&&\frac{(-1)^{k(N-1)}}{(2\pi \iota)^k} e^{k\tau/2}\oint\cdots\oint \prod_{1\le A<B\le k} \frac{w_A-w_B}{w_A-w_B+1}\prod_{j=1}^k \left(\prod_{m=1}^N \frac{1}{w_j+a_m}\right) {e^{-\tau w_j}dw_j}\,,
\end{eqnarray*}
where the $w_j$-contour contains $\{w_{j+1}-1,\cdots,w_k-1,-a_1,\cdots,-a_N\}$ and no other singularities for $j=1,\dots,k$.
\end{proposition}

The limiting form of Proposition~\ref{prop8FULLtzero} can be generalized to the following most general moment expectation:
\begin{proposition}\label{Proposition28FULLtzero}
Fix $k\geq 1$, $N\geq N_1\geq N_2 \cdots \geq N_k$, and $1\leq r_{\alpha}\leq N_{\alpha}$ for $1\leq \alpha\leq k$. Then
\begin{eqnarray*}
\lefteqn{\left\langle \prod_{\alpha=1}^{k} e^{-(T_{N_\alpha,N_\alpha}+\cdots +T_{N_{\alpha},N_{\alpha}-r_{\alpha}+1})}\right\rangle_{\W{a_1,\dots,a_N;\tau}} =}\\
&&\prod_{\alpha=1}^{k} \frac{(-1)^{r_{\alpha}(r_{\alpha}+1)/2 + N_\alpha r_{\alpha}}}{(2\pi \iota)^{r_{\alpha}} r_{\alpha}!} \prod_{\alpha=1}^{k} e^{\tau r_{\alpha}/2} \oint \cdots \oint \prod_{1\leq \alpha <\beta\leq k} \left(\prod_{i=1}^{r_{\alpha}}\prod_{j=1}^{r_{\beta}} \frac{w_{\alpha,i}-w_{\beta,j}}{w_{\alpha,i}-w_{\beta,j}+1}\right)\\
 &&\times \prod_{\alpha=1}^{k}\left( \left(\prod_{1\leq i<j\leq r_{\alpha}}(w_{\alpha,i}-w_{\alpha,j})^2\right) \left(\prod_{j=1}^{r_{\alpha}}\frac{ e^{-\tau w_{\alpha,j}}dw_{\alpha,j}}{(w_{\alpha,j}+a_1)\cdots (w_{\alpha,j}+a_{N_\alpha})}\right)\right)\,,
\end{eqnarray*}
where the $w_{\alpha,j}$-contour contains $\{w_{\beta,i}-1\}$ for all $i\in \{1,\ldots,r_{\beta}\}$ and $\beta>\alpha$, as well as $\{-a_1,\ldots,-a_N\}$ and no other singularities.
\end{proposition}

\begin{remark}
Since the Whittaker measure is supported on all of $\R^N$, it is clear that the above formulas can not be proved from the weak convergence result of Theorem~\ref{theorem26}. One could attempt to strengthen that result so as to imply the above formulas via a limiting procedure. One can alternatively derive the formulas by developing an analogous (degeneration) theory for Whittaker processes as we have done for Macdonald symmetric functions. After stating and proving the limit of Proposition~\ref{qlaplaceintegralform} momentarily, we will use this approach to prove Proposition~\ref{Proposition28FULLtzero} (and hence also \ref{Proposition27} and \ref{Proposition28}). Another rigorous alternative to prove the above results is to pursue the method of replicas. In Part \ref{replicassec} we develop this replica approach which, among other things, gives an independent check of the formula of Proposition~\ref{Proposition28} above.
\end{remark}

The methods we develop below in Section~\ref{sec7.5} allow us to prove integral formulas for joint exponential moments at different times that do not have obvious Macdonald analogs. We will just give one of those here.

\begin{proposition}\label{ktimeintform}
Fix $N\geq 1$ and $a=(a_1,\ldots, a_N)\in \R^N$. Consider the Whittaker 2d-growth model with drift $-a$, denoted by $T(\tau)$ for $\tau\geq 0$. For times $\tau_1\leq\tau_2\leq\cdots\leq\tau_k$,
\begin{eqnarray*}
&&\left\langle e^{-\sum_{i=1}^{k} T_{N,N}(\tau_i)} \right\rangle =\\
&&\frac{(-1)^{k(N-1)} e^{\sum_{i=1}^{k} \tau_i/2}}{(2\pi \iota)^k} \oint\cdots\oint \prod_{1\leq A<B\leq k} \frac{w_{A}-w_{B}}{w_{A}-w_{B}+1} \prod_{m=1}^{N} \prod_{j=1}^k \frac{1}{w_j+a_m} \prod_{j=1}^{k} e^{-\tau_j w_j}dw_j,
\end{eqnarray*}
where the $w_{A}$ contour contains only the poles at $\{w_B-1\}$ for $B>A$ as well as $\{-a_1,\ldots, -a_N\}$, and where the expectation is with respect to the growth model evolution.
\end{proposition}

\begin{remark}
The proofs of Propositions \ref{Proposition28FULLtzero} and \ref{ktimeintform} use spectral representations for relevant Markov kernels that were pointed out to us by O'Connell.
\end{remark}

The following Proposition is likely to be a limiting form of Proposition~\ref{qlaplaceintegralform}. It is a generalization of \cite{OCon} Corollary 4.2, and follows in a similar manner.
\begin{proposition}[\cite{OCon} Corollary 4.2]\label{Proposition29}
For any $u\in \C$ with $\Real(u)>0$ we have
\begin{equation*}
\left\langle e^{-ue^{-T_{N,N}}}\right\rangle_{\WM{a_1,\dots,a_N;\tau}} = \int \cdots \int \prod_{i,j=1}^{N} \Gamma(a_i+\iota \nu_j) \prod_{j=1}^{N} u^{-(a_j+\iota\nu_j)} e^{-\tfrac{\tau}{2}(\nu_j^2+a_j^2)} m_N(\nu)\prod_{j=1}^{N} d\nu_j\,,
\end{equation*}
where the integral is taken over lines parallel to the real axis with $\Imag(\nu_j)<\min_{i=1}^{N}(a_i)$.
\end{proposition}

\subsection{Whittaker difference operators and integral formulas}\label{sec7.5}

The Macdonald operator of rank $r$ with $t=0$ is given by (see Section~\ref{difopSEC})
\begin{equation*}
\sum_{\substack{I\subset\{1,\ldots,n\}\\|I|=r}} \prod_{\substack{i\in I\\j\notin I}} \frac{z_j}{z_j-z_i} \prod_{i\in I} T_{q,z_i}P_{\mu}(z_1,\ldots,z_n) = q^{\mu_n+\cdots +\mu_{n-r+1}} P_{\mu}(z_1,\ldots,z_n).
\end{equation*}

Take the limit as in Definition~\ref{scaleDef} -- i.e.,
\begin{equation*}
q=e^{-\e},\qquad z_k = e^{\iota \e \nu_k},\qquad \mu_{k} = (n+1-2k)m(\e) + \e^{-1} x_{k},\qquad m(\e) = -\left[\e^{-1}\log \e\right].
\end{equation*}
Under this scaling
\begin{eqnarray*}
\prod_{\substack{i\in I\\j\notin I}} \frac{z_j}{z_j-z_i}  &=& \prod_{\substack{i\in I\\j\notin I}} \frac{e^{\iota \e \nu_j}}{e^{\iota \e \nu_j}-e^{\iota \e \nu_i}} \sim (-\iota \e^{-1})^{r(n-r)} \prod_{\substack{i\in I\\j\notin I}} \frac{1}{\nu_j-\nu_i},\\
q^{\mu_n+\cdots +\mu_{n-r+1}} &=& e^{-\e\left((1+3+\cdots + (2r-1) - nr)\e^{-1}\log\e^{-1} + \e^{-1}(x_n+\cdots+x_{n-r+1})\right)} \sim \frac{e^{-x_n-\cdots-x_{n-r+1}}}{\e^{r(n-r)}},
\end{eqnarray*}
and
\begin{equation*}
P_{\mu}(x) \sim  \psi_{\nu}(x),\qquad T_{q,z_k}P_{\mu}(x) \sim \psi_{\nu_1,\ldots, \nu_k+\iota,\ldots, \nu_n}(x).
\end{equation*}

The conclusion is that
\begin{lemma}[\cite{KL}, 5.3(c)]
We have
\begin{equation}\label{Whitdiffeqnunintegerate}
\sum_{\substack{I\subset\{1,\ldots,n\}\\|I|=r}} \prod_{\substack{i\in I\\j\notin I}} \frac{-\iota}{\nu_j-\nu_i}\psi_{\nu_1,\ldots, \nu_k+\iota,\ldots, \nu_n}(x)=e^{-x_n-\cdots-x_{n-r+1}}\psi_{\nu}(x).
\end{equation}
\end{lemma}

By orthogonality of Whittaker functions, the following identity holds:
\begin{equation}\label{Whitdiffeqn}
\int_{\R^n} \psi_{\nu}(x)\overline{\psi_{\lambda}(x)} e^{-x_n-\cdots - x_{n-r+1}} dx = \frac{1}{n! m_n(\lambda)} \sum_{\substack{I\subset\{1,\ldots,n\}\\|I|=r}} \prod_{\substack{i\in I\\j\notin I}} \frac{\iota}{\nu_i-\nu_j} \sum_{\sigma\in S_n} \delta\big(\lambda - \sigma(\nu+\iota e_I)\big ),
\end{equation}
where $e_I = (\bfone_{k\in I})_{k=1,\ldots,n}$ is the vector with ones in the slots of label $I$ and zeros otherwise. This identity should be understood in a weak sense, in that it holds when both sides are integrated against $m_{n}(\lambda)W(\lambda)$ for $W(\lambda)$ such that
\begin{equation*}
\int_{\R^n} \overline{\psi_{\lambda}(x)} W(\lambda) m_n(\lambda) d\lambda
\end{equation*}
decays super-exponentially in $x$ and $W(\lambda)$ is analytic in $\lambda$ in at least a neighborhood of radius 1 around $\R^n$.

\begin{remark}
Note that if for $m$ fixed, $\lambda = \nu + \iota e_m$ then
\begin{eqnarray*}
\frac{1}{m_n(\lambda)} &=& (2\pi)^n n! \prod_{j\neq k} \Gamma(\iota\lambda_k - \iota \lambda_j) = (2\pi)^n n! (-1)^{n-1}\prod_{j\neq k} \Gamma(\iota \nu_k - \iota \nu_j) \prod_{\ell \neq m} \frac{\nu_m-\nu_\ell}{\nu_m -\nu_\ell +\iota}\\
&=& \frac{1}{m_n(\nu)} (-1)^{n-1} \prod_{\ell \neq m} \frac{\nu_m-\nu_\ell}{\nu_m -\nu_\ell +\iota} = \frac{1}{m_n(\nu)} (-1)^{n-1} \prod_{\ell \neq m} \frac{\nu_m-\nu_\ell}{\lambda_m -\lambda_\ell}.
\end{eqnarray*}
This equality agrees with the fact that if we conjugate both sides of (\ref{Whitdiffeqn}) then $\lambda$ and $\nu$ switch places (note that $m_n(\lambda)$ and $m_n(\nu)$ are real if we assume that $\lambda,\nu\in \R^n$).
\end{remark}

Recall the Baxter operator (\ref{baxterQ}) $\BaxOp{N}{N-1}{\kappa}$ which was defined via its integral kernel
\begin{equation*}
\BaxKer{N}{N-1}{\kappa}{x}{y}= e^{\iota \kappa \left(\sum_{i=1}^{N} x_i - \sum_{i=1}^{N-1} y_i\right) -\sum_{k=1}^{N} \left(e^{y_i-x_i} + e^{x_{i+1}-y_i}\right)}
\end{equation*}
where $x\in \R^N$ and $y\in \R^{N-1}$. Set
\begin{equation*}
\BaxOp{N}{N-k}{\kappa_N,\ldots,\kappa_{N-k+1}}= \BaxOp{N}{N-1}{\kappa_N}\circ \BaxOp{N-1}{N-2}{\kappa_{N-1}}\circ \cdots \circ \BaxOp{N-k+1}{N-k}{\kappa_{N-k+1}}.
\end{equation*}

Also, recall that Givental's integral representation of the Whittaker functions leads to the recursive formula
\begin{equation*}
\int_{\R^{N-1}}\BaxKer{N}{N-1}{\kappa}{x}{y} \psi_{\lambda}(y_1,\ldots,y_{N-1})dy = \psi_{(\lambda,\kappa)}(x_1,\ldots,x_N)
\end{equation*}
where $\lambda=(\lambda_1,\ldots,\lambda_{N-1})$ and $(\lambda,\kappa) = (\lambda_1,\ldots, \lambda_{N-1},\kappa)$. By using the orthogonality of Whittaker functions we obtain from this
\begin{equation*}
\BaxKer{N}{N-1}{\kappa}{x}{y} = \int_{\R^{N-1}} \psi_{(b,\kappa)}(x_1,\ldots,X_N)\overline{\psi_{b}(y_1,\ldots,y_{N-1})}m_{N-1}(b)db.
\end{equation*}
Further utilizing the orthogonality we compute the convolutions as
\begin{equation}\label{spectralBaxKer}
\BaxKer{N}{N-k}{\kappa_N,\ldots,\kappa_{N-k+1}}{x}{y}=\int_{\R^{N-k}} \psi_{(b,\kappa_{N-k+1},\ldots,\kappa_N)}(x_1,\ldots,x_N)\overline{\psi_{b}(y_1,\ldots,y_{N-k})}m_{N-k}(b)db,
\end{equation}
where $x\in \R^N$ and $y\in \R^{N-k}$. Also set $Q^{m\to m}(x,y) = \delta(x-y)$ for any $m\geq 1$.

\begin{proof}[Proof of Proposition~\ref{Proposition28FULLtzero}]
Recall that by Theorem~\ref{thm7.26}, $T(\tau)$ has an entrance law $\W{a;\tau}$. The density of the marginal distribution of the Whittaker process $\W{a;\tau}(\{x^k_j\}_{1\leq j\leq k \leq N})$ restricted to levels $x^{N_1},\ldots, x^{N_k}$ (where $N\geq N_1\geq N_2\geq \cdots \geq N_k$ and $x^{N_k}=(x^{N_k}_1,\ldots,x^{N_k}_{N_k})$) has the form

\begin{eqnarray*}
&&\left(e^{-\frac{\tau}{2}\sum_{j=1}^{N_1}a_j^2}\psi_{\iota a^{N_1}}(x^{N_1}_1,\ldots, x^{N_1}_{N_1}) \int_{\R^{N_1}} \overline{\psi_{\lambda}(x^{N_1})} e^{-\frac{\tau}{2} \sum_{j=1}^{N_1} \lambda_j^2} m_{N}(\lambda) d\lambda\right) \\
\nonumber &&\times \frac{\psi_{\iota a^{N_2}}(x^{N_2})}{\psi_{\iota a^{N_1}}(x^{N_1})} \BaxKer{N_1}{N_2}{\iota a_{N_1},\ldots, \iota a_{N_2+1}}{x^{N_1}}{x^{N_2}} \cdots \frac{\psi_{\iota a^{N_k}}(x^{N_k})}{\psi_{\iota a^{N_{k-1}}}(x^{N_{k-1}})} \BaxKer{N_{k-1}}{N_k}{\iota a_{N_{k-1}},\ldots, \iota a_{N_k+1}}{x^{N_{k-1}}}{x^{N_k}}
\end{eqnarray*}
where $a^M=(a_1,\ldots, a_M)$. Note that the term above in the large parentheses is just $\WM{a^{N_1};\tau}$. By canceling $\psi_{\iota a^j}$ factors and using the spectral representation (\ref{spectralBaxKer}) for the Baxter operators, we may rewrite this expression as
\begin{eqnarray*}
&& \left(e^{-\frac{\tau}{2} \sum_{j=1}^{N_1} a_{j}^2} \int_{\lambda^{N_1}\in \R^{N_1}} \overline{\psi_{\lambda^{N_1}}(x^{N_1})} e^{-\frac{\tau}{2}\sum_{j=1}^{N_1}(\lambda^{N_1}_j)^2} m_{N_1}(\lambda^{N_1}) d\lambda^{N_1}\right)\\
&&\times \int_{\lambda^{N_2}\in \R^{N_2}} \psi_{\lambda^{N_2},\iota a_{N_2+1},\ldots, \iota a_{N_1}}(x^{N_1}) \overline{\psi_{\lambda^{N_2}}(x^{N_2})} m_{N_2}(\lambda^{N_2}) d\lambda^{N_2} \\
&&\times \cdots
\int_{\lambda^{N_k}\in \R^{N_k}} \psi_{\lambda^{N_k},\iota a_{N_k+1},\ldots, \iota a_{N_k}}(x^{N_{k-1}}) \overline{\psi_{\lambda^{N_k}}(x^{N_k})} m_{N_k}(\lambda^{N_k}) d\lambda^{N_k} \cdot \psi_{\iota a_1,\ldots, \iota a_{N_k}}(x^{N_k}).
\end{eqnarray*}

Let us multiply this expression by
\begin{equation*}
e^{-(x^{N_1}_{N_1} + \cdots + x^{N_1}_{N_1-r_1+1})} \cdots e^{-(x^{N_k}_{N_k} + \cdots + x^{N_k}_{N_k-r_k+1})}
\end{equation*}
and integrate over all the $x$-variables using (\ref{Whitdiffeqn}). We start the integration with $x^{N_k}$, then go to $x^{N_{k-1}}$ and so on. The decay estimate of Proposition~\ref{supexpthetadecay} is needed at this place to make sure that the integrals converge and that we may apply the identity (\ref{Whitdiffeqn}). We obtain
\begin{eqnarray}\label{sumaboveint}
&&e^{-\frac{\tau}{2} ||a^{N_1}||^2} \sum_{\substack{I_1\subset\{1,\ldots, N_1\}\\|I_1|=r_1}} \cdots \sum_{\substack{I_k\subset\{1,\ldots, N_k\}\\|I_k|=r_k}} e^{\frac{\tau}{2} ||a^{N_1} + e_{I_1}+e_{I_2}+\cdots +e_{I_k}||^2}\\
\nonumber &&\times \prod_{\substack{i_1\in I_1\\j_1\in \{1,\ldots, N_1\}\setminus I_1}} \frac{1}{(a^{N_1}+e_{I_2}+\cdots+e_{I_k})_{i_1} -(a^{N_1}+e_{I_2}+\cdots+e_{I_k})_{j_1}} \cdots \prod_{\substack{i_k\in I_k\\j_k\in \{1,\ldots, N_k\}\setminus I_k}} \frac{1}{a_{i_k}-a_{j_k}}.
\end{eqnarray}
Above, the subscript around the parentheses as in $(a^{N_1}+e_{I_2}+\cdots+e_{I_k})_{i_1}$ refers to the $i_1$ entry of the vector in the parentheses.
For a sequence $b=\{b_1,\ldots,b_n\}$ and a subset $J$ of its indices denote
\begin{equation*}
|b|_{J} = \sum_{j\in J} b_j.
\end{equation*}
Then we may write
\begin{eqnarray*}
&&||a^{N_1} + e_{I_1} + \cdots + e_{I_k}||^2 - ||a^{N_1}||^2 = \\
&&|I_1|+\cdots + |I_k| + 2\left(|a|_{I_k} + |a+e_{I_k}|_{I_{k-1}} + |a+e_{I_k}+e_{I_{k-1}}|_{I_{k-2}} + \cdots + |a+e_{I_k}+e_{I_{2}}|_{I_{1}}\right).
\end{eqnarray*}
Using the above we may rewrite (\ref{sumaboveint}) as
\begin{eqnarray}\label{sumaboveint2}
&&e^{\frac{\tau}{2}(r_1+\cdots +r_k)} \sum_{\substack{I_k\subset \{1,\ldots,N_k\}\\|I_k|=r_k}} \prod_{\substack{i_k\in I_k\\j_k \in \{1,\ldots, N_k\}\setminus I_k}} \frac{1}{a_{i_k}-a_{j_k}} e^{\tau|a|_{I_k}}\\
\nonumber &&\times \sum_{\substack{I_{k-1}\subset \{1,\ldots,N_{k-1}\}\\|I_{k-1}|=r_{k-1}}} \prod_{\substack{i_{k-1}\in I_{k-1}\\j_{k-1} \in \{1,\ldots, N_{k-1}\}\setminus I_{k-1}}} \frac{1}{(a^{N_k}+e_{I_k})_{i_{k-1}}-(a^{N_k}+e_{I_k})_{j_{k-1}}} e^{\tau|a+e_{I_k}|_{I_{k-1}}}\\
\nonumber &&\times\cdots
\sum_{\substack{I_{1}\subset \{1,\ldots,N_{1}\}\\|I_{1}|=r_{1}}} \prod_{\substack{i_{1}\in I_{1}\\j_{1} \in \{1,\ldots, N_{1}\}\setminus I_{1}}} \frac{1}{(a^{N_1}+e_{I_2}+\cdots +e_{I_k})_{i_{1}}-(a^{N_1}+e_{I_2}+\cdots + e_{I_k})_{j_{1}}} e^{\tau|a+e_{I_2}+\cdots + e_{I_k}|_{I_{1}}}
\end{eqnarray}
The interior sums over $I_{j}$ have the exact same structure as the sum over $I_k$ but with $k$ replaced by $j$ and the sequence $(a_1,\ldots, a_{N_1})=a^{N_1}$ replaced by $a^{N_1}+e_{I_{j+1}}+\cdots +e_{I_k}$.

\begin{lemma}
Let $f(u)$ be an analytic function in a sufficiently large neighborhood of $x_1,\ldots, x_n\in \C$. Then
\begin{eqnarray*}
&&\sum_{\substack{I\subset \{1,\ldots, n\}\\|I|=r}} \prod_{\substack{i\in I\\j\in \{1,\ldots, n\}\setminus I}} \frac{1}{x_i-x_j}\prod_{k\in I} f(x_k)
=\\
&&\frac{(-1)^{r(r-1)/2}}{(2\pi \iota)^r r!} \oint\cdots \oint \prod_{1\leq k<\ell\leq r} (w_k-w_\ell)^2 \prod_{j=1}^{r}\left(\prod_{m=1}^{n} \frac{1}{w_j-x_m}\right) f(w_j) dw_j,
\end{eqnarray*}
where the contours for the $w_j$ are such that they only include the poles $w_j=x_m$ for all $j,m$.
\end{lemma}
\begin{proof}
This follows by a straightforward computation of residues. If different $w_j$'s pick the same pole $x_m$, the contribution is zero because of the Vandermonde factor. There are $r!$ ways to match $\{w_1,\ldots, w_r\}$ to $\{x_i\}_{i\in I}$, thus the $1/r!$.
\end{proof}

Using this lemma we can compute the inner most sum over $I_1$ in (\ref{sumaboveint2}) taking $n=N_1$, $x~=~a^{N_1}+e_{I_2}+\cdots+e_{I_k}$, and $f(u) = e^{\frac{\tau}{2}u}$. This gives
\begin{equation*}
\frac{(-1)^{r_1(r_1-1)/2}}{(2\pi \iota)^{r_1} r_1!} \oint\cdots \oint \prod_{1\leq k_1<\ell_1\leq r_1} (w^{(1)}_{k_1}-w^{(1)}_{\ell_1})^2 \prod_{j_1=1}^{r_1}\left(\prod_{m_1=1}^{N_1} \frac{1}{w^{(1)}_{j_1}-x_{m_1}}\right) e^{\tau w_{j_1}^{(1)}} dw_{j_1}^{(1)}.
\end{equation*}
We now want to do the same summation over $I_2$ with $\tilde n= N_2$, $\tilde x = a^{N_2}+e_{I_3}+\cdots + e_{I_k}$. Note that
\begin{equation*}
\prod_{m=1}^{N_2} \frac{1}{w-x_m} = \prod_{m\notin I_2} \frac{1}{w-\tilde x_m} \prod_{m\in I_2}\frac{1}{w-\tilde x_m - 1} = \prod_{m=1}^{N_2} \frac{1}{w-\tilde x_m} \prod_{m\in I_2} \frac{w-\tilde x_m}{w-\tilde x_m-1}.
\end{equation*}
Thus we take
\begin{equation*}
\tilde f(u) = \prod_{j=1}^{r_1} \frac{w^{(1)}_{j_1}-u}{w^{(1)}_{j_1} - u - 1} e^{\frac{\tau}{2} u}
\end{equation*}
and obtain
\begin{eqnarray*}
&&\frac{(-1)^{r_1(r_1-1)/2 + r_2(r_2-1)/2}}{(2\pi \iota)^{r_1+r_2} r_1!r_2!} \oint\cdots \oint \prod_{1\leq k_1<\ell_1\leq r_1} (w^{(1)}_{k_1}-w^{(1)}_{\ell_1})^2 \prod_{1\leq k_2<\ell_2\leq r_2} (w^{(2)}_{k_1}-w^{(2)}_{\ell_1})^2\\
&& \times \prod_{j_1=1}^{r_1}\left(\prod_{m_1=N_2+1}^{N_1} \frac{1}{w^{(1)}_{j_1}-a_{m_1}} \prod_{m_1=1}^{N_2} \frac{1}{w^{(1)}_{j_1} - \tilde x_{m_1}} \right) e^{\tau w_{j_1}^{(1)}}\\
&&\times \prod_{j_2=1}^{r_2}\left(\prod_{m_2=1}^{N_2} \frac{1}{w^{(2)}_{j_2}-\tilde x_{m_2}}\right) \prod_{j_1=1}^{r_1} \frac{w^{(1)}_{j_1}-w^{(2)}_{j_2}}{w^{(1)}_{j_1} -w^{(2)}_{j_2}-1} e^{\tau w_{j_2}^{(2)}} dw_{j_1}^{(1)}dw_{j_2}^{(2)},
\end{eqnarray*}
where the $w^{(2)}$ contours contain only $\tilde x_{m_2}$ as poles, while $w^{(1)}$ contours also contain the poles at points given by $w^{(2)}+1$. Iterating this procedure shows that (\ref{sumaboveint2}) can be rewritten as
\begin{eqnarray*}
&&\prod_{\alpha=1}^{k} \frac{(-1)^{r_\alpha(r_\alpha-1)/2} e^{\frac{\tau}{2}r_{\alpha}}}{(2\pi \iota)^{r_\alpha} r_\alpha!} \oint\cdots \oint
\prod_{1\leq \alpha<\beta\leq k} \left(\prod_{i=1}^{r_{\alpha}}\prod_{j=1}^{r_{\beta}} \frac{w^{(\alpha)}_{i}-w^{(\beta)}_{j}}{w^{(\alpha)}_{i}-w^{(\beta)}_{j}-1}\right)\\
\nonumber &&\times \prod_{\alpha=1}^k\left(\left(\prod_{1\leq i<j\leq r_\alpha} (w^{(\alpha)}_{i}-w^{(\alpha)}_{j})^2 \right)
\left(\prod_{j=1}^{r_\alpha}\frac{1}{(w^{(\alpha)}-a_1)\cdots (w^{(\alpha)}-a_{N_\alpha})} e^{\tau w^{(\alpha)}_{j}} dw^{(\alpha)}_j\right)\right),
\end{eqnarray*}
where the $w^{(\alpha)}_j$ contour contains the poles $\{a_1,\ldots, a_{N_{\alpha}}\}$ and also the poles $\{w^{(\beta)}_{i}+1\}$ for all $i\in~\{1,\ldots~r_{\beta}\}$ and $\beta>\alpha$.
Proposition~\ref{Proposition28FULLtzero} follows by taking $w^{(\alpha)}_i = -w_{\alpha,i}$.

\end{proof}

\begin{proof}[Proof of Proposition~\ref{ktimeintform}]
The evolution over time $\tau$ is given by the semi-group (recall Definition (\ref{semigpW2d}))
\begin{equation*}
e^{\tau\mathcal{L}} = e^{-\frac{\tau}{2}\sum_{j=1}^{N} a_j^2} \psi_{\iota a}^{-1} e^{\frac{1}{2} \tau H} \psi_{\iota a}.
\end{equation*}
Since $H$ is diagonalized by $\{\psi_{b}\}_{b\in \R^N}$ with eigenvalues given by $H\psi_{b} = (-\sum_{j=1}^{N} b_j^2)\psi_{b}$ we can write the kernel of the semi-group $e^{\frac{1}{2}tH}$ as
\begin{equation*}
e^{\frac{1}{2}\tau H}(x,y) = \int_{\R^N} e^{-\frac{\tau}{2} \sum_{j=1}^{N} b_j^2} \psi_{b}(x) \overline{\psi_b(y)} m_N(b) db.
\end{equation*}

We will focus on the case $k=2$ (i.e., the two time case) as the general case follows similarly. The above considerations show that the two-time distribution for our diffusion is
\begin{equation*}
e^{-\frac{\tau_1}{2}\sum_{j=1}^N a_j^2}e^{-\frac{\tau_2-\tau_1}{2}\sum_{j=1}^N a_j^2}\int_{\R^N} \overline{\psi_{\lambda}(x)} e^{-\frac{\tau_1}{2} \sum_{j=1}^N \lambda_j^2} m_N(\lambda) d\lambda
\int_{\R^N} e^{-\frac{\tau_2-\tau_1}{2}\sum_{j=1}^{N} b_j^2} \psi_b(x) \overline{\psi_{b}(y)} m_N(b) db \psi_{\iota a}(y).
\end{equation*}

We multiply this expression by $e^{-x_N-y_N}$ and integrate over $x,y\in \R^N$. Doing this amounts to computing $\left\langle e^{-\sum_{i=1}^{2} T_{N,N}(\tau_i)} \right\rangle$ as desired. Using (\ref{Whitdiffeqn}) we obtain, by first integrating over $y$ then over $x$ (the convergence of the integrals is justified by the decay of Proposition~\ref{supexpthetadecay}), that $\left\langle e^{-\sum_{i=1}^{2} T_{N,N}(\tau_i)} \right\rangle$ equals
\begin{eqnarray*}
 && e^{-\frac{\tau_2}{2} \sum_{j=1}^{N} a_j^2} \int_{\R^N} \overline{\psi_{\lambda}(x)} e^{-\frac{\tau_1}{2} \sum_{j=1}^{N} \lambda_j^2} m_N(\lambda)d\lambda \sum_{k=1}^{N} \prod_{j\neq k}^N \frac{1}{a_k-a_j} \psi_{\iota (a+e_k)}(x) e^{\frac{\tau_2-\tau_1}{2}\sum_{j=1}^{N} (a_j+\delta_{jk})^2}\\
&=& e^{-\frac{\tau_2}{2}}\sum_{k,\ell=1}^{N} \prod_{m\neq \ell}^{N} \frac{1}{a_\ell + \delta_{k\ell} -(a_m+\delta_{km})} \prod_{j\neq k}^{N} \frac{1}{a_k -a_j} e^{\frac{\tau_2}{2} \sum_{j=1}^{N} (a_j+\delta_{jk}+\delta_{j\ell})^2} e^{\frac{\tau_2-\tau_1}{2} \sum_{j=1}^{N} (a_j+\delta_{jk})^2}\\
&=& e^{\frac{\tau_1+\tau_2}{2}} \sum_{k,\ell=1}^{N} \prod_{m\neq \ell}^{N} \frac{1}{a_\ell + \delta_{k\ell} -(a_m+\delta_{km})} \prod_{j\neq k}^{N} \frac{1}{a_k -a_j} e^{\tau_1(a_{\ell} + \delta_{k\ell})} e^{\tau_2 a_k}\\
&=& e^{\frac{\tau_1+\tau_2}{2}} \oint \oint \frac{w_1-w_2}{w_1-w_2-1} \prod_{m=1}^{N} \frac{1}{(w_1-a_m)(w_2-a_m)} e^{\tau_1 w_1 + \tau_2 w_2} \frac{dw_1 dw_2}{(2\pi \iota)^2},
\end{eqnarray*}
where the $w_2$ contour contains $\{a_1,\ldots, a_N\}$ and the $w_1$ contour contains $\{w_2+1,a_1,\ldots,a_N\}$.
\end{proof}

\subsection{Fredholm determinant formulas for the Whittaker process}\label{WhitFredDetSec}

\subsubsection{Convergence lemmas}
We introduce two probability lemmas. The first will be useful when we perform asymptotics on the Whittaker process Fredholm determinant, whereas the second (proved similarly) will be useful presently.

\begin{lemma}\label{problemma1}
Consider a sequence of functions $\{f_n\}_{n\geq 1}$ mapping $\R\to [0,1]$ such that for each $n$, $f_n(x)$ is strictly decreasing in $x$ with a limit of $1$ at $x=-\infty$ and $0$ at $x=\infty$, and for each $\delta>0$, on $\R\setminus[-\delta,\delta]$ $f_n$ converges uniformly to $\bfone(x\leq 0)$. Define the $r$-shift of $f_n$ as $f^r_n(x) = f_n(x-r)$. Consider a sequence of random variables $X_n$ such that for each $r\in \R$,
\begin{equation*}
\EE[f^r_n(X_n)] \to p(r)
\end{equation*}
and assume that $p(r)$ is a continuous probability distribution function. Then $X_n$ converges weakly in distribution to a random variable $X$ which is distributed according to $\PP(X\leq r) = p(r)$.
\end{lemma}
\begin{proof}
Consider $s<t<u$. By the conditions on $f_n$ it follows that for all $\e>0$ there exists $n_0>0$ such that for all $n>n_0$
\begin{equation*}
\EE\big[f^s_n(X_n)\big] - \e \leq \PP(X_n \leq t) \leq \EE\big[f^u_n(X_n)\big] +\e.
\end{equation*}
The above fact follows from the uniform convergence outside of any fixed interval of the origin.
By the convergence of $\EE[f^s_n(X_n)]\to p(s)$ as $n\to \infty$, it follows further that there also exists $n_1>n_0$ such that for all $n>n_1$
\begin{equation*}
p(s) - 2\e \leq \PP(X_n \leq t) \leq p(u) +2\e.
\end{equation*}
This implies that
\begin{equation*}
p(s) - 2\e \leq \liminf_{n\to\infty} \PP(X_n \leq t)\leq \limsup_{n\to\infty} \PP(X_n \leq t)  \leq p(u) +2\e.
\end{equation*}
The above is established for arbitrary $\e$ and thus
\begin{equation*}
p(s) \leq \liminf_{n\to\infty} \PP(X_n \leq t)\leq \limsup_{n\to\infty} \PP(X_n \leq t)  \leq p(u).
\end{equation*}
The above is also established for arbitrary $s<t<u$ and by taking $s$ and $u$ to $t$ it follows that
\begin{equation*}
\lim_{s\to t^-}p(s) \leq \liminf_{n\to\infty} \PP(X_n \leq t)\leq \limsup_{n\to\infty} \PP(X_n \leq t)  \leq \lim_{u\to t^+}p(u).
\end{equation*}
Since $p(\cdot)$ is continuous it follows that the left and right bounds are identical, hence
\begin{equation*}
\lim_{n\to\infty} \PP(X_n \leq t)=p(t),
\end{equation*}
which is exactly as desired.
\end{proof}

\begin{lemma}\label{problemma2}
Consider a sequence of functions $\{f_n\}_{n\geq 1}$ mapping $\R\to [0,1]$ such that for each $n$, $f_n(x)$ is strictly decreasing in $x$ with a limit of $1$ at $x=-\infty$ and $0$ at $x=\infty$, and $f_n$ converges uniformly on $\R$ to $f$. Consider a sequence of random variables $X_n$ converging weakly in distribution to $X$.
Then
\begin{equation*}
\EE[f_n(X_n)] \to \EE[f(X)].
\end{equation*}
\end{lemma}
\begin{proof}
Follows from a similar sandwiching approach as above.
\end{proof}

\subsubsection{Asymptotics of the q-Whittaker process Fredholm determinant formulas}
%\begin{comment}
%\note{State this right now for $A_i$ near 1 , though the general result can be similarly stated if we are a little careful about the $s$ contour.}
%\end{comment}
\begin{theorem}\label{NeilPolymerFredDetThm}
Fix $0<\delta_2<1$, and $\delta_1<\delta_2 /2$; also fix $a_1,\ldots, a_N$ such that $|a_i|<\delta_1$. Then for all $u\in \C$ such that $\Real(u)>0$
\begin{equation*}
\left\langle e^{-u e^{-T_{N,N}}}  \right\rangle_{\WM{a_1,\ldots,a_N;\tau}} = \det(I+ K_{u})_{L^2(C_a)}
\end{equation*}
%where $\det(I+ K_{u})$ is the Fredholm determinant of
%\begin{equation*}
%K_{u}: L^2(C_a)\to L^2(C_a)
%\end{equation*}
where $C_a$ is a positively oriented contour containing $a_1,\ldots, a_N$ and such that for all $v,v'\in C_a$, $|v-v'|~<~\delta_2$. The operator $K_u$ is defined in terms of its integral kernel
\begin{equation}\label{kvvprime}
K_{u}(v,v') = \frac{1}{2\pi \iota}\int_{-\iota \infty + \delta_2}^{\iota \infty +\delta_2}ds \Gamma(-s)\Gamma(1+s) \prod_{m=1}^{N}\frac{\Gamma(v-a_m)}{\Gamma(s+v-a_m)} \frac{ u^s e^{v\tau s+\tau s^2/2}}{v+s-v'}.
\end{equation}
\end{theorem}
\begin{proof}
The proof splits into two pieces. Step 1: We prove that the left-hand side of equation (\ref{thmlaplaceeqn}) of Theorem~\ref{PlancherelfredThm}  converges to $\left\langle e^{-u e^{-T_{N,N}}}  \right\rangle_{\WM{a_1,\ldots,a_N;\tau}}$. This relies on combining Theorem~\ref{theorem26} (which provides weak convergence of the $q$-Whittaker process to the Whittaker process) with Lemma~\ref{problemma2} and the fact that the $q$-Laplace transform converges to the usual Laplace transform. Step 2: We prove that the Fredholm determinant expression coming from the right-hand side of Theorem~\ref{PlancherelfredThm} converges to the Fredholm determinant given in the theorem we are presently proving. This convergence is done via the Fredholm expansion and uniformly controlled term by term asymptotics.

We scale the parameters of Theorem~\ref{PlancherelfredThm} as
\begin{eqnarray*}
&&q =e^{-\e}, \qquad\gamma =\tau\e^{-2}, \qquad A_k=e^{-\e a_k}, 1\leq k\leq N\\
&&w= q^v, \qquad \zeta = -\e^{N} e^{\tau \e^{-1}} u, \qquad \lambda_N = \tau \e^{-2} + (N-1) \e^{-1} \log \e + T_N \e^{-1}.
\end{eqnarray*}

{\bf Step 1:}
Rewrite the left-hand side of equation (\ref{thmlaplaceeqn}) in  Theorem~\ref{PlancherelfredThm} as
\begin{equation*}
\left\langle \frac{1}{\left(\zeta q^{\lambda_N};q\right)_{\infty}}\right\rangle_{\MM_{t=0}(A_1,\ldots,A_N;\rho)}  = \left\langle e_q(x_q)\right\rangle_{\MM_{t=0}(A_1,\ldots,A_N;\rho)}
\end{equation*}
where
\begin{equation*}
x_q=(1-q)^{-1}\zeta q^{\lambda_N} =-ue^{-T_N} \e/(1-q)
\end{equation*}
and $e_q(x)$ is as in Section~\ref{classicalqfunctions}. Combine this with the fact that $e_q(x)\to e^{x}$ uniformly on $x\in~(-\infty,0)$ to show that, considered as a function of $T_N$,  $e_q(x_q)\to e^{-u e^{-T_N}}$ uniformly for $T_N\in \R$. By Theorem~\ref{theorem26}, the measure on $T_N$ (induced from the $q$-Whittaker measure on $\lambda_N$) converges weakly in distribution to the Whittaker measure $\WM{a_1,\ldots,a_N;\tau}$. Combining this weak convergence with the uniform convergence of $e_q(x_q)$ and Lemma~\ref{problemma2} gives that
\begin{equation*}
\left\langle \frac{1}{\left(\zeta q^{\lambda_N};q\right)_{\infty}}\right\rangle_{\MM_{t=0}(A_1,\ldots,A_N;\rho)} \to \left\langle e^{-u e^{-T_{N}}}  \right\rangle_{\WM{a_1,\ldots,a_N;\tau}}
\end{equation*}
as $q\to 1$ (i.e., $\e \to 0$).

{\bf Step 2:}
Recall the kernel in the right-hand side of equation (\ref{thmlaplaceeqn}) in  Theorem~\ref{PlancherelfredThm}. It can be rewritten as a Fredholm determinant of a kernel in $v$ and $v'$ (recall $w=q^v$) as follows
\begin{equation*}
K^{q}(v,v') = \frac{1}{2\pi \iota}\int_{-\iota \infty + \delta}^{\iota \infty +\delta}h^q(s) ds,
\end{equation*}
where
\begin{equation*}
h^q(s)=\Gamma(-s)\Gamma(1+s)\left(\frac{-\zeta}{(1-q)^N}\right)^s \frac{q^v \log q}{q^s q^v - q^{v'}} \prod_{m=1}^{N} \frac{\Gamma_q(\log_q(q^v/A_m))}{\Gamma_q(\log_q(q^s q^v/A_m))} \exp\{\gamma q^v(q^{s}-1)\}
\end{equation*}
where the new term $q^v \log q$ came from the Jacobian of changing $w$ to $v$.

Let us first demonstrate the point-wise limit before turning to the necessary uniformity to prove convergence of the Fredholm determinant. Consider the behavior of each term as $q\to 1$ (or equivalently as $\e\to 0$):
\begin{eqnarray*}
e^{-\tau s \e^{-1}}\left(\frac{-\zeta}{(1-q)^N}\right)^s &\to & u^s,\\%removed a minus sign -u^s
\frac{q^v \log q}{q^s q^v - q^{v'}} &\to& \frac{1}{v+s-v'},\\
\frac{\Gamma_q(v-a_m)}{\Gamma_q(s+v-a_m)} &\to& \frac{\Gamma(v-a_m)}{\Gamma(s+v-a_m)},\\
e^{\tau s \e^{-1}} \exp\{\gamma q^v(q^{s}-1)\} &\to& e^{v \tau s + \tau s^2/2}.
\end{eqnarray*}
Combining these point-wise limits together gives the kernel (\ref{kvvprime}). However, in order to prove convergence of the determinants, or equivalently the Fredholm expansion series, one needs more than just point-wise convergence.

In fact, we require three lemmas to complete the proof:

\begin{lemma}\label{uniflemma}
Fix any compact subset $D$ of the interval $\iota\R + \delta$. Then the following convergence as $q\to 1$ is uniform over all $s\in D$ and $v,v'\in C_{a}$:
\begin{equation}\label{uniflemmaeqn}
h^q(s) \to \Gamma(-s)\Gamma(1+s)  \prod_{m=1}^{N}\frac{\Gamma(v-a_m)}{\Gamma(s+v-a_m)} \frac{ u^s e^{v\tau s+\tau s^2/2}}{v+s-v'}.
\end{equation}
\end{lemma}
\begin{proof}
This strengthened version of the above point-wise convergence follows from the uniform convergence of the $\Gamma_q$ function to the $\Gamma$ function on compact regions away from the poles, as well as standard Taylor series estimates.
\end{proof}

\begin{lemma}\label{tailboundslemma}
There exists a constant $C$ such that for all $s\in \iota\R + \delta$, and all $q\in (1/2,1)$
\begin{equation*}
\left| \prod_{m=1}^{N} \frac{1}{\Gamma_q(s+v-a_m)} e^{\tau s \e^{-1}} \exp\{\gamma q^v(q^{s}-1)\} \right| \leq C.
\end{equation*}
\end{lemma}
\begin{proof}
Though not exactly, the left-hand side above is very close to a periodic function of $\Imag(s)$ of fundamental domain $[-\pi \e^{-1},\pi \e^{-1}]$. Observe that, as in equation (\ref{365}),
\begin{equation*}
\left|\frac{1}{\Gamma_q(s+v-a_m)}\right| \leq c e^{c' f^{\e}(s)}
\end{equation*}
where $c,c'$ are positive constants independent of $\e$ (or equivalently $q$) and $f^{\e} = \dist(\Imag(s),2\pi \e^{-1} \Z)$.

The other terms are periodic in $s$ of fundamental domain $[-\pi \e^{-1},\pi \e^{-1}]$. In order to control their absolute value we look to upper-bound the real part of their logarithm and find that for $s\in \iota \R + \delta_2$:
\begin{equation*}
\Real\left(\tau s \e^{-1} +\tau \e^{-2} e^{-\e v} (e^{-\e s}-1) \right) \leq - \Imag(s)^2 / 2.
\end{equation*}
This bound follows from careful Taylor series estimation and the bound that for $x\in [-\pi,\pi]$, $\cos(x)-1\leq -x^2/6$.
These two bounds shows that
\begin{equation*}
\left| \prod_{m=1}^{N} \frac{1}{\Gamma_q(s+v-a_m)} e^{\tau s \e^{-1}} \exp\{\gamma q^v(q^{s}-1)\} \right| \leq c e^{c' f^{\e}(s)} e^{-\Imag(s)^2 / 2}
\end{equation*}
which is easily bounded by a constant for all $s$.
\end{proof}

\begin{lemma}\label{fredtailemma}
For all $q\in (1/2,1)$, $|K^q(v,v')|\leq C$ for a fixed constant $C>0$ and all $v,v'\in C_{a}$.
\end{lemma}
\begin{proof}
Combine the tail decay estimate of Lemma~\ref{tailboundslemma} with the decay in $s$ (uniform over $q\in (1/2,1)$ and  $v,v'\in C_a$) of the other terms of the integrand of $K^{q}(v,v')$. This shows that for any constant $C>0$, there exists another constant $c>0$ such that for all $s$, $|s|>C$ we have $h^q(s)\leq e^{-c|s|}$, uniformly over over $q\in (1/2,1)$ and  $v,v'\in C_a$. By Lemma~\ref{uniflemma}, for $s$ such that $|s|\leq C$, there is uniform convergence to the right-hand side of equation (\ref{uniflemmaeqn}). Since the right-hand side of equation (\ref{uniflemmaeqn}) is uniformly bounded in $s$ and $v,v'\in C_{a}$, it follows that $h^q(s)$ is likewise bounded on the compact set of $s$ such that $|s|\leq C$. Combining this with the decay for large $s$ proves the lemma.
\end{proof}

Now combine the lemmas to finish the proof. By Lemma~\ref{fredtailemma} and Hadamard's bound, in order to proof convergence of the Fredholm expansion of $\det(I+K_{\zeta})$ to $\det(I+K_{u})$, it suffices to consider only a finite number of terms in the expansion (as the contribution of the later terms can be bounded arbitrarily close to zero by going out far enough in the expansion). By the tail bounds of Lemma~\ref{tailboundslemma} along with bounds for the other terms (which are uniform in $q,v,v'$) one shows that the integrals in $s$ variables in the Fredholm expansion can likewise be restricted to compact sets (as the contribution to the integrals from outside these set can be bounded arbitrarily close to zero by choosing large enough compact sets). Finally, restricted to compact sets, the uniform convergence result of Lemma~\ref{uniflemma} implies convergence of the remaining (compactly supported) integrals as $q$ goes to 1. This completes the proof of convergence of the Fredholm determinants.
\end{proof}

%\begin{comment}
%\subsubsection{Large contour formulas}
%It seems like this doesn't rescale that nicely... worth remarking on this.
%\end{comment}
\subsection{Tracy-Widom asymptotics}\label{TWasymsec}

In performing steepest descent analysis on Fredholm determinants, the following proposition allows one to deform contours to descent curves.

\begin{lemma}[Proposition 1 of \cite{TW3}]\label{TWprop1}
Suppose $s\to C_s$ is a deformation of closed curves $C_s$ and a kernel $L(\eta,\eta')$ is analytic in a neighborhood of $C_s\times C_s\subset \C^2$ for each $s$. Then the Fredholm determinant of $L$ acting on $C_s$ is independent of $s$.
\end{lemma}

\begin{definition}\label{digammadef}
The {\it Digamma function}\index{Diagamma function} \glossary{$\Psi(z)$} $\Psi(z) = [\log \Gamma]'(z)$. Define for $\kappa>0$
\begin{equation*}
\bfk = \inf_{t>0} (\kappa t - \Psi(t)),
\end{equation*}\glossary{$\bfk$} \glossary{$\btk$} \glossary{$\bgk$}
and let $\btk$ denote the unique value of $t$ at which the minimum is achieved. Finally, define the positive number (scaling parameter) $\bgk= -\Psi''(\btk)$.
\end{definition}

\begin{theorem}\label{TWasymptoticskappaTHM}
There exists a $\kappa^*>0$ such that for $\kappa>\kappa^*$ and $\tau = \kappa N$
\begin{equation*}
\lim_{N\to \infty} \PP_{\WM{0,\ldots,0;\tau}}\left( \frac{ -T_{N,N}(\tau) - N \bfk}{N^{1/3}}\leq r\right) = F_{{\rm GUE}}\left((\bgk / 2)^{-1/3}r\right).
\end{equation*}
\end{theorem}
\begin{remark}\label{remkappa}
This result should hold for all $\kappa^*>0$. However there are certain technical challenges which must be overcome to get the full set of $\kappa$. At the level of a critical point analysis, the results works equally well for all $\kappa^*>0$. The challenge comes in proving the negligibility of integrals away from the critical point. This is accomplished in \cite{BorCorFer} and the result is shown for all $\kappa^*>0$. The law of large numbers with the constant $\bfk$ was conjectured in \cite{OCon-Yor} and proved in \cite{OConnellMoriarty}. A tight one-sided bound on the fluctuation exponent for $\Fsd^{N}_1(t)$ was determined in \cite{SeppValko} Spohn \cite{SpohnPolymer} has recently described how the above theorem fits into a general physical KPZ scaling theory.
\end{remark}

\begin{proof}[Proof of Theorem~\ref{TWasymptoticskappaTHM}]
The starting point of our proof is the Fredholm determinant formula given in Theorem~\ref{NeilPolymerFredDetThm} for the expectation of $e^{-ue^{-T_{N,N}}}$ over Whittaker measure $\WM{0,\ldots,0;\tau}$. Consider the function $f_N(x) = e^{-e^{N^{1/3}x}}$ and define $f_N^r(x) = f_N(x-r)$. Observe that this sequence of functions meets the criteria of Lemma~\ref{problemma1}. Setting
\begin{equation*}
u=u(N,r,\kappa)=e^{-N\bfk - rN^{1/3}}
\end{equation*}
observe that
\begin{equation*}
e^{-ue^{-T_{N,N}}} = f_N^r\left(\frac{-T_{N,N} - N \bfk}{N^{1/3}}\right).
\end{equation*}
By Lemma~\ref{problemma1}, if for each $r\in \R$ we can prove that
\begin{equation*}
\lim_{N\to \infty} \left\langle f_N^r\left(\frac{-T_{N,N} - N \bfk}{N^{1/3}}\right) \right\rangle_{\WM{0,\ldots,0;\tau}} = p_{\kappa}(r)
\end{equation*}
for $p_{\kappa}(r)$ a continuous probability distribution function, then it will follow that
\begin{equation*}
\lim_{N\to \infty} \PP_{\WM{0,\ldots,0;\tau}}\left( \frac{ -T_{N,N}(\tau) - N \bfk}{N^{1/3}}\leq r\right) = p_{\kappa}(r)
\end{equation*}
as well.

In light of the above observations, it remains for us to analyze the asymptotics of the Fredholm determinant formula given in Theorem~\ref{NeilPolymerFredDetThm} for the case $a_m\equiv 0$.
The result of that theorem is that
\begin{equation*}
\left\langle f_N^r\left(\frac{-T_{N,N} - N \bfk}{N^{1/3}}\right) \right\rangle_{\WM{0,\ldots,0;\tau}}= \det(I+K_{u(N,r,\kappa)})_{L^2(C_0)}
\end{equation*}
%where $\det(I+K_{u(N,r,\kappa)})$ is the Fredholm determinant of
%\begin{equation*}
%K_{u(N,r,\kappa)}:L^2(C_0)\to L^2(C_0)
%\end{equation*}
where $C_0$ is positively oriented contour containing $0$ such that for all $v,v'\in C_{0}$, $|v-v'|<\delta_2$ where $\delta_2$ is any fixed real number in $(0,1)$. The operator $K_{u(N,r,\kappa)}$ is defined in terms of its integral kernel via
\begin{equation}\label{Kunrkappa}
K_{u(N,r,\kappa)}(v,v') = \frac{1}{2\pi \iota}\int_{-\iota \infty + \delta_2}^{\iota \infty +\delta_2}ds \Gamma(-s)\Gamma(1+s) \frac{\Gamma(v)^N}{\Gamma(s+v)^N} \frac{ u^s e^{v\tau s+\tau s^2/2}}{v+s-v'}.
\end{equation}

Let us first provide a critical point derivation of the asymptotics. We will then provide the rigorous proof. Besides standard issues of estimation, we must be very careful about manipulating contours due to a mine-field of poles (this ultimately leads to our present technical limitation that $\kappa>\kappa^*$).

We start by making a change of variables in the integral defining the kernel $K_{u(N,r,\kappa)}(v,v')$. Set $\zeta = s + v$ and rewrite
\begin{equation}\label{zetaKeqn}
K_{u(N,r,\kappa)}(v,v') = \frac{1}{2\pi \iota}\int_{v-\iota \infty + \delta_2}^{v+\iota \infty +\delta_2}\frac{\pi}{\sin (\pi(v-\zeta))}\exp\left\{N \left(G(v)-G(\zeta)\right)+ r N^{1/3} (v-\zeta)\right\}\frac{d\zeta}{\zeta-v'}.
\end{equation}
where
\begin{equation*}
G(z) = \log \Gamma(z) - \kappa z^2/2 + \bfk z
\end{equation*}
and where we have also replaced $\Gamma(-s)\Gamma(1+s) = \pi / \sin(-\pi s)$. The problem is now prime for steepest descent analysis of the integral defining the kernel above.

The idea of steepest descent is to find critical points for the function being exponentiated, and then to deform contours so as to go close to the critical point. The contours should be engineered such that away from the critical point, the real part of the function in the exponential decays and hence as $N$ gets large, has negligible contribution. This then justifies localizing and rescaling the integration around the critical point. The order of the first non-zero derivative (here third order) determines the rescaling in $N$ (here $N^{1/3}$) which in turn corresponds with the scale of the fluctuations in the problem we are solving. It is exactly this third order nature that accounts for the emergence of Airy functions and hence the Tracy Widom (GUE) distribution.

The critical point equation for $G$ is given by $G'(z)=0$ where
\begin{equation*}
G'(z) = \Psi(z) - \kappa z + \bfk.
\end{equation*}
The Digamma function $\Psi(z)$ is given in Definition~\ref{digammadef} as is $\btk$, which clearly satisfies $G'(\btk) = 0$ and hence is a critical point. We will Taylor expand around this point $z=\btk$. Consider the second and third derivatives of $G(z)$:
\begin{eqnarray*}
G''(z) &=& \Psi'(z) - \kappa \\
G'''(z) &=& \Psi''(z).
\end{eqnarray*}
By the definition of $\btk$ as the argmin of $\kappa t - \Psi(t)$, it is clear that $G''(\btk)=0$, and recalling the definition of $\bgk$ we have that $G'''(\btk) = -\bgk$.
This indicates that near $z=\btk$,
\begin{equation*}
G(v) - G(\zeta) = -\frac{\bgk (v-\btk)^3}{6} + \frac{\bgk (\zeta-\btk)^3}{6}
\end{equation*}
plus lower order terms. This cubic behavior suggests rescaling around $\btk$ by the change of variables
\begin{equation*}
\tilde{v} = N^{1/3}(v-\btk), \qquad \tilde{\zeta} = N^{1/3}(\zeta-\btk).
\end{equation*}
Clearly the steepest descent contour for $v$ from $\btk$ departs at an angle of $\pm 2\pi /3$ whereas the contour for $\zeta$ departs at angle $\pm \pi/3$. The $\zeta$ contour must lie to the right of the $v$ contour (so as to avoid the pole from $1/(\zeta-v')$). As $N$ goes to infinity, neglecting the contribution away from the critical point, the point-wise limit of the kernel becomes
\begin{equation}\label{Krkappa}
K_{r,\kappa}(\tilde{v},\tilde{v}') = \frac{1}{2\pi \iota}\int \frac{1}{\tilde{v}-\tilde{\zeta}}\frac{\exp\left\{\tfrac{-\bgk}{6}\tilde{v}^3+ r \tilde{v}\right\}}{\exp\left\{\tfrac{-\bgk}{6}\tilde{\zeta}^3+ r \tilde{\zeta}\right\}}\frac{d\tilde{\zeta}}{\tilde{\zeta}-\tilde{v}'}.
\end{equation}
where the kernel acts on the contour $\pm 2\pi/ 3$ (oriented from negative imaginary part to positive imaginary part) and the integral in $\tilde{\zeta}$ is on the contour $\pm \pi/3 + \delta$ for any $\delta>0$ (likewise oriented). This is owing to the fact that
\begin{equation*}
N^{-1/3}\frac{\pi}{\sin (\pi(v-\zeta))} \to \frac{1}{\tilde{v}-\tilde{\zeta}}, \qquad \frac{d\zeta}{\zeta-v'} \to \frac{d\tilde{\zeta}}{\tilde{\zeta}-\tilde{v}'}.
\end{equation*}
where the $N^{-1/3}$ comes from the Jacobian associated with the change of variables in $v$ and $v'$.

Another change of variables to rescale by $(\tilde{g}_{\kappa}/2)^{1/3}$ results in
\begin{equation*}
K_{r,\kappa}(\tilde{v},\tilde{v}') = \frac{1}{2\pi \iota}\int \frac{1}{\tilde{v}-\tilde{\zeta}}\frac{\exp\left\{-\tilde{v}^3 /3+ (\bgk/2)^{-1/3} r \tilde{v}\right\}}{\exp\left\{-\tilde{\zeta}^3 /3 + (\bgk/2)^{-1/3} r \tilde{\zeta}\right\}}\frac{d\tilde{\zeta}}{\tilde{\zeta}-\tilde{v}'}.
\end{equation*}
Noting that $\Real(\tilde{\zeta}-\tilde{v}')>0$ we can linearize
\begin{equation*}
\frac{1}{\tilde{\zeta}-\tilde{v}'} = \int_{0}^{\infty} e^{-t(\tilde{\zeta}-\tilde{v}')} dt.
\end{equation*}
This allows us to factor the kernel into the composition of three operators $ABC$ and by the cyclic nature of the Fredholm determinant, $\det(I+ABC) = \det(I+CAB)$. This reordering allows us to evaluate the Airy integrations and we find the standard form of the Airy kernel. This shows that the limiting expectation $p_{\kappa}(r)= F_{{\rm GUE}}\left((\bgk / 2)^{-1/3}r\right)$ which shows that it is a continuous probability distribution function and thus Lemma~\ref{problemma1} applies. This completes the critical point derivation.

The challenge now is to rigorously prove that $\det(I+K_{u(N,r,\kappa)})$ converges to $\det(I+K_{r,\kappa})$ where the two operators act on their respective $L^2$ spaces and are defined with respect to the kernels above in (\ref{Kunrkappa}) and (\ref{Krkappa}). In fact, everything done above following this convergence claim is rigorous.

We will consider the case where $\kappa>\kappa^*$ for $\kappa^*$ large enough and perform certain estimates given that assumption. Let us record some useful such estimates:
\begin{eqnarray*}
\btk &=& \frac{1}{\sqrt{\kappa}} + O(\kappa^{-3/2})\\
\bfk &=& 2\sqrt{\kappa} + O(\kappa^{-1/2}).
\end{eqnarray*}
For $z$ close to zero we also have ($\gamma$ is the Euler Mascheroni constant $\sim .577$)
\begin{eqnarray*}
\log\Gamma(z) &=& -\log z -\gamma z + O(z^2)\\
\Psi'''(z) &=& \frac{6}{z^4} + O(1).
\end{eqnarray*}
Now consider replacing $z= \kappa^{-1/2}\tilde{z} $:
\begin{eqnarray}\label{festimate}
G(z)-G(\btk) &=&\log\Gamma(z) - \kappa z^2/2 + \kappa^{1/2} 2z -\log\Gamma(\btk)+\kappa \btk^2/2 - \bfk \btk   +O(\kappa^{-1}) \\
%&=&\log\Gamma(\kappa^{-1/2}\tilde{z}) - \tilde{z}^2/2 + 2\tilde{z} -\log\Gamma(\btk)+\kappa \btk^2/2 - \bfk \btk\\
&=& f(\tilde{z}) + \gamma\kappa^{-1/2}(1-\tilde{z}) + O(\kappa^{-1}),
\end{eqnarray}

where the error is uniform for $\tilde{z}$ in any compact domain and
\begin{equation*}
f(\tilde{z}) = -\log \tilde{z}  - \tilde{z}^2/2 + 2\tilde{z} -3/2.
\end{equation*}

Our approach will be as follows: Step 1: We will deform the contour $C_{0}$ on which $v$ and $v'$ are integrated as well as the contour on which $\zeta$ is integrated so that they both locally follow the steepest descent curve for $G(z)$ coming from the critical point $\btk$ and so that along them there is sufficient global decay to ensure that the problem localizes to the vicinity of $\btk$. Step 2: In order to show this desired localization we will use a Taylor series with remainder estimate in a small ball of radius approximately $1/\sqrt{\kappa}$ around $\btk$, and outside that ball we will use the estimate (\ref{festimate}) for the $v$ and $v'$ contour, and a similarly straight-forward estimate for the $\zeta$ contour. Step 3: Given these estimates we can show convergence of the Fredholm determinants as desired.

\noindent {\bf Step 1:} Define the contour $C_{f}$ which corresponds to a steep (not steepest though) descent contour of the function $f(\tilde{z})$ leaving $\tilde{z}=1$ at an angle of $2\pi/3$ and returning to $\tilde{z}=1$ at an angle of $-2\pi/3$, given a positive orientation. In particular we will take $C_{f}$ to be composed of a line segment from $\tilde{z}=1$ to $\tilde{z}=1 + e^{\iota 2\pi/3}$, then a circular arc (centered at 0) going counter-clockwise until $\tilde{z}=1+e^{-\iota 2\pi/3}$ and then a line segment back to $\tilde{z}=1$. It is elementary to confirm that along $C_f$, $\Real(f)$ achieves its maximal value of 0 at the unique point $\tilde{z}=1$ and is negative everywhere else.

Now we can define the $v$ contour $C_{G,\kappa,c}$ which will serve as our steep descent contour for $G(v)$ (see Figure \ref{contour1}). The contour $C_{G,\kappa,c}$ starts at $v=\btk$ and departs at an angle of $2\pi/3 + O(\kappa^{-1})$ along a straight line of length $c/\sqrt{\kappa}$ where the angle is chosen so that the endpoint of this line-segment touches the scaled contour $\kappa^{-1/2}C_{f}$ given above. The contour then continues along $\kappa^{-1/2}C_{f}$. The contour in the lower half plane is defined via reflecting through the real axis. That the error to the angle is $O(\kappa^{-1})$ is readily confirmed due to the $O(\kappa^{-3/2})$ error between $\btk$ and $1/\sqrt{\kappa}$. The constant $c>0$ will be fixed soon as needed for the estimates.

One should note that for $\kappa$ large, the contour $C_{G,\kappa,c}$ is such that for $v$ and $v'$ along it, $|v-v'|<\delta_2$ where $\delta_2$ can be chosen as any real number in $(0,1)$. By virtue of this fact, we may employ Proposition~\ref{TWprop1} to deform our initial contour $C_{0}$ (given from Theorem~\ref{NeilPolymerFredDetThm}) to the contour $C_{G,\kappa,c}$ without changing the value of the Fredholm determinant.

Likewise, we must deform the contour along which the $\zeta$ integration in (\ref{zetaKeqn}) is performed. Given that $v,v'\in C_{G,\kappa,c}$ now, we may again use Proposition~\ref{TWprop1} to deform the $\zeta$ integration to a contour $C_{\langle,N}$ which is defined symmetrically over the real axis by a ray from $\btk+N^{-1/3}$ leaving at an angle of $\pi /3$ (again, see Figure \ref{contour1}). It is easy to see that this deformation does not pass any poles, and with ease one confirms that due to the quadratic term in $G(\zeta)$ the deformation of the infinite contour is justified by suitable decay bounds near infinity between the original and final contours.

Thus, the outcome for this first step is that the $v,v'$ contour is now given by $C_{G,\kappa,c}$ and the $\zeta$ contour is now given by $C_{\langle,N}$ which is independent of $v$.

\begin{figure}
\begin{center}
\includegraphics[width=.5\textwidth,height=.5\textwidth]{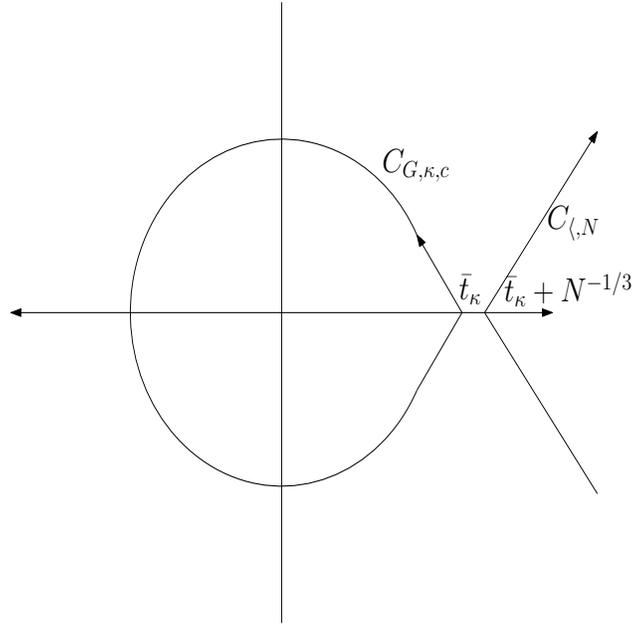}
\end{center}
\caption{Steep descent contours.}\label{contour1}
\end{figure}

\noindent {\bf Step 2:} We will presently provide two types of estimates for our function $G$ along the specified contours: those valid in a small ball around the critical point and those valid outside the small ball. Let us first focus on the small ball estimate.

\begin{lemma}
There exists a constant $c$ and $\kappa^*>0$ such that for all $\kappa>\kappa^*$ the following two facts hold:

\noindent(1) There exists a constant $c'>0$ such that for all $v$ along the line segments of length $c/\sqrt{\kappa}$ in the contour $C_{G,\kappa,c}$:
\begin{equation*}
\Real\left[G(\zeta)-G(\btk)\right] \leq \Real\left[ -c'\kappa^{3/2}(v-\btk)^3 \right].
\end{equation*}

\noindent(2)  There exists a constant $c'>0$ such that for all $\zeta$ along the line segments of length $c/\sqrt{\kappa}$ in the contour $C_{\langle,N}$:
\begin{equation*}
\Real\left[G(v)-G(\btk)\right] \geq \Real\left[ -c'\kappa^{3/2}(\zeta-\btk)^3 \right].
\end{equation*}
\end{lemma}
\begin{proof}
Recall the Taylor expansion remainder estimate for a function $F(z)$ expanded around $z^*$,
\begin{eqnarray*}
&&\left| F(z) - \left( F(z^*) + F'(z^*)(z-z^*)+ \tfrac{1}{2} F''(z^*)(z-z^*)^2 + \tfrac{1}{6} F'''(z^*)(z-z^*)^3 \right)\right| \\
&&\leq \max_{w\in B(z^*,|z-z^*|)} \tfrac{1}{24}|F''''(w)||z-z^*|^4.
\end{eqnarray*}
We may apply this to our function $G(z)$ around the point $\btk$ giving
\begin{equation*}
\left| G(z) - G(\btk) + \tfrac{1}{6} \bgk(z-\btk)^3\right| \leq \max_{w\in B(\btk,|z-\btk|)} \tfrac{1}{24}|G''''(w)||z-\btk|^4.
\end{equation*}
As before let $z=\kappa^{-1/2}\tilde{z}$ and also let $w=\kappa^{-1/2}\tilde{w}$. Note that for $\kappa$ large,
\begin{equation*}
\bgk = 2\kappa^{3/2} + O(\kappa^{1/2}), \qquad G''''(w) = \Psi'''(w) = \frac{6\kappa^2}{\tilde{w}^4} + O(1).
\end{equation*}
This implies the key estimate
\begin{equation*}
\left| G(z) - G(\btk) + \tfrac{1}{3}(\tilde{z}-1)^3\right| \leq \frac{1}{4} \left|\frac{\tilde{z}-1}{\tilde{z}}\right|^4 + O(\kappa^{-1}).
\end{equation*}
Both parts of the lemma follow readily from this estimate.
\end{proof}

We may now turn to the estimate outside the ball of size $c/\sqrt{\kappa}$.
\begin{lemma}
For every constant $c\in (0,1)$ there exists $\kappa^*>0$ and $c'>0$ such that for all $v$ along the circular part of the contour $C_{G,\kappa,c}$, the following holds:
\begin{equation*}
\Real\left[G(v)-G(\btk)\right] \leq -c'.
\end{equation*}
\end{lemma}
\begin{proof}
Writing $v=\kappa^{-1/2}\tilde{v}$ we may appeal to the estimate (\ref{festimate}) and the fact that along the circular part of the contour $C_{G,\kappa,c}$, the function $f(\tilde v)$ is strictly negative and the error in the estimate is $O(\kappa^{-1/2})$.
\end{proof}

The above bound suffices for the $v$ contour since it is finite. However, the $\zeta$ contour is infinite so our estimate must be sufficient to ensure that the contour's tails do not play a role.
\begin{lemma}
For every constant $c\in (0,1)$ there exists  $\kappa^*>0$ and  $c'>0$ such that for all $\zeta$ along the contour $C_{\langle,N}$ of distance exceeding $c/\sqrt{\kappa}$ from $\btk$, the following holds:
\begin{equation*}
\Real\left[G(\zeta)-G(\btk)\right] \geq \Real\left[-c'\kappa \zeta^2/2\right].
\end{equation*}
\end{lemma}
\begin{proof}
This estimate is best established in three parts. We first estimate for $\zeta$ between distance $c/\sqrt{\kappa}$ and distance $c_1/\sqrt{\kappa}$ from $\btk$ ($c_1$ large). Second we estimate for $\zeta$ between distance $c_1/\sqrt{\kappa}$ and distance $c_2$ from $\btk$. Finally we estimate for all $\zeta$ yet further. This third estimate is immediate from the first line of (\ref{festimate}) in which the quadratic term in $z$ clearly exceeds the other terms for $\kappa$ large enough and $|z|>c_2$.

To make the first estimate we use the bottom line of (\ref{festimate}). The function $f(\tilde{z})$ has growth along this contour sufficient to overwhelm the terms $\gamma \kappa^{-1/2}(1-\tilde{z}) + O(\kappa^{-1})$ as long as $\kappa$ is large enough. The $O(\kappa^{-1})$ error in (\ref{festimate}) is only valid for $\tilde{z}$ in a compact domain though. So for the second estimate we must use the cruder bound that $G(z)-G(\btk) = f(\tilde{z}) +O(1)$ for $z$ along the contour and of size less than $c_2$ from $\btk$. Since in this regime of $z$, $f(\tilde{z})$ behaves like $-\kappa z^2/2$, one sees that for $\kappa$ large enough, this $O(1)$ error is overwhelmed and the claimed estimate follows.
\end{proof}

\noindent{\bf Step 3:} We now employ the estimates given above to conclude that $\det(I+K_{u(N,r,\kappa)})$ converges to $\det(I+K_{r,\kappa})$ where the two operators act on their respective $L^2$ spaces and are defined with respect to the kernels above in (\ref{Kunrkappa}) and (\ref{Krkappa}). The approach is standard so we just briefly review what is done. Convergence can either be shown at the level of Fredholm series expansion or trace-class convergence of the operators. Focusing on the Fredholm series expansion, the estimates provided in Step 2, along with Hadamard's bound show that for any $\e$, there is a $k$ large enough such that for all $N$,  the contribution of the terms of the Fredholm series expansion past index $k$ can be bounded by $\e$. This localizes the problem of asymptotics to a finite number of integrals involving the kernels. The estimates of Step 2 then show that these integrals can be localized in a large window of size $N^{-1/3}$ around the critical point $\btk$. The cost of throwing away the portion of the integrals outside this window can be uniformly (in $N$) bounded by $\e$, assuming the window is large enough. Finally, after a change of variables to rescale this window by $N^{1/3}$ we can use the Taylor series with remainder to show that as $N$ goes to infinity, the integrals coming from the kernel $K_{u(N,r,\kappa)}$ converge to those coming from $K_{r,\kappa}$. This last step is essentially the content of the critical point computation given earlier.
\end{proof}

\section{The $\alpha$-Whittaker process limit}\label{alphawhitprocessSec}

\subsection{Weak convergence to the $\alpha$-Whittaker process}

We start by defining the $\alpha$-Whittaker process and measure as introduced and first studied in \cite{COSZ} in relation to the image of the log-gamma discrete directed polymer under the tropical Robinson-Schensted-Knuth correspondence (see Section~\ref{logsec}).

\begin{definition}\label{alphadefs}
Fix integers $n\geq N\geq 1$ and vectors $\alpha = (\alpha_1,\ldots, \alpha_n)$ and $a=(a_1,\ldots, a_N)$ such that $\alpha_i>0$ for all $1\leq i\leq n$ and $\alpha_i+a_j>0$ for all combinations of $1\leq i\leq n$  and $1\leq j\leq N$. Define the {\it $\alpha$-Whittaker process}\index{$\alpha$-Whittaker process} as a probability measure on $\R^{\frac{N(N+1)}2}$ with the density function given by
\begin{equation*}
\alphaW{a;\alpha,n}\left(\{T_{k,j}\}_{1\le j\le k\le N}\right)=\prod_{i=1}^n \prod_{k=1}^{N}\frac{1}{\Gamma(\walpha_i + a_j)}\exp({\mathcal{F}_{\iota a}(T)}) \theta_{\walpha,n}(T_{N,1},\ldots, T_{N,N})
\end{equation*}
\glossary{$\alphaW{a;\alpha,n}$}\glossary{$\theta_{\walpha,n}(\ul{x}{N})$}
where
\begin{equation*}
\theta_{\walpha,n}(\ul{x}{N})=\int_{\R^N} \overline{\psi_{\ul{\nu}{N}}(\ul{x}{N})}\prod_{i=1}^n \prod_{k=1}^{N} \Gamma(\walpha_i - \iota \nu_{k}) m_N(\ul{\nu}{N})d\ul{\nu}{N}.
\end{equation*}

\begin{remark}
As explained in Remark \ref{diffformofmeasures}, the difference in appearance between this measure and the Macdonald measure is due to the use of the torus product representation for the Macdonald $Q$ function. However, since we are dealing with finite pure-alpha specializations for the $Q$ function, it should be possible to write the $\alpha$-Whittaker process in terms of the product of two Whittaker functions. Such a formula when $N=n$ can be found in \cite{COSZ}.
\end{remark}

The nonnegativity of the density is not obvious but it does follow from Theorem~\ref{alphalimitthm} below. To see that the total integral is equal to 1, one first integrates over the variables $T_{k,i}$ with $k<N$, which yields the following
{\it $\alpha$-Whittaker measure}\index{$\alpha$-Whittaker measure}\glossary{$\alphaWM{a;\alpha,n}$}
\begin{equation*}
\alphaWM{a;\alpha,n}\left(\{T_{N,i}\}_{1\le i\le N}\right)= \prod_{i=1}^n \prod_{k=1}^{N}\frac{1}{\Gamma(\walpha_i + a_j)}\psi_{\iota a}(T_{N,1},\ldots, T_{N,N}) \theta_{\walpha,n}(T_{N,1},\ldots, T_{N,N})
\end{equation*}
and then over the remaining variables, using the orthogonality of the Whittaker functions.

\begin{remark}\label{starstar83}
In order to make use of the orthogonality we need to justify the analytic continuation of this relation off the real axis. As in Section~\ref{whitprobmeas} this requires strong decay estimates for $\theta_{\alpha,n}$ which presently we do not make.
\end{remark}
\end{definition}

\begin{theorem}[modulo decay estimates]\label{alphalimitthm}
In the limit regime
\begin{eqnarray*}
&&q=e^{-\epsilon},\qquad \malpha_i=e^{-\walpha_i \e},\  1\le i\le n, \qquad A_k=e^{-\epsilon a_k}, \  1\le k\le N, \\
&&\lambda^{(k)}_j=n\e^{-1}\log\e^{-1} + (k+1-2j)\e^{-1}\log\e^{-1}+T_{k,j}\epsilon^{-1}, \quad 1\le j\le k\le N,
\end{eqnarray*}
the $q$-Whittaker process $\M_{asc,t=0}(A_1,\dots,A_N;\rho)$ with pure alpha specialization
$\rho$ determined by $\alpha=(\alpha_1,\ldots,\alpha_n)$ (and $\gamma=0$, $\beta_i=0$ for all $i\geq 0$) as in equation (\ref{tag1}), weakly converges to the $\alpha$-Whittaker process
$\alphaW{a;\alpha,n}\bigl(\{T_{k,i}\}_{1\le i\le k\le N}\bigr)$.
\end{theorem}
\begin{proof}[Proof, modulo decay estimates]
The proof of this result follows along the same lines as the proof of Theorem~\ref{theorem26}. As such, we provide the main steps only, and forgo repeating arguments when they remain unchanged.

As before, recall that by definition
\begin{equation*}
\M_{asc,t=0}(A_1,\ldots, A_N;\rho)(\lambda^{(1)},\ldots,\lambda^{(N)})= \frac{P_{\lambda^{(1)}}(A_1)P_{\lambda^{(2)}/\lambda^{(1)}}(A_2)\cdots P_{\lambda^{(N)}/\lambda^{(N-1)}}(A_N) Q_{\lambda^{(N)}}(\rho)}{\Pi(A_1,\ldots, A_N;\rho)},
\end{equation*}
where $t=0$ in the Macdonald symmetric functions.

It suffices to control the convergence for $\{T_{k,i}\}$ varying in any given compact set -- thus for the rest of the proof fix a compact set for the $\{T_{k,i}\}$. Again we provide three lemmas which combine to the result. Recall also the notation $\A(\e) = -\frac{\pi^2}{6} \frac{1}{\e} -\frac{1}{2}\log\frac{\e}{2\pi}$.

Let us start by considering the skew Macdonald polynomials with $t=0$.
\begin{lemma}\label{Plemmaalpha}
Fix any compact subset $D\subset \R^{N(N+1)/2}$. Then
\begin{equation}\label{skewPprodeqnalpha}
P_{\lambda^{(1)}}(A_1)P_{\lambda^{(2)}/\lambda^{(1)}}(A_2)\cdots P_{\lambda^{(N)}/\lambda^{(N-1)}}(A_N) = e^{-\frac{(N-1)(N-2)}{2}\A(\e)}\e^{n\sum_{k=1}^{N}a_k} \mathcal{F}_{\iota a}(T) e^{o(1)}
\end{equation}
where the $o(1)$ error goes uniformly (with respect to $\{T_{k,i}\}_{1\leq i\leq k\leq N}\in D$) to zero as $\e$ goes to zero.
\end{lemma}
\begin{proof}
This is proved exactly as in Lemma~\ref{Skewlemma} by making the replacement $\tau= \e n \log\e^{-1}$.
\end{proof}

\begin{lemma}\label{Pilemmaalpha}
We have
\begin{equation*}
\Pi(A_1,\ldots,A_N;\rho) = \prod_{i=1}^n \prod_{k=1}^{N}\frac{\Gamma(\walpha_i + a_j)}{e^{\A(\e)}\e^{1-\walpha_i - a_k}} e^{o(1)}
\end{equation*}
where the $o(1)$ error goes to zero as $\e$ goes to zero.
\end{lemma}
\begin{proof}
Note that
\begin{equation*}
\Pi(A_1,\ldots,A_N;\rho) = \prod_{k=1}^{N} \prod_{i=1}^{n} \frac{1}{(\malpha_i A_k;q)_{\infty}}.
\end{equation*}
From the definition of the $q$-Gamma function and its convergence to the usual Gamma function as $q\to 1$, we have
\begin{equation}\label{PiqGammaeqn}
\frac{1}{(q^x;q)} = \frac{\Gamma(x)}{e^{\A(\e)} \e^{1-x}} e^{o(1)}
\end{equation}
where the $o(1)$ term goes to zero as $\e$ goes to zero. Recalling the scalings we have, and using this convergence, the lemma follows.
\end{proof}

\begin{lemma}[modulo decay estimate in Step 4]\label{Qlemmaalpha}
Fix any compact subset $D\subset \R^{N(N+1)/2}$. Then
\begin{equation*}
Q_{\lambda^{(N)}}(\rho) = \left(e^{\frac{(N-1)(N-2)}{2}\A(\e)} \e^{\frac{N(N+1)}{2}}\prod_{i=1}^{n}\prod_{k=1}^{N}\frac{1}{\e^{1-\alpha_i}}\right)\theta_{\walpha,n}(T_{N,1},\ldots,T_{N,N}) e^{o(1)}
\end{equation*}
where the $o(1)$ error goes uniformly (with respect to $\{T_{k,i}\}_{1\leq i\leq k\leq N}\in D$) to zero as $\e$ goes to zero.
\end{lemma}
\begin{proof}

We employ the torus scalar product, see Section~\ref{torusSec}, with respect to which the Macdonald polynomials are orthogonal (we keep $t=0$):
\begin{equation*}
\langle f,g\rangle'_N = \int_{\T^N} f(z) \overline{g(z)} m_N^{q}(z)\prod_{i=1}^{N} \frac{ dz_i}{z_i} , \qquad m_N^q(z) = \frac{1}{(2\pi \iota)^{N} N!}\prod_{i\neq j} (z_iz_j^{-1};q)_{\infty}.
\end{equation*}
Recalling the definition of $\Pi$ from equation (\ref{PIeqn}) we may write
\begin{equation*}
Q_{\lambda}(\rho) = \frac{1}{\langle P_\lambda,P_\lambda\rangle'_N} \langle \Pi(z_1,\ldots, z_N;\rho),P_{\lambda}(z_{1},\ldots, z_{N})\rangle'_N.
\end{equation*}

Let us break up the study of the asymptotics of $Q$ into a few steps. The first three steps follow easily, while the fourth requires a strong decay estimate which we presently do not give.

{\bf Step 1:} Just as in Lemma~\ref{Qlemma}
\begin{equation*}
\langle P_{\lambda^{(N)}},P_{\lambda^{(N)}}\rangle'_N = e^{o(1)},
\end{equation*}
where the $o(1)$ error goes to zero as $\e$ goes to zero.

The next three steps deal with the convergence of the integrand of the torus scalar product, and then the convergence of the integral itself. We introduce the change of variables $z_j= \exp\{\iota \e \nu_j\}$.

{\bf Step 2:} Fix a compact subset $V\subset \R^N$. Then
\begin{equation*}
\Pi(z_1,\ldots, z_N;\rho) = E_{\Pi} \prod_{i=1}^{n}\prod_{k=1}^{N} \Gamma(\walpha_i - \iota\nu_k)e^{o(1)}, \qquad E_{\Pi} = \prod_{i=1}^{n}\prod_{k=1}^{N}\frac{1}{e^{\A(\e)} \e^{1-\walpha_i + \iota\nu_j}},
\end{equation*}
and
\begin{equation*}
P_{\lambda^{(N)}}(z_1,\ldots, z_N) =E_{P}\psi_{\nu_1,\ldots, \nu_N}(T_{N,1},\ldots, T_{N,N}) e^{o(1)}, \qquad E_{P} =  \e^{-\frac{N(N-1)}{2}} e^{-\frac{(N-1)(N+2)}{2}\A(\e)}\e^{-n\iota\sum_{k=1}^{N} \nu_k},
\end{equation*}
where the $o(1)$ error goes uniformly (with respect to $\{T_{k,i}\}_{1\leq i\leq k\leq N}\in D$ and $\{\nu_i\}_{1\leq i \leq N}\in V$) to zero as $\e$ goes to zero.

{\bf Step 3:} Fix a compact subset $V\subset \R^N$. Then
\begin{equation*}
m_N^{q}(z)\prod_{i=1}^{N} \frac{dz_i}{z_i} = E_{m} m_N(\nu) \prod_{i=1}^{N} d\nu_i e^{o(1)}, \qquad E_m= \e^{N^2}e^{N(N-1)\A(\e)}
\end{equation*}
where the $o(1)$ error goes uniformly (with respect to $\{\nu\}_{1\leq N}\in V$) to zero as $\e$ goes to zero.

{\bf Step 4:} We have
\begin{equation}\label{PiPscalareqnalpha}
\langle \Pi(z_1,\ldots, z_N;\rho),P_{\lambda}(z_{1},\ldots, z_{N})\rangle'_N = \left(e^{\frac{(N-1)(N-2)}{2}\A(\e)} \e^{\frac{N(N+1)}{2}}\prod_{i=1}^{n}\prod_{k=1}^{N}\frac{1}{\e^{1-\alpha_i}}\right)\theta_{\walpha,n}(T_{N,1},\ldots,T_{N,N}) e^{o(1)}
\end{equation}
where the $o(1)$ error goes uniformly (with respect to $\{T_{k,i}\}_{1\leq i\leq k\leq N}\in D$) to zero as $\e$ goes to zero.

The proof of this step is involved and relies on precise asymptotic estimates involving the tails of the $q$-Whittaker functions. We presently forgo this estimate.

Steps 2 and 3 readily show that the integrand of the scalar product on the left-hand side of equation (\ref{PiPscalareqnalpha}) converges to the integrand of $\theta_{\alpha,n}$ on the right-hand side times $E_{\Pi}\overline{E_{P}} E_m$ (recall that we must take the complex conjugate of $P_{\lambda}$ and hence also $E_{P}$). These three terms combine to equal the prefactor of $\theta_{\alpha,n}$ above. This convergence, however, is only on compact sets of the integrand variables $\nu$. In order to conclude convergence of the integrals we must show suitable tail decay for large $\nu$.

In particular, the estimate we need to show is that if we define
\begin{equation*}
V_M = \{(z_1,\ldots, z_N): \forall 1\leq k\leq N, z_k=e^{\iota\e \nu_k} \textrm{ and } |\nu_k|>M\},
\end{equation*}
then
\begin{equation*}
\lim_{M\rightarrow \infty} \lim_{\e\rightarrow 0}\int_{\nu\in V_{M}} \frac{\Pi(z_1,\ldots, z_N;\rho)}{E_{\Pi}} \frac{\overline{P_{\lambda^{(N)}}(z_1,\ldots, z_N)}}{\overline{E_{P}}} \frac{m_N^q(z)}{E_m} \prod_{i=1}^{N} \frac{ dz_i}{z_i}=0
\end{equation*}
with the limits uniform with respect to $\{T_{k,i}\}_{1\leq i\leq k\leq N}\in D$. Such an estimate (which would complete the proof of this lemma) uses the Mellin Barnes integral representation for Whittaker functions in Section~\ref{WMellinBarnes} and an analogous representation for the $q$-Whittaker functions.

%\begin{comment}
%AAAAAAAAAAAAAAAAAAAA:
% \note{the below stuff is not fixed and is just copied from the previous part. There are a few difficulties we now encounter. The Pi term behalves like a product of qGamma functions. The m term like a product of reciprocal qGamma functions. It seems as though we can not simply bound the P term by a constant though... we need to understand the decay (like the gamma function on the imaginary axis) for the qWhittaker functions. Then it seems likely to me that we need to impose $n\geq N$ for else we can't take the limit. This is because the limiting process is only defined for $n\geq N$. Given this, we should be able to bound the tail in the integral.  We will need to use asymptotics which can be extracted from the Mellin Barnes formula.}
%
%\note{ Still we will roughly follow the approach used in the other convergence proof... I had previously pasted it below but now erase it to save space}
%\end{comment}
\end{proof}

We may combine the above lemmas, just as in the proof of Theorem~\ref{theorem26}, to complete the proof Theorem~\ref{alphalimitthm}, modulo the above decay estimates which we do not presently make.
\end{proof}

\subsection{Integral formulas for the $\alpha$-Whittaker process}\label{alphamoment}

Similarly to Section~\ref{whitintform}, the moment formulas for the $q$-Whittaker processes give rise to moment formulas for the $\alpha$-Whittaker processes. The proof of the following statements follows the same path as that for Propositions \ref{Proposition27}, \ref{Proposition28}, \ref{Proposition28FULLtzero} and use decay noted in Remark \ref{starstar83} in replacement of the decay estimate of Proposition~\ref{supexpthetadecay}.

\begin{proposition}[modulo decay assumption in Remark \ref{starstar83}]\label{Proposition27alpha}
For any $1\le r\le N$,
\begin{eqnarray*}
\lefteqn{\left\langle e^{-(T_{N,N}+T_{N,N-1}+\cdots+T_{N,N-r+1})}\right\rangle_{\alphaWM{a_1,\dots,a_N;\alpha_1,\ldots, \alpha_n}} =}\\
&&\frac{(-1)^{\frac{r(r+1)}2+Nr}}{(2\pi \iota)^rr!} \oint\cdots\oint \prod_{1\le k<l\le r} (w_k-w_l)^2\prod_{j=1}^r\left(\prod_{m=1}^N \frac{1}{w_j+a_m}\right)\left(\prod_{m=1}^{n} (\alpha_m - w_j)\right)dw_j\,,
\end{eqnarray*}
where the contours include all the poles of the integrand.
\end{proposition}

The limiting form of Proposition~\ref{prop8tzero} is the following:
\begin{proposition}[modulo decay assumption in Remark \ref{starstar83}]\label{Proposition28alpha}
For any $k \ge 1$,
\begin{eqnarray*}
\lefteqn{\left\langle e^{-kT_{N,N}}\right\rangle_{\alphaWM{a_1,\dots,a_N;\alpha_1,\ldots, \alpha_n}} =}  \\
&&\frac{(-1)^{k(N-1)}}{(2\pi \iota)^k} \oint\cdots\oint \prod_{1\le A<B\le k} \frac{w_A-w_B}{w_A-w_B+1}\prod_{j=1}^k \left(\prod_{m=1}^N \frac{1}{w_j+a_m}\right)\left(\prod_{m=1}^{n} (\alpha_i - w_j)\right) dw_j\,,
\end{eqnarray*}
where the $w_j$-contour contains $\{w_{j+1}-1,\cdots,w_k-1,-a_1,\cdots,-a_N\}$ and no other singularities for $j=1,\dots,k$.
\end{proposition}

The limiting form of Proposition~\ref{prop8FULLtzero} is the following:
\begin{proposition}[modulo decay assumption in Remark \ref{starstar83}]\label{Proposition28FULLtzeroalpha}
Fix $k\geq 1$ and $r_{\kappa}\geq 1$ for $1\leq \kappa\leq k$. Then
\begin{eqnarray*}
\lefteqn{\left\langle \prod_{\kappa=1}^{k} e^{-(T_{N,N}+\cdots +T_{N,N-r_{\kappa}+1})}\right\rangle_{\alphaWM{a_1,\dots,a_N;\tau}} =}\\
&&\prod_{\kappa=1}^{k} \frac{(-1)^{r_{\kappa}(r_{\kappa}+1)/2 + Nr_{\kappa}}}{(2\pi \iota)^{r_{\kappa}} r_{\kappa}!} \oint \cdots \oint \prod_{1\leq \kappa_1 <\kappa_2\leq k} \left(\prod_{i=1}^{r_{\kappa_1}}\prod_{j=1}^{r_{\kappa_2}} \frac{w_{\kappa_1,i}-w_{\kappa_2,j}}{w_{\kappa_1,i}-w_{\kappa_2,j}+1}\right)\\
 &&\times \prod_{\kappa=1}^{k}\left( \left(\prod_{1\leq i<j\leq r_{\kappa}}(w_{\kappa,i}-w_{\kappa,j})^2\right) \left(\prod_{j=1}^{r_{\kappa}}\frac{
\prod_{m=1}^{n}(\alpha_n-w_{\kappa,j}) dw_{\kappa,j}}{(w_{\kappa,j}+a_1)\cdots (w_{\kappa,j}+a_N)}\right)\right)\,,
\end{eqnarray*}
where the $w_{\kappa_1,j}$-contour contains $\{w_{\kappa_2,i}-1\}$ for all $i\in \{1,\ldots,r_{\kappa_2}\}$ and $\kappa_2>\kappa_1$, as well as $\{-a_1,\ldots,-a_N\}$ and no other singularities.
\end{proposition}

A statement similar to Proposition~\ref{Proposition29} and Proposition~\ref{qlaplaceintegralform} can be found in \cite{COSZ}.

\subsection{Fredholm determinant formulas for the $\alpha$-Whittaker process}\label{alphaFred}

Let us first write the analog to Theorem~\ref{PlancherelfredThm} (which was done for Plancherel specializations) in the case of pure alpha specializations.

\begin{theorem}\label{purealphafredThm}
Fix $\rho$ a pure alpha (see Definition~\ref{specializationtype}) Macdonald nonnegative specialization determined by positive numbers $\hat\alpha=(\hat\alpha_1,\ldots,\hat\alpha_n)$. Fix $0<\delta<1$ and $A_1,\ldots, A_N$ such that $|A_i -1|\leq d$ for some constant $d <\frac{1-q^{\delta}}{1+q^{\delta}}$. Then for all $\zeta\in \C\setminus \Rplus$

\begin{equation*}
\left\langle \frac{1}{\left(\zeta q^{\lambda_N};q\right)_{\infty}}\right\rangle_{\MM_{t=0}(A_1,\ldots, A_N;\rho)} = \det(I+K_{\zeta})_{L^2(C_{A})}
\end{equation*}
%where $\det(I+K_{\zeta})$ is the Fredholm determinant of
%\begin{equation*}
%K_{\zeta}:L^2(C_{A})\to L^2(C_{A})
%\end{equation*}
where $C_A$ is a positively oriented circle $|w-1|=d$ and the operator $K_{\zeta}$ is defined in terms of its integral kernel
\begin{equation*}
K_{\zeta}(w,w') = \frac{1}{2\pi \iota}\int_{-\iota \infty + \delta}^{\iota \infty +\delta} \Gamma(-s)\Gamma(1+s)(-\zeta)^s g_{w,w'}(q^s)ds
\end{equation*}
where
\begin{equation*}
g_{w,w'}(q^s) = \frac{1}{q^s w - w'} \prod_{m=1}^{N} \frac{(q^{s}w/A_m;q)_{\infty}}{(w/A_m;q)_{\infty}} \prod_{i=1}^{n} \frac{(\hat\alpha_i w;q)_{\infty}}{(q^{s}\hat\alpha_i w;q)_{\infty}}.
\end{equation*}
The operator $K_{\zeta}$ is trace-class for all $\zeta\in \C\setminus \Rplus$.
\end{theorem}
\begin{proof}
The proof of this result is a straightforward modification of that of Theorem~\ref{PlancherelfredThm}.
\end{proof}

From this we can prove the analog of Theorem~\ref{NeilPolymerFredDetThm}.

\begin{theorem}[modulo decay estimates of Theorem~\ref{alphalimitthm}]\label{logGammaPolymerFredDetThm}
Recall the conditions of Definition~\ref{alphadefs}.
Fix $0<\delta_2<1$, and $\delta_1<\delta_2 /2$. Fix $a_1,\ldots, a_N$ such that $|a_i|<\delta_1$. Then
\begin{equation*}
\left\langle e^{-u e^{-T_{N}}}  \right\rangle_{\alphaWM{a_1,\ldots,a_N;\alpha_1,\ldots,\alpha_n}} = \det(I+ K_{u})_{L^2(C_a)}
\end{equation*}
%where $\det(I+ K_{u})$ is the Fredholm determinant of
%\begin{equation*}
%K_{u}: L^2(C_a)\to L^2(C_a)
%\end{equation*}
where $C_a$ is a positively oriented contour containing $a_1,\ldots, a_N$ such that for all $v,v'\in C_a$, $|v-v'|<\delta_2$. The operator $K_u$ is defined in terms of its integral kernel

\begin{equation}\label{kvvprimealpha}
K_{u}(v,v') = \frac{1}{2\pi \iota}\int_{-\iota \infty + \delta_2}^{\iota \infty +\delta_2}ds \Gamma(-s)\Gamma(1+s) \left(\prod_{m=1}^{N}\frac{\Gamma(v-a_m)}{\Gamma(s+v-a_m)}\right) \left(\prod_{m=1}^{n} \frac{\Gamma(v+\alpha_m)}{\Gamma(s+v+\alpha_m)}\right)  \frac{ u^s}{v+s-v'}.
\end{equation}
\end{theorem}

\begin{proof}
The proof of this result is a straightforward modification of that of Theorem~\ref{NeilPolymerFredDetThm}. The proof relies on the weak convergence result of Theorem~\ref{alphalimitthm} which is proved modulo certain decay estimates.
\end{proof}
%\subsection{Tracy-Widom asymptotics}

\chapter{Directed polymer in a random media}

\section{General background}\label{genback}

In this section we will focus on a class of models introduced first by Huse and Henley \cite{HuHe} which we will call {\it directed polymers in a random media}\index{directed polymers in a random media} (DPRM). Such polymers are directed in what is often referred to as a {\it time} direction, and then are free to configure themselves in the remaining $d$ {\it spatial} dimensions. The probability of a given configuration of the polymer is then given (relative to an underlying path measure on paths $\pi(\cdot)$) as a Radon Nikodym derivative which is often written as a {\it Boltzmann weight} \index{directed polymers in a random media!Boltzmann weight} involving a {\it Hamiltonian}\index{directed polymers in a random media!Hamiltonian} which assigns an energy to the path:
\begin{equation*}
dP_{Q}^{\beta}(\pi(\cdot)) = \frac{1}{Z^{\beta}_{Q}} \exp\{\beta H_{Q}(\pi(\cdot))\} dP_0(\pi(\cdot)).
\end{equation*}
In the above equation $dP_0$ represents the underlying path measure (which is independent of the Hamiltonian and its randomness). The parameter $\beta$ is known as the inverse temperature since modifying its value changes the balance between underlying path measure (entropy) and the energetic rewards presented by the disordered or random media in which the path lives. The term $H_{Q}$ represents the Hamiltonian which assigns an energy to a given path. The subscript $Q$ stands for {\it quenched}\index{directed polymers in a random media!quenched} which means that this $H_{Q}(\pi(\cdot))$ is actually a random function of the disorder $\omega$ which we think of as an element of a probability space. Finally, $Z^{\beta}_{Q}$ is the quenched partition function which is defined as necessary to normalize $dP_{Q}^{\beta}$ as a probability measure:
\begin{equation*}
Z^{\beta}_{Q}= \int \exp\{\beta H_{Q}(\pi(\cdot))\} dP_0(\pi(\cdot)).
\end{equation*}
The measure $dP_{Q}^{\beta}$ is a quenched polymer measure since it is still random with respect to the randomness of the Hamiltonian $H_{Q}$. This is to say that $dP_{Q}^{\beta}$ is also a function of the disorder $\omega$. We denote averages with respect to the disorder $\omega$ by $\langle \cdot \rangle$, so that $\langle Z^{\beta}_Q\rangle$ represents the averaged value of the partition function. We use $\PP$ for the probability measure for the disorder $\omega$ and denote the variance with respect to the disorder as $\var{\cdot}$.

At infinite temperature, $\beta=0$, and under standard hypotheses on $dP_0$ (i.e., i.i.d.\ finite variance increments) the measure $dP^{\beta}_{Q}(\pi(\cdot))$ rescales diffusively to that of a Brownian motion and thus the polymer is purely maximizing entropy. At zero temperature, $\beta=\infty$, the polymer measure concentrates on the path (or paths) $\pi$ which maximize the polymer energy $H_Q(\pi)$. A well studied challenge is to understand the effect of quenched disorder at positive $\beta$ on the behavior of a $dP^{\beta}_{Q}$-typical path of the free energy $F^\beta_{Q}:=\beta^{-1} \log (Z^{\beta}_Q)$. A rough description of the behaviour is given by the {\it transversal fluctuation exponent}\index{directed polymers in a random media!transversal fluctuation exponent}  $\xi$ and the {\it longitudinal fluctuation exponent}\index{directed polymers in a random media!longitudinal fluctuation exponent} $\chi$. There are many different ways these exponents have been defined, and it is not at all obvious that they exist for a typical polymer model -- though it is believed that they do. As $n$ goes to infinity, the first exponent describes the fluctuations of the endpoint of the path $\pi$: typically $|\pi(n)| \approx n^{\xi}$. The second exponent describes the fluctuations of the free energy: $\var{F_{\beta,Q}} \approx n^{2\chi}$. Assuming the existence of these exponents, in order to have a better understanding of the system it is of essential interest to understand the statistics for the properly scaled location of the endpoint and fluctuations of the free energy.

%When $\beta=0$ the above model reduces to the original path measure $dP_0$. Let us now focus on the case when $dP_0$ is the path measure for a standard simple symmetric random walk (SSRW). This means that for $\beta=0$, $\pi(\cdot)$ will rescale diffusively to a Brownian motion (or Brownian bridge if we pin the endpoint). A general question of interest in the study of polymers is to understand the effect of a random Hamiltonian at positive $\beta$ on the behavior and energy of a $dP_{Q}^{\beta}$ typical path. This is generally recorded in terms of two scaling exponents: the {\it transversal fluctuation exponent}\index{directed polymers in a random media!transversal fluctuation exponent}  $\xi$ and the {\it longitudinal fluctuation exponent}\index{directed polymers in a random media!longitudinal fluctuation exponent} $\chi$. When $dP_0$ is supported on $n$-step SSRWs, as $n$ goes to infinity, the first exponent describes the fluctuations of the endpoint of the path $\pi$: $\var{\pi(n)} \approx n^{2\xi}$. The second exponent likewise describes the fluctuations of the free energy: $\var{\beta^{-1}\log Z^{\beta}} \approx n^{\chi}$. On top of these scaling exponents it is of essential interest to understand the statistics for the properly scaled location of the endpoint and fluctuations of the free energy.

We will now focus entirely on Hamiltonians which take the form of a path integral through a space-time independent noise field. In the discrete setting of $dP_0$ as SSRW of length $n$, the noise field can be chosen as IID random variables $w_{i,x}$ and then $H_{Q}(\pi(\cdot)) = \sum_{i=0}^{n} w_{i,\pi(i)}$.

The first rigorous mathematical work on directed polymers was by Imbrie and Spencer \cite{IS} in 1988 where (by use of an elaborate expansion) they proved that in dimensions $d\ge 3$ and with small enough $\beta$, the walk is diffusive ($\xi=1/2$). Bolthausen \cite{Bolt} strengthened the result (under same same $d\ge 3$, $\beta$ small assumptions) to a central limit theorem for the endpoint of the walk. His work relied on the now fundamental observation that renormalized partition function (for $dP_0$ a SSRW of length $n$) $W_n=Z^{\beta}_{Q} /\langle Z^{\beta}_Q\rangle$ is a martingale.

By a zero-one law, the limit $W_\infty=\lim_{n\to \infty} W_n$ is either almost surely $0$ or almost surely positive. Since when $\beta=0$, the limit is 1, the term {\it strong disorder}\index{directed polymers in a random media!strong disorder} has come to refer to the case of $W_{\infty}=0$ since then the disordered noise has, indeed, a strong effect. The case $W_{\infty}>0$ is called {\it weak disorder}\index{directed polymers in a random media!weak disorder}.

There is a critical value $\beta_c$ such that weak disorder holds for $\beta<\beta_c$ and strong for $\beta>\beta_c$.  It is known that $\beta_c=0$ for $d\in\{1,2\}$ \cite{CSY} and $0<\beta_c\le\infty$ for $d\ge 3$. In $d\ge 3$ and weak disorder the walk converges to a Brownian motion, and the limiting diffusion matrix is the same as for standard random walk \cite{come-yosh-aop-06}.

On the other hand, in strong disorder it is known (see \cite{CSY}) that there exist (random) points at which the path $\pi$ has a positive probability (under $dP_{Q}^{\beta}$) of ending. This is certainly different behavior than that of a Brownian motion.

The behavior of directed polymer when restricted to $d=1$ has drawn significant attention and the scaling exponents $\xi,\chi$ and fluctuation statistics are believed to be universal with respect to the underlying path measure and underlying random Hamiltonian. Establishing such universality has proved extremely difficult (see \cite{SeppLog} for a review of the progress so far in this direction).

The KPZ universality belief is that in $d=1$, for all $\beta>0$ and all distributions for $w_{i,j}$ (up to certain conjectural conditions on finite moments) the exponents $\xi=2/3$ and $\chi=1/3$. A stronger form of this conjecture is that, up to centering and scaling, there exists a unique limit
\begin{equation*}
\lim_{\e\to 0} R_{\e} \frac{\log Z^{\beta}(t,x)}{\beta}
\end{equation*}
where $Z^{\beta}(t,x)$ is the point to point partition function of polymers ending at $x$ at time $t$. The operator $R_{\e}$ is the KPZ renormalization operator and acts on a space-time function as $(R_{\e} f)(t,x) = \e f(\e^{-3}t,\e^{-2}x)$ minus the necessary centering. This limit point is described in \cite{CQ2} where it is called the {\it KPZ renormalization fixed point}\index{directed polymers in a random media!KPZ renormalization fixed point}. Information about this fixed point (such as the fact that for a fixed $t$, it is spatially distributed as an Airy$_2$ process \cite{PS2}) has generally come from studying ground-state, or zero temperature models such as last passage percolation, TASEP or PNG (see the review \cite{ICreview}).

It has only been very recently that exactly solvable finite temperature polymer models have been discovered and analyzed. In the following sections we introduce these models, provide background as to what was previously known about their solvability and then demonstrate how the methods developed in this paper enhance the solvability.

\section{The O'Connell-Yor semi-discrete directed polymer}\label{OConmodel}
\index{O'Connell-Yor semi-discrete directed polymer}
\subsection{Definitions and results from \cite{OCon}}

\begin{definition}
An {\it up/right path}\index{O'Connell-Yor semi-discrete directed polymer!up/right path} in $\R\times \Z$ is an increasing path which either proceeds to the right or jumps up by one unit. For each sequence $0<s_1<\cdots<s_{N-1}<t$ we can associate an up/right path $\phi$ from $(0,1)$ to $(t,N)$ which jumps between the points $(s_i,i)$ and $(s_{i},i+1)$, for $i=1,\ldots, N-1$, and is continuous otherwise. Fix a real vector $a=(a_1,\ldots, a_N)$ and let $B(s) = (B_1(s),\ldots, B_N(s))$ for $s\geq 0$ be independent standard Brownian motions such that $B_i$ has drift $a_i$. Define the {\it energy} \index{O'Connell-Yor semi-discrete directed polymer!energy} of a path $\phi$ to be
\begin{equation*}
E(\phi) = B_1(s_1)+\left(B_2(s_2)-B_2(s_1)\right)+ \cdots + \left(B_N(t) - B_{N}(s_{N-1})\right).
\end{equation*}\glossary{$E(\phi)$}
Then the {\it semi-discrete directed polymer partition function}\index{O'Connell-Yor semi-discrete directed polymer!partition function} $\Zsd^{N}(t)$ is given by
\begin{equation*}
\Zsd^{N}(t) = \int e^{E(\phi)} d\phi,
\end{equation*}\glossary{$\Zsd^{N}(t)$}
where the integral is with respect to Lebesgue measure on the Euclidean set of all up/right paths $\phi$ (i.e., the simplex of jumping times $0<s_1<\cdots<s_{N-1}<t$). One can introduce the {\it hierarchy of partition functions}\index{O'Connell-Yor semi-discrete directed polymer!partition function hierarchy} $\Zsd^{N}_{n}(t)$ for $0\leq n\leq N$ via  $\Zsd^{N}_{0}(t)=1$ and for $n\geq 1$,
\begin{equation*}
\Zsd^{N}_{n}(t) = \int_{D_{n}(t)} e^{\sum_{i=1}^{n} E(\phi_i)} d\phi_1 \cdots d\phi_n,
\end{equation*}\glossary{$\Zsd^{N}_{n}(t)$}
where the integral is with respect to the Lebesgue measure on the Euclidean set $D_n(t)$ of all $n$-tuples of non-intersecting (disjoint) up/right paths with initial points $(0,1),\ldots, (0,n)$ and endpoints $(t,N-n+1),\ldots, (t,N)$.

The {\it hierarchy of free energies}\index{O'Connell-Yor semi-discrete directed polymer!free energy hierarchy} $\Fsd^{N}_{n}(t)$ for $1\leq n\leq N$ is defined via
\begin{equation*}
\Fsd^{N}_{n}(t) = \log\left(\frac{ \Zsd^{N}_{n}(t)}{ \Zsd^{N}_{n-1}(t)}\right).
\end{equation*}\glossary{$\Fsd^{N}_{n}(t)$}
\end{definition}

The triangular array
\begin{equation*}
\left\{\Fsd^{k}_{j}(\cdot)\right\}_{1\leq j\leq k \leq N}:[0,\infty) \to \R^{N(N+1)/2}
\end{equation*}
can be recognized as being almost surely (path-wise) equal to the trajectories of a certain diffusion Markov process. In order to state this result we define an auxiliary diffusion Markov process
\begin{equation*}
\left\{\Gsd^{k}_{j}(\cdot)\right\}_{1\leq j\leq k \leq N}:[0,\infty) \to \R^{N(N+1)/2}
\end{equation*}
recursively as follows: Let $d\Gsd^{1}_{1} = d\tilde{B}_1$ and, for $k=2,\ldots,N$
\begin{eqnarray}\label{nonsymdyn}
\nonumber d\Gsd^{k}_1 &=& d\Gsd^{k-1}_1 + e^{\Gsd^{k}_2 - \Gsd^{k-1}_1} dt\\
\nonumber d\Gsd^{k}_2 &=& d\Gsd^{k-1}_2 + \left(e^{\Gsd^{k}_3 - \Gsd^{k-1}_2} - e^{\Gsd^{k}_2 - \Gsd^{k-1}_1}\right) dt\\
&\vdots&\\
\nonumber d\Gsd^{k}_{k-1} &=& d\Gsd^{k-1}_{k-1} + \left(e^{\Gsd^{k}_k - \Gsd^{k-1}_{k-1}} - e^{\Gsd^{k}_{k-1} - \Gsd^{k-1}_{k-2}}\right) dt\\
\nonumber d\Gsd^{k}_k &=& d\tilde{B}_k  - e^{\Gsd^{k}_k - \Gsd^{k-1}_{k-1}} dt,
\end{eqnarray}
where the $\tilde{B}_{k}=-B_k$ and $B_k$ are the independent standard Brownian motions with drifts $a_k$ which serve as the inputs for the polymer model.

It is shown in \cite{OCon} that:

\begin{theorem}\label{OConthm}
Fix $N\geq 1$ and a vector of drifts $a=(a_1,\ldots, a_N)$, then:
\begin{enumerate}
\item If we define $\Gsd$ in terms of $\Fsd$ via
\begin{equation*}
\left\{\Fsd^{k}_j(\cdot)\right\}_{1\leq j\leq k\leq N}=\left\{-\Gsd^{k}_{k-j+1}(\cdot)\right\}_{1\leq j\leq k\leq N}
\end{equation*}
then, almost surely, $\Gsd$ satisfies the diffusion Markov process defined in (\ref{nonsymdyn}).
\item $\left\{\Gsd^N_n(\cdot)\right\}_{1\leq n\leq N}$ evolves as a diffusion Markov process in $\R^N$ (with respect to its own filtration) with infinitesimal generator given by
\begin{equation*}
\mathcal{L}_{a} = \tfrac{1}{2} \psi_{\iota  a}^{-1} \left(H-\sum_{i=1}^{N} a_i^2\right) \psi_{\iota a} = \tfrac{1}{2}\Delta + \nabla \log \psi_{\iota a} \cdot \nabla
\end{equation*}
where $\psi_{\iota a}(x) = \psi_{\iota a_1,\ldots, \iota a_N}(x_1,\ldots, x_N)$ is the Whittaker functions (see Section~\ref{Whittakerfunctiondefs} and note the difference between the given definition and that of \cite{OCon}), where $H$ is the quantum $\mathfrak{gl}_N$ Toda lattice Hamiltonian
\begin{equation*}
H=\Delta - 2\sum_{i=1}^{N-1} e^{x_{i+1}-x_i}.
\end{equation*}
The entrance law at time $t$ for this diffusion is given by the Whittaker measure $\WM{a;t}$ of  (\ref{WMdef}).

\item For each $t>0$, the conditional law of $\Gsd(t)=\left\{\Gsd^{k}_{j}(t)\right\}_{1\leq j\leq k \leq N}$ given
\begin{equation*}
\left\{\left\{\Gsd^{N}_n(s)\right\}_{1\leq n\leq N}:s\leq t, \left\{\Gsd^{N}_n(t)\right\}_{1\leq n\leq N} = x\right\}
\end{equation*}
is given by the density
\begin{equation*}
\psi_{\iota a}(x)^{-1} \exp({\mathcal{F}_{\iota a}(\Gsd(t))}).
\end{equation*}
\end{enumerate}
\end{theorem}

\begin{remark}\label{whitsymrem}
It is useful to note the following symmetry: The transformation $T_{k,i}\leftrightarrow -T_{k,k+1-i}$ maps $\W{a;\tau}$ to $\W{-a;\tau}$ (the sign of $a_j$'s changes). This easily follows from the definition of $\W{a;\tau}$.
\end{remark}

\begin{corollary}\label{Fsdcor}
Fix $N\geq 1$ and a vector of drifts $a=(a_1,\ldots, a_N)$, then $\Fsd(t)=\left\{\Fsd^{k}_j(\cdot)\right\}_{1\leq j\leq k\leq N}$ is distributed according to the Whittaker process $\W{-a;t}$ defined in (\ref{Wdef}).
\end{corollary}

\begin{remark}\label{probresrem}
As mentioned in the proof of Proposition~\ref{equiv1}, the above result provides probabilistic means to seeing that the Whittaker measure is a probability measure.
\end{remark}

\begin{remark}
As a corollary of the above result and the Bump-Stade identity O'Connell \cite{OCon} derived a integral formula for the Laplace transform of $\Zsd^{N}_{1}(t)$ (see Proposition~\ref{Proposition29}). In \cite{BorCorRem} we proved an identity which shows that this integral formula is equivalent to the Fredholm determinant formula we derived earlier (in the polymer context see in Section~\ref{restatedFredDet}).
\end{remark}

The dynamics given by the coupled stochastic ODEs in (\ref{nonsymdyn}) do not correspond to the Whittaker 2d-growth model given in Section~\ref{contlimit2dgrowth}. For instance, one readily sees that the dynamics of (\ref{nonsymdyn}) is driven by $N$ noises, whereas that of the Whittaker 2d-growth model is driven by $N(N+1)/2$ noises.
O'Connell does consider the Whittaker 2d-growth model and refers to it as a ``symmetric'' version of the dynamics of (\ref{nonsymdyn}) -- see Remark \ref{Whittaker2dgrowthprojections}. Restricted to the top level, both processes are Markovian with generator $\mathcal{L}$. Also, restricted to the edge of the triangular array, the dynamics of $\{T^{k}_{k}(\cdot)\}_{1\leq k\leq N}$ and of $\{\Gsd^{k}_{k}(\cdot)\}_{1\leq k\leq N}$ coincide and are both Markov processes. These edge dynamics correspond to the limit of the $q$-TASEP process.

\subsection{Integral formulas}\label{restatedIntForm}

The following result follows from Theorem~\ref{OConthm} and Proposition~\ref{Proposition27}.
\begin{proposition}\label{Proposition27OCon}
Fix $N\geq 1$ and a drift vector $a=(a_1,\ldots,a_N)$. Then for $1\leq r\leq N$ and $t\geq 0$
\begin{equation*}
\left\langle \Zsd^{N}_{r}(t) \right\rangle= \frac{(-1)^{\frac{r(r-1)}2} e^{rt/2}}{(2\pi \iota)^rr!} \oint\cdots\oint \prod_{1\le k<\ell\le r} (w_k-w_\ell)^2\prod_{j=1}^r\left(\prod_{m=1}^N \frac{1}{w_j-a_m}\right) {e^{t w_j}dw_j}\,,
\end{equation*}
where the contours include all the poles of the integrand.
\end{proposition}

The following result follows from Theorem~\ref{OConthm} and Proposition~\ref{Proposition28}. We give an independent check of this result in Part \ref{replicassec} using the replica approach (see Proposition~\ref{discDBG}).
\begin{proposition}\label{Proposition28OCon}
Fix $N\geq 1$ and a drift vector $a=(a_1,\ldots,a_N)$. Then for $k\geq 1$ and $t\geq 0$
\begin{equation*}
\left\langle  \left(\Zsd^{N}_{1}(t)\right)^k    \right\rangle = \frac{e^{kt/2}}{(2\pi \iota)^k} \oint\cdots\oint \prod_{1\le A<B\le k} \frac{w_A-w_B}{w_A-w_B-1}\prod_{j=1}^k \left(\prod_{m=1}^N \frac{1}{w_j-a_m}\right) {e^{t w_j}dw_j}\,,
\end{equation*}
where the $w_{A}$ contour contains only the poles at $\{w_B+1\}$ for $B>A$ as well as $\{a_1,\ldots, a_N\}$.
\end{proposition}

The following more general statement follows from Theorem~\ref{OConthm} and Proposition~\ref{Proposition28FULLtzero}.
\begin{proposition}\label{Proposition28FULLtzeroOCon}
Fix $N\geq 1$ and a drift vector $a=(a_1,\ldots,a_N)$. Then for $k\geq 1$, $r_{\alpha}\geq 1$ for $1\leq \alpha\leq k$, $N\geq N_1\geq N_2\geq \cdots \geq N_k$, and $t\geq 0$,
\begin{eqnarray*}
\lefteqn{\left\langle \prod_{\alpha=1}^{k} \Zsd^{N_\alpha}_{r_{\alpha}}(t)\right\rangle =}\\
&&\prod_{\alpha=1}^{k} \frac{(-1)^{r_{\alpha}(r_{\alpha}-1)/2 }}{(2\pi \iota)^{r_{\alpha}} r_{\alpha}!} \prod_{\alpha=1}^{k} e^{t r_{\alpha}/2} \oint \cdots \oint \prod_{1\leq \alpha <\beta\leq k} \left(\prod_{i=1}^{r_{\alpha}}\prod_{j=1}^{r_{\beta}} \frac{w_{\alpha,i}-w_{\beta,j}}{w_{\alpha,i}-w_{\beta,j}-1}\right)\\
 &&\times \prod_{\alpha=1}^{k}\left( \left(\prod_{1\leq i<j\leq r_{\alpha}}(w_{\alpha,i}-w_{\alpha,j})^2\right) \left(\prod_{j=1}^{r_{\alpha}}\frac{ e^{t w_{\alpha,j}}dw_{\alpha,j}}{(w_{\alpha,j}-a_1)\cdots (w_{\alpha,j}-a_{N_\alpha})}\right)\right)\,,
\end{eqnarray*}
where the $w_{\alpha,j}$-contour contains $\{w_{\beta,i}+1\}$ for all $i\in \{1,\ldots,r_{\beta}\}$ and $\beta>\alpha$, as well as $\{a_1,\ldots,a_N\}$ and no other singularities.
\end{proposition}

The following result follows from Theorem~\ref{OConthm} and Proposition~\ref{ktimeintform}.
\begin{proposition}\label{ktimeintformOCon}
Fix $N\geq 1$ and a drift vector $a=(a_1,\ldots,a_N)$. Then for $k$ time moments $0\leq t_1\leq t_2\leq \cdots \leq t_k$,
\begin{equation*}
\left\langle \prod_{i=1}^{k} \Zsd^{N}_{1}(t_i) \right\rangle = \frac{e^{\sum_{i=1}^{k} t_i/2}}{(2\pi \iota)^k} \oint\cdots\oint \prod_{1\leq A<B\leq k} \frac{w_{A}-w_{B}}{w_{A}-w_{B}-1} \prod_{m=1}^{N} \left(\prod_{j=1}^k \frac{1}{w_j-a_m} \right) e^{t_j w_j} dw_j,
\end{equation*}
where the $w_{A}$ contour contains only the poles at $\{w_B+1\}$ for $B>A$ as well as $\{a_1,\ldots, a_N\}$.
\end{proposition}

\subsection{Fredholm determinant formula}\label{restatedFredDet}

The following result follows from Theorem~\ref{NeilPolymerFredDetThm}, Corollary \ref{Fsdcor} and Remark \ref{whitsymrem}.

\begin{theorem}\label{OConFredDet}
Fix $N\geq 1$ and a drift vector $a=(a_1,\ldots,a_N)$. Fix $0<\delta_2<1$, and $\delta_1<\delta_2 /2$ such that $|a_i|<\delta_1$. Then for $t\geq 0$,
\begin{equation*}
\left\langle e^{-u \Zsd^{N}_{1}(t) }  \right\rangle = \det(I+ K_{u})_{L^2(C_{a})}
\end{equation*}
%where $\det(I+ K_{u})$ is the Fredholm determinant of
%\begin{equation*}
%K_{u}: L^2(C_{a})\to L^2(C_{a})
%\end{equation*}
where $C_{a}$ is a positively oriented contour containing $a_1,\ldots, a_N$ and such that for all $v,v'\in C_{a}$, we have $|v-v'|<\delta_2$. The operator $K_u$ is defined in terms of its integral kernel
\begin{equation*}
K_{u}(v,v') = \frac{1}{2\pi \iota}\int_{-\iota \infty + \delta_2}^{\iota \infty +\delta_2}ds \Gamma(-s)\Gamma(1+s) \prod_{m=1}^{N}\frac{\Gamma(v-a_m)}{\Gamma(s+v-a_m)} \frac{ u^s e^{vt s+t s^2/2}}{v+s-v'}.
\end{equation*}
\end{theorem}

%\begin{definition}
%The Digamma function $\Psi(z) = [\log \Gamma]'(z)$. Define for $\kappa>0$
%\begin{equation*}
%\bfk = \inf_{t>0} (\kappa t - \Psi(t))
%\end{equation*}
%and let $\btk$ denote the unique value of $t$ at which the minimum is achieved. Finally, define the positive number (scaling parameter) $\bgk= -\Psi''(\btk)$.
%\end{definition}

The following result follows from Theorem~\ref{OConthm} and Theorem~\ref{TWasymptoticskappaTHM}. Recall Definition~\ref{digammadef}.
\begin{theorem}
Fix drifts $a_i\equiv 0$ and for all $N$, set $t=N\kappa$. Then there exists $\kappa^*>0$ such that for $\kappa>\kappa^*$,
\begin{equation*}
\lim_{N\to \infty} \PP\left( \frac{ \Fsd^{N}_{1}(t) - N \bfk}{N^{1/3}}\leq r\right) = F_{{\rm GUE}}\left((\bgk / 2)^{-1/3}r\right).
\end{equation*}
\end{theorem}
\begin{remark}
This result should hold for all $\kappa^*>0$. See Remark \ref{remkappa} for further developments achieve this.
\end{remark}

\subsection{Parabolic Anderson model}\label{PAMsec}

Consider an infinitesimal generator $L$ for a continuous time Markov process on $\Z$ with potential $V:\Rplus\times \Z\to \R$. Then, under fairly general conditions on $L$ and $V$, the initial value problem
\begin{equation*}
\partial_t u= L u - V u, \qquad u(0,y)=f(y),
\end{equation*}
has a unique solution which can be written via the Feynman-Kac representation\index{Feynman-Kac representation!discrete space} as
\begin{equation*}
u(t,x) = \EE_x\left[e^{-\int_0^t V(s,X(s))ds} f(X(t))\right],
\end{equation*}
where $X(\cdot)$ is a diffusion with infinitesimal generator $L$ and $\EE_x$ represents the expectation over trajectories of such a diffusion started from $x$ \cite{CarMol}.

This average over paths integrals can be thought of as a polymer (in the potential $V$) started at $x$, ending according to a potential given by $\log(f(y))$ and evolving according to the underlying path measure given by the Markov process with generator $L$. When $f(y)=\delta_{y=j}$ (the Kronecker delta) and $\tilde{L}$ is the time reversal of $L$, then there is an alternative representation as
\begin{equation*}
u(t,x) = \tilde{\EE}_{j}\left[e^{-\int_0^t \tilde{V}(s,\tilde{X}(s))ds} \delta(\tilde{X}(t)=x)\right],
\end{equation*}
where $\tilde{\EE}_{j}$ is the expectation of the time reversed process $\tilde{X}$ started at $j$ at time 0 in time reversed potential $\tilde{V}(s,i) = \tilde{V}(t-s,i)$. Often it is interesting to consider potentials $V$ which are random.

The O'Connell-Yor semi-discrete directed polymer partition function $\Zsd^{N}(t)$ with drift vector $\{a_1,\ldots, a_N\}$ fits into this framework. Let the time reversed generator $\tilde{L}$ be a Poisson jump process $\tilde{X}(t)$ which increases from $i$ to $i+1$ at rate one, let $f(y)=\delta(y=1)$, and let $V(s,i) = dW(s,i)+a_i$ where $dW(\cdot,i)$ are independent identically distributed 1-dimensional white noises (i.e., distributional derivatives of independent Brownian motions). Let the integral $\int_0^t V(s,\tilde{X}(s))ds$ be treated as the It\^{o} integral $\int_0^t dW(s,\tilde{X}(s))$ for a given realization of $\tilde{X}(s)$. Then the partition function $\Zsd^{N}(t)=e^{t} u(t,N)$. The reason for the $e^t$ factor comes from the fact that a Poisson point process with exactly $N-1$ jumps on the time interval $[0,t]$ is distributed according to the uniform density $e^{-t}t^{N-1}/(N-1)!$ on the simplex of jump times $0<s_1<\cdots<s_{N-1}<t$. This differs by the $e^t$ factor from the measure we integrated against when defining $\Zsd^{N}(t)$.

The symmetric {\it parabolic Anderson model}\index{O'Connell-Yor semi-discrete directed polymer!parabolic Anderson model} on $\Rplus\times \Z$ fits into the same setup but has a generator corresponding to that of a simple symmetric continuous time random walk. The analysis performed in \cite{OCon} does not seem to extend to this model or to any other generators $L$ than that of the Poisson jump process described above. However, in Part \ref{replicassec} we will demonstrate a new ansatz for computing polymer moments which applies to some degree to the symmetric parabolic Anderson model as well (cf. Remark \ref{PAMrem}).

\section{Log-gamma discrete directed polymer}\label{logsec}
\index{log-gamma discrete directed polymer}
%\begin{comment}
%\note{introduce the model: first the single path polymer with exponentiated weights and SSRW. mention how the below definition is just a rotation and considering the exponentiated variables.}
%\begin{definition}\label{ddrpexpform}
%state formally the model and the partition function and free energy in terms of SSRW and Boltzmann weights
%\end{definition}
%\end{comment}

We presently follow the notation of \cite{COSZ}.

\begin{definition}
Fix $N\geq 1$ and a semi-infinite matrix $d=(d_{ij}: i\ge 1,  1\leq j\leq N)$ of positive real weights $d_{ij}\in \R^{>0}$. For each $n\ge 1$  form  the  $n\times N$ matrix $d^{[1,n]}=(d_{ij}: 1\le i\le n,  1\leq j\leq N) $. For $1\leq \ell \leq k\le N$ let $\Pi_{n,k}^{\ell}$ denote the set of $\ell$-tuples $\pi=(\pi_1,\ldots, \pi_\ell)$ of non-intersecting lattice paths in $\Z^2$ such that, for $1\le r\le \ell$,  $\pi_{r}$ is a lattice path from $(1,r)$ to $(n,k+r-\ell)$. A {\it lattice path}\index{log-gamma discrete directed polymer!lattice path} only takes unit steps in the coordinate directions between nearest-neighbor lattice points of $\Z^2$ (i.e., up or right); non-intersecting means that paths do not touch. The {\it weight}\index{log-gamma discrete directed polymer!weight} of an $\ell$-tuple $\pi = (\pi_1,\ldots, \pi_\ell)$ of such paths is
\begin{equation*}
wt(\pi) = \prod_{r=1}^{\ell} \prod_{(i,j)\in \pi_{r}} d_{ij}.
\end{equation*}

Let us assume $n\geq N$ (otherwise some additional case must be taken -- see \cite{COSZ}). Then for $1\le \ell\le k\le N$, let
\begin{equation*}
\tau^{k}_{\ell}(n) = \sum_{\pi\in\Pi_{n,k}^{\ell}} wt(\pi).   \label{tau}
\end{equation*}
%For $0\leq n<\ell< k\leq N$  the set of paths $\Pi_{n,k}^\ell$ is empty and we take the empty sum to equal zero. At $\ell=k$ there is a unique $\ell$-tuple, and in fact we have the equation
%\begin{equation*}
%\tau^k_\ell(n)=\delta_{k,\ell}\,\tau^{k}_{n}(n)\qquad\text{for}\qquad  0\leq n< \ell\leq k\leq N
%\end{equation*}
%where  $\delta_{k,\ell}$ is the Kronecker delta.
%For $\ell=0$ set $\tau^{k}_{0}(n)=1$ for $1\le k\leq N$ and otherwise set
%\begin{equation*}
%\tau_{k,\ell}(n) = \sum_{\pi\in\Pi_{n,k}^{\ell}} wt(\pi).
%\end{equation*}
Define the {\it free energy hierarchy}\index{log-gamma discrete directed polymer!free energy hierarchy} $f(n)=\{f^{k}_{\ell}(n): 1\le \ell\le k\le N\}$  via
\begin{equation}\label{discfreedef}
f^{k}_{\ell}(n) = \log\left(\frac{\tau^{k}_{\ell}(n)}{\tau^{k}_{\ell-1}(n)}\right).
\end{equation}

%The elements  $\big(\zarr_{k\ell}(n):\ n< \ell\le k\le N\big)$ we regard as undefined, even though strictly speaking one more element, namely $\zarr_{n+1,n+1}(n)$, could be consistently defined as $1$.   In the spirit of Remark \ref{sing-rem} we could also  replace the undefined array elements with particular singular values. See Figure \ref{pathfig} for an illustration.

We express the mapping  (\ref{discfreedef}) that defines $f(n)$ from $d^{[1,n]}$ as
\begin{equation}\label{PnN}
f(n)=P_{n,N}(d^{[1,n]}).
\end{equation}
This corresponds to the {\it tropical P tableaux} of the image of $d^{[1,n]}$ under A.N. Kirillov's {\it tropical Robinson-Schensted-Knuth correspondence} \cite{Kir}\index{log-gamma discrete directed polymer!tropical RSK correspondence}.
\end{definition}

\begin{definition} Let $\theta$ be a positive real. A random variable $X$ has {\it inverse-gamma distribution with parameter $\theta>0$}\index{inverse-gamma distribution}  if it is supported on the positive reals
where it has distribution
\begin{equation}\label{invgammadensity}
\PP(X\in dx) = \frac{1}{\Gamma(\theta)} x^{-\theta-1}\exp\left\{-\frac{1}{x}\right\} dx.
\end{equation}
We abbreviate this as $X\sim \Gamma^{-1}(\theta)$.
\end{definition}

\begin{definition}
An {\it inverse-gamma weight matrix}, with respect to a {\it parameter matrix} $\thetarc=(\thetarc_{i,j}>0:  i\ge 1, 1\le j\le N)$, is a matrix of positive weights $(d_{i,j}:  i\ge 1, 1\le j\le N)$ such that the entries are  independent random variables $d_{i,j}\sim \Gamma^{-1}(\thetarc_{i,j})$. We call a parameter matrix $\thetarc$ {\it solvable} if $\thetarc_{i,j} =  \thetar_i+ \thetac_j>0$ for real parameters $(\thetar_i: i\geq 1)$ and $(\thetac_j: 1\le j\le N)$.   In this case we also  refer to the associated weight matrix as solvable.
%Column $n$ of the parameter matrix $\thetarc$ is denoted by $\thetarc^{[n]}=(\thetarc_{n,j})_{1\le j\le N}$. We denote the vector $\thetacv=(\thetac_j: 1\le j\le )$ for later use.
\label{def-d}\end{definition}

\begin{remark}
The weights $d_{i,j}$ can be considered as $e^{-\tilde{d}_{i,j}}$ where $\tilde{d}_{i,j}$ are distributed as log-gamma random variables. Since traditionally one identified the distribution of the weights of the Hamiltonian for a directed polymer this is the origin of the name (see \cite{SeppLog}).
\end{remark}

%\note{Maybe we do want to look at $x_1$ not $x_N$ hence we need to other Fredholm determinant?}

From now on we will fix that the weight matrix $d$ is an inverse-gamma weight matrix with a solvable parameter matrix. Analogously to the O'Connell-Yor semi-discrete polymer, the hierarchy $f(n)$ evolves as a Markov chain in $n$ with state space in $\R^{N(N+1)/2}$. An explicit construction of this Markov chain as a function of the weights $d$ is given in \cite{COSZ} (appealing to the recursive formulation of the tropical RSK correspondence which is given in \cite{NY}). We will not restate the kernel of this Markov chain on $\R^{N(N+1)/2}$, but only remark that it is not the same as the limiting Markov chain which corresponds to the scaling limit of the dynamics given in Proposition~\ref{prop15}. However, we have the following:

\begin{remark}
The marginal distribution of $\{f^{N}_{\ell}(n)\}_{1\leq \ell \leq N}$ is given by the $\alpha$-Whittaker measure $\alphaWM{a;\alpha,n}$ with $a_j=\theta_j$ for $1\leq j\leq N$ and $\alpha_i= \thetar_i$ for $1\leq i\leq n$. The formulas given in Section~\ref{alphamoment} provide formulas for the moments of certain terms in the hierarchy of free energies. The Fredholm determinant formula in Section~\ref{alphaFred} provides a formula for the expectation of $\exp(-se^{f^N_N(n)})$. It should be noted that this is not the single path free energy of the model, but rather the free energy for a dual single path polymer model. Remark \ref{firstlambdaremark}, however, provides the necessary formulas which when developed in the manner of this paper should yield a Fredholm determine for $f^{N}_1(n)$ as well.
\end{remark}

%\begin{comment}
%Despite this, the marginal distribution of $\{f^{N}_{\ell}(n)\}_{1\leq \ell \leq N}$ is computable and turns out to be a
%
%does evolve as a Markov chain in its own filtration which should coincide with the limiting Markov chain dynamics for $n\geq N$ is given by the pure alpha Whittaker measure.
%
%Before stating this result of \cite{COSZ}, let us introduce the transition kernel which describes the evolution of $\{f^{N}_{\ell}(n)\}_{1\leq \ell \leq N}$ for $n\geq N$.
%
%-- just state the P kernel. Then state the theorem and use our notation not COSZ.
%Define a time $n$ positive kernel from $\R^N$ to itself by
%\begin{equation}\label{PNdef}
%\Pker{N}{\thetarc^{[n]}}(x,d\tilde x) = \frac{\psi_{-\iota\theta}(\tilde{x})}{\psi_{-\iota\theta}(x)}\prod_{i=1}^{N-1} \exp\left\{-e^{\tilde x_{i+1}-x_i}\right\} \prod_{j=1}^{N} \biggl( \Gamma(\thetarc_{n,j})^{-1} e^{(x_j-\tilde x_j)\thetar_{n}} \exp\biggl\{-e^{x_j-\tilde x_j}\biggr\} d\tilde x_j\,\biggr).
%\end{equation}
%Here $\psi$ are Whittaker functions as before.
%
%The following is proved in \cite{COSZ}:
%
%\begin{theorem}
%
%Also mention the Pker is a markov kernel.
%
%state the main result in terms of the alpha WM and alpha WP and bottom row kernel
%
%\end{theorem}
%
%
%
%
%
%
%
%observe equivalence with our model (and that the dynamics on the edge coincide -- remark).
%
%Rephrase our results in terms of this model -- for distribution and for moments
%
%Later: compare to known results (COSZ, timos)
%
%\note{this above section needs to be finished later}
%\end{comment}

\section{Continuum directed random polymer and Kardar-Parisi-Zhang stochastic PDE}
\index{continuum directed random polymer}
\begin{definition}
O'Connell and Warren introduced the {\it partition function hierarchy}\index{continuum directed random polymer!partition function hierarchy} for the continuum directed random polymer (CDRP). It is a continuous function $\mathcal{Z}_{n}(T,X)$ of $(n,T,X)$ varying over $(\Zgzero,\Rplus,\R)$, which is formally written as \glossary{$\mathcal{Z}_{n}(T,X)$}
\begin{equation}\label{wick}
\mathcal{Z}_{n}(T,X)= p(T,X)^n \EE\left[:{\rm exp}:\, \left\{\sum_{i=1}^{n} \int_{0}^{T} \whitenoise(t,b_i(t)) dt \right\} \right]
\end{equation}
where $\EE$ is the expectation of the law on $n$ independent Brownian bridges $\{b_i\}_{i=1}^{n}$ starting at $0$ at time 0 and ending at $X$ at time $T$. The $:{\rm exp}:$ \glossary{$:{\rm exp}:$} is the {\it Wick exponential}\index{continuum directed random polymer!Wick exponential} and $p(T,X)=(2\pi)^{-1/2}e^{-X^2/T}$ is the standard heat kernel. Intuitively these path integrals represent energies of non-intersecting paths, and thus the expectation of their exponential represents the partition function for this multiple path directed polymer model. The Wick exponential must be carefully defined and one approach to doing this is via {\it Weiner-It\^{o} chaos series}: For $n\in \Zgzero$, $T\geq 0$ and $X\in \R$ define \cite{OConWar}
\begin{equation}\label{Zpartfunc}
\mathcal{Z}_{n}(T,X) = p(T,X)^n \sum_{k=0}^{\infty} \int_{\Delta_k(T)}\int_{\R^k} R_k^{(n)}\left((t_1,x_1),\ldots, (t_k,x_k)\right) \whitenoise(dt_1 dx_1)\cdots \whitenoise(dt_k dx_k),
\end{equation}
where $\Delta_k(T) = \{0<t_1<\cdots <t_k<T\}$, and $R_k^{(n)}$ is the $k$-point correlation function for a collection of $n$ non-intersecting Brownian bridges which all start at $0$ at time $0$ and end at $X$ at time $T$. For notational simplicity set $\mathcal{Z}_0(T,X)\equiv 1$. These series are convergent in $L^2(\whitenoise)$. For $n=1$ one readily observes that the above series satisfies the well-posed stochastic heat equation with multiplicative noise and delta function initial data:
\begin{equation}\label{SHE}
\partial_T \mathcal{Z}_1 = \tfrac{1}{2}\partial_X^2 \mathcal{Z}_1 -\mathcal{Z}_1\dot{\mathscr{W}}, \qquad \mathcal{Z}_1(0,X)=\delta_{X=0}.
\end{equation}
In light of this, the Wick exponential provides a generalization of the Feynman-Kac representation\index{Feynman-Kac representation!continuous space}. There are other equivalent approaches to defining the Wick exponential (see \cite{ICreview}).

The entire set of $\mathcal{Z}_{n}(T,X)$ are almost surely everywhere positive for $T>0$ fixed as one varies $n\in \Zgzero$ and $X\in \R$ \cite{CH2}. Hence it is justified to define, for each $T>0$ the {\it KPZ$_T$ line ensemble} which is a continuous $\Zgzero$-indexed line ensemble  $\mathcal{H}^T = \{\mathcal{H}^T_{n}\}_{n\in \Zgzero}: \Zgzero\times \R \to \R$ given by
\begin{equation*}
\mathcal{H}^T_{i}(X) = \log \left(\frac{\mathcal{Z}_n(T,X)}{\mathcal{Z}_{n-1}(T,X)}\right).
\end{equation*}\index{stochastic heat equation}\glossary{$\mathcal{H}^T_{i}(X)$}

Taking $n=1$, $\mathcal{H}^{T}_1(\cdot)$ is the {\it Hopf-Cole solution to the Kardar-Parisi-Zhang stochastic PDE}. Formally, the KPZ stochastic PDE is given by
\begin{equation}\label{KPZ}
\partial_T \mathcal{H}_1 = \tfrac{1}{2}\partial_X^2 \mathcal{H}_1 + \tfrac{1}{2}(\partial_X \mathcal{H}_1)^2 +\whitenoise.
\end{equation}\index{Kardar-Parisi-Zhang equation}
This is an ill-posed equation -- hence its interpretation as the logarithm of the stochastic heat equation. The delta initial data for the stochastic heat equation is called {\it narrow wedge} initial data for KPZ.\index{Kardar-Parisi-Zhang equation!narrow wedge initial data}
\end{definition}

The CDRP can be thought of in terms of the general directed polymer framework. The CDRP defined corresponds to an underlying path measure of Brownian bridge which starts at $0$ at time $0$ and ends at $X$ at time $T$, an inverse temperature $\beta=1$ ($\beta$ can be scaled into $T$ in fact) and a quenched Hamiltonian given by
\begin{equation*}
H_{Q}(\pi(\cdot)) = -\int_{0}^{T} \dot{\mathscr{W}}(t,b(t))dt.
\end{equation*}
The reason for the Wick exponential is due to the roughness of the potential as well as the Brownian path.

\subsection{CDRP as a polymer scaling limit}\label{CDRPlimitres}
The CDRP occurs as limits of discrete and semi-discrete polymers under what has been called {\it intermediate disorder} scaling \index{directed polymers in a random media!intermediate disorder}. This means that the inverse temperature should be scaled to zero in a critical way as the system size scales up. For the discrete directed polymer it was observe independently by Calabrese, Le Doussal and Rosso \cite{CDR} and by Alberts, Khanin and Quastel \cite{AKQ} that if one scaled $\beta$ as $n^{-1/4}$ and the end point of an $n$-step simple walk as $n^{1/2}y$ then the properly scaled partition function converges to CDRP partition function. Using convergence of discrete to continuum chaos series, \cite{AKQ2} provide a proof of this result that is universal with respect to the underlying IID random variables which form the random environment (subject to certain moment conditions).

Concerning the O'Connell-Yor semi-discrete directed polymer, a similar approach to that of \cite{AKQ2} is used in \cite{QM} to prove convergence of the entire partition function hierarchy. Define a scaling coefficient
\begin{equation}\label{scalingconst}
C(N,T,X) = \exp\left(N+ \frac{\sqrt{TN}+X}{2} + XT^{-1/2}N^{1/2}\right) (T^{1/2}N^{-1/2})^N.
\end{equation}
Then for any $n$ fixed, consider (for $N\geq n$), the scaled partition functions
\begin{equation*}
\mathcal{Z}^N_n(T,X)=\frac{\Zsd^{N}_{n}(\sqrt{TN}+X)}{C(N,T,X)^n}.
%\mathcal{Z}^N_n(T,X)=\frac{N^{n(n+1)/4}}{|D_n(\sqrt{N}-TX)|} \Zsd^{N}_{n}(\sqrt{N}-TX) e^{-\frac{1}{2} T^2 n \sqrt{N}},
\end{equation*}
In the next section we will provide evidence that as $N\to \infty$, this rescaled polymer partition function hierarchy converges to that of the CDRP hierarchy $\mathcal{Z}_n(T,X)$. This evidence is at the level of showing that various moment formulas converge to their analogous CDRP formulas. A proof of the convergence of this collection of space-time stochastic processes will appear in forthcoming work \cite{QM}.

%where $D_n(\tau)$ is the Euclidean set of all $n$-tuples of non-intersecting (disjoint) up/right paths with initial points $(0,1),\ldots, (0,n)$ and endpoints $(\tau,N-n+1),\ldots, (\tau,N)$, and $|D_n(\tau)|$ is its Lebesgue measure. Then as $N\to \infty$, this rescaled hierarchy converges to the CDRP partition function hierarchy $\mathcal{Z}_n(T,X)$ (see \cite{QM} for precise statement).

\subsection{Integral formulas for the CDRP}\label{CDRPintSec}

The convergence result of \cite{QM} mentioned in Section~\ref{CDRPlimitres} shows that the semi-discrete polymer partition function hierarchy converges weakly to that of the CDRP. This convergence does not imply convergence of moments (though one might attempt to strengthen it). Without reference to the convergence, it is fairly easy to see the scalings under which to take asymptotics of the moment expressions given in Section~\ref{restatedIntForm}. One can check, after the fact, that these agree with the intermediate disorder scaling of \cite{QM} which strongly suggests that these limits provide expressions for the moments of the CDRP hierarchy. By using replica methods we can confirm this fact for some of these formulas.

In what follows we specialized to zero drifts (i.e., $a_i\equiv 0$). Analogous formulas can be written down for general drifts.

\begin{proposition}\label{prop28CDRP}
Fix $T>0 $, $X\in \R$, and a drift vector $a=(a_1,\ldots,a_N)=(0,\ldots, 0)$. Then for $r\geq~1$,
\begin{equation}\label{eqn28CDRP}
\lim_{N\to \infty} \left\langle \frac{\Zsd^{N}_{r}(\sqrt{TN}+X)}{C(N,T,X)^r} \right\rangle=
\frac{1}{(2\pi \iota)^rr!} \int\cdots\int \prod_{k\neq \ell}^r (z_k-z_\ell)\prod_{j=1}^{r} e^{\frac{T}{2} z_j^2 + X z_j}dz_j,
\end{equation}
where the contours can all be taken to be $\iota \R$.
\end{proposition}
\begin{proof}
This follows from a simple asymptotic analysis of the formula of Proposition~\ref{Proposition27OCon}. Changing variables $w_j$ to $-w_j$ and replacing $\prod_{1\leq k\leq \ell \leq r} (w_k-w_\ell)^2$ by $\prod_{k\neq \ell}^r (w_k-w_\ell)$ clears the powers of $-1$. Under the above scaling, $t=\sqrt{TN}+X$ so that the  $e^{rt/2}$ term cancels the analogous term in the denominator of the left-hand side of (\ref{eqn28CDRP}). The function $e^{tw}/w^N$ which shows up in the integrand can be written as $e^{tw-N\log w}$. This function $tw-N\log w$ has a critical point at $N/t$. Let $w_c$ correspond to the critical point when $X=0$, i.e, $w_c= T^{-1/2} N^{1/2}$. Set $w=w_c+z$ and observe that
\begin{equation*}
e^{tw-N\log w} = e^{tw_c + tz-N\log(w_c+z)} = e^{tw_c - N\log w_c} e^{tz - N\log(1+z/w_c)}.
\end{equation*}
By plugging in the value of $t$ and $w_c$, and expanding $\log(1+z/w_c)$ around $z/w_c=0$ we find
\begin{eqnarray*}
e^{tw-N\log w} &=& e^{N+XT^{-1/2}N^{1/2}-N\log(T^{-1/2}N^{1/2})} e^{T^{1/2}N^{1/2} z + X z - N\tfrac{z}{w_c} + N\tfrac{z^2}{(w_c)^2}+ O(N^{-1/2})}\\
 &=& \frac{e^{N+XT^{-1/2}N^{1/2}}}{(T^{-1/2}N^{1/2})^N} e^{\frac{T}{2} z^2+X z + O(N^{-1/2})}.
\end{eqnarray*}
By deforming the integration contours so as to be $w_c+\iota \R$ and by standard estimates on the decay of $tw-N\log w$ away from the critical point along said contours, we find that in the $N\to \infty$ limit, we are left with $e^{\frac{T}{2}z^2+X z}$ (that is after canceling the diverging prefactor with the rest of the denominator on the left-hand side of (\ref{eqn28CDRP})). The only other term left to consider is the Vandermonde determinant, but since it involves differences of variables, it remains unchanged aside for having $z$'s substituted for $w$'s. This proves the limiting formula.
\end{proof}
\begin{remark}
The above limit result along with the scalings of Section~\ref{CDRPlimitres} suggest the following formula: For $T>0$ and $r\geq 1$,
\begin{equation*}
\langle \mathcal{Z}_r(T,X)\rangle =\frac{1}{(2\pi \iota)^rr!} \int\cdots\int \prod_{k\neq \ell}^r (z_k-z_\ell)\prod_{j=1}^{r} e^{\frac{T}{2} z_j^2+X z_j}dz_j,
\end{equation*}
where the contours can all be taken to be $\iota \R$.
\end{remark}

\begin{proposition}
Fix $T>0 $, $X\in \R$, and a drift vector $a=(a_1,\ldots,a_N)=(0,\ldots, 0)$. Then for $k\geq 1$,
\begin{equation*}
\lim_{N\to \infty} \left\langle  \frac{\left(\Zsd^{N}_{1}(\sqrt{TN}+X)\right)^k}{C(N,T,X)^k}\right\rangle =
\frac{1}{(2\pi \iota)^k} \int\cdots\int \prod_{1\le A<B\le k} \frac{z_A-z_B}{z_A-z_B-1}\prod_{j=1}^{k} e^{\frac{T}{2}z_j^2 + X z_j}dz_j,
\end{equation*}
where the $z_A$-contour is along $C_A+\iota \R$ for any $C_1>C_2+1>C_3+2>\cdots >C_k+(k-1)$.
\end{proposition}
\begin{proof}
This follows immediately from the same type of argument as used in the proof of Proposition~\ref{prop28CDRP} but with the role of Proposition~\ref{Proposition27OCon} replaced by Proposition~\ref{Proposition28OCon}. The only change is that the contours must be ordered correctly in the asymptotics, which introduces the ordering above.
\end{proof}
\begin{remark}\label{BCrem}
In Section~\ref{CDRPreplica} we will prove that the above asymptotics correctly gives the moments of the CDRP suggested by Section~\ref{CDRPlimitres}: For $T>0$,
\begin{equation}\label{Zkform}
\langle \mathcal{Z}_1(T,X)^k\rangle =
\frac{1}{(2\pi \iota)^k} \int\cdots\int \prod_{1\le A<B\le k} \frac{z_A-z_B}{z_A-z_B-1}\prod_{j=1}^{k} e^{\frac{T}{2}z_j^2+X z_j}dz_j,
\end{equation}
where the $z_A$-contour is along $C_A+\iota \R$ for any $C_1>C_2+1>C_3+2>\cdots >C_k+(k-1)$. Note that the above formula agrees with the fact that $\mathcal{Z}_n(T,X)/p(T,X)^n$ is a stationary process in $X$. The stationarity was observed for $n=1$ in \cite{ACQ} as a result of the invariance of space time white noise under affine shift. This argument extends to general $n\geq 1$.

Let us work out the $k=1$ and $k=2$ formulas explicitly. For $k=1$, the above formula gives $\langle \mathcal{Z}_1(T,0)\rangle = (2\pi T)^{-1/2}$ which matches $p(T,0)$ as on expects. When $k=2$ we have
\begin{eqnarray*}
\langle \mathcal{Z}_1(T,0)^2\rangle &=& \frac{1}{(2\pi \iota)^2} \int\int \frac{z_1-z_2}{z_1-z_2-1}e^{\frac{T}{2}(z_1^2+ z_2^2)}dz_1 dz_2\\
&=& \left(\frac{1}{2\pi \iota} \int_{\iota \R} e^{\frac{T}{2} z^2}dz\right)^2 + \frac{1}{(2\pi \iota)^2} \int\int \frac{1}{z_1-z_2-1}e^{\frac{T}{2}(z_1^2+ z_2^2)}dz_1 dz_2\\
&=& \frac{1}{2\pi T} + \int_{0}^{\infty} du e^u \left(\frac{1}{2\pi \iota} \int_{\iota \R} e^{\frac{T}{2} z^2 - uz} dz\right)^2 \\
&=& \frac{1}{2\pi T} + \frac{1}{(2\pi)^2} \frac{2\pi}{T} \int_{0}^{\infty} e^{u-\frac{u^2}{T}} du\\
&=& \frac{1}{2\pi T}\left(1+ \sqrt{\pi T} e^{\frac{T}{4}} \Phi(\sqrt{T/2})\right),
\end{eqnarray*}
where
\begin{equation*}
\Phi(s) = \frac{1}{\sqrt{2\pi}} \int_{-\infty}^{s} e^{-t^2/2}dt,
\end{equation*}
and where to get from the second to third lines we used that since $\Real(z_1-z_2-1)>0$  we may use the linearization formula
\begin{equation*}
\frac{1}{z_1-z_2-1} = \int_0^{\infty} e^{-u(z_1-z_2-1)}du.
\end{equation*}
This formula for $k=2$ matches formula (2.27) of \cite{BC} where it was rigorously proved via local time calculations.
\end{remark}

%\note{generalize below to $N_1>N_2\cdots$}
\begin{proposition}
Fix $T>0 $, $X\in \R$, and a drift vector $a=(a_1,\ldots,a_N)=(0,\ldots, 0)$. Then for $k\geq 1$, $r_{\alpha}\geq 1$ for $1\leq \alpha\leq k$, and $s_1\leq s_2\leq \cdots \leq s_k$,
\begin{eqnarray*}
\lefteqn{\lim_{N\to \infty} \left\langle \prod_{\alpha=1}^{k} \frac{\Zsd^{N_\alpha}_{r_{\alpha}}(\sqrt{TN}+X)}{C(N,T,X)^{r_{\alpha}}}\right\rangle =}\\
&&\prod_{\alpha=1}^{k} \frac{1}{(2\pi \iota)^{r_{\alpha}} r_{\alpha}!} \int \cdots \int \prod_{1\leq \alpha <\beta\leq k} \left(\prod_{i=1}^{r_{\alpha}}\prod_{j=1}^{r_{\beta}} \frac{(z_{\alpha,i}-s_{\alpha})-(z_{\beta,j}-s_{\beta})}{(z_{\alpha,i}-s_\alpha)-(z_{\beta,j}-s_{\beta})-1}\right)\\
 &&\times \prod_{\alpha=1}^{k}\left( \left(\prod_{i\neq j}^{r_{\alpha}}(z_{\alpha,i}-z_{\alpha,j})\right) \left(\prod_{j=1}^{r_{\alpha}}e^{\frac{T}{2} z_{\alpha,j}^2 + X z_{\alpha,j}}dz_{\alpha,j}\right)\right)\,,
\end{eqnarray*}
where $N_{\alpha}=(N^{1/2} - s_{\alpha}T^{1/2})^2$ and the $(z_{\alpha,j}-s_{\alpha})$-contour is along $C_{\alpha}+\iota \R$ for any $C_1>C_2+1>C_3+2>\cdots > C_k+(k-1)$, for all $j\in \{1,\ldots, r_\alpha\}$.
\end{proposition}
\begin{proof}
Follows by asymptotic analysis along the same lines as the above results using Proposition~\ref{Proposition28FULLtzeroOCon}.
\end{proof}
\begin{remark}
The above limit result along with the scalings of Section~\ref{CDRPlimitres} suggest the following formula: Fix $T>0$. Then for $k\geq 1$, $r_{\alpha}\geq 1$ for $1\leq \alpha\leq k$, and $X_1\leq~X_2\leq~\cdots~\leq~X_k$,
\begin{eqnarray*}
&&\left\langle \prod_{\alpha=1}^{k} \mathcal{Z}_{r_\alpha}(T,X_\alpha)\right\rangle = \prod_{\alpha=1}^{k} \frac{1}{(2\pi \iota)^{r_{\alpha}} r_{\alpha}!} \int \cdots \int \prod_{1\leq \alpha <\beta\leq k} \left(\prod_{i=1}^{r_{\alpha}}\prod_{j=1}^{r_{\beta}} \frac{z_{\alpha,i}-z_{\beta,j}}{z_{\alpha,i}-z_{\beta,j}-1}\right)\\
&&\times \prod_{\alpha=1}^{k}\left( \left(\prod_{i\neq j}^{r_{\alpha}}(z_{\alpha,i}-z_{\alpha,j})\right) \left(\prod_{j=1}^{r_{\alpha}}e^{\frac{T}{2} z_{\alpha,j}^2 + X_{\alpha}z_{\alpha,j}}dz_{\alpha,j}\right)\right)\,,
\end{eqnarray*}
where the $z_{\alpha,j}$-contour is along $C_{\alpha}+\iota \R$ for any $C_1>C_2+1>C_3+2>\cdots > C_k+(k-1)$, for all $j\in \{1,\ldots, r_\alpha\}$. The moment is, of course, symmetric in the order of the $X_\alpha$'s -- the integral however, is not (due to the ordering of the contours in the integral).
\end{remark}

\begin{proposition}
Fix $T>0$ and a drift vector $a=(a_1,\ldots,a_N)=(0,\ldots, 0)$. Then for $k$ real numbers $X_1\leq X_2\leq \cdots \leq X_k$,
\begin{equation*}
\lim_{N\to \infty} \left\langle \prod_{i=1}^{k} \frac{\Zsd^{N}_{1}(\sqrt{TN} + X_i)}{C(N,T,X_i)} \right\rangle = \frac{1}{(2\pi \iota)^k} \int\cdots\int \prod_{1\leq A<B\leq k} \frac{z_{A}-z_{B}}{z_{A}-z_{B}-1} \prod_{j=1}^{k} e^{\frac{T}{2}z_j^2 + X_j z_j}dz_j,
\end{equation*}
where the $z_{A}$-contour is along $C_{A}+\iota\R$ for any $C_1>C_2+1>C_3+2>\cdots >C_k+(k-1)$.
\end{proposition}
\begin{proof}
Follows by asymptotic analysis along the same lines as the above results using Proposition~\ref{ktimeintformOCon}.
\end{proof}
\begin{remark}
In Section~\ref{CDRPreplica} we will prove that the above asymptotics correctly gives the moments of the CDRP suggested by Section~\ref{CDRPlimitres}: For $T>0$, and $X_1\leq X_2\leq \cdots \leq X_k$,
\begin{equation}\label{nptSHEform}
\left\langle \prod_{i=1}^{k} \mathcal{Z}_1(T,X_i) \right\rangle =\frac{1}{(2\pi \iota)^k} \int\cdots\int \prod_{1\leq A<B\leq k} \frac{z_{A}-z_{B}}{z_{A}-z_{B}-1} \prod_{j=1}^{k} e^{\frac{T}{2}z_j^2 + X_j z_j}dz_j,
\end{equation}
where the $z_A$-contour is along $C_A+\iota \R$ for any $C_1>C_2+1>C_3+2>\cdots> C_k+(k-1)$.

For example, when $k=2$ this formula leads to the spatial two point function
\begin{equation*}
\langle \mathcal{Z}_1(T,0)\mathcal{Z}_1(T,X)\rangle = \frac{1}{2\pi T} e^{-\frac{X^2}{2T}} \left( 1+ \sqrt{\pi T} e^{\frac{1}{4}(T^{1/2} - \frac{X}{T^{1/2}})^2} \Phi\big(\tfrac{1}{\sqrt{2}}(T^{1/2}-\tfrac{X}{T^{1/2}})\big)\right),
\end{equation*}
or equivalently,
\begin{equation*}
\langle \mathcal{Z}_1(T,0)\mathcal{Z}_1(T,X)\rangle - \langle \mathcal{Z}_1(T,0)\rangle \langle\mathcal{Z}_1(T,X)\rangle = \frac{1}{2\sqrt{\pi T}} e^{\frac{1}{4}(T^{1/2} - \frac{X}{T^{1/2}})^2} \Phi\big(\tfrac{1}{\sqrt{2}}(T^{1/2}-\tfrac{X}{T^{1/2}})\big).
\end{equation*}
\end{remark}

\subsection{Fredholm determinant formula for the CDRP}\label{CDRPfreddetSec}
We present a critical point asymptotic analysis of our Fredholm determinant formula. The critical scaling makes the analysis more involved and we do not attempt a rigorous proof presently but note that in \cite{BorCorFer} this derivation is rigorously performed.

\begin{formal}
Fix $T>0$ and a drift vector $a=(a_1,\ldots,a_N)=(0,\ldots, 0)$. Then, for $u=s \exp\{-N-\tfrac{1}{2} \sqrt{TN} - \tfrac{1}{2} N \log(T/N)\}$,
\begin{equation*}
\lim_{N\to\infty} \langle e^{-u\Zsd^{N}_{1}(\sqrt{TN})}\rangle  =\det(I- K_{se^{-T/24}})_{L^2(0,\infty)}
\end{equation*}
where the operator $K_s$ is defined via is integral kernel
\begin{equation*}
K_s(r,r') = \int_{-\infty}^{\infty} \frac{s}{s+e^{-\kappa_T t}} \Ai(t+r) \Ai(t+r') dt
\end{equation*}
where $\kappa_T= 2^{-1/3}T^{1/3}$.
\end{formal}
\begin{remark}
The above critical point derivation is made rigorous in \cite{BorCorFer}. Assuming the above result it follows rigorously from the results of \cite{QM} explained in Section~\ref{CDRPlimitres} that
\begin{equation*}
\langle e^{-s e^{T/24} \mathcal{Z}(T,0)} \rangle =\det(I- K_{s})_{L^2(0,\infty)}.
\end{equation*}
This Laplace transform can be inverted and doing so one recovers the exact solution to the KPZ equation which was simultaneously and independently discovered in \cite{ACQ,SaSp} and proved rigorously in \cite{ACQ}. Soon after, \cite{CDR} derived this same Laplace transform through a non-rigorous replica approach. From \cite{ACQ} one can easily provide a rigorous proof of the above stated formula. The $T/24$ which shows up is the law of large numbers term for the KPZ equation.
\end{remark}
\begin{proof}[Derivation]
%Now lets turn  to the crossover scaling.
%Let $t=\beta^2 N^{1/2}$ and take
%$$u=u's$$ where
%$$u'= \exp\left\{-N -\frac{\beta^2}{2} N^{1/2} -N\log (\beta^2N^{-1/2}) \right\}$$
%where $s=e^{-a}$.

It may be helpful to the reader to first study the proof of Theorem~\ref{TWasymptoticskappaTHM}, especially the critical point derivation which can be found at the beginning of the proof. The mine-field of poles (coming from the $\Gamma(-s)\Gamma(1+s)$ term) which one encountered in performing the asymptotics does not get scaled away. In the GUE asymptotics, only the pole at zero remained relevant in the scaling limit. In the present limit the poles at every integer contribute non-trivially in the limit as we will see.

The idea will be to manipulate the kernel of the Fredholm determinant into a form which is good for asymptotics and then to (n\"{a}ively) take those asymptotics, disregarding a plethora of mathematical issues (such as poles and tail decay bounds).

We will start from the formula given by Theorem~\ref{OConFredDet} with all $a_m=0$ and $u$ as specified in the statement of the result we seek to prove. The first manipulation is to change variables to set $w=v+s$. This changes the contour of integration for $s$ to a somewhat odd contour for $w$. As this is a formal derivation, we will disregard restrictions on how this contour can be deformed and hence not specify it.

Now recall the identity
\begin{equation*}
s^z\frac{\pi}{\sin(\pi z)} = \int_{-\infty}^{\infty} \frac{se^{zt}}{s+e^t}dt,
\end{equation*}
which, for instance, can be found in \cite{ACQ} page 19. Of course one must restrict the set of $z$ for which this holds, but again we proceed formally, so let us not. Using this we can rewrite our formula in terms of the kernel
\begin{equation*}
K(v,v') = \int \left(\int_{-\infty}^{\infty} \frac{se^{(w-v)t}}{s+e^{t}}dt\right) \frac{1}{w-v'} \frac{ F(v)}{F(w)} dw
\end{equation*}
where
\begin{equation*}
F(z) = (\Gamma(z))^N (u')^{-z} e^{-\frac{\sqrt{TN}}{2}z^2}.
\end{equation*}

Let $T=\beta^4$ in what follows. Perform the change of variables $v=N^{1/2}\tilde v$ and likewise for $v'$, $w$, and $t$. This yields, after taking into account the effect on the determinant $L^2$ space,
\begin{equation*}
\tilde K(\tilde v,\tilde v') = \int \left(\int_{-\infty}^{\infty} \frac{se^{(\tilde w-\tilde v)\tilde t}}{s+e^{N^{1/2}\tilde t}}d\tilde t\right) \frac{1}{\tilde w-\tilde v'}\frac{\exp\{N^{3/2} G(\tilde v)\}}{\exp\{N^{3/2}G(\tilde w)\}} d\tilde w,
\end{equation*}
where
\begin{equation*}
G(z) = N^{-1/2}\log\Gamma(N^{1/2} z)+z + z\tfrac{\beta^2}{2}N^{-1/2} +z\log(\beta^2 N^{-1/2}) - \tfrac{\beta^2}{2}z^2.
\end{equation*}

For $z$ near infinity we have the asymptotics
\begin{equation*}
\log\Gamma(z) = z\log z - z + \frac{1}{2}\left[-\log z + \log 2\pi\right] +t\frac{1}{12 z} +O(z^{-3}).
\end{equation*}
Plugging in this expansion to our formula for $G$ it follows that
\begin{eqnarray*}
G(z)&=& z\log z +z\log \beta^{2} - \tfrac{\beta^2}{2}z^2 + N^{-1/2} \left[-\tfrac{1}{2}\log z + \tfrac{\beta^2}{2}z\right] \\
&&+ N^{-1}\frac{1}{12z} + O(N^{-2})+N^{-1/2}(-\tfrac{1}{2}\log N^{1/2} +\log\sqrt{2\pi}).
\end{eqnarray*}

Observe that $z\log z + z\log \beta^2 - \tfrac{\beta^2}{2}z^2$ has a critical point at $z=\beta^{-2}$ at which the second derivative also disappears. Therefore we expand each grouping of terms in $G(z)$ around this point. Call all order 1 terms $G_1(z)$, all order $N^{-1/2}$ terms $G_2(z)$, and the $N^{-1}$ term $G_3(z)$.
Observe then that
\begin{eqnarray*}
G_1(z) &=& \frac{1}{2}\beta^{-2}  -\frac{\beta^4}{6}(z-\beta^{-2})^3 +O\big((z-\beta^{-2})^4\big),\\
G_2(z) &=& \log\beta +\frac{1}{2} + \frac{\beta^4}{4} (z-\beta^{-2})^2 + O\big((z-\beta^{-2})^3\big),\\
G_3(z) &=& \frac{\beta^2}{12} -\frac{\beta^4}{12} (z-\beta^{-2}) + O\big((z-\beta^{-2})^2\big).
\end{eqnarray*}

Now make the change of variables $\hat z = N^{1/2}(z-\beta^{-2})$ and we find
\begin{equation*}
G(z) = c + N^{-3/2}\hat G(\hat z)+ O(N^{-2}),
\end{equation*}
where
\begin{equation*}\hat G(\hat z) = -\frac{\beta^4}{12}\hat z+ \frac{\beta^4}{4}\hat z^2 -\frac{\beta^4}{6} \hat z^3.
\end{equation*}
The constant $c$ can be explicitly written, but is inconsequential as we ultimately have a ratio of $G$ functions and hence this cancels.

Returning to $\tilde K$ we perform the change of variables $\hat{w} =N^{1/2}(\tilde w - \beta^{-2})$ and likewise for $\hat v$ and $\hat v'$. This yields (after an additional change of variables in the $t$ integral)

\begin{equation*}
\hat K(\hat  v,\hat  v') = \int \left(\int_{-\infty}^{\infty} \frac{se^{(\hat  w-\hat  v)\hat  t}}{s+e^{\hat  t}}d\hat  t\right) \frac{1}{\hat  w-\hat  v'}\frac{\exp\{\hat G(\hat  v)\}}{\exp\{G(\hat w)\}} d\hat w.
\end{equation*}

Let us momentarily consider the contours. The $\hat{w}$ and $\hat{v}$ contours about should be such that $\Real(\hat w-\hat v)\in (0,1)$. In the limit we would like them to become infinite lines parallel to the imaginary axis, such that $\Real(\hat w-\hat v)\in (0,1)$.

%We should now consider the contours. The $v$ contour was originally a small circle around zero and the $w$ contour was such that $\Real(w-v)\in (0,1)$. The ensures the convergence of the $t$ integral. After the first change of variables we can deform the contours so as to avoid poles in such a way that they both go near $\beta^{-2}$, with the $\tilde w$ contour to the right of the $\tilde v$ contour. When we make the second change of variables we can estimate away the rest of the contours which are not near the critical point, so that in terms of the $\hat{w}$ and $\hat{v}$ variables, the contours become infinite lines parallel to the imaginary axis, with $\Real(\hat w-\hat v)\in (0,1)$.

Taking stock of what we have done so far, our Laplace transform is given by $\det(I+\hat K)$ with the contours and kernels as specified above. Let us no longer write the terms in the exponent which go to zero, and rewrite our kernel without any hats or tildes as
\begin{equation*}
\int \left(\int_{-\infty}^{\infty} \frac{se^{(  w-  v)  t}}{s+e^{ t}}d t\right) \frac{1}{  w-  v'}\frac{\exp\{ F(v)\}}{\exp\{F( w)\}} dw,
\end{equation*}
where
\begin{equation*}
F(z) = -\frac{\beta^4}{12} (2z^3-3z^2+z).
\end{equation*}
If we replace $w$ by $w+\tfrac{1}{2}$ and likewise for $v$ and $v'$ the formula for $F$ reduces to
\begin{equation*}
F(z+\tfrac{1}{2}) = -\tfrac{\beta^4}{24}(4 z^3-z)
\end{equation*}
so that we get
\begin{equation*}
\int \left(\int_{-\infty}^{\infty} \frac{se^{(  w-  v)\hat  t}}{s+e^{ t}}d t\right) \frac{1}{w-  v'}\frac{\exp\{ -\tfrac{\beta^4}{6} v^3 + \tfrac{\beta^4}{24} v\}}{\exp\{-\tfrac{\beta^4}{6} w^3 + \tfrac{\beta^4}{24} w\}} dw.
\end{equation*}
Since $\Real(w-v')>0$ we can substitute
\begin{equation*}
\frac{1}{w-v'} = \int_0^{\infty} e^{-r(w-v')}dr.
\end{equation*}
Inserting this in, we find that the kernel $\hat{K}$ is given by
\begin{equation*}
\int \left(\int_{-\infty}^{\infty} \frac{se^{(w- v)  t}}{s+e^{ t}}d t\right) \int_0^{\infty} e^{-r(w-v')}dr\frac{\exp\{ -\tfrac{\beta^4}{6} v^3 + \tfrac{\beta^4}{24} v\}}{\exp\{-\tfrac{\beta^4}{6} w^3 + \tfrac{\beta^4}{24} w\}} dw,
\end{equation*}

This can be factored as $\hat{K}=ABC$ where
\begin{equation*}
A:L^2(0,\infty)\rightarrow L^2(C_{v}), \qquad B: L^2(C_{w})\rightarrow L^2(0,\infty), \qquad C:L^2(C_{v})\rightarrow L^2(C_{w})
\end{equation*}
are integral operators given by their kernels
\begin{eqnarray*}
A(v,r) &=& e^{-rv}, \qquad B(r,w) = e^{rw},\\
C(w,v) &=& \int_{-\infty}^{\infty} dt\frac{s}{s+e^{ t}}\frac{\exp\{ -\tfrac{\beta^4}{6} v^3 + \tfrac{\beta^4}{24} v - tv + r' v\}}{\exp\{-\tfrac{\beta^4}{6} w^3 + \tfrac{\beta^4}{24} w - tw +rw\}}.
%\frac{\exp\{ -\tfrac{\beta^4}{6} v^3 + \tfrac{\beta^4}{12} v-tv-r'v\}}{\exp\{-\tfrac{\beta^4}{6} w^3 + \tfrac{\beta^4}{12} w-tw-rw\}}.
\end{eqnarray*}
The $C_{w}$ and $C_{v}$ represent the contours on which these act.

Now observe that $\det(I+ABC)=\det(I+BCA)$, where $BCA(r,r')$ acts on $L^2(0,\infty)$ with kernel
\begin{equation*}
\int_{-\infty}^{\infty} dt\frac{s}{s+e^{ t}} \left(\int_{C_v}dv \exp\{ -\tfrac{\beta^4}{6} v^3 + \tfrac{\beta^4}{24} v - tv + r' v\}\right)\left(\int_{\Gamma_w}dw \exp\{\tfrac{\beta^4}{6} w^3 - \tfrac{\beta^4}{24} w + tw - rw\}\right).
\end{equation*}
These two interior integrals can be expressed as Airy functions after a change of variables. After recalling that $T=\beta^4$, the resulting expression is exactly that of the kernel $K_{se^{-T/24}}(r,r')$.
\end{proof}

\subsection{CDRP directly from $q$-TASEP}\label{CDRPqTASEP}
%\note{this part is basically just copied from the notes}
We present a soft argument for why $q$-TASEP should be related to the stochastic heat equation.

%This is effectively a microscopic Hopf-Cole transform (similar to that of G\"{a}rtner \cite{G}). A subsequent work will deal with the actual convergence to the stochastic heat equation (as in \cite{BG,ACQ}).

Recall that that particles of $q$-TASEP evolve according to a Markov chain with the rate of increase of $x_N$ by 1 equal to $(1-q^{x_{N-1}-x_N-1})$. We assume that initially $x_1>x_2>\cdots$. Then the update rule will preserve the order. Consider the scaling
\begin{equation*}
q=e^{-\e},\qquad t=\e^{-2}\tau,\qquad x_k+k = \e^{-2}\tau -(k-1)\e^{-1} \log \e^{-1} - \e^{-1} y_k, \quad k\geq 1.
\end{equation*}
In a time interval $dt= \e^{-2} d\tau$ (small in the rescaled time $\tau$ but large in $q$-TASEP time $t$). Then the change in $x_N$ over the time interval $dt$ is equal to the jump rate
\begin{equation*}
1-q^{x_{N-1}-x_N-1} = 1- e^{-\e(\e^{-1}\log\e^{-1} + \e^{-1}(y_N-y_{N-1})-1)} = 1 -\e e^{y_{N-1}-y_N}
\end{equation*}
that remains approximately constant over $dt$, multiplied by the number of Poisson jumps over $dt$. The latter can be thought of as approximately $\e^{-2} d\tau + \e^{-1} (B_N(\tau+d\tau)-B_N(d\tau))$, where $B_N$ is a Brownian motion which comes from the limit of the centered Poisson jumps of the clock from particle $N$. Thus,
\begin{eqnarray*}
\e^{-2} d\tau - \e^{-1}\big(y_N(\tau+d\tau)-y_N(\tau)\big) = x_N(t+dt)-x_N(t) \\
= \big(1 -\e e^{y_{N-1}-y_N}\big) \big(\e^{-2} d\tau + \e^{-1} (B_N(\tau+d\tau)-B_N(d\tau))\big).
\end{eqnarray*}
Terms of order $\e^{-2}$ cancel, and collecting $\e^{-1}$ terms gives
\begin{equation*}
e^{y_N} \frac{ y_N(\tau+d\tau) - y_N(\tau)}{d\tau} = e^{y_{N-1}} e^{y_N} \frac{(B_N(\tau+d\tau)-B_N(d\tau))}{d\tau}
\end{equation*}
or after rearranging and taking $d\tau\to 0$,
\begin{equation*}
\frac{d}{d\tau} e^{y_N} = e^{y_{N-1}} - \dot{B}_N(\tau)  e^{y_N}.
\end{equation*}
Note that this equation is satisfied by the O'Connell-Yor partition function by setting $e^{y_N}(\tau)= \Zsd^{N}(\tau)$. If we pass to $\tilde{\Zsd}^N =e^{-\tau} \Zsd^N$, which means that $\tilde{\Zsd}^N$ is the expectation of Poisson point paths taking $N-1$ jumps from time 0 to $\tau$, then
\begin{equation*}
\frac{d}{d\tau} \tilde{\Zsd}^N = (\tilde{\Zsd}^{N-1}-\tilde{\Zsd}^N) - \dot{B}_N(\tau) \tilde{\Zsd}^N.
\end{equation*}
This is a discretized version of the stochastic heat equation with multiplicative noise, and a further limit reproduces the continuum stochastic heat equation.

This also suggests the existence of a microscropic Hopf-Cole transform for $q$-TASEP, similar to that of \cite{G} for ASEP. This transform and the analog to the work of \cite{BG,ACQ} will be the subject of a subsequent work.

At this point, however, we would like to figure out the scaling that takes $q$-TASEP to the stochastic heat equation. This can be done by studying under what scalings the integral representations for moments of $q$-TASEP converge to those for the stochastic heat equation. Recall that $x_N(t)=\lambda^{(N)}_N-N$ is the location of particle $N$ at time $t$ where initially for $N\geq 1$, $x_N(0)=-N$. Let us assume all $q$-TASEP speeds $a_k\equiv 1$.

\begin{proposition}
Consider the following scaling limit of $q$-TASEP. Let $q=e^{-\e}$ and $\tau:= \e^4 t^2/N$. Let $\e\to 0$, $t\to \infty$ and $N\to \infty$ so that $\e^{-2} \ll t\ll\e^{-3}$, $N\ll \e^{-2}$, and $\tau$ has a limit. Define $\ell^{(N)}_N(\tau)$ by
\begin{equation*}
\lambda^{(N)}_N(\tau) = t- \frac{\e^3 t^2}{\tau} - \frac{\e^3 t^2}{\tau}\left| \log\left(\frac{\e^3 t}{\tau}\right)\right| +\e^{-1}\log \e^{-1} + \e^{-1} \ell^{(N)}_N(\tau).
\end{equation*}
Then for any $k\geq 0$,
\begin{equation*}
\lim_{\e\to 0,t\to \infty,N\to \infty} \left\langle e^{-k \ell^{(N)}_N(\tau)}\right\rangle  = \frac{1}{(2\pi \iota)^k} \int \cdots \int \prod_{1\le A<B\le k} \frac{z_A-z_B}{z_A-z_B-1}\prod_{j=1}^k e^{\frac{\tau}{2} z_j^2}dz_j\,,
\end{equation*}
where the $z_A$-contour is along $C_A+\iota \R$ for any $C_1>C_2+1>C_3+2>\cdots >C_k+(k-1)$. Note that the right-hand side is as in Remark \ref{BCrem} or Proposition~\ref{ADBGprop} (with all $x$'s set to 0).
\end{proposition}

\begin{proof}
We start with Proposition~\ref{prop8tzero} which shows
\begin{equation*}
\langle q^{k\lambda^N_N(t)}\rangle = \frac{(-1)^{k}q^{\frac{k(k-1)}2}}{(2\pi \iota)^k} \int\cdots\int \prod_{1\le A<B\le k} \frac{z_A-z_B}{z_A-qz_B} \prod_{j=1}^k e^{F(z_j)}\frac{dz_j}{z_j},
\end{equation*}
where $z_j$-contour contains $\{qz_{j+1},\dots,qz_k,1\}$ and no other singularities, and where
\begin{equation*}
F(z) = (q-1)t z_j - N\log(1-z).
\end{equation*}
We will perform steepest descent analysis. The exponential function $F(z)$ has a critical point at $z_c$ such that $F'(z_c)=0$. This is solved by
\begin{equation*}
z_c = 1-\frac{N}{(1-q)t}.
\end{equation*}
With our assumptions we have
\begin{equation*}
1-z_c = \frac{N}{(1-q)t} \sim \frac{\e^4 t^2/\tau}{\e t} = \frac{ \e^{3} t}{\tau},
\end{equation*}
whence $\e\ll 1-z_c \ll 1$. We will deform the contours to form longer (in $\e$ scale) stretches of vertical lines near $z_c$ with distance of order $\e$ between the lines. The standard steepest descent type arguments allow to estimate away the contributions away from $z_c$. Let us do the change of variables
\begin{equation*}
w_j = \e^{-1} (z_j - z_c), \qquad j=1,\ldots, k.
\end{equation*}
Note that under this change, $\frac{dz_j}{z_j}$ goes to $\frac{\e dw_j}{\e w_j +z_c}$ which behaves like $\e dw_j$ for $\e$ small. This account for the $\e^{-1}\log\e^{-1}$ term in the scalings.
We have
\begin{equation*}
\frac{z_A-z_B}{z_A-qz_B} = \frac{\e(w_A-w_B)}{\e(w_A-w_B)+(1-q)(z_c+O(\e))} = \frac{w_A-w_B}{w_A-w_B + 1+o(1)},
\end{equation*}
\begin{eqnarray*}
F(z_j) - F(z_c) &=& \frac{1}{2} F''(z_c)(\e w_j)^2 + lot = \frac{1}{2} \frac{N\e^2}{(1-z_c)^2} w_j^2 + \textrm{lot}\\
&=& \frac{1}{2} \frac{ \e^4 t^2/\tau \cdot \e^2}{(\e^3t/\tau)^2} w_j^2 +lot = \frac{\tau}{2} w_j^2 + \textrm{lot},
\end{eqnarray*}
(where $\textrm{lot}$ represents lower order terms) and finally
\begin{eqnarray*}
F(z_c) &=& (q-1)t\left(1-\frac{N}{(1-q)t}\right) - N\log \left(\frac{N}{(1-q)t}\right) \\
       &=& (q-1)t+ Nt -N\log\left(\frac{N}{\e t}(1+O(\e))\right)\\
       &=& \left(-\e t + \frac{\e^4 t^2}{\tau} - \frac{\e^4 t^2}{\tau} \log\left(\frac{\e^3 t}{\tau}\right)\right)(1+O(\e)).
\end{eqnarray*}
Hence,
\begin{equation*}
q^{\lambda^{(N)}_N(t) -\e^{-1} \ell^{(N)}_{N}(\tau)} = q^{t-\frac{\e^3 t^2}{\tau} + \frac{\e^3 t^2}{\tau} \log(\frac{\e^3 t}{\tau})} = e^{-\e(t-\frac{\e^3t^2}{\tau} + \frac{\e^3 t^2}{\tau}\log(\frac{\e^3 t}{\tau}))} = e^{F(z_c)}(1+o(1)),
\end{equation*}
and substituting into the formula for $\langle q^{k\lambda^{(N)}_N}\rangle$ and changing the signs of the integration variables $\{w_j\}$, we obtain the desired relation.
\end{proof}

The above proposition provides the time and fluctuation scaling for $q$-TASEP to the stochastic heat equation, however it does not give the spatial scaling. For the stochastic heat equation we have joint moments, but not for $q$-TASEP, so a similar calculation can not presently be made. The spatial scaling should also come up naturally from the rescaling of the $q$-TASEP dynamics (after the microscopic Hopf-Cole transform) necessary to converge to the stochastic heat equation. We pursue this in a subsequent work.

\chapter{Replicas and quantum many body systems}\label{replicassec}

\section{Replica analysis of semi-discrete directed polymers}\label{replicaSemidisc}

\subsection{The many body system}
\index{many body system!discrete space}
Consider an infinitesimal generator $L$ for a continuous time autonomous (time-homogeneous) Markov process on $\Z$ and independent $\Z$-indexed one dimensional white noises $\{W(\cdot,i)\}_{i\in \Z}$ where the $\cdot$ represents time (which varies over $\Rplus$). We can define a directed polymer partition function $Z_{L,\beta}(t,x)$ as
\begin{equation*}
Z_{L,\beta}(t,x) = E_{\pi(0)=0} \left[\delta(\pi(t)=x) \exp\left\{\int_{0}^{t} \beta W(s,\pi(s))\right\}\right],
\end{equation*}
where $\delta$ is the Kronecker delta, and the expectation is with respect to the path measure on $\pi(\cdot)$ given by the trajectories of the continuous Markov process started at 0 with infinitesimal generator given by $L$. There is a polymer measure for which this is the partition function -- it is the reweighting of the above mentioned path measure with respect to the Boltzmann weights given by $\exp\left\{\int_{0}^{t} \beta W(s,\pi(s))\right\}$. In Section~\ref{PAMsec} we saw that, via the Feynman-Kacs representation and time reversal, this provides a solution to an initial value problem $\partial_t u = L^* u - Vu$ with $L^*$ the adjoint of $L$ (resulting from the time-reversal).

Define
\begin{equation}\label{barZ}
\bar{Z}_{L,\beta}(\vec{x};t) = \left\langle \prod_{i=1}^{k}Z_{L,\beta}(t,x_i)\right\rangle,
\end{equation}
where $\langle \cdot \rangle$ denotes the average with respect to the white noises.

Define $\Omega^k_{\vec{x};t}$ to be the set of all $k$ paths $\vec{\pi}=\{\pi_i(\cdot)\}_{i=1}^{k}$ which, at time $t$ are such that $\pi_i(t)=x_i$ and $\pi_i(0)=0$ for all $1\leq i\leq k$.
Denote $W(s,\vec{\pi}(s)) = \sum_{i=1}^{k} W(s,\pi_i(s))$.
Then
\begin{equation*}
\bar{Z}_{L,\beta}(\vec{x};t)  = \int_{\Omega^k_{\vec{x};t}} d\vec{\pi} \left\langle \exp\left\{\int_{0}^{t} \beta W(s,\vec{\pi}(s))\right\}\right\rangle.
\end{equation*}
The outer integral is with respect to the Lebesgue measure on $\Omega^k_{\vec{x};t}$.

\begin{remark}
Any finite range generator $L$ can be written (possibly, using time homothety) via its action on functions $f:\Z\to \R$ as
\begin{equation}\label{LStar}
Lf(x) = \sum_{i\in I} p_i f(x+i) - f(x),
\end{equation}
where $p_i>0$, $\sum_{i\in I} p_i=1$, and $I$ is a finite set such that $p_i=0$ for $i\notin I$. These $p_i$ represent the jump rates of the process from $x$ to $x+i$. The adjoint of $L$ is denoted by $L^*$ and given by
\begin{equation*}
L^*g(x) = \sum_{i\in I} p_i g(x-i) - g(x),
\end{equation*}
so that $(Lf,g)=(f,L^* g)$ where $(f,g) =\sum_{x\in \Z} f(x)g(x)$. The adjoint generator corresponds to running the process backwards and accounts for the reversal in the jump directions: $p_i$ is now the jump rates of the process from $x$ to $x-i$. The $k$-fold tensor product of the adjoint generator $L^*$ is defined by its action on functions $g:\Z^k\to \R$ via
\begin{equation*}
L^{*,k}g(x_1,\ldots, x_k) = \sum_{j=1}^{k} \sum_{i\in I} p_i g(x_1,\ldots, x_j-i,\ldots, x_k) - k g(x_1,\ldots, x_k).
\end{equation*}

The O'Connell-Yor semi-discrete directed polymer corresponds to the generator
\begin{equation}\label{OCongen}
Lf(x) = f(x+1) - f(x),
\end{equation}
whereas the parabolic Anderson model corresponds to the generator
\begin{equation}\label{PAMgeb}
Lf(x) = \tfrac{1}{2}f(x+1)+\tfrac{1}{2}f(x-1) - f(x).
\end{equation}
That is to say, in the first case, the continuous time random walk only increases at rate one, whereas in the second case, the continuous time random walk is simple and symmetric.
\end{remark}

\begin{remark}
The following result is not new (cf. \cite{CarMol} Theorem 2.3.2) but we include a proof below for completeness.
\end{remark}
\begin{proposition}\label{DiscreteMBP}
Assume $L$ is of the form (\ref{LStar}). For $k\geq 1$ and $\vec{x}\in \Z^k$, $\bar{Z}_{L,\beta}(\vec{x};t)$ is a solution to the following many-body system:
\begin{equation*}
\partial_t u(\vec{x};t) = H u(\vec{x};t), \qquad \textrm{where} \qquad H = L^{*,k} +\frac{1}{2}\beta^2 \sum_{a,b=1}^{k} \delta(x_a-x_b)
\end{equation*}
with $L^{*,k}$ is the k-fold tensor product of the adjoint of $L$, subject to the initial condition
\begin{equation*}
\bar{Z}_{L,\beta}(\vec{x};0) = \prod_{i=1}^{k} \delta(x_i=0),
\end{equation*}
and the symmetry condition that $u(\sigma \vec{x};t) = u(\vec{x};t)$ for any permutation $\sigma\in S_k$.
\end{proposition}

\begin{proof}

This result relies heavily on the fact that the random environment given by the white noises $W(\cdot,i)$ is Gaussian. It is the Gaussian nature of the noise which makes the interaction term $\sum_{a,b=1}^{K} \delta(x_a-x_b)$ only two-body (i.e., only involving pairs of two variables rather than more). The essence of the proof is to consider a small change in the time and to derive the differential effect of this $dt$ change on $\bar{Z}_{L,\beta}(\vec{x};t)$. Since the jumps of $\pi(\cdot)$ are determined by Poisson events at rates $p_i$, the probability of multiple jumps is sufficiently unlikely. This makes the analysis easy since we can separately consider the small $p_i dt$ probability of different jumps, without considering the effect of their joint occurrence.

%First observe that uniqueness is clear as is the fact that $\bar{Z}_{L,\beta}(\vec{x};t)$ satisfies the initial condition at $t=0$. Likewise, the symmetry condition is clear.

The symmetry condition is clear. In order to prove the evolution equation, consider the differential increment of time $dt$ and for simplicity of notation we will now write (recalling the notation $\Omega^k_{\vec{x};t}$ introduced after (\ref{barZ}))
\begin{eqnarray*}
A&=&\{\Omega^k_{\vec{x};t+dt}\cap \Omega^k_{\vec{x};t}\}\\
B&=&\{\Omega^k_{\vec{x};t+dt}\cap (\Omega^k_{\vec{x};t})^{c}\}\\
C&=&\{(\Omega^k_{\vec{x};t+dt})^{c}\cap \Omega^k_{\vec{x};t}\}.
\end{eqnarray*}
The set $A$ represents paths which are at $\vec{x}$ at time $t$ and at time $t+dt$; set $B$ represents paths at $\vec{x}$ at time $t$ but not at time $t+dt$; and set $C$ represents paths at $\vec{x}$ at time $t+dt$ but not at time $t$. Observe then that
\begin{equation*}
\Omega^k_{\vec{x};t+dt} = A\cup B \qquad \textrm{and},\qquad \Omega^k_{\vec{x};t} = A\cup C.
\end{equation*}

Therefore
\begin{equation*}
\bar{Z}_{L,\beta}(\vec{x};t+dt) - \bar{Z}_{L,\beta}(\vec{x};t) = (1)+(2)-(3)-(4),
\end{equation*}
where
\begin{eqnarray*}
(1) &=& \int_{A}d\vec{\pi} \left\langle \exp\left\{\int_0^t \beta W(s,\vec{\pi}(s))ds + \int_t^{t+dt} \beta W(s,\vec{\pi}(s))ds \right\} \right\rangle\\
(2) &=& \int_{B}d\vec{\pi} \left\langle \exp\left\{\int_0^t \beta W(s,\vec{\pi}(s))ds + \int_t^{t+dt} \beta W(s,\vec{\pi}(s))ds \right\}  \right\rangle\\
(3) &=& \int_{A}d\vec{\pi} \left\langle \exp\left\{\int_0^t \beta W(s,\vec{\pi}(s))ds \right\}  \right\rangle\\
(4) &=& \int_{C}d\vec{\pi} \left\langle \exp\left\{\int_0^t \beta W(s,\vec{\pi}(s))ds \right\}  \right\rangle.
\end{eqnarray*}

Now we use the independence of the intervals of the white-noise as well as the fact that
\begin{equation*}
\left\langle \exp\left\{\int_{a}^{b}\beta \ell W(s,x)ds\right\}\right\rangle = e^{\frac{1}{2} \beta^2 \ell^2 (b-a)}.
\end{equation*}

The first term of the four above can be rewritten as
\begin{equation*}
(1) = \frac{\int_{A} d\vec{\pi}}{\int_{\Omega^k_{\vec{x};t}} d\vec{\pi}} \bar{Z}_{L,\beta}(\vec{x};t) \exp\left\{\tfrac{1}{2}\beta^2dt \sum_{a,b=1}^{k} \delta(x_a-x_b)\right\}.
\end{equation*}
The last part of the above expression deserves explanation: On the event $A$, all particles stay in the same place for the time $dt$. This means that the contribution from the exponential of the integral of the white noise from $t$ to $t+dt$ needs to take into account how many particles are at the same place. The contribution we get is the square of this number, summed over all clusters of particles. The above expression in terms of $\delta(x_a-x_b)$ is conveniently the same as this. This is why Gaussian noise produces two-body interactions, and no more.

Thus, up to an $O(dt^2)$ correction, the first term is
\begin{equation*}
(1) = \bar{Z}_{L,\beta}(\vec{x};t) \left(\frac{\int_{A} d\vec{\pi}}{\int_{\Omega^k_{\vec{x};t}} d\vec{\pi}} +\tfrac{1}{2}\beta^2 \sum_{a,b=1}^{k} \delta(x_a-x_b) dt + O(dt^2)\right).
\end{equation*}

Turn now to the third term and likewise express it as
\begin{equation*}
(3) = \bar{Z}_{L,\beta}(\vec{x};t) \left(\frac{\int_{A} d\vec{\pi}}{\int_{\Omega^k_{\vec{x};t}} d\vec{\pi}}\right).
\end{equation*}
Thus
\begin{equation*}
(1) - (3) = \left( \tfrac{1}{2}\beta^2 \sum_{a,b=1}^{k} \delta(x_a-x_b) dt + O(dt^2) \right) \bar{Z}_{L,\beta}(\vec{x};t).
\end{equation*}

The second term can be expressed as
\begin{equation*}
(2) =  \sum_{j=1}^{k}\sum_{i\in I} \frac{\int_{B_{j,-i}} d\vec{\pi}}{\int_{\Omega^k_{\vec{x};t}} d\vec{\pi}} \bar{Z}_{L,\beta}(\vec{x}_{j,-i};t) \exp\left\{\tfrac{1}{2}\beta^2dt \sum_{a,b=1}^{k} \delta(x_a-x_b)\right\} + O(dt^2),
\end{equation*}
where $\vec{x}_{j,i} = (x_1,\ldots, x_{j-1},x_{j}+i,x_{j+1},\ldots,x_k)$ and $B_{j,i}$ is the event that $\vec{\pi}(t) = \vec{x}_{j,i}$ and $\vec{\pi}(t+dt) = \vec{x}$ (which partitions $B$ up to an error of $O(dt^2)$ which comes from two jumps in the interval $dt$).
Now observe that
\begin{equation*}
\frac{\int_{B_{j,-i}} d\vec{\pi}}{\int_{\Omega^k_{\vec{x};t}} d\vec{\pi}} = p_i dt+O(dt^2)
\end{equation*}
due to the fact that up jumps from $x_j-i$ to $x_j$ occur at rate $p_i$ and in a short interval $dt$ by probability of more than one such jump is $O(dt^2)$.
Therefore we have
\begin{equation*}
(2) =  \sum_{j=1}^{k}\sum_{i\in I}p_i \bar{Z}_{L,\beta}(\vec{x}_{j,-i};t) dt + O(dt^2),
\end{equation*}
where the exponential term also disappeared since it was $1+O(dt)$. The minus sign in front of $i$ above is where the adjoint $L^*$ comes in.
Likewise one shows that
\begin{equation*}
(4) =  \sum_{i=1}^{k} \bar{Z}_{L,\beta}(\vec{x};t)dt + O(dt^2).
\end{equation*}

Combining $(1)+(2)-(3)-(4)$ shows that
\begin{equation*}
\frac{\bar{Z}_{L,\beta}(\vec{x};t+dt) - \bar{Z}_{L,\beta}(\vec{x};t)}{dt} = H\bar{Z}_{L,\beta}(\vec{x};t)+O(dt)\bar{Z}_{L,\beta}(\vec{x};t).
\end{equation*}
As $dt$ goes to zero (due to a priori boundedness of $\bar{Z}_{L,\beta}(\vec{x};t)$) and hence this converges to the desired relation $\partial_t \bar{Z}_{L,\beta}(\vec{x};t) = H\bar{Z}_{L,\beta}(\vec{x};t)$, This proves the proposition.
\end{proof}

\subsection{Solving the semi-discrete quantum many body system}

Proposition~\ref{DiscreteMBP} gives us a way to check the previously derived formula (\ref{Proposition28OCon}) for the $k$-th moment of the O'Connell-Yor semi-discrete directed polymer: $\left\langle  \prod_{i=1}^k \Zsd^{N_i}_{1}(t) \right\rangle$. Observe that for the generator $L$ in (\ref{OCongen}), and $\beta=1$
\begin{equation*}
\bar{Z}_{L,1}(\vec{N};t) = e^{-kt} \left\langle  \prod_{i=1}^k \Zsd^{N_i}_{1}(t) \right\rangle.
\end{equation*}
The reason for the $e^{-kt}$ factor is that in the definition of $\bar{Z}_{L,1}(\vec{N};t)$, the averaging is with respect to a Poisson process $\pi(\cdot)$ whereas in $\Zsd^{N_i}_{1}(t)$ the averaging is over the simplex of volume $t^{N-1}/(N-1)!$, not $e^{-t}t^{N-1}/(N-1)!$ as the Poisson process dictates.

Inspired by the form of Proposition~\ref{Proposition28OCon} one may write down an ansatz for the solution to the many body problem in Proposition~\ref{DiscreteMBP} which will give us a formula for $\bar{Z}_{L,1}(\vec{N};t)$.

\begin{proposition}\label{discDBG}
Fix $L$ as in (\ref{OCongen}) and $k\geq 1$. For $\vec{N}=(N_1,\ldots,N_k)$ such that $N_1\geq N_2\geq \cdots \geq N_k$, the following integral solves the many body problem given in Proposition~\ref{DiscreteMBP} but with initial condition given by delta functions at $1$ rather than 0:
\begin{equation}\label{ZbarL1}
\bar{Z}_{L,1}(\vec{N};t) = e^{-kt/2}\frac{1}{(2\pi \iota)^k} \int\cdots\int \prod_{1\le A<B\le k} \frac{w_A-w_B}{w_A-w_B-1}\prod_{j=1}^k \frac{e^{t w_j}}{w_j^{N_j}} dw_j\,,
\end{equation}
where the $w_j$-contour contains $\{w_{j+1}+1,\cdots,w_k+1,0\}$ and no other singularities for $j=1,\dots,k$.
\end{proposition}
\begin{remark}
This integral formula is only valid on the subspace where $N_1\geq N_2\geq \cdots \geq N_k$ and must be extended outside that subspace by symmetry. Its value when the $N_j$ are unordered does not give the polymer moment. By taking all of the $N_j=N$ we recover a proof of Proposition~\ref{Proposition28OCon}.
\end{remark}

\begin{remark}\label{PAMrem}
One might hope that the solution to the many body system which corresponds to the moments of the general nearest neighbor parabolic Anderson model (continuous time, discrete space) with delta function initial data can also be written in terms of the above ansatz. In \cite{BorCorPAM} this is accomplished for the second moment and thus yields an exact contour integral formula for the second moment.
\end{remark}

\begin{proof}
We need to check the initial condition and the evolution equation. Let us start with the initial condition. When $t=0$ the formula gives
\begin{equation*}
\bar{Z}_{L,1}(\vec{N};0) =\frac{1}{(2\pi \iota)^k} \int\cdots\int \prod_{1\le A<B\le k} \frac{w_A-w_B}{w_A-w_B-1}\prod_{j=1}^k \frac{1}{w_j^{N_j}} dw_j.
\end{equation*}
The contour for $w_1$ can be expanded to infinity without crossing any poles except possibly at infinity. If $N_1=1$ there is a first order pole at infinity for $w_1$, and for $N_1>1$ there is no pole. Thus if $N_1>1$, $\bar{Z}_{L,1}(\vec{N};0)=0$. This shows $N_1\leq 1$. Likewise, we can collapse $w_k$ to zero. Doing this we find the $\bar{Z}_{L,1}(\vec{N};0)=0$ unless $N_k\geq 1$. Combined this implies $N_1=\cdots =N_k=1$ gives the only nonzero contribution to $\bar{Z}_{L,1}$ (which can be checked to equal 1).

We now turn to the proof that the evolution equation is satisfied by this integral formula. Let us first assume all $N_j$ are distinct. Taking the time derivative of $\bar{Z}_{L,1}(\vec{N};t)$ we can bring the differentiation inside the integrals to show that
\begin{equation*}
\partial_t \bar{Z}_{L,1}(\vec{N};t) = e^{-kt/2}\frac{1}{(2\pi \iota)^k} \int\cdots\int \prod_{1\le A<B\le k} \frac{w_A-w_B}{w_A-w_B-1} \left(-\tfrac{k}{2} + w_1+\cdots + w_k\right)\prod_{j=1}^k \frac{e^{t w_j}}{w_j^{N_j}} dw_j.
\end{equation*}
On the other hand, one immediately checks that the effect of applying the discrete derivatives to $\bar{Z}_{L,1}(\vec{N};t)$ also brings down a factor of $w_j-1$ for each $j=1,\ldots, k$. Summing these factors up gives $w_1+\cdots +w_k - k$. From this we must subtract $\tfrac{1}{2}k$ due to the fact that the $k$ of the terms in the delta interaction are 1 (i.e., when $a=b$). Thus the effect of applying the Hamiltonian $H$ to $\bar{Z}_{L,1}(\vec{N};t)$ is to introduce a factor of $w_1+\cdots +w_k-\tfrac{k}{2}$ in the integral. This, however, matches with the formula we got from taking a time derivative, and hence implies the evolution equation holds.

When the $N_j$ are not all distinct, the argument becomes slightly more involved. Assume that the $N_j$ form several groups within which each $N_j$ has the same value. Each group can be treated similarly so let us consider the group of $N_1$: $N_1=N_2=\cdots = N_m=n$ for some $1\leq m\leq k$. We now make a basic observation:
\begin{lemma}\label{switchcontours}
For any $i=1,\ldots ,k-1$, and any function of $k$ variables $f(w_1,\ldots, w_k)$ such that $f(w_1,\ldots, w_i,w_{i+1},\ldots w_k)=f(w_1,\ldots, w_{i+1},w_{i},\ldots w_k)$ (i.e., symmetric in the $i$ variable),
\begin{equation*}
\int\cdots\int \prod_{1\le A<B\le k} \frac{w_A-w_B}{w_A-w_B-1}  (w_i-w_{i+1}-1) f(w_1,\ldots, w_k) \prod_{j=1}^k dw_j = 0,
\end{equation*}
where the $w_j$-contour contains $\{w_{j+1}+1,\cdots,w_k+1,0\}$ and no other singularities for $j=1,\dots,k$.
\end{lemma}
\begin{proof}
The $w_{i}-w_{i+1}-1$ factor cancels with one term in the denominator of the product over $A<B$. This kills the pole for $w_{i}$ associated with $w_{i+1}+1$. This means that we can deform the $w_i$ contour to the $w_{i+1}$ contour. We can then relabel to switch the labels of $w_{i}$ and $w_{i+1}$. This does not change the value of $f$ (due to its symmetry). On the other hand, due to the anti-symmetry of the Vandermonde determinant, this results in sign change and nothing more. The upshot of the asymmetry is that this shows that the integral we are considering is equal to its negative, and hence zero.
\end{proof}
Applying this lemma, it follows that multiplying the integrand by $w_{1}+\cdots + w_{m}$ has the same effect as multiplying it by
\begin{equation*}
(w_m+m-1) + (w_m + m-2) + \cdots + w_m = m w_m + \frac{m(m-1)}{2}.
\end{equation*}
Multiplication by $w_m$ is equivalent to reducing one of the number $N_1,\ldots, N_m$ by one. One the other hand, letting $\nabla_j f(N) = f(N-1)-f(N)$, one sees that applying $\nabla_j+1$ for $1\leq j \leq m$ also has the effect of reducing one of $N_1,\ldots, N_m$ by one.  Thus, multiplication of the integrand by $(w_1+\cdots + w_m)$ is equivalent to
\begin{equation*}
\left(\sum_{j=1}^{m} (\nabla_j +1) + \frac{m(m-1)}{2}\right) \bar{Z}_{L,1}(\vec{N};t).
\end{equation*}
In addition to the sum $w_1+\cdots + w_m$, differentiation with respect to $t$ also introduces an additional $-\tfrac{m}{2}$ (for each grouping of equal $N$'s). Thus we get
\begin{equation*}
\left(\sum_{j=1}^{m} \nabla_j + \frac{m^2}{2}\right) \bar{Z}_{L,1}(\vec{N};t).
\end{equation*}
Note that for a grouping of size $m$ of equal $N$'s, the term $\sum_{a,b} \delta(N_a-N_b) = m^2$. Thus, summing these over each grouping of each $N$'s we find that this is equivalent to $H\bar{Z}_{L,1}(\vec{N};t)$, as desired.
\end{proof}

\section{Delta Bose gas and the continuum directed polymer}\label{CDRPreplica}
\index{many body system!continuous space}
Let $\Weyl{N}=\{x_1<x_2<\cdots<x_N\}$ be the Weyl chamber. \glossary{$\Weyl{N}$}\index{Weyl chamber}

\begin{definition}
A function $u:\Weyl{N}\times \Rplus\to \R$ solves the delta Bose gas with coupling constant $\kappa\in \R$ and delta initial data if:
\begin{itemize}
\item For $x\in  \Weyl{N}$,
\begin{equation*}
\partial_t u = \tfrac{1}{2}\Delta u,
\end{equation*}
\item On the boundary of $\Weyl{N}$,
\begin{equation*}
(\partial_{x_{i}}-\partial_{x_{i+1}}-\kappa)u \big\vert_{x_{i+1}=x_{i}+0} = 0,
\end{equation*}
\item and for any $f\in L^{2}(\Weyl{N})\cap C_b(\overline{\Weyl{N}})$, as $t\to 0$
\begin{equation*}
N! \int_{\Weyl{N}} f(x) u(x;t) dx \to f(0).
\end{equation*}
\end{itemize}
\index{many body system!delta Bose gas}
When $\kappa>0$ this is called the attractive case, whereas when $\kappa<0$ this is the repulsive case. By scaling one can assume $\kappa=\pm 1$.
\end{definition}

Note that the boundary condition is often included in PDE so as to appear as
\begin{equation*}
\partial_t u = \tfrac{1}{2}\Delta u + \tfrac{1}{2} \kappa \sum_{i\neq j} \delta(x_i-x_j)u.
\end{equation*}

\begin{remark}\label{CDRPDBG}
It is widely accepted in the physics literature that the moments of the CDRP correspond to the solution of the attractive delta Bose gas. We will state this correspondence below, but note that we do not know of a rigorous mathematical treatment of this fact. One approach to such a rigorous proof would be to smooth out the white noise in space and then observe that the moments solve the Lieb-Liniger many body problem with smoothed delta potential. It then suffices to prove that the smoothed moments converge to the moments without smoothing and likewise for the solutions to the Lieb-Liniger many body problem. In the case $N=2$ this is treated in \cite{Alb}, Section I.3.2.

Fix $N\geq 1$ and fix $x=(x_1,\ldots, x_N)\in \R^N$. From the above discussion we expect that
\begin{equation*}
u(x;t) = \left\langle \prod_{i=1}^{N} \mathcal{Z}_1(t,x_i) \right\rangle
\end{equation*}
solves the attractive delta Bose gas with coupling constant $\kappa=1$ and delta initial data. This fact can be checked when $N=1,2$ from \cite{BC} as we did in Remark \ref{BCrem}.

From the moments of the CDRP one might hope to recover the polymer partition function's distribution function. However, the moments grow far too fast to uniquely characterize the distribution hence no such approach could be mathematically rigorous. The same issue occurs at the level of the O'Connell-Yor polymer. Nevertheless, by summing strongly divergent series coming from moment formulas, \cite{CDR} and \cite{Dot} were able to recover the formulas of \cite{ACQ,SaSp}. In some sense, $q$-TASEP provides a regularization of these models at which level the analogous series are convergent and the moment problems are well-posed.
\end{remark}

\begin{proposition}\label{ADBGprop}
Fix $N\geq 1$. The solution to the delta Bose gas with coupling constant $\kappa\in \R$ and delta initial data can be written as
\begin{equation}\label{ADBGeqn}
u(x;t) = \frac{1}{(2\pi \iota)^N} \int\cdots\int \prod_{1\leq A<B\leq N} \frac{z_A-z_B}{z_A-z_B-\kappa}e^{\frac{t}{2}\sum_{j=1}^{N} z_j^2 + \sum_{j=1}^{N}x_j z_j}\prod_{j=1}^{N} dz_j,
\end{equation}
where the $z_j$-contour is along $\alpha_j+\iota \R$ for any $\alpha_1>\alpha_2+\kappa>\alpha_3+2\kappa>\cdots >\alpha_N+(N-1)\kappa$.
\end{proposition}

\begin{proof}
We will first check the PDE for $x=(x_1<x_2<\cdots<x_N)$ (i.e., $x\in \Weyl{N}$) and $t>0$, and then the boundary condition and finally the initial condition. For $x\in \Weyl{N}$ we may compute $\partial_t u(x;t)$ by bring the differentiation inside the $dz$ integrals (as is justified by the ample Gaussian decay). This has the effect of introducing a new multiplicative factor of $\tfrac{1}{2}\sum_{j=1}^{N} z_j^2$ in the integrand. Likewise we may compute $\tfrac{1}{2}\Delta u(x;t)$ by differentiating inside the integral which also leads to a factor of $\tfrac{1}{2}\sum_{j=1}^{N} z_j^2$ in the integrand. Since the two resulting expressions coincide, the PDE is satisfied.

In order to check the boundary condition observe that if we apply $(\partial_{x_{i}}-\partial_{x_{i+1}}-\kappa)$ to $u(x;t)$ this introduces a factor of $(z_i-z_{i+1}-\kappa)$ in the integrand and the remaining argument is as in the proof of Lemma~\ref{switchcontours} above.
%This factor cancels a single term in the denominator of the product over $A<B$. Now set $x_i=x_{i+1}$ and call the resulting integrand $I$. Since the pole at $z_i=z_{i+1}+1$ is no longer present, we can deform the $z_i$ and $z_{i+1}$ contours so as to switch places (this deformation is easily justified by the Gaussian decay of the integrand). Finally, we may switch the labels of $z_i$ and $z_{i+1}$. This whole operation has rather limited effect on the integrand: the exponential terms are unchanged (since $x_{i}=x_{i+1}$) and in the Vandermond factor the only term which changes is $(z_{i}-z_{i+1})$ which becomes $(z_{i+1}-z_{i})$. Therefore, only the sign changes and the integrand becomes $-I$. On the other hand, the contours are exactly as they were originally, and therefore we find that $(\partial_{x_{i}}-\partial_{x_{i+1}}-\kappa)u(x;t)$, at $x_{i+1}=x_{i}+0$ is the integral of $I$ and also the integral of $-I$, and therefore 0 as desired.

The initial condition requires a little more work to prove. Fix a function $f\in L^{2}(\Weyl{N})\cap~C_b(\overline{\Weyl{N}})$, (and set $f(x)=f(\sigma x)$ for all permutations $\sigma\in S_N$) and assume $t>0$. Compute
\begin{equation*}
\int_{\R^N} u(x;t)f(x)dx =  \frac{1}{(2\pi \iota)^N} \int\cdots\int \prod_{1\le A<B\le N} \frac{z_A-z_B}{z_A-z_B-\kappa}e^{\frac{t}{2}\sum_{j=1}^{N} z_j^2}\hat{f}(-\iota z_1,\ldots,-\iota z_N)\prod_{j=1}^{N} dz_j,
\end{equation*}
where $\hat{f}(z_1,\ldots,z_N)$ is the Fourier transform
\begin{equation*}
\int_{\R^N} e^{\sum_{i=1}^{N}\iota x_i z_i}f(x_1,\ldots, x_N)\prod_{i=1}^{N} dx_i.
\end{equation*}

Observe that
\begin{equation}\label{prodZAB}
\prod_{1\le A<B\le N} \frac{z_A-z_B}{z_A-z_B-\kappa} = \prod_{1\le A<B\le N} \left(1 + \frac{\kappa}{z_A-z_B-\kappa}\right).
\end{equation}
This product can be expanded, and there is a single term which has no denominator and is just $1$. That term leads to
\begin{equation*}
\frac{1}{(2\pi \iota)^N} \int\cdots\int e^{\frac{t}{2}\sum_{j=1}^{N} z_j^2}\hat{f}(-\iota z_1,\ldots,-\iota z_N)\prod_{j=1}^{N} dz_j.
\end{equation*}
As there are no poles in $z_j$ we can deform all contours back to $\iota \R$ and then take $t=0$ which results in
\begin{equation*}
\frac{1}{(2\pi \iota)^N} \int\cdots\int \hat{f}(-\iota z_1,\ldots,-\iota z_N)\prod_{j=1}^{N} dz_j.
\end{equation*}
This is just the integral of the Fourier transform and therefore equals $f(0)$. Since we are trying to show that $\int_{\R^N} f(x) u(x;t) dx \to f(0)$ as $t\to 0$, it now suffices to prove that all of the other terms in the expansion of the product (\ref{prodZAB}) have contribution which goes to zero as $t\to 0$. We will do this for $u(x;t)$ directly (rather than its Fourier transform).

Besides the term $1$ dealt with above, the expansion of (\ref{prodZAB}) involves terms
\begin{equation*}
u^{K}(x;t) = \frac{1}{(2\pi \iota)^N} \int\cdots\int \prod_{(A,B)\in K} \frac{\kappa}{z_A-z_B-\kappa}e^{\frac{t}{2}\sum_{j=1}^{N} z_j^2 + \sum_{j=1}^{N}x_j z_j}\prod_{j=1}^{N} dz_j,
\end{equation*}
where $K$ is a non-empty subset of $\{(A,B): 1\leq A<B\leq N\}$ and $k=|K|$. Let $v=x/||x||$ be the unit vector in the direction of $x$. Deform the $z_j$-contours to lie along $-t^{-1/2}v_j + \iota\R$ (due to the ordering of coordinates of $x$ and the Gaussian decay near infinity, this deformation can be made as long as $t$ is small enough so that no poles are crossed in deforming). Change variables so that $w_j=t^{1/2}z_j$. Thus
\begin{equation*}
u^{K}(x;t) = \kappa^k t^{k/2} t^{-N/2} \frac{1}{(2\pi \iota)^N} \int\cdots\int \prod_{(A,B)\in K} \frac{1}{w_A-w_B-\kappa t^{1/2}}e^{\frac{1}{2}\sum_{j=1}^{N} w_j^2} e^{t^{-1/2}\sum_{j=1}^{N}x_j w_j}\prod_{j=1}^{N} dw_j.
\end{equation*}
Due to the choice of $w_j$-contours, $\Real(x\cdot w) = -||x||$ and so by taking absolute values we get
\begin{equation*}
\left|u^{K}(x;t)\right| \leq |\kappa|^k t^{k/2} t^{-N/2} e^{-t^{-1/2}||x||} \frac{1}{(2\pi \iota)^N} \int\cdots\int \prod_{(A,B)\in K} \left|\frac{1}{w_A-w_B-\kappa t^{1/2}}e^{\frac{1}{2}\sum_{j=1}^{N} w_j^2}\right| \prod_{j=1}^{N} dw_j.
\end{equation*}
Since the integral is now clearly convergent and independent of $x$ we get
\begin{equation}\label{ubound}
\left|u^{K}(x;t)\right| \leq  C t^{k/2} t^{-N/2} e^{-t^{-1/2}||x||}
\end{equation}
for some constant $C>0$.

For any function $f\in L^{2}(\Weyl{N})\cap C_b(\overline{\Weyl{N}})$ we may compute the integral against $u^{K}$. The conditions on $f$ imply that  $|f(x)|<M$ for some constant $M$ as $x$
varies in $\overline{\Weyl{N}}$. By the triangle inequality
\begin{equation*}
\left|\int_{\Weyl{N}} u^{K}(x;t) f(x)dx\right| \leq \int_{\Weyl{N}} \left|u^{K}(x;t)\right| M dx \leq C' t^{k/2}.
\end{equation*}
The second inequality follows by substituting (\ref{ubound}) and performing the change of variables $y_i=t^{-1/2} x_i$ in order to bounding the resulting integral (this change of variables results in the cancelation of the $t^{-N/2}$ term). Observe that since $k=|K|\geq 1$, as $t\to 0$ these terms disappear for all non-empty $K$. Since these terms account for the other terms in the expansion of the right-hand side of (\ref{prodZAB}) this shows that only the delta function coming from the empty $K$ plays a role as $t\to 0$. This completes the proof of the initial conditions.
\end{proof}

\begin{remark}
Inserting a symmetric factor $F(z_1,\ldots,z_N)$ into the integrand of the right-hand side of (\ref{ADBGeqn}) preserves the evolution equation and boundary condition, but leads to different initial conditions.
\end{remark}
\begin{remark}
One should note that the above proposition deals with both the $\kappa>0$ (attractive) and $\kappa<0$ (repulsive) delta Bose gas. Looking in the arguments of Section 4 of \cite{HeckOp} (where they were proving the Plancherel theory for the delta Bose gas) it is possible to extract the formula of the above proposition. This type of formula reveals a symmetry between the two cases which is not present in the Bethe Ansatz eigenfunctions.
\end{remark}

\index{many body system!Bethe ansatz}
An alternative and much earlier taken approach to solving the delta Bose gas is by demonstrating a complete basis of eigenfunctions (and normalizations) which diagonalize the Hamiltonian and respect the boundary condition. The eigenfunctions were written down in 1963 for the repulsive delta interaction by Lieb and Liniger \cite{LL} by Bethe ansatz. Completeness was proved by Dorlas \cite{Dorlas} on $[0,1]$ and by Heckman and Opdam \cite{HeckOp} and then recently by Prolhac and Spohn (using formulas of Tracy and Widom \cite{TWBose}) on $\R$ (as we are considering presently). For the attractive case, McGuire \cite{McGuire} wrote the eigenfunctions in terms of {\it string states} in 1964. As opposed to the repulsive case, the attractive case eigenfunctions are much more involved and are not limited to bound state eigenfunctions (hence a lack of symmetry with respect to the eigenfunctions). The norms of these states were not derived until 2007 in \cite{CalCaux} using ideas of algebraic Bethe ansatz. Dotsenko \cite{Dot} later worked these norms out very explicitly through combinatorial means. Completeness in the attractive case was shown by Oxford \cite{Oxford}, and then by Heckman and Opdam \cite{HeckOp}, and recently by Prolhac and Spohn \cite{ProSpoComp}.

The work \cite{LL,TWBose,Dot,CDR,ProSpoComp} provide formulas for the propagators (i.e., transition probabilities) for the delta Bose gas with general initial data. These formulas involve either summations over eigenstates or over the permutation group. In the repulsive case it is fairly easy to see how the formula of Proposition~\ref{ADBGprop} is recovered from these formulas.

\begin{remark}
The reason why the symmetry which is apparent in Proposition~\ref{ADBGprop} is lost in the eigenfunction expansion is due to the constraint on the contours. In the repulsive case $\kappa<0$ and the contours are given by having the $z_j$-contour along $\alpha_j+\iota \R$ for any $\alpha_1>\alpha_2+\kappa>\alpha_3+2\kappa>\cdots >\alpha_N+(N-1)\kappa$. The contours, therefore, can be taken to be all the same. It is a straight-forward calculation to turn the Bethe ansatz eigenfunction expansion into the contour integral formula we provide. The attractive case leads to contours which are not the same. In making the contours coincide we encounter a sizable number of poles which introduces many new terms which coincide with the fact that there are many other eigenfunctions coming from the Bethe ansatz in this case.
\end{remark}

For the attractive case, the following proposition (which is a degeneration of Proposition~\ref{mukprop} and can be proved similarly) provides a way to turn the moment formula of Proposition~\ref{ADBGprop} into the formula given explicitly in Dotsenko's work \cite{Dot}.

\begin{proposition}
For an entire function $f(z)$ and $k\geq 1$, setting
\begin{equation*}
\mu_k= \frac{1}{(2\pi \iota)^k} \int \cdots \int \prod_{1\leq A<B\leq k} \frac{z_A-z_B}{z_A-z_B-1} \frac{f(z_1)\cdots f(z_k)}{z_1\cdots z_k} dz_1\cdots dz_k,
\end{equation*}
we have
\begin{eqnarray*}
\mu_k&=& k! \sum_{\substack{\lambda\vdash k\\ \lambda=1^{m_1}2^{m_{2}}\cdots}} \frac{1}{m_1!m_2!\cdots} \\
 && \frac{1}{(2\pi \iota)^{\ell(\lambda)}} \int \cdots \int \det\left[\frac{1}{\lambda_i+w_i-w_j}\right]_{i,j=1}^{\ell(\lambda)} \prod_{j=1}^{\ell(\lambda)}  f(w_j)f(w_j+1)\cdots f(w_j+\lambda_j-1) dw_j,
\end{eqnarray*}
where the $z_j$-contour is along $\alpha_j+\iota \R$ for any $\alpha_1>\alpha_2+1>\alpha_3+2>\cdots >\alpha_k+(k-1)$
\end{proposition}

\begin{remark}\label{endansatzcomment}
The contour integral formulas used to solve the quantum many body systems in Proposition~\ref{ADBGprop} and Proposition~\ref{discDBG} should be understood as coming about from a more general contour integral ansatz. The correspondence is as follows. Consider a many body system which evolves in a Weyl chamber according to a certain PDE and has delta function initial condition. Assume it also has a boundary condition on the boundary of the Weyl chamber involving consecutively labeled particles. Then an ansatz for the solution is an $N$-dimensional contour integral in $z_1,\ldots, z_N$ which involves three terms. The first term is $\exp(z\cdot x)$; the second term is the exponential of the Fourier transform of the PDE generator applied to $z$; the third term is a product for $1\leq A<B\leq N$ of $(z_A-z_B)$ divided by a term which is chosen to cancel for $i$, $i+1$ when the corresponding boundary condition is checked.
\end{remark}

\clearpage
 \addcontentsline{toc}{chapter}{Index}
\printnotation
\printindex

\end{document}